\documentclass[11pt,reqno]{amsart}

\usepackage[a4paper,top=3.2cm,bottom=3.2cm,left=2.7cm,right=2.7cm]{geometry}

\usepackage{amssymb}
\usepackage{mathtools}
\usepackage{bm}

\usepackage[final,protrusion=true,expansion=true]{microtype}

\usepackage{graphicx}
\usepackage[all]{xy}
\usepackage{tikz}
\usetikzlibrary{arrows.meta, positioning, calc}

\usepackage[dvipsnames]{xcolor}

\usepackage[colorlinks=true,
            citecolor=red,
            linkcolor=blue,
            urlcolor=red,
            bookmarksnumbered=true,
            bookmarksopen=true]{hyperref}

\numberwithin{equation}{section}

\newcommand{\bb}{\bm{b}}
\newcommand{\Mm}{{\bf{M}}}
\newcommand{\Nn}{{\bf{N}}}

\newcommand{\Pp}{{\bf{P}}}

\newcommand{\Dd}{{\mathfrak{D}}}

\newcommand{\exc}{\mathrm{exc}}

\newcommand{\Cc}{\mathfrak{C}}

\newcommand{\Qq}{\mathbb{Q}}

\newcommand{\Rr}{\mathbb{R}}

\newcommand{\vol}{\operatorname{vol}}
\newcommand{\nvol}{\widehat{\operatorname{vol}}}

\newcommand{\Center}{\operatorname{center}}

\newcommand{\Exc}{\operatorname{Exc}}

\newcommand{\Bir}{\operatorname{Bir}}

\newcommand{\Aut}{\operatorname{Aut}}

\newcommand{\gilct}{\operatorname{gilct}}

\newcommand{\rk}{\operatorname{rank}}

\newcommand{\ninv}{\operatorname{ninv}}
\newcommand{\inv}{\operatorname{inv}}

\newcommand{\aif}{{\operatorname{aif}}}

\newcommand{\tmld}{{\operatorname{tmld}}}

\newcommand{\Weil}{\operatorname{Weil}}
\newcommand{\lct}{\operatorname{lct}}
\newcommand{\ilct}{\operatorname{ilct}}
\newcommand{\eb}{\operatorname{eb}}
\newcommand{\aux}{\operatorname{aux}}
\newcommand{\iglct}{\operatorname{iglct}}

\newcommand{\pet}{\operatorname{pet}}
\newcommand{\gt}{\operatorname{big}}
\newcommand{\Supp}{\operatorname{Supp}}

\newcommand{\Diff}{\operatorname{Diff}}

\newcommand{\mult}{\operatorname{mult}}
\newcommand{\Rct}{\operatorname{Rct}}

\newcommand{\Aa}{{\mathfrak{A}}}

\newcommand{\Bb}{{\mathfrak{B}}}
\newcommand{\Ff}{\mathcal{F}}

\newcommand{\Ii}{\Gamma}

\newcommand{\bNef}{\mathrm{bNef}}
\newcommand{\Sing}{\mathrm{Sing}}

\newtheorem{thm}{Theorem}[section]
\newtheorem{conj}[thm]{Conjecture}
\newtheorem{cor}[thm]{Corollary}
\newtheorem{lem}[thm]{Lemma}
\newtheorem{prop}[thm]{Proposition}
\newtheorem{prodef}[thm]{Proposition-Definition}
\newtheorem{claim}[thm]{Claim}

\newtheorem{alphthm}{Theorem}

\newtheorem{alphconj}[alphthm]{Conjecture}

\theoremstyle{definition}
\newtheorem{defn}[thm]{Definition}
\newtheorem{ques}[thm]{Question}
\newtheorem{rem}[thm]{Remark}
\newtheorem{deflem}[thm]{Definition-Lemma}
\newtheorem{setup}[thm]{Set-up}
\newtheorem{ex}[thm]{Example}
\newtheorem{nota}[thm]{Notation}

\begin{document}

\title{Birational boundedness of stable families}
\author{Paolo Cascini, Jihao Liu, Calum Spicer, and Roberto Svaldi}

\date{\today}

\begin{abstract}
We prove that normal projective stable families of maximal variation, of fixed dimension, and with bounded adjoint volume are birationally bounded. This is a consequence of a substantially stronger statement, formulated a priori independently of stable families: algebraically integrable foliations of fixed dimension and bounded adjoint volume are log birationally bounded. In this way, the birational geometry of foliations provides a systematic framework for approaching classical boundedness problems for fibrations.

A key input is our proof of M\textsuperscript{c}Kernan's ACC conjecture for interpolated log canonical thresholds of algebraically integrable foliations. This may be viewed as the foliated analogue of Shokurov's ACC conjecture for log canonical thresholds, proved in the classical setting by Hacon--M\textsuperscript{c}Kernan--Xu.

As applications, we establish two boundedness criteria for Fano algebraically integrable adjoint foliated structures: Birkar's criterion for exceptional Fanos, and Jiang's criterion for Fanos for which both Tian's $\alpha$-invariant and the anti-canonical volume are bounded away from zero.

We also obtain several results on the birational geometry of algebraically integrable adjoint foliated structures, including lower bounds for adjoint volumes, boundedness of automorphism groups, and ACC theorems for pseudo-effective thresholds, $\mathbb{R}$-complementary thresholds, and the Fano spectrum.
\end{abstract}
\address{Department of Mathematics, Imperial College London, 180 Queen’s
Gate, London SW7 2AZ, UK}
\email{p.cascini@imperial.ac.uk}

\address{Department of Mathematics, Peking University, No. 5 Yiheyuan Road, Haidian District, Beijing 100871, China}
\address{Beijing International Center for Mathematical Research, Peking University, No. 5 Yiheyuan Road, Haidian District, Beijing 100871, China}
\email{liujihao@math.pku.edu.cn}

\address{Department of Mathematics, King’s College London, Strand, London WC2R 2LS, UK}
\email{calum.spicer@kcl.ac.uk}

\address{Dipartimento di Matematica ``F. Enriques'', Universit\`a degli Studi di Milano, Via Saldini 50, 20133 Milano (MI), Italy}
\email{roberto.svaldi@unimi.it}

\maketitle

\tableofcontents

\section{Introduction}\label{sec: introduction}

We work over the field of complex numbers $\mathbb C$.

\subsection*{Background}
Birational geometry seeks to classify projective varieties up to birational equivalence, and the Minimal Model Program provides a systematic framework for this study. 
A recurring theme is that many classification problems become accessible once one can establish boundedness results and moduli theories in terms of natural numerical invariants. 
From this viewpoint, stable varieties (varieties of general type), Fano varieties, Calabi--Yau varieties, and fibered varieties form four fundamental building blocks. 
For stable and Fano varieties, boundedness (cf.~\cite{HMX18,Bir21a}) and the corresponding moduli theory (cf.~\cite{Kol23,Xu25a}) are by now rather well-developed, whereas the picture remains substantially more open for Calabi--Yau and fibered varieties despite significant recent progress (cf.~\cite{Bir23b,BFMT25,Eng+25,FHS25}). 

In this paper, we study boundedness questions for fibered varieties, with emphasis on stable families, and relate them to the birational geometry of algebraically integrable foliations. Stable families are the basic objects in KSBA moduli theory: they provide a framework for compactifying moduli of canonically polarized varieties and encode how the fibers vary in moduli.

Historically, boundedness questions for stable families already arose in the slope inequalities of Xiao and Cornalba--Harris \cite{Xia87,CH88}. Refinements and higher-dimensional analogues have been investigated in a number of recent works; see, for instance, \cite{CTV23,CPT25} and the references therein. 

From the perspective of moduli for stable varieties and stable pairs, stable families in the KSBA sense are the natural way to define moduli functors; the associated moduli spaces appear in the ``moduli of parametrizations'' constructions (cf.~\cite{AB19,AB21,Inc20,AB22,AB23}). The moduli space of stable families over a fixed base, e.g.\ a rational curve, is closely related to the moduli of stable quasimaps to moduli spaces (cf.~\cite{DLI24,HH25b,ISZ25,IZ25,Hat26}).
Finally, any stable family carries the algebraically integrable foliation induced by its fibers, so the birational geometry of stable families naturally interacts with boundedness and moduli problems for foliations, cf.~\cite{SS23,SSV25}.

\subsection{Boundedness of stable families}
Let $f\colon X \to T$ be a stable family. After a finite base change on $T$, one may factor $f$ through a stable family of maximal variation \cite[Corollary~6.20]{KP17}. More precisely, there exist a stable family $g\colon Y\to S$ of maximal variation and a morphism $\pi\colon T\to S$ (not necessarily finite) such that $X$ is the pullback of $Y$ along $\pi$, i.e.\ we have a cartesian diagram
\[
\xymatrix{
Y \ar[d]^g_{\text{max.\ variation}} & & X=Y\times_S T \ar[ll]_{\psi} \ar[d]^f \\
S & & T \ar[ll]_{\pi}^{\text{not necessarily finite}}
}
\]
where $\psi$ is the projection.

In addition, if we assume that $Y$ is normal, then $K_{Y/S}$ is big and nef \cite[Proposition~2.15]{PX17}; moreover, under the above identification, $\psi^\ast K_{Y/S}=K_{X/T}$.
This suggests that boundedness questions for stable families should be approached first in the maximal-variation case, and that one might try to mimic the boundedness and moduli theory for stable varieties by working with the relative canonical divisor $K_{X/T}$ in place of $K_X$.
However, several basic obstructions show that this naive strategy cannot work in general:
\begin{enumerate}
    \item (Failure of abundance) $K_{X/T}$ can be big and nef but not semi-ample, cf.~\cite[Example~5.5]{ACSS21} and \cite[Theorem~IV.2.2]{McQ08}. 
    Hence, canonical models may not exist within the category of projective varieties, in contrast to what happens for the canonical class, cf.~\cite[Corollary~1.1.1]{BCHM10}. 
    \item (Failure of effective birationality) Effective birationality of big relative pluricanonical systems 
    $\vert mK_{X/T}\vert$ 
    fails already when $\dim X=2$ \cite[Theorem~1.3]{Lü25}, unlike for $\vert mK_X\vert$, cf.~\cite[Theorem~1.3(3)]{HMX14} (see also \cite{HM06,Tak06,Tsu07}).
    \item (Failure of boundedness) 
    Even assuming that
    $K_{X/T}$ 
    is ample and fixing $\vol(K_{X/T})$, 
    fibrations $X \to T$ with total space $X$ of dimension $2$ satisfying these assumptions can be unbounded, see \cite[Example~4.1]{Pas24}; 
    this is the opposite of what happens for stable varieties of fixed dimension $d$ and canonical volume, cf.~\cite[Theorem~1.1]{HMX18}.
\end{enumerate}
These examples show that we should not expect to achieve a satisfactory boundedness statement for stable families by simply substituting $K_X$ with $K_{X/T}$.

A natural question is whether one can nevertheless formulate and prove a birational boundedness result for stable families as fibered varieties using only the canonical data of the family, without fixing any auxiliary polarization on the fibers or on the base. 
Such a result would be particularly well-suited for moduli-theoretic applications.

To bypass these issues and obtain such an intrinsic result, we interpolate between the relative and absolute canonical classes: namely, for $t\in(0,1)$ we consider the $\mathbb R$-divisor
\[K_t:=tK_{X/T}+(1-t)K_X.\]
These $\mathbb R$-divisors are defined intrinsically: they depend only on the original data $K_{X/T}$ and $K_X$ and require no auxiliary polarization. Moreover, when $T$ is a point we have $K_t=K_X$ for all $t$, so we recover the classical setting. Thus the class $K_t$ provides a workable substitute for $K_{X/T}$ in effective birationality and boundedness arguments. Our first main result is the following birational boundedness theorem for stable families. 

\begin{alphthm}[Birational boundedness of stable families]\label{thm: stable family birational boundedness}
    Let $d$ be a positive integer. Then there exists a number $t=t(d)\in(0,1)$ depending only on $d$ with the following property.

    Assume that $f\colon X\to T$ is a stable family of maximal variation such that $X$ is normal and projective, and $\dim X=d$. Then:
    \begin{enumerate}
        \item $K_{t}:=tK_{X/T}+(1-t)K_X$ is big.
        \item Assume that $\vol(K_t)\leq C$ for some constant $C$. Then $X$ is birationally bounded.
    \end{enumerate}
\end{alphthm}
See \cite[Definition~6.16]{KP17} for a precise definition of stable families of maximal variation and Definition~\ref{defn: bounded pairs} for a precise definition of birational boundedness. Many boundedness and moduli results for fibered varieties are formulated in the setting of polarized fibrations, where one fixes, as part of the input, ample classes on the fibers and/or on the base; see, for instance, \cite{Bir22,Bir23a,Jia23,Bej+24,Bir24,BDCS24,Fil24,HH25a,HH25b,ISZ25,IZ25,Jia25a,Jia25b,JJZ25,Zhu25}. In line with the question above, we seek an intrinsic birational boundedness statement for stable families as fibered varieties that depends only on canonical data and involves no auxiliary choice of polarization.

Theorem~\ref{thm: stable family birational boundedness} gives an affirmative answer: it is stated solely in terms of $K_{X/T}$ and $K_X$ via the interpolated divisor $K_t$, and it does not require fixing any additional polarization. Once $d$ (and hence $t=t(d)$) is fixed, the birational boundedness conclusion in~(2) depends only on the constant $C$ in the inequality $\vol(K_t)\leq C$. When $T$ is a point, Theorem~\ref{thm: stable family birational boundedness} recovers the birational boundedness theorem of Hacon--M\textsuperscript{c}Kernan--Xu (\cite[Theorem~3.1]{HMX13}, \cite[Theorem~1.3]{HMX14}). Moreover, our method not only shows that stable families are birationally bounded as fibered varieties, but also controls the fibration structure $f\colon X\to T$ in the foliated sense; see Theorem~\ref{thm: birationally bounded main} below.

\subsection{Boundedness of algebraically integrable foliations of general type}
The proof of Theorem~\ref{thm: stable family birational boundedness} is not confined to morphisms $f\colon X\to T$. Guided by the failures above, we pass to the broader category of algebraically integrable foliations and make systematic use of the birational geometry of foliations and adjoint foliated structures developed in recent years (cf.~\cite{PS19,ACSS21,CHLX23,SS23,Cas+24,LW24,Cas+25a,Cas+25b,CS25c,LX25,LMX25}).

The key point is that insisting on keeping the fibration structure forces one to confront the pathologies above. In particular, finite generation may fail, and consequently one loses access to effective birationality and boundedness statements formulated purely in terms of the relative canonical divisor $K_{X/T}$.

We therefore pass to the induced algebraically integrable foliation and work in the foliated setting, which is better behaved under birational modifications. This trade-off is justified: the foliated framework allows us to circumvent the three failures above. For instance, canonical models exist for algebraically integrable adjoint foliated structures of general type \cite{Cas+25a}, and effective birationality and boundedness are known for adjoint foliated surfaces (cf.~\cite{SS23,SSV25}).

Within this foliated framework, the birational boundedness statement in Theorem~\ref{thm: stable family birational boundedness} is a corollary of the following stronger theorem.

\begin{alphthm}[Birational boundedness of algebraically integrable foliations]\label{thm: birationally bounded main}
Let $d$ be a positive integer, 
$\Ii\subset (0,1)$ 
a set satisfying the descending chain condition (DCC), and
$C$
a positive real number.
Then the following set of projective algebraically integrable foliated pairs
\begin{align*}
    \left \{
(X,\Ff)
\ \middle \vert \
\dim X=d, 
\exists\ t\in\Ii\
\text{s.t. $(X,\Ff,t)$ is lc and } 
0<\vol\!\left(K_X+\frac{t}{1-t}K_{\Ff}\right)\le C
   \right \}
\end{align*}
forms a log birationally bounded family.
\end{alphthm}

We refer the reader to 
Definition~\ref{defn: sing of afs} 
for the definition of singularities for \emph{adjoint foliated structures} of the form $(X,\Ff,t)$, 
and 
Definition~\ref{defn: lbb} for the definition of log birational boundedness for foliations. 
When $t$ is fixed (as in Theorem~\ref{thm: stable family birational boundedness}~(1)), Theorem~\ref{thm: birationally bounded main} immediately implies Theorem~\ref{thm: stable family birational boundedness}~(2) by taking $\Ff$ to be the foliation induced by $f\colon X\to T$, since $K_{X/T}=K_{\Ff}$ when $f$ is a stable family. More generally, any fibration induces an algebraically integrable foliation. Thus Theorem~\ref{thm: birationally bounded main} applies in particular to foliations arising from fibrations, while being formulated for a much broader class of algebraically integrable foliations. Even for foliations induced by fibrations, Theorem~\ref{thm: birationally bounded main} is stronger than Theorem~\ref{thm: stable family birational boundedness}~(2): if $X$ is log canonical, requiring $(X,\Ff,t)$ to be log canonical is strictly weaker than requiring that $f$ is (locally) stable.

We refer the reader to Theorems~\ref{thm: birational boundedness} and~\ref{thm: precise birational boundedness} for more general versions of Theorem~\ref{thm: birationally bounded main}. In particular, a boundary-polarized version of Theorem~\ref{thm: birationally bounded main} also holds, which allows us to study moduli of Fano or polarized Calabi--Yau adjoint foliated structures and the corresponding fibration structures. See \cite{KX20,Asc+23,BL24} and references therein for the corresponding classical theory. 

Finally, although Theorem~\ref{thm: birationally bounded main} is formulated for algebraically integrable foliations, the strategy is not conceptually restricted to algebraic integrability.
Extending it to general foliations would require a satisfactory resolution and adjunction theory in arbitrary dimension.
We refer to \cite{SS23,Pas24,SSV25} for progress in settings where such tools have been established, particularly when rank $=1$, and to Subsection~\ref{subsec: variation} for a more detailed discussion.

\subsection{M\textsuperscript{c}Kernan's ACC conjecture for interpolated lc thresholds}

We now explain the uniform parameter $t$ in Theorem~\ref{thm: stable family birational boundedness}~(1). From the foliated viewpoint, it is a special case of M\textsuperscript{c}Kernan's ascending chain condition (ACC) conjecture for interpolated log canonical thresholds. This conjecture is the natural foliated analogue of Shokurov's ACC conjecture for log canonical thresholds \cite{Sho92}, proved in the classical setting by Hacon--M\textsuperscript{c}Kernan--Xu \cite[Theorem~1.1]{HMX14}.

\begin{alphconj}[M\textsuperscript{c}Kernan's ACC conjecture for interpolated lc thresholds, {\cite[28:22]{McK22}}]\label{conj: acc interpolated}
Let $d$ be a positive integer. Then there exists an ACC set $\Ii\subset [0,1]$ depending only on $d$ such that for any $\mathbb Q$-Gorenstein foliation $\Ff$ on a $\mathbb Q$-Gorenstein variety $X$, the interpolated log canonical (lc) threshold
\[
\lct(X;\Ff):=\sup\{t\geq 0\mid (X,\Ff,t)\text{ is log canonical}\}
\]
belongs to $\Ii$.
\end{alphconj}
Here we say that $\Ff$ (resp. $X$) is $\mathbb Q$-Gorenstein if $K_{\Ff}$ (resp. $K_X$) is $\mathbb Q$-Cartier. In this paper, we prove M\textsuperscript{c}Kernan's conjecture for algebraically integrable foliations, which is the case relevant to our applications to fibered varieties since the foliation induced by a fibration is algebraically integrable. Theorem~\ref{thm: main ACC} and the accompanying global ACC theorem (Theorem~\ref{thm: global ACC main}) form the main technical engine of the paper. They yield an ACC statement for pseudo-effective thresholds (cf.~Theorem~\ref{thm: pet gap main}), produce the uniform parameter $t$ in Theorem~\ref{thm: stable family birational boundedness}~(1), and are used throughout to deduce effective birationality and, ultimately, birational boundedness.

\begin{alphthm}[ACC for interpolated lc thresholds]\label{thm: main ACC}
Let $d$ be a positive integer. Then the set of interpolated lc thresholds 
\begin{align*}
    \left \{
\lct(X;\Ff)
\ \middle \vert \ 
\dim X=d, 
\text{ $X$ and $\Ff$ are $\mathbb Q$-Gorenstein, $\Ff$ is algebraically integrable} 
    \right \}
\end{align*}
satisfies the ACC.
\end{alphthm}

Theorem~\ref{thm: main ACC} is also of independent interest for the birational study of fibrations. If $X$ is a $\mathbb Q$-factorial klt variety equipped with a fibration $f\colon X\to T$ and $\Ff$ is the induced foliation, then $\lct(X;\Ff)\in (0,1]$ is a new birational invariant measuring quantitatively how close the foliation is to being induced (up to quotient) by a locally stable family. For instance, when the base $T$ is smooth, one has $\lct(X;\Ff)=1$ whenever $\Ff$ satisfies Property $(*)$ in the sense of \cite[Definition~3.5]{ACSS21}, and Property $(*)$ is realized by quotients of locally stable families over a smooth base \cite[Proposition~3.5]{FS25}. 

It is worth noting that several known conditions on fibrations can be rephrased in terms of interpolated lc thresholds. For example, in the study of irrationality of degenerations of Fano varieties, \cite[Theorem~1.3(3)]{BQ24} proposes a technical condition on the log canonicity of $(X,tF)$, where $F$ is a reduced fiber of a fibration $f\colon X\to Z$ to a curve $Z$, with $t$ bounded away from zero. From the viewpoint of interpolated lc thresholds, this condition is simply the requirement that $\lct(X;\Ff)$ is bounded away from zero.

The interpolated lc threshold also has applications to algebraic geometry in positive characteristic. In a recent preprint, Hiroi shows that, for $1$-foliations over a field of characteristic $p>0$ (which correspond to algebraically integrable foliations in characteristic $0$), the interpolated lc threshold is closely related to the singularities of the quotient space $X/\Ff$: the adjoint foliated structure $(X,\Ff,(p-1)/p)$ is lc (resp.\ klt) if and only if $X/\Ff$ is lc (resp.\ klt) \cite[Theorem~3.2.5]{Hir26}.

We remark that M\textsuperscript{c}Kernan proved Conjecture~\ref{conj: acc interpolated} in dimension~$2$ in his talk at Shokurov's 70\,$+$\,2nd birthday conference in 2022 \cite{McK22}. The third and fourth authors showed in \cite[Lemma~2.19]{SS23} that if $X$ is a smooth surface and $(X,\Ff,t)$ is lc for some $t>\frac{5}{6}$, then $\Ff$ is lc. We refer the reader to Theorem~\ref{thm: acc lct general version} for a more general version of Theorem~\ref{thm: main ACC}, formulated for adjoint foliated structures with a boundary and a nef part.

\subsection{Boundedness of Fano algebraically integrable foliations}
Although our main theorems focus on stable families and algebraically integrable foliations of general type, the ACC technology we develop also has important applications in the Fano setting. In \cite{Cas+25a}, a BAB-type criterion for Fano algebraically integrable adjoint foliated structures was established as a foliated analogue of Birkar's BAB theorem \cite[Theorem~1.1]{Bir21a}. In the context of K-stability and K-moduli for Fano varieties, one often invokes a different boundedness criterion due to Jiang \cite[Theorem~1.6]{Jia20}, which controls boundedness in terms of Tian's $\alpha$-invariant together with volume; see \cite{Che20,LLX20} for refinements and alternative proofs. In this paper, we prove an analogue of Jiang's criterion for algebraically integrable adjoint foliated structures.

\begin{alphthm}[Jiang's boundedness criterion]\label{thm: jiang afs main}
Let $d$ be a positive integer and $\epsilon$ a positive real number. Then the set of projective varieties $X$ such that
\begin{enumerate}
    \item $\Aa:=(X,\Ff,t)$ is a Fano algebraically integrable adjoint foliated structure,
    \item $\dim X=d$,
    \item $\alpha(\Aa)^d\cdot\vol(-K_{\Aa})\geq\epsilon$, and
    \item $t\leq 1-\epsilon$
\end{enumerate}
forms a bounded family. 
Moreover, the set of $(X,\Ff)$ such that (1-4) hold and
\begin{enumerate}
    \item[(5)] $t\geq\epsilon$
\end{enumerate}
forms a bounded family.
\end{alphthm}

Here we say that $\Aa$ is Fano if $-K_{\Aa}:=-tK_{\Ff}-(1-t)K_X$ is ample, and $\alpha(\Aa)$ denotes Tian's $\alpha$-invariant of the adjoint foliated structure $\Aa$ (see Definition~\ref{defn: alpha}). We refer the reader to Theorem~\ref{thm: jiang principle general} for a more general version of Theorem~\ref{thm: jiang afs main}. For varieties, Jiang's criterion implies that K-semistable Fano varieties with anti-canonical volume bounded away from $0$ form a bounded family, since K-semistable Fano varieties of dimension $d$ satisfy $\alpha\geq\frac{1}{d+1}$ (cf.~\cite[Theorem~3.5]{FO18}). For foliations, a differential-geometric notion of ``K-stability'' has not yet been formulated; nevertheless, Theorem~\ref{thm: jiang afs main} suggests that an algebro-geometric stability condition should be captured by a $\delta$-invariant for foliations, analogous to the $\delta$-invariant \cite{FO18,BJ20} defined by Fujita--Odaka and Blum--Jonsson, and should be expressible in terms of adjoint foliated structures. Mascharak and Vassiliadis informed us of forthcoming work of theirs in this direction \cite{MV26}. Papazachariou \cite{Pap26} also informed us that, by K-stability methods, he obtains a boundedness result for Fano adjoint foliated structures in the spirit of Theorem~\ref{thm: jiang afs main}. We do not pursue these directions in this paper.

Another fundamental boundedness statement in the Fano setting is Birkar's criterion for exceptional Fanos, i.e.\ those with Tian's $\alpha$-invariant strictly greater than $1$ \cite[Theorem~1.11]{Bir19}. Note that exceptional Fano varieties are always K-stable \cite{Tia87,OS12}. We prove the following foliated analogue.

\begin{alphthm}[Birkar's boundedness criterion for exceptional Fanos]\label{thm: exceptional afs main}
Let $d$ be a positive integer and $\Ii\subset [0,1]$ a DCC set. Then the set of foliated pairs $(X,\Ff)$ such that
\begin{enumerate}
    \item $\Aa:=(X,\Ff,t)$ is a Fano algebraically integrable adjoint foliated structure,
    \item $\dim X=d$ and $t\in\Ii\backslash\{0\}$, and
    \item $\alpha(\Aa)>1$
\end{enumerate}
forms a bounded family.
\end{alphthm}

We refer the reader to Theorem~\ref{thm: exceptional bdd} for a more general version of Theorem~\ref{thm: exceptional afs main}.

\subsection*{Roadmap}
For the reader's convenience, Figure~\ref{fig:roadmap} summarizes the main logical dependencies among the results used in the proofs, with references to the relevant sections.

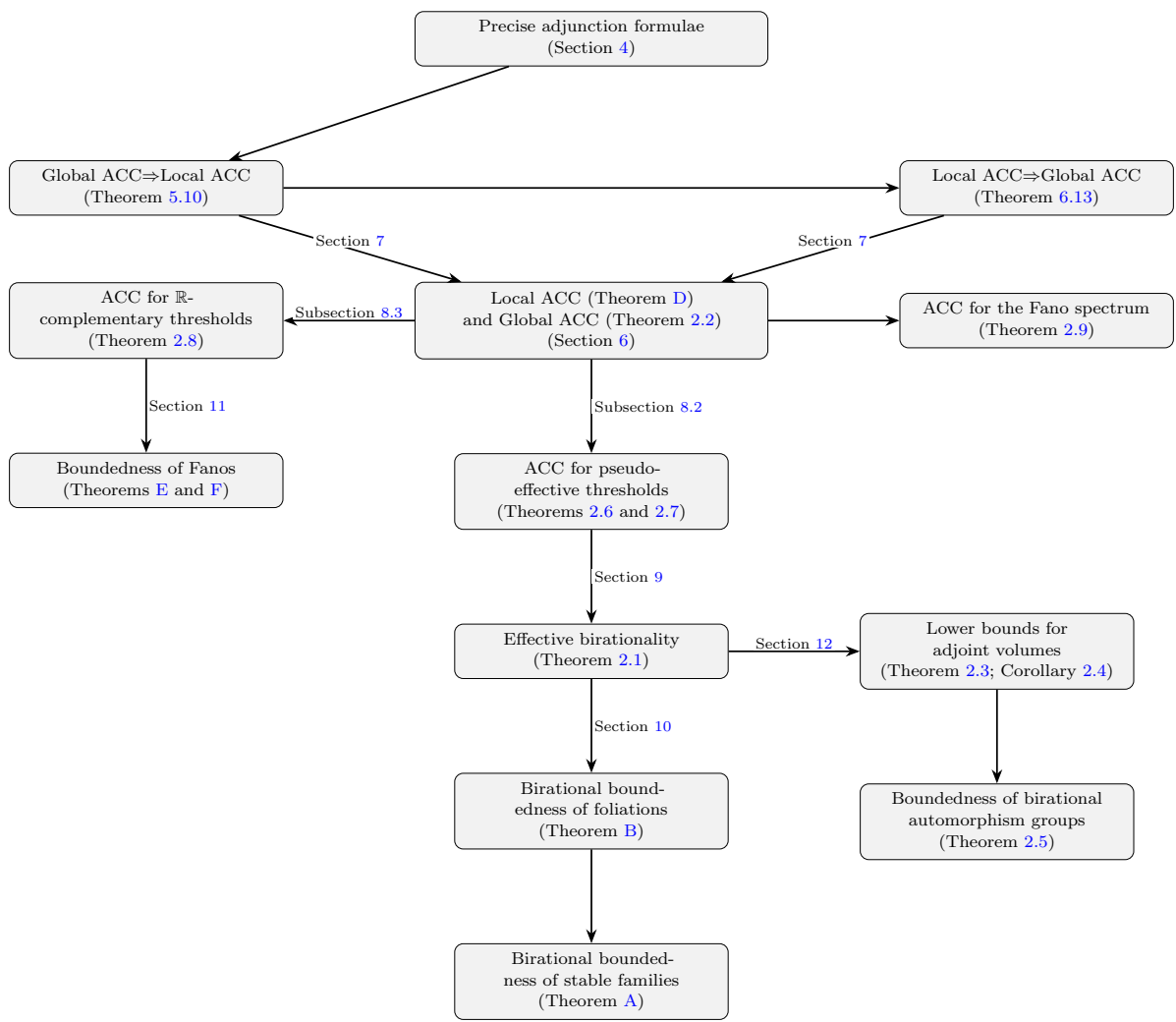
\begin{figure}[!htbp]
\centering
\resizebox{\textwidth}{!}{%
\begin{tikzpicture}[
  node distance=1.55cm and 2.15cm,
  box/.style={
    draw,
    rounded corners,
    align=center,
    font=\scriptsize,
    text width=4.2cm,
    inner sep=4pt,
    fill=gray!10
  },
  boxwide/.style={box, text width=5.5cm},
  boxsmall/.style={box, text width=3.35cm},
  arrow/.style={-Stealth, thick},
  lab/.style={font=\tiny, fill=white, inner sep=1pt}
]

\node[boxwide] (adj) {Precise adjunction formulae\\(Section~\ref{sec: precise adj})};

\node[box, below left=of adj] (g2l)
  {Global ACC$\Rightarrow$Local ACC\\
  (Theorem~\ref{thm: global to local})};

\node[box, below right=of adj] (l2g)
  {Local ACC$\Rightarrow$Global ACC\\
   (Theorem~\ref{thm: acc to global acc})};

\coordinate (mid23) at ($(g2l)!0.5!(l2g)$);
\node[boxwide, below=of mid23] (acc)
  {Local ACC (Theorem~\ref{thm: main ACC})\\
   and Global ACC (Theorem~\ref{thm: global ACC main})\\
   (Section~\ref{sec: local to global})};

\node[box, below=of acc] (pet)
  {ACC for pseudo-effective thresholds\\
   (Theorems~\ref{thm: pet main} and~\ref{thm: pet gap main})};

\node[box, below=of pet] (eb)
  {Effective birationality\\
   (Theorem~\ref{thm: afs main})};

\node[box, below=of eb] (bbfol)
  {Birational boundedness of foliations\\
   (Theorem~\ref{thm: birationally bounded main})};

\node[box, below=of bbfol] (bbfam)
  {Birational boundedness of stable families\\
   (Theorem~\ref{thm: stable family birational boundedness})};

\node[box, left=of acc] (rct)
  {ACC for $\mathbb{R}$-complementary thresholds\\
   (Theorem~\ref{thm: r-complementary main})};

\node[box, below=of rct] (fano)
  {Boundedness of Fanos\\
   (Theorems~\ref{thm: jiang afs main} and~\ref{thm: exceptional afs main})};

\node[box, right=of acc] (spec)
  {ACC for the Fano spectrum\\
   (Theorem~\ref{thm: fano spectrum main})};

\node[box, right=of eb] (vol)
  {Lower bounds for adjoint volumes\\
   (Theorem~\ref{thm: volume lower bound};
    Corollary~\ref{cor: low bound volume})};

\node[box, below=of vol] (aut)
  {Boundedness of birational automorphism groups\\
   (Theorem~\ref{thm: bir auto group afs main})};

\draw[arrow] (adj) -- (g2l);
\draw[arrow] (g2l) -- (l2g);

\draw[arrow] (g2l) -- node[midway, lab, above]{Section~\ref{sec: proof of acc and global acc}} (acc);
\draw[arrow] (l2g) -- node[midway, lab, above]{Section~\ref{sec: proof of acc and global acc}} (acc);

\draw[arrow] (acc) -- node[midway, lab, right]{Subsection~\ref{subsec: acc pet}} (pet);
\draw[arrow] (pet) -- node[midway, lab, right]{Section~\ref{sec: eb}} (eb);
\draw[arrow] (eb) -- node[midway, lab, right]{Section~\ref{sec: birational boundedness}} (bbfol);
\draw[arrow] (bbfol) -- (bbfam);

\draw[arrow] (eb) -- node[midway, lab, above]{Section~\ref{sec: volume lower bound}} (vol);
\draw[arrow] (vol) -- (aut);

\draw[arrow] (acc) -- node[midway, lab, above]{Subsection~\ref{subsec: acc rct}} (rct);
\draw[arrow] (rct) -- node[midway, lab, right]{Section~\ref{sec: application}} (fano);

\draw[arrow] (acc) -- (spec);

\end{tikzpicture}%
}
\caption{Roadmap of the main implications.}
\label{fig:roadmap}
\end{figure}

\subsection*{Acknowledgements} The authors would like to thank Hamid Abban, Caucher Birkar, Christopher D. Hacon, Masafumi Hattori, Zhengyu Hu, Giovanni Inchiostro, Junpeng Jiao, Roktim Mascharak, James M\textsuperscript{c}Kernan, Fanjun Meng, Wenhao Ou, Zsolt Patakfalvi, Jorge Vit\'orio Pereira, Luca Tasin, Stefania Vassiliadis, Lingyao Xie, Chenyang Xu, Zheng Xu, and Junyan Zhao for useful discussions. Key ideas of the paper originated when the authors attended the CIRM workshop \emph{Foliations, birational geometry and applications} during the period 3--7 February 2025, and the authors would like to thank the local organizers for their hospitality.

This work is supported by the National Key R\&D Program of China \#\allowbreak 2024YFA1014400. 
PC is partially funded by a Simons Collaboration Grant.
JL would like to thank Professor Gang Tian for constant support and encouragement. 
CS is partially funded by EPSRC. 
RS was partially supported by ``Programma per giovani ricercatori Rita Levi Montalcini'' of MUR and by PSR 2022 -- Linea 4 of the University of Milan. 
He is a member of the GNSAGA group of INDAM.

\section{Applications of the main results and sketch of the proofs}\label{sec: sketch of proof}

The main results that we have just presented have several additional consequences for the birational geometry of adjoint foliated structures. 
In this section, we collect those applications of the main theorems introduced above that play an important role in the paper, together with some results of independent interest. 
We will also sketch a proof of the ACC for interpolated lc thresholds (Theorem~\ref{thm: main ACC}). 
Moreover, we discuss how this result is used to deduce effective birationality (Theorem~\ref{thm: afs main}) and hence birational boundedness (Theorems~\ref{thm: stable family birational boundedness} and~\ref{thm: birationally bounded main})
for adjoint foliated structures.

\subsection*{Effective birationality} First, we prove the following effective birationality theorem, which can be viewed as an analogue of \cite[Theorem~1.3(3)]{HMX14} and is crucial for us to prove the birational boundedness theorem (Theorem~\ref{thm: birationally bounded main}).

\begin{thm}[Effective birationality]\label{thm: afs main}
Let $d$ be a positive integer and $\Ii\subset [0,1)$ a DCC set. Then there exist positive integers $m_0$ and $n_0$ depending only on $d$ and $\Ii$ with the following property. 

Assume that $(X,\Ff,t)$ is a projective lc algebraically integrable adjoint foliated structure of general type such that $\dim X=d$ and $t\in\Ii$. Then for any integers $m\geq m_0$ and $n\geq\min\left\{\left\lfloor\frac{mt}{1-t}\right\rfloor,mn_0\right\}$, the linear system
    $$|mK_X+nK_{\Ff}+L|$$
    defines a birational map for any pseudo-effective integral divisor $L$.
\end{thm}

We refer the reader to Theorem~\ref{thm: afs eb} for a more general version of Theorem~\ref{thm: afs main}. As we mentioned above, one should not expect an effective birationality result for pluricanonical systems of the form $|mK_{\Ff}|$, due to a famous example of Lü \cite[Theorem~1.3]{Lü25}.

\subsection*{Global ACC} We prove a global ACC statement for algebraically integrable adjoint foliated structures. This can be seen as an analogue of the global ACC in \cite[Theorem~1.5]{HMX14}, which plays an important role in the study of polarized Calabi--Yau varieties and Fano varieties.

\begin{thm}[Global ACC]\label{thm: global ACC main}
    Let $d$ be a positive integer. Then there exists an ACC set $\Ii=\Ii_{\gilct}\subset [0,1]$ depending only on $d$ with the following property. Assume that $(X,\Ff,t)$ is a projective lc algebraically integrable adjoint foliated structure of dimension $d$ such that
    $$tK_{\Ff}+(1-t)K_X\equiv 0.$$
    Then either $t\in\Ii$, or $K_X\equiv K_{\Ff}\equiv 0$.
\end{thm}

We refer the reader to Theorem~\ref{thm: global acc general} for a more general version of Theorem~\ref{thm: global ACC main}.

\subsection*{Lower bound on the volume} As an application of effective birationality (Theorem~\ref{thm: afs main}), we obtain a lower bound on the volume of algebraically integrable adjoint foliated structures. This can be seen as an analogue of \cite[Theorem~1.3(2)]{HMX14}.

\begin{thm}\label{thm: volume lower bound}
    Let $d$ be a positive integer and $\Ii\subset [0,1]$ a DCC set. Then there exists a positive real number $\epsilon$ depending only on $d$ and $\Ii$ with the following property. 
    
    Assume that $(X,\Ff,t)$ is a projective lc algebraically integrable adjoint foliated structure of general type such that $\dim X=d$ and $t\in\Ii$. Then
    $$\vol(tK_{\Ff}+(1-t)K_X)\geq (1-t)^d\epsilon.$$
\end{thm}

We refer the reader to Theorem~\ref{thm: low bound volume} for a more general version of Theorem~\ref{thm: volume lower bound}. We also remark that, when $d=2$ and $t$ is fixed and sufficiently close to $1$, Theorem~\ref{thm: volume lower bound} was proved in \cite[Theorem~1.5]{SS23} without assuming algebraic integrability.

Theorem~\ref{thm: volume lower bound} is related to a question of 
Pereira 
(cf.\ also Hacon--Langer \cite[Question~4]{HL21} and \cite[Section~3]{Cas21}) which asks whether $\vol(\Ff):=\vol(K_{\Ff})$ has a positive lower bound when $\Ff$ is (log) canonical, of general type, and of fixed dimension. This question has been studied in \cite{PS19,Che21,HL21,LT22,SS23,HJLL24,CPT25,Fan25,ST26}, but remains open in general even for algebraically integrable foliations in dimension $2$. We mention in particular the recent work of Spicer--Tasin \cite{ST26}, which establishes a precise adjunction-type formula $K_X+\Gamma\sim K_{\Ff}+D$ for log canonical rank-one foliations $\Ff$ on log homogeneous (in particular, toric) varieties, and uses it to derive several results on the birational geometry of such foliations, including effective lower bounds for $\vol(\Ff)$ in the toric setting. As a corollary of Theorem~\ref{thm: volume lower bound}, we obtain the following result on the lower bound of $\vol(\Ff)$.

\begin{cor}\label{cor: low bound volume}
Let $d$ be a positive integer. Then there exists a positive real number $\epsilon$ depending only on $d$ with the following property. 

Assume that $\Ff$
is an lc algebraically integrable foliation of general type on a normal projective variety $X$ of dimension $d$, such that $X$ is lc and $-K_X$ is pseudo-effective (for instance, when $X$ is Calabi--Yau or of Fano type). Then 
$$\vol(\Ff)\geq\epsilon.$$
\end{cor}

\subsection*{Boundedness of birational automorphism groups} We obtain the following result on the boundedness of birational automorphism groups for algebraically integrable foliations of general type. This can be viewed as an analogue of \cite[Theorem~1.1]{HMX13}. 

\begin{thm}\label{thm: bir auto group afs main}
    Let $d$ be a positive integer. Then there exist a positive integer $c$ and a positive real number $\lambda$, depending only on $d$, with the following property.

 Assume that $\Ff$ is an lc algebraically integrable adjoint foliated structure of general type on a normal projective variety $X$ of dimension $d$, such that $X$ is klt. Then
    $$\#\Bir(X,\Ff)\leq c\vol\left(K_X+\lambda K_{\Ff}\right).$$
\end{thm}

We remark that when $d=2$, $X$ is smooth, and $\Ff$ is canonical, Theorem~\ref{thm: bir auto group afs main} was proved in \cite[Theorem~1.6]{SS23} without assuming algebraic integrability. Note that by \cite[Theorem~F]{BPRT22}, the cardinality $\#\Bir(X,\Ff)$ of the group
$\Bir(X,\Ff)$
of birational transformations of $X$ that preserve $\Ff$ is always finite. 

\subsection*{Other ACC-type results} The ACC for lc thresholds and the global ACC have several variants. We obtain analogous statements in the algebraically integrable foliated setting. First, we have the ACC for pseudo-effective thresholds, a variant of \cite[Theorem~1.6]{DC16}.

\begin{thm}[ACC for pseudo-effective thresholds]\label{thm: pet main}
    Let $d$ be a positive integer. Then there exists an ACC set $\Ii\subset [0,1]$ depending only on $d$ such that, for any lc algebraically integrable foliation $\Ff$ on an lc variety $X$ of dimension $d$, if $K_{\Ff}$ is pseudo-effective, then the pseudo-effective threshold
    $$\pet(X;\Ff):=\inf\{t\geq 0\mid tK_{\Ff}+(1-t)K_X\text{ is pseudo-effective}\}$$
    belongs to $\Ii$.
\end{thm}

We refer the reader to Theorem~\ref{thm: pet general} for a more general version of Theorem~\ref{thm: pet main}. A special case of Theorem~\ref{thm: pet main} is the following gap theorem, which immediately implies Theorem~\ref{thm: stable family birational boundedness}~(1) by applying it to the case when $\Ff$ is the foliation induced by $f\colon X\to T$.

\begin{thm}\label{thm: pet gap main}
Let $d$ be a positive integer and $\Ii\subset [0,1]$ a DCC set. Then there exists a positive real number $\tau>0$ depending only on $d$ and $\Ii$ with the following property. 

Assume that $(X,\Ff,t)$ is a projective lc algebraically integrable adjoint foliated structure of general type such that $t\in\Ii$. Assume that either $t<1$ or $X$ is lc. 
Then
$$\pet(X;\Ff):=\inf\{s\in [0,1]\mid sK_{\Ff}+(1-s)K_X\text{ is big}\}\leq (1-\tau)t.$$
\end{thm}
Theorem~\ref{thm: pet gap main} can be viewed as an analogue of \cite[Lemma~7.3]{HMX14} and is also an immediate consequence of Theorem~\ref{thm: pet main}. Combining Theorems~\ref{thm: afs main} and~\ref{thm: pet gap main}, we shall obtain the desired effective birationality result for general type foliations on lc varieties. See Theorem~\ref{thm: eb when t=1} for a detailed statement. We refer the reader to \cite[Theorems~1.4, 4.1]{SS23} and \cite[Theorems~1.2, 1.3]{Vas25} for similar statements in dimension $2$ without assuming that $\Ff$ is algebraically integrable. 

Next, we have the ACC for $\mathbb R$-complementary thresholds, an analogue of \cite[Theorem~5.12]{HLS24} and \cite[Theorem~21]{Sho20}, which is crucial for us to prove Birkar's criterion for the boundedness of exceptional Fanos (Theorem~\ref{thm: exceptional afs main}).

\begin{thm}[ACC for $\mathbb R$-complementary thresholds]\label{thm: r-complementary main}
    Let $d$ be a positive integer. Then there exists an ACC set $\Ii\subset [0,1]$ depending only on $d$ such that, for any $\mathbb Q$-Gorenstein algebraically integrable foliation $\Ff$ on a Fano type variety $X$ of dimension $d$, the $\mathbb R$-complementary threshold
    $$\Rct(X;\Ff):=\sup\{t\geq 0\mid (X,\Ff,t)\text{ is }\mathbb R\text{-complementary}\}$$
    belongs to $\Ii$. Here we say that $(X,\Ff,t)$ \emph{is $\mathbb R$-complementary} if $(X,\Ff,B,t)$ is lc and
    $$tK_{\Ff}+(1-t)K_X+B\sim_{\mathbb R}0$$
    for some $\mathbb R$-divisor $B\geq 0$.
\end{thm}

We refer the reader to Theorem~\ref{thm: rct general} for a more general version of Theorem~\ref{thm: r-complementary main}. 

Finally, we have an ACC statement for the Fano spectrum, an analogue of \cite[Corollary~1.10]{HMX14}. 

\begin{thm}\label{thm: fano spectrum main}
    Let $d$ be a positive integer and $\Ii\subset [0,1]$ a DCC set. Then there exists an ACC set $\Ii'$ depending only on $d$ with the following property. Assume that
\begin{enumerate}
    \item $\Aa:=(X,\Ff,t)$ is an lc Fano algebraically integrable adjoint foliated structure,
    \item $\dim X=d$ and $t\in\Ii$,
    \item $H$ is a Cartier divisor on $X$, and
    \item $-K_{\Aa}\sim_{\mathbb R}rH$.
\end{enumerate}
Then $r\in\Ii'$.
\end{thm}

We refer the reader to Theorem~\ref{thm: fano spectrum general} for a more general version of Theorem~\ref{thm: fano spectrum main}.

\subsection*{Sketch of the proof of the main theorems} We first sketch the proof of Theorem~\ref{thm: main ACC}, the ACC for interpolated lc thresholds for algebraically integrable foliations. 

A natural first step is to apply adjunction to lc places and reduce to a global ACC statement (Theorem~\ref{thm: global ACC main}). This requires a precise adjunction formula. We remark that, in the case of varieties, this part is substantially easier, since one can cut the ambient variety by general hyperplane sections and reduce to a surface problem (cf.~\cite[Section~16]{Kol+92}). Another key point is that it is not sufficient to prove only that ``DCC coefficients adjoint to DCC coefficients'' (cf.~Theorem~\ref{thm: dcc adjunct to dcc}). One also needs to prove that ``if the coefficients after adjunction form a finite set, then the coefficients before adjunction already form a finite set'' (cf.~Theorem~\ref{thm: finite inversion of adjunction to finite 2}). The latter statement is technical but necessary, and an analogous argument appears in \cite[Lemma~4.10]{BZ16} in the proof of the ACC for generalized lc thresholds. We refer to Section~\ref{sec: precise adj} for more details on the precise adjunction formulae.

Once the precise adjunction formulae are available, we would like to reduce the ACC for interpolated lc thresholds to a global ACC statement in lower dimensions, following the strategy of \cite[Section~5]{HMX14}. However, unlike the case of varieties, we can only reduce the problem to the \emph{relative global ACC} (Theorem~\ref{thm: global acc general}). The reason is that we cannot apply adjunction to a general fiber, since Bertini-type theorems fail for foliations. 

From this point on, the arguments of \cite{HMX14} no longer apply directly. In particular, the inductive approach in \cite{HMX14} uses effective birationality (Theorem~\ref{thm: afs main}) in an essential way. In the proof of effective birationality in \cite[Section~7]{HMX14}, a subadjunction formula is used, which in turn relies on the canonical bundle formula. The canonical bundle formula, however, is not known in general for adjoint foliated structures. 

At this stage we need a different approach. The guiding idea is simple: we deduce the ACC from BAB. In the classical setting of varieties, this strategy was suggested by M\textsuperscript{c}Kernan--Prokhorov \cite{MP04} but was not carried out, since the ACC for log canonical thresholds \cite{HMX14} predates BAB \cite{Bir21a} and enters essentially into its proof. In our setting there is no circularity, because BAB for algebraically integrable adjoint foliated structures is already available \cite{Cas+25a}. Interestingly, an idea originating in work of M\textsuperscript{c}Kernan--Prokhorov \cite{MP04} can be combined with our BAB input to settle a conjecture of M\textsuperscript{c}Kernan.

There are, of course, substantial technical difficulties. We highlight two of them. First, because we work with a \emph{relative} global ACC statement, the two-ray game in \cite{MP04} has to be replaced by a relative version (cf.~Lemma~\ref{lem: relative two-ray}), and one must also control vertical divisors. We handle the latter by induction on the relative dimension of the fibration (the integer $d'$ in the proof of Theorem~\ref{thm: acc to global acc}). Second, the parameter $t$ can be arbitrarily close to $1$, so the BAB theorem of \cite{Cas+25a} does not apply directly. We circumvent this by working with a BAB statement at parameter $1-\tau$, where $\tau>0$ is a uniform constant determined by the ACC for interpolated lc thresholds in the same dimension (available by induction on dimension). Iterating the ACC for interpolated lc thresholds together with the lower-dimensional global ACC then yields the desired conclusion.

Once the ACC and the global ACC are in place, we can deduce the remaining statements recorded in this section. Theorems~\ref{thm: pet main}, \ref{thm: pet gap main}, and \ref{thm: r-complementary main} follow from the ACC for interpolated lc thresholds together with the global ACC, while Theorem~\ref{thm: fano spectrum main} follows from a generalized global ACC applied to the adjoint foliated structures $\left(\Aa,r\overline{H}\right)$. For effective birationality (Theorem~\ref{thm: afs main}), we combine Theorem~\ref{thm: pet gap main} with the finiteness of non-terminal places for klt algebraically integrable adjoint foliated structures, which reduces the problem to the $\epsilon$-lc case for some $\epsilon>0$; we then apply Birkar's effective birationality theorem \cite[Theorem~4.2]{Bir23a}. The volume lower bound (Theorem~\ref{thm: volume lower bound} and Corollary~\ref{cor: low bound volume}) and the boundedness of birational automorphism groups (Theorem~\ref{thm: bir auto group afs main}) are direct consequences of effective birationality.

We now turn to birational boundedness, namely Theorem~\ref{thm: birationally bounded main}. The argument again starts from effective birationality of pluricanonical systems, but the deduction is more delicate: in the classical setting one uses Kawamata--Viehweg vanishing \cite[Lemma~3.2]{HMX13}, which fails in general for foliations. We instead follow the strategy of \cite[Proof of Theorem~6.1]{SS23}. Effective birationality yields a birational model $W$ of $X$ and a birational morphism $f\colon W\to X$ such that we may write
$$f^*|mK_{\Aa}|=|A_W|+R_W$$
where $A_W$ is base-point-free and induces a birational morphism to a bounded variety $Y$ (using the volume bound), and $R_W$ is the fixed part of $f^*|mK_{\Aa}|$. We then show that $W$ carries an adjoint foliated structure $\Aa_W$ with good properties (Lemma~\ref{lem: special model for eb}). In particular, $\Aa$ is log birationally bounded if and only if $\Aa_W$ is. From this point on, we adapt Steps~3--6 of \cite[Proof of Theorem~6.1]{SS23}: first we run a $K_{\Aa_W}$-MMP$/Y$ to obtain a model $\Aa_Z/Y$ with ambient variety $Z$, and then run a $(K_Z+\text{(boundary)})$-MMP$/Y$ to obtain a model $X'$. We show that $X'$ is bounded using the relative BAB \cite{Bir24}. Finally, boundedness of intersection numbers together with the nefness of $K_{\Aa_Z}$ yields boundedness of $\Aa'$, where $\Aa'$ is the image of $\Aa_Z$ on $X'$. This implies the log birational boundedness of $\Aa_W$, and hence of $\Aa$. See Theorem~\ref{thm: precise birational boundedness} for details.

Since birational boundedness and the property of being a stable family are preserved under base change, Theorem~\ref{thm: stable family birational boundedness} follows immediately---after a suitable base change---from Theorem~\ref{thm: birationally bounded main} and the ACC for pseudo-effective thresholds (Theorem~\ref{thm: pet gap main}).

Finally, we prove the two boundedness criteria for Fano algebraically integrable foliations. To prove Birkar's criterion on the boundedness of exceptional Fanos (Theorem~\ref{thm: exceptional afs main}), we first use the ACC for $\mathbb R$-complementary thresholds (Theorem~\ref{thm: r-complementary main}) to reduce to the case where $t$ is bounded away from $1$, and then apply it again to reduce to the case where $\Aa$ is $\epsilon$-lc for some fixed $\epsilon>0$. In this situation, the BAB theorem \cite[Theorem~B]{Cas+25a} applies. To prove Jiang's criterion (Theorem~\ref{thm: jiang afs main}), we may assume that $\tmld(\Aa)$ approaches $0$; otherwise we are done by BAB \cite[Theorem~B]{Cas+25a}. We then extract a divisor $E$ by a birational morphism $f\colon Y\rightarrow X$ that computes $\tmld(\Aa)$. The key idea is to construct a birational model of $(Y,(1-e)E)$ on which $E$ is not contracted and the pair is $\epsilon'$-lc for some fixed $\epsilon'>0$, while $e\to 0$. To do so, we need to estimate the lower bound of the local volumes of the ambient variety in terms of $\alpha^d\cdot\vol$, similarly to the approach in \cite{LLX20}. This leads to a contradiction as we deduce that $e>\epsilon'$, which is incompatible with $e\to 0$.

\subsection*{Structure of the paper}
In Section~\ref{sec: preliminaries} we recall preliminary material and fix notation. In particular, we introduce \emph{normalized foliated structures}, which allow us to keep track of coefficients uniformly throughout the paper. In Section~\ref{sec: precise adj} we establish precise adjunction formulae for algebraically integrable adjoint foliated structures; in particular, we show that the DCC property of coefficient sets is preserved under adjunction (Theorem~\ref{thm: dcc adjunct to dcc}). In Section~\ref{sec: global to local} we state more general versions of the local ACC and the global ACC theorems (Theorems~\ref{thm: acc lct general version} and~\ref{thm: global acc general}), and we show that the global ACC in dimension $d-1$ implies the local ACC in dimension $d$. In Section~\ref{sec: local to global}, we show that the global ACC in dimension $d-1$, together with the local ACC in dimension $d$, implies the global ACC in dimension $d$. In Section~\ref{sec: proof of acc and global acc} we prove the global ACC (Theorem~\ref{thm: global ACC main}) and the local ACC (Theorem~\ref{thm: main ACC}). In Section~\ref{sec: acc type results} we prove several further ACC-type results, namely the ACC for pseudo-effective thresholds (Theorems~\ref{thm: pet main} and \ref{thm: pet gap main}), the ACC for $\mathbb R$-complementary thresholds (Theorem~\ref{thm: r-complementary main}), and the ACC for the Fano spectrum (Theorem~\ref{thm: fano spectrum main}). In Section~\ref{sec: eb} we prove effective birationality for adjoint pluricanonical systems (Theorem~\ref{thm: afs main}). In Section~\ref{sec: birational boundedness} we prove birational boundedness for algebraically integrable foliations of general type (Theorem~\ref{thm: birationally bounded main}), which immediately implies the birational boundedness of stable families of maximal variation (Theorem~\ref{thm: stable family birational boundedness}). In Section~\ref{sec: application} we prove Birkar's boundedness criterion for exceptional Fanos (Theorem~\ref{thm: exceptional afs main}) and Jiang's boundedness criterion (Theorem~\ref{thm: jiang afs main}). In Section~\ref{sec: volume lower bound} we prove the lower bound on the volume (Theorem~\ref{thm: volume lower bound}, Corollary~\ref{cor: low bound volume}) which implies the finiteness of the birational automorphism group (Theorem~\ref{thm: bir auto group afs main}). Finally, in Section~\ref{sec: discussion} we discuss variations of our main theorems and formulate several open questions.

\section{Preliminaries}\label{sec: preliminaries}

We adopt the standard notation and definitions for the MMP from \cite{Sho92,KM98,BCHM10} and use them freely. We also adopt the same notation and definitions as in \cite{Cas+24,Cas+25a} for adjoint foliated structures.

\subsection{Basic notation}

\begin{nota}
Let $X\rightarrow U$ be a projective morphism between normal quasi-projective varieties.
We use the following standard terminology.

A \emph{contraction}$/U$ $f: X\rightarrow Y$ is a projective morphism$/U$ such that $f_\ast \mathcal{O}_X=\mathcal{O}_Y$.
Since $X$ and $U$ are normal, the fibers of $f$ are connected.

For any birational map$/U$ $h\colon X\dashrightarrow X'$, we denote by $\Exc(h)$ the reduced divisor supported on the codimension one part of the exceptional locus of $h$.

$X$ is called \emph{of Fano type}$/U$ if there exists a klt pair $(X,\Delta)$ such that $-(K_X+\Delta)$ is ample$/U$ for some $\Delta\geq 0$. 
$X$ is called \emph{potentially klt} if $X$ is of Fano type$/X$.
\end{nota}

\begin{defn}[$\bb$-divisors]
We use the notation as in \cite[Definition~2.4]{HL23a}. Let $X$ be a normal quasi-projective variety. A $\bb$-divisor $\Mm$ on $X$ is a (possibly infinite) $\mathbb R$-linear combination of divisorial valuations $\nu_P$ over $X$
\[\Mm=\sum_Pr_P\nu_P\]
such that for any birational map $\phi\colon X\dashrightarrow Y$, 
\[\{P\mid \Center_{Y}P\text{ is a divisor}, r_P\neq 0\}\]
is a finite set. By definition, $\Mm$ is also a $\bb$-divisor on $Y$. We say that $\Mm$ is a $\mathbb Q$-$\bb$-divisor if $r_P\in\mathbb Q$ for any prime divisor $P$ over $X$. For any birational map $\phi\colon X\dashrightarrow Y$, we denote by
\[\Mm_{Y}:=\sum_{P\mid \Center_{Y}P\text{ is a divisor}}r_P\left(\Center_{Y}P\right).\]
We say that $\Mm$ \emph{descends to} $X$ if $\Mm_{X}$ is $\mathbb R$-Cartier, and for any projective birational morphism $h\colon X'\rightarrow X$, we have $\Mm_{X'}=h^*\Mm_X$. For any $\mathbb R$-Cartier $\mathbb R$-divisor $D$ on $X$, we denote by $\overline{D}$ the $\bb$-divisor $\Mm$ such that $\Mm$ descends to $X$ and $\Mm_X=D$. We denote by $\bm{0}:=\overline{0}$. We say that $\Mm$ is \emph{$\bb$-Cartier} if there exists a birational map $\phi\colon X\dashrightarrow Y$ such that $\Mm$ descends to $Y$ and $\Mm_Y$ is Cartier. For any projective morphism $\pi: X\rightarrow U$, we say that $\Mm$ is \emph{nef$/U$} (resp. $\Mm\equiv_{U}\bm{0}$) if there exists a birational map$/U$ $\phi\colon X\dashrightarrow Y$ such that $\Mm$ descends to $Y$ and $\Mm_Y$ is nef$/U$ (resp. $\Mm_Y\equiv_{U}0$), and if $U$ is a closed point, then we say that $\Mm$ is \emph{nef} (resp. $\Mm\equiv\bm{0}$). We denote by $\bNef(X/U)$ the set of nef$/U$ $\bb$-divisors on $X$. We denote by $\bNef(X):=\bNef(X/\{pt\})$. We refer the reader to \cite[Remark~4.2(2), Definition~4.7]{BZ16} for restriction of $\bb$-divisors to the general fiber of a contraction or to the normalization of a prime divisor, which we will often use in this paper.
\end{defn}

\begin{nota}
Let $\Ii\subset\mathbb R$ be a set. 
We denote by $\overline{\Ii}$ the closure of $\Ii$ in $\mathbb R$. 
For any real number $a$, we define $a\Ii:=\{a\gamma\mid \gamma\in\Ii\}$. For any non-negative integer $m$ and $\bm{v}=(v_1,\dots,v_m)\in\mathbb R^m$, we write $\bm{v}\in\Ii$ if $v_i\in\Ii$ for each $i$.

Let $X$ be a normal quasi-projective variety. 
For any $\mathbb R$-divisor $D$ on $X$, we write $D\in\Ii$ if $D=\sum d_iD_i$, where $D_i$ are the irreducible components of $D$ and $d_i\in\Ii$ for each $i$. 
For any projective morphism $X\rightarrow U$, we denote by $\bNef(X/U,\Ii)$ the set of $\bb$-divisors $\Mm$ on $X$ such that $\Mm=\sum \mu_j\Mm_j$, where each $\mu_j\in\Ii$ and each $\Mm_j$ is nef$/U$ $\bb$-Cartier.
\end{nota}

\begin{nota}\label{nota: special notation}
Let $X$ be a normal variety, $P$ a prime divisor on $X$, $E$ a prime divisor over $X$, and $D$ an $\mathbb R$-divisor on $X$. We denote by $\mult_PD$ the coefficient of $P$ along $D$. If $D$ is $\mathbb R$-Cartier then, given any projective birational morphism $h\colon X'\rightarrow X$ such that $E$ is on $X'$, we denote by $\mult_ED:=\mult_Eh^*D$. We define
    \[\Weil_{\mathbb R}(X)_{P=0}:=\left\{\sum_{finite} a_iD_i\middle| D_i\text{ are distinct prime divisors on } X, a_i\in\mathbb R, D_i\neq P\right\}.\]
\end{nota}

\subsection{Adjoint foliated structures}

We begin by recalling some basic notions about algebraically integrable foliations.

\begin{defn}\label{defn: foliation}
Let $X$ be a normal quasi-projective variety. 
A \emph{foliation} $\Ff$ on $X$ is a coherent sheaf $T_{\Ff}\subset T_X$ such that $T_{\Ff}$ is saturated in $T_X$ and $T_{\Ff}$ is closed under the Lie bracket. 
The \emph{rank} of $\Ff$ is the rank of $T_\Ff$ as a sheaf and is denoted by $\rk\Ff$. 
The \emph{canonical divisor} of $\Ff$ is a divisor $K_\Ff$ such that $\mathcal{O}_X(-K_{\Ff})\cong\det(T_\Ff)$. 
We say that $\Ff$ is \emph{$\mathbb Q$-Gorenstein} if $K_{\Ff}$ is $\mathbb Q$-Cartier.

A subvariety $S\subset X$ is called \emph{$\Ff$-invariant} if for any open subset $U\subset X$ and any section $\partial\in H^0(U,T_\Ff)$, we have $\partial(\mathcal{I}_{S\cap U})\subset \mathcal{I}_{S\cap U}$, where $\mathcal{I}_{S\cap U}$ denotes the ideal sheaf of $S\cap U$ in $U$. 

Given any dominant map 
$h\colon Y\dashrightarrow X$, we denote by $h^{-1}\Ff$ the \emph{pullback} of $\Ff$ on $Y$ (cf.~\cite[\S~3.2]{Dru21}). We say that $\Ff$ is \emph{algebraically integrable} if there exists a dominant map $f\colon X\dashrightarrow Z$ such that $\Ff=f^{-1}\Ff_Z$, where $\Ff_Z$ is the trivial foliation (i.e.\ $T_{\Ff_Z}=0$), and we say that $\Ff$ is \emph{induced by} $f$. 

Fix a birational map
$g\colon X\dashrightarrow X'$, and define $g_\ast \Ff\coloneqq (g^{-1})^{-1}\Ff$. 
For any prime divisor $P$ over $X$ such that $P':=\Center_{X'}P$ is a divisor, we write $\epsilon_{\Ff}(P):=0$ (resp. $\epsilon_{\Ff}(P):=1$) if $P'$ is (resp. is not) $g_*\Ff$-invariant. 
For any $\mathbb R$-divisor $D=\sum_{i=1}^na_iD_i$ on $X$ where each $D_i$ is a prime divisor, we write $D_{\Ff}^{\ninv}:=\sum_{i=1}^n\epsilon_{\Ff}(D_i)a_iD_i$ and $D_{\Ff}^{\inv}:=D-D_{\Ff}^{\ninv}$. 
If $\Ff$ is clear in the context then we may write
$D^{\ninv}:=D_{\Ff}^{\ninv}$ and $D^{\inv}:=D_{\Ff}^{\inv}$.

A \emph{foliated pair} is a pair $(X,\Ff)$ consisting of a normal variety $X$ and a foliation $\Ff$ on $X$.
\end{defn}

\begin{defn}
Given a projective foliated pair $(X,\Ff)$, we denote by $\Aut(X,\Ff)$ (resp. $\Bir(X,\Ff)$) the group of automorphisms $\phi\colon X\to X$ (resp. birational automorphisms $\mu\colon X\dashrightarrow X$) such that $\phi_\ast\Ff=\Ff$ (resp. $\mu_\ast\Ff=\Ff$).
\end{defn}

We now introduce the notion of adjoint foliated structure that was already introduced and discussed in 
\cite{Cas+24,Cas+25a}.

\begin{defn}\label{defn: afs}
An \emph{adjoint foliated structure}, denoted by
\[
\Aa/U=(X,\Ff,B,\Mm,t)/U,
\]
consists of a normal quasi-projective variety $X$ equipped with a projective morphism $X\to U$, a foliation $\Ff$ on $X$, an $\Rr$-divisor $B\geq 0$ on $X$, a nef$/U$ $\bb$-divisor $\Mm$ on $X$, and a real number $t\in [0,1]$, such that the $\Rr$-divisor
\[
K_{\Aa}:=tK_{\Ff}+(1-t)K_X+B+\Mm_X
\]
is $\Rr$-Cartier. The \emph{dimension} $\dim\Aa$ of $\Aa$ is defined to be $\dim X$.

If we allow $B$ to have negative coefficients, then we add the prefix ``sub-''. If $B$ is a $\Qq$-divisor, $\Mm$ is a $\Qq$-$\bb$-divisor, and $t\in\mathbb Q$, then we add the prefix ``$\Qq$-''. If we do not require $K_{\Aa}$ to be $\Rr$-Cartier, then we add the prefix ``quasi-''.
\end{defn}

\begin{nota}
Let $\Aa/U=(X,\Ff,B,\Mm,t)/U$ be a quasi-sub-adjoint foliated structure.

We call $X,\Ff,B,\Mm,t$ the \emph{ambient variety}, \emph{foliation part}, \emph{boundary part}, \emph{nef part} (or \emph{moduli part}), and \emph{parameter} of $\Aa$, respectively. $K_{\Aa}$ is called the \emph{canonical $\mathbb R$-divisor} of $\Aa$. The $\mathbb R$-divisor $B+\Mm_X$ is called the \emph{generalized boundary} of $\Aa$. 
We adopt the following abbreviations:
\begin{align*}
  \Aa=(X,\Ff,B,\Mm) &\quad\text{if } t=1,\ \text{and then }\Aa/U\text{ is called a \emph{generalized foliated quadruple}},\\
  \Aa=(X,B,\Mm) &\quad\text{if } \Ff=T_X\ \text{or } t=0,\ \text{and then }\Aa/U\text{ is called a \emph{generalized pair}},\\
  \Aa=(X,\Ff,B) &\quad\text{if } \Mm=\bm{0}\ \text{and } t=1,\ \text{and then }\Aa/U\text{ is called a \emph{foliated triple}},\\
  \Aa=(X,B) &\quad\text{if }\Aa\text{ is a generalized pair and }\Mm=\bm{0},\ \text{and then }\Aa/U\text{ is called a \emph{pair}}.
\end{align*}

If $B=0$, or if $\Mm=\bm{0}$, or if $U$ is not important, then we may drop $B,\Mm,U$ respectively. If $U=\{pt\}$ then we also drop $U$ and say that $\Aa$ is \emph{projective}. 

For any $\Rr$-divisor $D$ on $X$ and nef$/U$ $\bb$-divisor $\Nn$ on $X$, we denote by $(\Aa,D,\Nn):=(X,\Ff,B+D,\Mm+\Nn,t)$. If $D=0$ or $\Nn=\bm{0}$ then we may drop $D,\Nn$ respectively. 

We say that $\Aa$ is \emph{algebraically integrable} if $\Ff$ is, and that $\Aa$ is \emph{$\mathbb Q$-factorial} if $X$ is.

Let $\Aa_i/U=(X,\Ff,B_i,\Mm_i,t_i)/U$, $i=1, \dots, k$ be quasi-sub-adjoint foliated structures. 
For all $1 \leq i \leq k$, let $a_i\in\mathbb R_{\geq 0}$ be real numbers such that 
$\sum_{i=1}^k a_it_i\leq 1$. 
We write
\begin{align*}
\sum_{i=1}^k a_i\Aa_i:=\left(X,\Ff,\sum_{i=1}^k a_iB_i,\sum_{i=1}^k a_i\Mm_i,\sum_{i=1}^k a_it_i\right).
\end{align*}
Let us note that if $\sum_{i=1}^k a_i=1$, then $K_{\sum_{i=1}^k a_i\Aa_i}=\sum_{i=1}^k a_iK_{\Aa_i}$.
\end{nota}

We now recall the different types of singularities for adjoint foliated structures, in analogy with the case of log pairs.

\begin{defn}\label{defn: sing of afs}
Let $\Aa/U=(X,\Ff,B,\Mm,t)/U$ be a quasi-sub-adjoint foliated structure.

For any birational map$/U$ $\phi\colon X\dashrightarrow X'$, we define $\phi_*\Aa:=(X',\phi_*\Ff,\phi_*B,\Mm,t)$ and say that $\phi_*\Aa$ is the \emph{image} of $\Aa$ on $X'$. 

Let us now assume that $K_{\Aa}$ is $\mathbb R$-Cartier. For any projective birational morphism $h\colon X'\rightarrow X$, we denote by
$h^*\Aa:=(X',\Ff',B',\Mm,t)$ where $\Ff':=h^{-1}\Ff$ and $B'$ is the unique $\Rr$-divisor such that $K_{h^*\Aa}=h^*K_{\Aa}$. 

For any prime divisor $E$ on $X'$, we denote by
$a(E,\Aa):=-\mult_EB'$
the \emph{discrepancy} of $E$ with respect to $\Aa$. 
The \emph{total minimal log discrepancy} $\tmld(\Aa)$ of $\Aa$ is
\[\tmld(\Aa):=\inf\{a(E,\Aa)+t\epsilon_{\Ff}(E)+(1-t)\mid E\text{ is over }X\}.\]
For any non-negative real number $\epsilon$, we say that $\Aa$ is \emph{$\epsilon$-lc} (resp. \emph{$\epsilon$-klt}) if 
\begin{align*}
    \tmld(\Aa)\geq\epsilon \qquad 
    \text{(resp. $ \tmld(\Aa)>\epsilon$)}.
\end{align*}
We say that $\Aa$ is \emph{lc} (resp. \emph{klt}) if $\Aa$ is $0$-lc (resp. $0$-klt). 

An \emph{lc place} of $\Aa$ is a prime divisor $E$ over $X$ such that $a(E,\Aa)=-t\epsilon_{\Ff}(E)-(1-t)$. An \emph{lc center} of $\Aa$ is the image of an \emph{lc place} of $\Aa$ on $X$. We say that $\Aa$ is \emph{plt} if $\Aa$ is lc and any lc place of $\Aa$ is a divisor on $X$.
\end{defn}

We also introduce the notion of the $\alpha$-invariant for projective adjoint foliated structures.

\begin{defn}\label{defn: alpha}
    Let $\Aa=(X,\Ff,B,\Mm,t)$ be a projective adjoint foliated structure and let $L$ be an $\mathbb R$-divisor on $X$. We define \emph{Tian's $\alpha$-invariant}
    $\alpha(\Aa;L)$ of $L$ with respect to $\Aa$ as
    \[\alpha(\Aa;L):=\sup\{s\geq 0\mid (\Aa,sD)\text{ is lc for any }0\leq D\sim_{\mathbb R}L\}.\]

If $|-K_{\Aa}|_{\mathbb R}\neq\emptyset$, then we define 
\emph{Tian's $\alpha$-invariant} 
$\alpha(\Aa)$ of $\Aa$ as $\alpha(\Aa):=\alpha(\Aa;-K_{\Aa})$. For any pair $(X,\Delta)$ and any point $x \in X$, we also define the \emph{local Tian's $\alpha$-invariant} 
  $\alpha(X\ni x,\Delta;L)$ of $L$ at $x \in X$
  as
    \[\alpha(X\ni x,\Delta;L):=\sup\{s\geq 0\mid (X,\Delta+sD)\text{ is lc near }x\text{ for any }0\leq D\sim_{\mathbb R}L\}.\]
\end{defn}

\begin{defn}
Let $\Aa/U$ be an adjoint foliated structure. 

We say that $\Aa/U$ is \emph{of general type} if $K_{\Aa}$ is big$/U$. We say that $\Aa/U$ is \emph{$\mathbb R$-complementary} if there exists $D\geq 0$ such that $(\Aa,D)/U$ is lc and $K_{\Aa}+D\sim_{\mathbb R,U}0$. If $U$ is a point and $-K_{\Aa}$ is ample, then we say that $\Aa$ is \emph{Fano}.
\end{defn}

\subsection{Normalized foliated structures}\label{subsec: normalized foliated structure}

In this section we introduce normalized foliated structures. They provide a convenient repackaging of adjoint foliated structures, in which the invariant part of the boundary is rescaled by the parameter.

\begin{defn}
A \emph{normalized foliated structure}, denoted by $$\Bb/U=(X,\Ff,B,\Mm)(t)/U,$$ 
consists of a normal quasi-projective variety $X$ equipped with a projective morphism $X\to U$, a foliation $\Ff$ on $X$, an $\mathbb R$-divisor $B\geq 0$ on $X$, a nef$/U$ $\bb$-divisor $\Mm$ on $X$, and a real number $t\in[0,1]$, such that
\[
\Aa/U\coloneqq (X,\Ff,B^{\ninv}+(1-t)B^{\inv},\Mm,t)/U
\]
is an adjoint foliated structure.

We shall say that the adjoint foliated structure $\Aa/U$ defined above is the \emph{adjoint foliated structure associated to} $\Bb/U$, and we denote this relation by
\[
\Aa=\log\Bb.
\]

The \emph{dimension} $\dim\Bb$ of $\Bb$ is defined to be $\dim X$.

If we allow $B$ to have negative coefficients, then we add the prefix ``sub-''. If we do not require $K_{\log\Bb}$ to be $\Rr$-Cartier, then we add the prefix ``quasi-''.
\end{defn}

\begin{rem}
A (quasi-)(sub-)normalized foliated structure $\Bb/U=(X,\Ff,B,\Mm)(t)/U$ uniquely determines the associated (quasi-)(sub-)adjoint foliated structure
\[
\log\Bb\coloneqq (X,\Ff,B^{\ninv}+(1-t)B^{\inv},\Mm,t)/U.
\]
Conversely, if $t<1$ and $\Aa/U=(X,\Ff,\Delta,\Mm,t)/U$ is a (quasi-)(sub-)adjoint foliated structure, then there exists a unique (quasi-)(sub-)normalized foliated structure $\Bb/U$, of the form $\Bb/U=(X,\Ff,B,\Mm)(t)/U$, with
\[
B=\Delta^{\ninv}+\frac{1}{1-t}\Delta^{\inv},
\]
such that $\Aa=\log\Bb$. We denote this inverse correspondence by $\Bb=\exp(\Aa)$.
\end{rem}

\begin{nota}
Let $\Bb/U=(X,\Ff,B,\Mm)(t)$ be a quasi-sub-normalized foliated structure, and set $\Aa:=\log\Bb$.
We refer to $X,\Ff,B,\Mm,t$ as the \emph{ambient variety}, \emph{foliation part}, \emph{boundary part}, \emph{nef part} (or \emph{moduli part}), and \emph{parameter} of $\Bb$, respectively.
We write $K_{\Bb}\coloneqq K_{\Aa}$ and call it the \emph{canonical $\mathbb R$-divisor} of $\Bb$.

If $B=0$, or $\Mm=\bm{0}$, or the base $U$ is clear from the context, then we may drop $B$, $\Mm$, or $U$, respectively.
If $U=\{pt\}$, then we also drop $U$ and say that $\Bb$ is \emph{projective}.

We say that $\Bb/U$ is \emph{of general type} if $(\log\Bb)/U$ is of general type.
For any $\mathbb R$-divisor $D$ on $X$ and nef$/U$ $\bb$-divisor $\Nn$ on $X$, we set
\[
(\Bb,D,\Nn)\coloneqq (X,\Ff,B+D,\Mm+\Nn)(t).
\]
If $D=0$ or $\Nn=\bm{0}$, then we may drop $D$ or $\Nn$, respectively.

We say that $\Bb$ is \emph{algebraically integrable} (resp. \emph{$\mathbb Q$-factorial}) if $\log\Bb$ is algebraically integrable (resp. $\mathbb Q$-factorial), i.e.\ if $\Ff$ is (resp. $X$ is).

A \emph{foliated log resolution} of $\Bb$ is a foliated log resolution of $(X,\Ff,\Supp B,\Mm)$ (cf.~\cite[Definition~6.2.3]{CHLX23}).

For any set $\Ii\subset [0,+\infty)$, we write $\Bb/U\in\Ii$ if $B,t\in\Ii$ and $\Mm=\sum\mu_j\Mm_j$, where each $\mu_j\in\Ii$ and each $\Mm_j$ is nef$/U$ $\bb$-Cartier.
\end{nota}

We now define how normalized foliated structures behave under birational maps.

\begin{defn}
Let $\Bb/U=(X,\Ff,B,\Mm)(t)$ be a quasi-sub-normalized foliated structure. 

For any birational map$/U$ $\phi\colon X\dashrightarrow X'$, we define $\phi_*\Bb:=(X',\phi_*\Ff,\phi_*B,\Mm)(t)$ and we say that $\phi_*\Bb$ is the image of $\Bb$ on $X'$. 

Assume that $\Bb/U=(X,\Ff,B,\Mm)(t)$ is a sub-normalized foliated structure. For any projective birational morphism $h\colon X'\rightarrow X$ and any sub-normalized foliated structure $\Bb'/U$, we write $\Bb'=h^*\Bb$ if $\Bb'=(X',\Ff',B',\Mm)(t)$ and $K_{\Bb'}=h^*K_{\Bb}$. 

Let $\Bb''=(X'',\Ff'',B'',\Mm'')(t)/U$ be a sub-normalized foliated structure and assume that there exists a birational map$/U$ $\phi\colon X\dashrightarrow X''$ such that $\Ff''=\phi_*\Ff$. 
We say that $\Bb$ and $\Bb''$ are \emph{crepant} if there exist projective birational morphisms $p\colon W\rightarrow X$ and $q\colon W\rightarrow X''$ and a sub-normalized foliated structure $\Bb_W$ such that 
\begin{align*}
p^*\Bb=\Bb_W=q^*\Bb''.    
\end{align*}
\end{defn}

\begin{rem}
With the notation of the above definition, we highlight that the sub-normalized foliated structure $\Bb'= h^\ast \Bb$ is uniquely determined when $t<1$.
That may not be the case, though, when $t=1$:
indeed, writing $\Bb' = (X',\Ff',B',\Mm)(1)$, then for any choice of an $\mathbb R$-divisor $H$ such that $H=H^{\inv}$, also $\Bb'' \coloneqq (X',\Ff',B'+H,\Mm)(1)$ will satisfy $K_{\Bb''}=h^\ast K_{\Bb}$, since 
\[
h^\ast K_{\Bb}=K_{\Bb'}=K_{X'}+B'^{\ninv}+\Mm=K_{X'}+(B'+H)^{\ninv}+\Mm= K_{\Bb''}.
\]
\end{rem}

\begin{defn}\label{defn: sing of nfs}
Let $\Bb/U=(X,\Ff,B,\Mm)(t)$ be a sub-normalized foliated structure. 

For any prime divisor $E$ over $X$, we define 
\begin{align*}
a(E,\Bb):=a(E,\log\Bb)
\quad \text{and} \quad
\tmld(\Bb):=\tmld(\log\Bb).
\end{align*}
We say that $\Bb$ is $\epsilon$-lc (resp. $\epsilon$-klt, lc, klt, plt) if $\log\Bb$ is $\epsilon$-lc (resp. $\epsilon$-klt, lc, klt, plt). An lc place (resp. lc center) of $\Bb$ is an lc place (resp. lc center) of $\log\Bb$. We say that $\Bb$ is \emph{qdlt} if $\Bb$ is lc, $(X,B,\Mm)$ is qdlt (cf.~\cite[Definition~7.1.1]{CHLX23}), and any lc center of $\Bb$ is an lc center of $(X,B,\Mm)$. 
\end{defn}

Having defined log canonical and qdlt singularities for normalized foliated structures, we can also discuss the existence of qdlt modifications.

\begin{defn}
Let $\Bb/U=(X,\Ff,B,\Mm)(t)$ be a sub-normalized foliated structure. 

When $t<1$, we define a $\mathbb Q$-factorial qdlt modification of $\Bb$ simply as a $\mathbb Q$-factorial ACSS modification $h\colon X'\rightarrow X$ of $\log\Bb$, cf.~\cite[Definition~3.7]{Cas+24}.
We say that $\Bb':=(h^{-1}_*\Bb,\Exc(h))$ is a $\mathbb Q$-factorial qdlt model of $\Bb$. 

When $t=1$, a $\mathbb Q$-factorial ACSS modification of $\Bb$ is a $\mathbb Q$-factorial ACSS modification $h\colon X'\rightarrow X$ of $\log\Bb$ associated with a contraction $f: X'\rightarrow Z$ and a divisor $G$ as in \cite[Definition~2.18]{Cas+24}, such that $G\geq\Exc(h)+h^{-1}_*B^{\inv}$, and we say that $\Bb':=(h^{-1}_*\Bb,\Exc(h))$ is a $\mathbb Q$-factorial ACSS model (resp. $\mathbb Q$-factorial qdlt model) of $\Bb$. 
\end{defn}

We now define structures associated to Fano normalized foliated structures.

\begin{defn}
Let $\Bb/U$ be an lc normalized foliated structure. An \emph{$\mathbb R$-complement} of $\Bb/U$ is an lc normalized foliated structure $(\Bb,D)/U$ such that $D\geq 0$ and $K_{\Bb}+D\sim_{\mathbb R,U}0$. We say that $\Bb/U$ is \emph{$\mathbb R$-complementary} if $\Bb/U$ has an $\mathbb R$-complement. 

Let $\Bb$ be a projective normalized foliated structure. We denote by $\alpha(\Bb):=\alpha(\log\Bb)$. We say that $\Bb$ is \emph{Fano} if $-K_{\Bb}$ is ample. We say that $\Bb$ is \emph{exceptional} if $|-K_{\Bb}|_{\mathbb R}\neq\emptyset$ and $(\Bb,D)$ is klt for any $0\leq D\sim_{\mathbb R}-K_{\Bb}$.
\end{defn}

We will repeatedly use two basic operations on normalized foliated structures: rescaling by $s\in[0,1]$, denoted by an exponentiation $\Bb^s$, and addition $\Bb+\Bb'$ when the sum of the corresponding parameters satisfies $t+t'\leq 1$; cf.~Subsection~\ref{subsec: log canonicity under order}.

\begin{defn}
    Let 
    $\Bb/U=(X,\Ff,B,\Mm)(t)/U$, 
    $\Bb'/U=(X,\Ff,B',\Mm')(t')/U$ 
    be quasi-sub-normalized foliated structures. 

    For any real number $s\in [0,1]$, we define
        \begin{align*}
        \Bb^s:=(X,\Ff,sB,s\Mm)(st).    
        \end{align*}
    
    If 
    $t+t'\leq 1$, 
    we define the sum 
    $\Bb+\Bb'$ 
    as
        \begin{align*}
        \Bb+\Bb':=(X,\Ff,B+B',\Mm+\Mm')(t+t').    
        \end{align*}
\end{defn}

One key reason to define normalized foliated structures is that they admit a partial ordering, which we denote by $\geq$. 
This will be particularly important when $t=1$.

\begin{defn}\label{defn: order of nfs}
    Let $\Bb/U=(X,\Ff,B,\Mm)(t)/U$ and $\Bb'/U=(X,\Ff,B',\Mm')(t')/U$ be two quasi-sub-normalized foliated structures. We write $\Bb/U\geq\Bb'/U$
    (equivalently, $\Bb'/U\leq\Bb/U$) 
    when $B\geq B',t\geq t'$, and $\Mm-\Mm'$ is nef$/U$.
\end{defn}

The following lemma shows how the structure of normalized foliated structures behaves when applying the operations $\bullet^s$ and $+$.

\begin{lem}\label{lem: geometric average of nfs}
Let $\Bb/U=(X,\Ff,B,\Mm)(t)/U$ and $\Bb'/U=(X,\Ff,B',\Mm')(t')/U$ be two quasi-sub-normalized foliated structures and let $s\in [0,1]$ be a real number. Let $\Aa=\log\Bb$ and $\Aa'=\log\Bb'$. Set $t_s := st+(1-s)t'$ and denote
\[(X,\Ff,\Gamma_s,\Nn_s,t_s):= s\Aa+(1-s)\Aa'\] and \[(X,\Ff,B_s,\Mm_s)(t_s):=\Bb^s+\Bb'^{1-s}\]
(note that it follows from the definition that the parameter in each of these structures is equal to $t_s$).
Similarly, let us denote \[(X,\Ff,\Delta_s,\Mm_s,t_s):=\log(\Bb^s+\Bb'^{1-s})\]
for any $s\in [0,1]$. Then the following hold. 
\begin{enumerate}
    \item For any $s\in [0,1]$, we have 
\[\Nn_s=\Mm_s=s\Mm+(1-s)\Mm'.\]
\item For any $s\in [0,1]$, we have 
\[\Delta_s-\Gamma_s=s(1-s)(t-t')(B^{\inv}-B'^{\inv}).\]
In particular,
$\Gamma_s^{\ninv}=\Delta_s^{\ninv}=(sB+(1-s)B')^{\ninv}$.
\item Assume that $\Bb/U\geq\Bb'/U$ or $\Bb'/U\geq\Bb/U$. Then $\Delta_s\geq\Gamma_s$ for any $s\in [0,1]$. Moreover, $\Delta_s=\Gamma_s$ if and only if $s=0$ or $1$, or $t=t'$, or $B^{\inv}=B'^{\inv}$.
\item Assume that $t\neq1$ (resp. $t'\neq1$). Then for any $s\in (0,1]$ (resp. $s \in [0,1)$),  $\exp(s\Aa+(1-s)\Aa')$ is well defined. Moreover:
\begin{enumerate}
    \item If $\Bb/U\geq\Bb'/U$ or $\Bb'/U\geq\Bb/U$, then 
\[(\Bb^s+\Bb'^{1-s})/U\geq\exp(s\Aa+(1-s)\Aa')/U,\]
and equality holds when $s=0$ or $1$, or $t=t'$, or $B^{\inv}=B'^{\inv}$. 
\item If $\Bb/U\geq\Bb'/U$, then for any $1>s\geq r>0$, 
\[\exp\left(s\Aa+(1-s)\Aa'\right)/U\geq\exp\left(r\Aa+(1-r)\Aa'\right)/U.\]
Moreover, $\Bb/U\geq \exp(s\Aa+(1-s)\Aa')/U\geq\Bb'/U$.
\end{enumerate}
\end{enumerate}
\end{lem}
\begin{proof}
By definition,
\begin{align*}
\Aa = (X, \mathcal F, B^{\ninv} +(1-t)B^{\inv}, \Mm, t)
\ \text{and} \
\Aa' = (X, \mathcal F, B'^{\ninv}+(1-t')B'^\inv, \Mm', t').
\end{align*}
Thus,
\begin{align*}
&s\Aa+(1-s)\Aa' =  \\ &(X,\Ff,s(B^{\ninv}+(1-t)B^{\inv})+(1-s)(B'^{\ninv}+(1-t')B'^{\inv}),s\Mm+(1-s)\Mm',t_s)= \\
&(X, \mathcal F, (sB+(1-s)B')^{\ninv}+(s(1-t)B+(1-s)(1-t')B')^{\inv}, s\Mm+(1-s)\Mm', t_s).
\end{align*}
On the other hand, we have
\begin{align*}
\Bb^s+\Bb'^{1-s} = (X, \mathcal F, sB+(1-s)B', s\Mm+(1-s)\Mm', t_s).    
\end{align*}
Hence, 
\begin{align*}
&\Aa_s:= \log (\Bb^s+\Bb'^{1-s}) = \\
&(X, \mathcal F, (sB+(1-s)B')^{\ninv}+(1-t_s)(sB+(1-s)B')^{\inv}, s\Mm+(1-s)\Mm', t_s).
\end{align*}
From this we can conclude that:
\begin{itemize}
    \item 
(1) follows immediately from the above expressions.
    \item 
For (2), the above expressions imply that 
\[\Gamma_s = (sB+(1-s)B')^{\ninv}+(s(1-t)B+(1-s)(1-t')B')^{\inv}\]
and 
\[\Delta_s = (sB+(1-s)B')^{\ninv}+(1-t_s)(sB+(1-s)B')^{\inv},\]
from which in turn the claimed equality follows at once.
    \item 
Since $\Bb/U\geq\Bb'/U$ or $\Bb'/U\geq\Bb/U$, $(t-t')(B^{\inv}-B'^{\inv})\geq 0$. Hence, we can conclude that (3) follows from (2). 
    \item 
Finally, (4.a) follows from (3) since $t_s<1$ implies that $\exp(s\Aa+(1-s)\Aa')$ is well-defined. 
To prove (4.b), let us note that when $\Bb/U\geq\Bb'/U$, then $B\geq B'$; hence,
$\Gamma_s^{\ninv}-\Gamma_r^{\ninv}=(s-r)(B^{\ninv}-B^{\inv})\geq 0$,
and
\begin{align*}
    \frac{1}{1-t_s}\Gamma_s^{\inv}-\frac{1}{1-t_r}\Gamma_r^{\inv}=\frac{(s-r)(1-t')(1-t)}{(1-t_s)(1-t_r)}(B^{\inv}-B'^{\inv})\geq 0,
\end{align*}
which implies that 
$\exp(s\Aa+(1-s)\Aa')/U\geq\exp(r\Aa+(1-r)\Aa')/U$.
\end{itemize}
\end{proof}

\begin{deflem}\label{deflem: nfs geometric average}
Let $\Bb/U=(X,\Ff,B,\Mm)(t)/U$ and $\Bb'/U=(X,\Ff,B',\Mm')(t')/U$ be two quasi-sub-normalized foliated structures and let $s\in [0,1]$ be a real number. We define
\begin{align*}
\Bb^s\cdot\Bb'^{1-s} \coloneqq
\begin{cases}
\Bb
&\text{if $s=1$},\\
\Bb'
&\text{if $s=0$},\\
\exp(s\log\Bb+(1-s)\log\Bb')
&\text{if $t\neq1$ or $t'\neq1$, and $s\in (0,1)$},\\
\Bb^s+\Bb'^{1-s}
&\text{if $t=t'=1$}.
\end{cases}
\end{align*}

Lemma~\ref{lem: geometric average of nfs} shows that $\Bb^s\cdot\Bb'^{1-s}$ is well-defined, and that
\[K_{\Bb^s\cdot\Bb'^{1-s}}=sK_{\Bb}+(1-s)K_{\Bb'}.\]
Thus, if $\Bb,\Bb'$ are (sub-)normalized foliated structures, then the same holds for $\Bb^s\cdot\Bb'^{1-s}$.
\end{deflem}

\begin{rem}
Let 
$\Bb/U$ and $\Bb'/U$ 
be normalized foliated structures with ambient variety $X$. By the above definition, if 
$\Bb/U\geq\Bb'/U$ or $\Bb'/U\geq\Bb/U$, 
then
\begin{align*}
(\Bb^s+\Bb'^{1-s})/U\geq (\Bb^s\cdot\Bb'^{1-s})/U,
\end{align*}
and equality holds if and only if either $s=0, 1$, or $t=t'$, or $B^{\inv}=B'^{\inv}$.
If $\Bb/U\geq\Bb'/U$, then
  \[\Bb/U\geq(\Bb^s+\Bb'^{1-s})/U\geq\Bb^s\cdot\Bb'^{1-s}/U\geq\Bb'/U.\]
Moreover, the function $s\mapsto \Bb^s\cdot\Bb'^{1-s}/U$ is increasing with respect to the ordering $\geq$.

When $s \in (0, 1)$ it is usually unclear whether $\Bb^s+\Bb'^{1-s}$ is in turn a normalized foliated structure: indeed, it may not be the case in general that $K_{\Bb^s+\Bb'^{1-s}}$ is $\mathbb R$-Cartier.
\end{rem}

\begin{lem}\label{lem: construct smaller bb}
    Let $\Bb/U$ and $\Bb'/U$ be two normalized foliated structures such that $\Bb/U\geq\Bb'/U\geq\Bb^{s}/U$ for some $s>0$. Then for any $r\in [0,s)$, there exists a normalized foliated structure $\Bb_r$ such that $\Bb_r/U\geq\Bb^r/U$ and $\Bb'=\Bb^{a_r}\cdot\Bb_r^{1-a_r}$ for some $a_r\in (0,1)$. 
\end{lem}
\begin{proof}
We may assume that $s<1$, otherwise $\Bb=\Bb'$ and we may take $\Bb_r=\Bb$ for any $r\in [0,1)$. Write $\Bb=(X,\Ff,B,\Mm)(t)/U$ and $\Bb'=(X,\Ff,B',\Mm')(t')/U$. We take $a_r:=\frac{s-r}{1-r}$. 
If $t=t'$, then it suffices to define
    \begin{align*}
        \Bb_r:=\left(X,\Ff,\frac{1-r}{1-s}B'-\frac{s-r}{1-s}B,\frac{1-r}{1-s}\Mm'-\frac{s-r}{1-s}\Mm\right)(t).
    \end{align*}
If $t>t'$, then it suffices to define
    \begin{align*}
    \Bb_r:=\left(X,\Ff,B_r,\frac{1-r}{1-s}\Mm'-\frac{s-r}{1-s}\Mm\right)\left(\frac{(1-r)t'-(s-r)t}{1-s}\right),
    \end{align*}
where
    \begin{align*}
    B_r:=\left(\frac{1-r}{1-s}B'-\frac{s-r}{1-s}B\right)^{\ninv}+\frac{1}{(1-t')-a_r(1-t)}((1-t')B'^{\inv}-a_r(1-t)B^{\inv}).
    \end{align*}
\end{proof}

\subsection{Bounded families}\label{subsec: bounded}

\begin{defn}[{cf.~\cite[Definition~3.5.1]{HMX14},~\cite[2.19]{Bir19}}]\label{defn: bounded pairs}
We say that a set 
$\mathcal{P}$ 
of projective varieties is 
\emph{birationally bounded} 
if there exists a projective morphism 
$f\colon Z\rightarrow T$, 
where 
$T$ 
is of finite type, such that for every 
$X\in\mathcal{P}$, 
there exists a closed point 
$t\in T$
such that
$Z_t:=f^{-1}(t)$
is irreducible, together with
a birational map 
$\phi_t \colon Z_t\dashrightarrow X$.

We say that a set 
$\mathcal{P}$ 
of pairs is \emph{log birationally bounded} 
(resp. \emph{bounded}) 
if there is a pair 
$(Z,C)$ 
such that 
$C$ 
is reduced, 
and a projective morphism 
$f \colon Z\rightarrow T$ 
where 
$T$ 
is of finite type, such that for any pair
$(X,B)\in\mathcal{P}$, 
there exists a closed point 
$t\in T$
such that 
$Z_t:=f^{-1}(t)$
is irreducible, together with
a birational map (resp. isomorphism) 
$\phi_t \colon Z_t\dashrightarrow X$ 
such that 
\begin{align*}
\Supp C_t\neq Z_t
\quad 
\text{and}
\quad  
(\phi^{-1}_t)_\ast \Supp B
\cup
\Exc(\phi_t)
\subset
\Supp C_t\ \left(\text{resp. }\phi_t(C_t)=B\right),
\end{align*}
where $C_t:=f^{-1}(t)\cap C$. 
\end{defn}

\begin{defn}[{\cite[Definition~3.31]{Cas+25a}}]\label{defn: bounded foliated pairs}
Let $\mathcal{P}$ be a set of projective foliated pairs $(X,\Ff)$. We say that $\mathcal{P}$ is \emph{bounded} if there exist finitely many flat projective morphisms 
$f_i\colon X_i\to Z_i$, 
$i=1,\dots,N$ 
of normal varieties with normal fibers and foliations $\mathcal G_i$ of rank $r_i$ on $X_i$ satisfying the following properties:
 \begin{enumerate}
    \item For any closed point $z\in Z_i$
    with $X_{i,z}:=f_i^{-1}(z)$, we have $T_{X_i/Z_i}|_{X_{i,z}}\simeq T_{X_{i,z}}$.
    \item $T_{\mathcal G_i}\subset T_{X_i/Z_i}$, and for any closed point 
    $z\in Z_i$, we have that $T_{\mathcal G_i}|_{X_{i,z}}\subset T_{X_{i,z}}$ defines a foliation ${\Ff_{i,z}}$ of rank $r_i$ on $X_{i,z}$.
    \item For any $(X,\mathcal F)\in \mathcal P$, there exist $i\in\{1,\dots,N\}$, a closed point $z\in Z_i$, and an isomorphism $\phi_z\colon X\to X_{i,z}$, such that $(\phi_z)_\ast\mathcal F\simeq \mathcal F_{i,z}$.
    \end{enumerate}
    In particular, $\{X\mid X\in\mathcal{P}\}$ is bounded.
\end{defn}

\begin{defn}\label{defn: bounded nfs}
Let $\mathcal{P}$ be a set of projective quasi-normalized foliated structures. We say that $\mathcal{P}$ is \emph{bounded} if there exist a bounded set of projective foliated pairs $\mathcal{Q}$ and a positive integer $r$ satisfying the following. For any $\Bb:=(X,\Ff,B,\Mm)(t)\in\mathcal{P}$, there exists a very ample divisor $H$ on $X$, such that:
\begin{enumerate}
    \item $(X,\Ff)\in\mathcal{Q}$.
    \item $H^{\dim X}\leq r$.
    \item $-r\leq (K_X+\Supp B+\Mm_X)\cdot H^{\dim X-1}\leq r$.
\end{enumerate}
Note that this aligns with Definition~\ref{defn: bounded pairs} (cf.~\cite[Lemma~2.20]{Bir19}) and the definition of bounded generalized pairs in the literature (cf.~\cite[Definition~1.1(5)]{Bir21b}).
\end{defn}

\begin{defn}\label{defn: lbb}
Let $\mathcal{P}$ be a set of projective quasi-normalized foliated structures. We say that $\mathcal{P}$ is \emph{log birationally bounded} if there exists a bounded set of projective quasi-normalized foliated structures $\mathcal{Q}$, such that for any $\Bb=(X,\Ff,B,\Mm)(t)\in\mathcal{P}$ with ambient variety $X$, there exists $\Bb'=(X',\Ff',B',\Mm')(t)\in\mathcal{Q}$ with irreducible ambient variety $X'$ together with a birational map $\phi\colon X\dashrightarrow X'$, such that 
\begin{enumerate}
    \item $\Supp B'\supset\phi_*\Supp B\cup\Exc(\phi^{-1})$,   
    \item $\phi_*\Ff=\Ff'$, and
    \item $\Mm=\Mm'$ as $\bb$-divisors. In particular, this implies that there exists a resolution of indeterminacies $p\colon W\rightarrow X$ and $q\colon W\rightarrow X'$ of $\phi$ such that $\Mm,\Mm'$ descend to $W$ and $\Mm_W=\Mm'_W$.
\end{enumerate}
\end{defn}

\begin{defn}[{cf.~\cite[Definition~3.5.3]{HMX14}}]
    Let $X$ be a normal projective variety and $D$ a $\mathbb Q$-Cartier $\mathbb Q$-divisor on $X$. We say that $D$ is \emph{potentially birational} if for any general closed points $x,y\in X$, possibly switching $x$ and $y$, we may find $0\leq\Delta\sim_{\mathbb Q}(1-\epsilon)D$ for some $0<\epsilon<1$, such that $(X,\Delta)$ is not klt at $y$, $(X,\Delta)$ is lc at $x$, and $x$ is an lc center of $(X,\Delta)$.
\end{defn}

\section{Precise adjunction formulae}\label{sec: precise adj}

The goal of this section is to prove the following two theorems on precise adjunction formulae. The first theorem shows that, for algebraically integrable adjoint foliated structures with DCC coefficients, adjunction preserves the DCC property.

\begin{thm}\label{thm: dcc adjunct to dcc}
Let $\Ii\subset [0,+\infty)$ be a DCC set. Then there exists a DCC set $\Ii'\subset [0,1]$ depending only on $\Ii$ satisfying the following. Let
$$\Bb/U:=(X,\Ff,B,\Mm)(t)/U$$
be an lc algebraically integrable normalized foliated structure such that $B,t\in\Ii$ and $\Mm=\sum\mu_j\Mm_j$, where each $\mu_j\in\Ii$ and each $\Mm_j$ is nef$/U$ $\bb$-Cartier. Let $\widetilde S$ be an irreducible component of $\lfloor B\rfloor$ and $S$ the normalization of $\widetilde S$. Let $\Aa:=\log\Bb$ and let $\Aa_S/U$ be the lc algebraically integrable adjoint foliated structure induced by adjunction (cf.~\cite[Theorem~1.8]{Cas+24})
$$K_{\Aa_S}:=K_{\Aa}|_S.$$
Then we may write
$$\Aa_S=(X,\Ff_S,B_S^{\ninv}+(1-t)B_S^{\inv},\Mm^S,t),$$ 
where $\Ff_S:=\Ff|_S$, $B_S\in\Ii'$, and $\Mm^S:=\sum\mu_j\Mm_j|_S$. In particular, if $t<1$, then
$$\Bb_S/U:=\exp(\Aa_S)/U=\left(S,\Ff_S,B_S,\Mm^S\right)(t)/U\in\Ii'\cup\Ii$$
and $\Bb_S$ is the normalized foliated structure induced by adjunction
$$K_{\Bb_S}:=K_{\Bb}|_S.$$
\end{thm}

The second theorem is more technical. Roughly speaking, it says that for algebraically integrable adjoint foliated structures with DCC coefficients, if the adjoint foliated structure induced on $S$ by adjunction has coefficients in a finite set, then the original adjoint foliated structure also has coefficients in a finite set ``near $S$''. For the nef part, one has to make precise what it means for the coefficients to be finite ``near $S$''. The precise statement is as follows.

\begin{thm}\label{thm: finite inversion of adjunction to finite 2}
 Let $\Ii\subset [0,+\infty)$ be a DCC set and $\Ii_0\subset [0,+\infty)$ a discrete set. Then there exists a finite set $\Ii_0'$ depending only on $\Ii$ and $\Ii_0$ satisfying the following. Assume that:
\begin{enumerate}
\item $\Bb/U:=(X,\Ff,B,\Mm)(t)/U$ is a $\mathbb Q$-factorial lc algebraically integrable normalized foliated structure,
\item $B\in\Ii$, $t\in\Ii\backslash\{1\}$, and $\Mm=\sum \mu_j\Mm_j$, where $\mu_j\in\Ii$ and $\Mm_j$ is nef$/U$ $\bb$-Cartier.
\item $D\in\Ii$ is an $\Rr$-divisor on $X$ and $\Nn=\sum\nu_k\Nn_k$, where each $\Nn_k$ is nef$/U$ $\bb$-Cartier and each $\nu_k\in\Ii$.
\item $S$ is a normal prime divisor on $X$ such that $S$ is an lc place of 
$$\Bb':=(\Bb,D,\Nn),$$
$S$ is not a component of $D$, and $(X,S)$ is plt.
\item $\Bb_S/U=\left(S,\Ff_S,\Delta_S,\Pp^S\right)(t)/U$ is the lc algebraically integrable normalized foliated structure induced by adjunction
$$K_{\Bb_S}:=K_{\Bb'}|_S,$$
where $\Pp^S=\sum \mu_j\Mm_j^S+\sum\nu_k\Nn_k^S$, $\Mm_j^S :=\Mm_j|_S$, and $\Nn_k^S :=\Nn_k|_S$ for any $j,k$.
\item $F$ is a general fiber of $g: S\rightarrow Z$ where $S\xrightarrow{g} Z\xrightarrow{\tau} U$ is the Stein factorization of the induced morphism $S\rightarrow Z$.
\item $\Delta_S|_F\in\Ii_0$.
\item For each $k$, either $\nu_k\in\Ii_0$, or $\Nn_{k}^S\big|_F\equiv\bm{0}$.
\item For any irreducible component $L$ of $D$, $L|_S\neq0$ over the generic point of $Z$.
\end{enumerate}
Then 
$$D\in\Ii_0'\quad \text{and} \quad \{\nu_k\mid (\Nn_{k,X}|_S)|_F\not\equiv 0\}\subset\Ii_0'.$$
\end{thm}

\subsection{Adjunction and singularities}

\begin{lem}\label{lem: CS25 3.13(4) general case}
Let $(X,S)$ be a plt pair where $S$ is a prime divisor, $\Ff$ a foliation on $X$, and $D$ a prime divisor on $S$ with generic point $\eta_D$. Let $r$ be the Cartier index of $K_X+S$ near $\eta_D$ and let $\Ff_S:=\Ff|_S$. Then there exists a non-negative integer $a$ such that, after possibly shrinking $X$ to a neighborhood of $\eta_D$, the following hold.
\begin{enumerate}
    \item For any prime divisor $P$ on $X$, $rD$ is Cartier, and we have
    $$(K_X+S)|_S=K_S+\frac{r-1}{r}D.$$
    \item We may write
    $$(K_{\Ff}+\epsilon_{\Ff}(S)S)|_S=K_{\Ff_S}+\frac{a+(r-1)\epsilon_{\Ff_S}(D)}{r}D.$$
    \item If $\epsilon_{\Ff}(S)=0$, then $a=0$ if and only if $\eta_D\not\in\Sing(\Ff)$.
    \item Let $V\neq S$ be an $\mathcal F$-invariant prime divisor such that $D \subset V\cap S$. If $\epsilon_{\mathcal F}(S) = 0$ and $\epsilon_{\mathcal F_S}(D) = 1$, then $\eta_D \in \Sing ~ \mathcal F$.
\end{enumerate}
\end{lem}
\begin{proof}
Since $(X,S)$ is plt, $X$ is klt near $\eta_D$ and $S$ is normal, so we may assume that $X$ is $\mathbb Q$-factorial by shrinking $X$ to a neighborhood of $\eta_D$.

(1) follows from \cite[16.6 Proposition]{Kol+92}. Moreover, by  \cite[16.6 Proposition]{Kol+92}, $X$ has a cyclic quotient singularity at $\eta_D$ and $K_X+S$ generates the local class group at $\eta_D$. Let $q\colon Y\rightarrow X$ be the index one cover associated to $K_X+S$ and let $T:=q^{-1}(S)$. By \cite[16.6 Proposition]{Kol+92}, $T$ is irreducible and normal, and the ramification index of $q|_T: T\rightarrow S$ at $E:=q^{-1}(D)$ is exactly $r$. 

Set $\Ff_Y:=q^{-1}\Ff$ and $\Ff_T:=\Ff_Y|_T$, then $K_{\Ff_Y}$ and $T$ are both Cartier. By \cite[Proposition-Definition~3.7]{CS25a}, we may write
$$(K_{\Ff_T}+\epsilon_{\Ff_Y}(T)T)|_T=K_{\Ff_T}+aE$$
where $a\geq 0$ is an integer.

We show that $a$ satisfies our requirements. Write
$$(K_{\Ff}+\epsilon_{\Ff}(S)S)|_S=K_{\Ff_S}+cD$$
near the generic point of $D$. Since $r$ is the ramification index of $q|_T$ at $E$, by the Riemann-Hurwitz formula (cf.~\cite[Lemma~3.4]{Dru21} and \cite[Proposition~2.2]{CS25b}), we have
$$a=rc-\epsilon_{\Ff_S}(D)(r-1)$$
or, equivalently,
$$c=\frac{a+\epsilon_{\Ff_S}(D)(r-1)}{r}.$$
This implies (2).

We now prove (3) and (4). By \cite[Corollary~5.14]{Dru21}, $\eta_D \in \Sing~\Ff$ if and only if $\eta_E \in \Sing~\Ff_Y$, where $\eta_E$ is the generic point of $E$. Thus possibly replacing $(X,S),D$ and $\Ff$ by $(Y,T),E$ and $\Ff_Y$, we may assume that $(X,S)$ is log smooth. 

We prove (3). It suffices to show that $a \neq 0$ if and only if $\eta_D \in \Sing(\Ff)$. By \cite[Proposition~3.14(1)]{CS25a} we see that if $\eta_D\in\Sing(\Ff)$, then $a \neq 0$. To show the converse direction, let $\phi\colon \Omega_X^k \to \mathcal O_X$ be the Pfaff field defining $\mathcal F$. By the construction of the different, cf.~\cite[Lemma~3.5, Remark~3.6]{CS25a}, since $S$ is normal, $a$ is the largest integer such that the image of $\phi|_S\colon \Omega^r_X|_S \to \mathcal O_S$ is contained in $\mathcal O_S(-aD)$.
Since $D\subset S$, if $a \neq 0$, then the image of $\phi$ is contained in the ideal of $D$ in $X$, i.e.\ $\eta_D \in \Sing(\mathcal F)$.

We are left to prove (4). By shrinking $X$ to a neighborhood of $\eta_D$, we may freely assume that $T_{\mathcal F}$ is locally free (cf.~\cite[Corollary~1.4]{Har80}) and generated by vector fields $\partial_1, \dots, \partial_m$ where $m = \rk \mathcal F$.
Since $D$ is the intersection of two different $\mathcal F$-invariant divisors, $D$ is also $\mathcal F$-invariant. Since $D$ is not $\mathcal F_S$-invariant, there exists a local section $\delta \in T_{\mathcal F_S}$ such that $\delta(I_{D/S}) \not\subset I_{D/S}$, where $I_{D/S}$ is the ideal of $D$ in $S$. If $D$ is not contained in $\Sing(\mathcal F)$, then the natural morphism $T_{\mathcal F}|_S \to T_{\mathcal F_S}$ is surjective, and we may let $\partial$ be a lift of $\delta$. Since $\delta(I_{D/S}) \not\subset I_{D/S}$,
it follows that $\partial(I_{D/X}) \not\subset I_{D/X}$, i.e.\ $D$ is not $\mathcal F$-invariant, a contradiction. Thus, $D$ is contained in $\Sing(\Ff)$, as claimed. 
\end{proof}

\begin{lem}\label{lem: invariant adjunction to non-invariant is invariant}
    Let $X$ be a $\mathbb Q$-factorial normal quasi-projective variety, $\Ff$ an algebraically integrable foliation on $X$, $S$ a normal prime non-$\Ff$-invariant divisor, $\Ff_S:=\Ff|_S$, and $D$ an $\Ff$-invariant prime divisor. Suppose that $(X,\mathcal F,S)$ is lc. Then any irreducible component of $D|_S$ is $\Ff_S$-invariant.
\end{lem}
\begin{proof}
Let $h\colon X'\rightarrow X$ be a $\mathbb Q$-factorial ACSS modification of $\Ff$ \cite[Theorem~2.5.1]{CHLX23} associated with the contraction $f: X'\rightarrow Z$. Let $\Ff':=h^{-1}\Ff$, $S':=h^{-1}_*S$, $D':=\Supp h^*D$, $h_S:=h|_{S'}: S'\rightarrow S$, and $\Ff_{S'}:=\Ff'|_{S'}$. Then $\Ff_{S'}$ is the foliation induced by $f|_{S'}: S'\rightarrow Z$ \cite[Proposition~3.2, Case II]{ACSS21}. For any $h$-exceptional prime divisor $E$ on $X'$, $E$ is an lc center of $\Ff$, hence $\Center_XE\not\subset S$.

Let $C$ be an irreducible component of $D\cap S$. Possibly shrinking $X$ to the generic point of $C$, we may assume that $h^{-1}$ is an isomorphism along $C$ and, in particular, $D'$ is $\Ff'$-invariant. Thus $D'\subset \Supp f^*D_Z$ for some prime divisor $D_Z\subset Z$. Let $C':=(h_S^{-1})_*C$. Then we have 
$$C'\subset D'\cap S'\subset \Supp f^*D_Z\cap S'=\Supp f|_{S'}^*D_Z,$$
hence $C'$ is $\Ff_{S'}$-invariant. Thus $C$ is $\Ff_S$-invariant, as claimed. 
\end{proof}

\begin{lem}\label{lem: restriction ii}
Let $(X,S)$ be a $\mathbb Q$-factorial plt pair and $\Ff$ an algebraically integrable foliation on $X$. Assume that $S$ is a prime non-$\Ff$-invariant divisor, $\Ff_S:=\Ff|_S$, and $D$ a non-$\Ff_S$-invariant prime divisor on $S$ with generic point $\eta_D$. Assume that
$$(K_{\Ff}+S)|_S=K_{\Ff_S}+\frac{r-1}{r}D$$
near the generic point of $D$, where $r$ is the Cartier index of $K_X+S$ near $\eta_D$. 

Then for any $\Ff$-invariant divisor $T$ on $X$, $\eta_D\not\in T\cap S$.
\end{lem}
\begin{proof}
Possibly taking a cover over the generic point of $\eta_D$, we may assume that $r=1$. Let $h\colon X'\rightarrow X$ be a $\mathbb Q$-factorial ACSS modification of $(X,\Ff,S)$ \cite[Theorem~2.5.1]{CHLX23} and write
    $$K_{\Ff'}+S'+\sum e_iE_i=h^*(K_{\Ff}+S)$$
    where the sum runs over the prime $h$-exceptional divisors, $\Ff':=h^{-1}\Ff$, and $S':=h^{-1}_*S$. Let $\Ff_{S'}:=\Ff'|_{S'}$. Then we have 
    $$\left(K_{\Ff'}+S'+\sum e_iE_i\right)\bigg|_{S'}=h|_{S'}^*(K_{\Ff}+S)=h|_{S'}^*K_{\Ff_S}.$$
By comparing the coefficients of $D$ and $(h|_{S'})^{-1}_*D$, we have that $E_i|_{S'}=0$ over the generic point of $D$ for any $i$ such that $e_i>0$. Therefore, $(X',\Ff',S')=(X',\Ff',S'+\sum e_iE_i)$ is lc near the generic point of $D'$ for any codimension $1$ subvariety $D'\subset S'$ such that $h(D')=D$. By Lemma~\ref{lem: invariant adjunction to non-invariant is invariant}, for any $\Ff'$-invariant divisor $T'$, $T'|_{S'}$ does not contain $D'$. In particular, $E_i|_{S'}$ does not contain $D'$ for any $D'$ and any $i$. Moreover, for any $\Ff$-invariant divisor $T$ on $X$, $h^{-1}_*T$ does not contain $D'$, hence $h^*T$ does not contain $D'$, so $T$ does not contain $D$. Thus $\eta_D\not\in T\cap S$.
\end{proof}

The following example shows that the condition ``$(X,\mathcal F, S)$ is lc'' in Lemma~\ref{lem: invariant adjunction to non-invariant is invariant} is necessary. It remains interesting to ask whether Lemma~\ref{lem: invariant adjunction to non-invariant is invariant} holds for non-algebraically integrable foliations.
    
\begin{ex}
Let $\Ff$ be the foliation on $\mathbb P^4$ given by the projection $\pi\colon \mathbb P^4 \dashrightarrow \mathbb P^1$ from a linear subspace $L\subset\mathbb P^4$ of dimension $2$. 
Let $S\subset\mathbb P^4$ be a general divisor containing $L$, and let $D$ be a general hyperplane containing $L$. Note that $S$ is not $\mathcal F$-invariant and $D$ is $\mathcal F$-invariant.
We claim that $L \subset S$ is not $\mathcal F_S$-invariant. Let $b\colon \widetilde{\mathbb P^4} \to \mathbb P^4$ be the blow-up at $L$, $\widetilde{S}$ the strict transform of $S$ on $\widetilde{\mathbb P^4}$, and set $\widetilde{L} = \widetilde{S} \cap b^{-1}(L)$.
Since $\mathcal F_{\widetilde{S}}$ is induced by the projection $\widetilde{S} \to \mathbb P^1$, $\widetilde{L}$ is not $\mathcal F_{\widetilde{S}}$-invariant. Since $\widetilde{L} \to L$ is finite, $L$ is not $\mathcal F_S$-invariant.

We conclude the example by observing that $L$ is an irreducible component of $D \cap S$.
\end{ex}

\subsection{Definition of the different and precise adjunction}

\begin{lem}\label{lem: bnef adjunction effectivity}
    Let $X\rightarrow U$ be a projective morphism between normal quasi-projective varieties, $\widetilde{S}$ a prime divisor on $X$, and $\nu: S\rightarrow\widetilde{S}$ the normalization of $\widetilde{S}$. Let $\Mm$ be a nef$/X$ $\bb$-divisor on $X$ such that $\Mm_X$ is $\Rr$-Cartier, and let $\Mm^S:=\Mm|_S$. Then
    $$\Mm^S_S+D_S=\Mm_X|_S$$
    for some $D_S\geq 0$.
\end{lem}
\begin{proof}
    Let $h\colon X'\rightarrow X$ be a log resolution of $(X,\widetilde S)$ such that $\Mm$ descends to $X'$. Let $S':=h^{-1}\widetilde{S}$ and let $h_S: S'\rightarrow S$ be the induced birational morphism. Then $\Mm^S=\overline{\Mm_{X'}|_{S'}}$ and $(h_S)_*\Mm^S_{S'}=\Mm^S_S$. By the negativity lemma,
    $$\Mm_{X'}=h^*\Mm_X-E$$
    for some $E\geq 0$ that is exceptional$/X$, hence
    \begin{align*}
    \Mm_X|_S&=(h_*h^*\Mm_X)|_S=(h_S)_*(h^*\Mm_X|_{S'})=(h_S)_*(\Mm_{X'}|_{S'}+E|_{S'})=(h_S)_*\Mm^S_{S'}+(h_S)_*E|_{S'}\\
    &=\Mm^S_S+(h_S)_*E|_{S'}\geq \Mm^S_S.
\end{align*}
The lemma follows.
\end{proof}

\begin{prodef}[{cf.~\cite[Proposition-Definition~3.7]{CS25a}}]\label{prodef: CS25 3.7 r gpair}
Assume that $X$ is a normal quasi-projective variety, $\Ff$ a foliation on $X$, and $\widetilde{S}$ a prime divisor on $X$ with normalization $\nu: S\rightarrow\widetilde S$. Assume that there exists a closed subset $Z\subset X$ such that $X\backslash Z$ is $\mathbb Q$-factorial, and $\dim (\widetilde S\cap Z)\leq\dim\widetilde S-2$. Then there exists a canonically defined function $$\Diff_{S}(\Ff,\cdot,\cdot,\cdot): \Weil_{\mathbb R}(X)_{\widetilde{S}=0}\times\bNef(X/X)\times [0,1]\rightarrow\Weil_{\mathbb R}(S)$$
defined in the following way (cf.~Notation~\ref{nota: special notation} for the definition of $\Weil_{\mathbb R}(X)_{\widetilde{S}=0}$).

For any $B\in\Weil_{\mathbb R}(X)_{\widetilde{S}=0}$, $\Mm\in\bNef(X/X)$, and $t\in [0,1]$, we define:
\begin{itemize}
    \item $\Diff_S(\Ff,B,\Mm,1):=\Diff(\Ff,B)+(\Mm_X|_S-\Mm^S_S)$ is the unique $\mathbb R$-divisor such that
    $$(K_{\Ff}+\epsilon_{\Ff}(\widetilde S)\widetilde S+B+\Mm_X)|_S=K_{\Ff_S}+\Diff_S(\Ff,B,\Mm,1)+\Mm^S_S$$
    where $\Mm^S:=\Mm|_S$, and $\Ff_S$ and $\Diff(\Ff,B)$ are the canonically defined foliation and different as in \cite[Proposition-Definition~3.7]{CS25a}. Note that \cite[Proposition-Definition~3.7]{CS25a} defines $\Diff(\Ff,B)$ only when $B\geq 0$ but $\Diff(\Ff,\cdot)$ linearly extends to a function over $\Weil_{\mathbb R}(X)_{\widetilde{S}=0}$.
    \item $\Diff_S(\Ff,B,\Mm,0)$ is the unique $\mathbb R$-divisor such that
        $$(K_{X}+\widetilde S+B+\Mm_X)|_S=K_{S}+\Diff_S(\Ff,B,\Mm,0)+\Mm^S_S$$
        by the usual adjunction formula for generalized pairs.
        \item $\Diff_S(\Ff,B,\Mm,t):=t\Diff_S(\Ff,B,\Mm,1)+(1-t)\Diff_S(\Ff,B,\Mm,0)$.
\end{itemize}
The following properties hold:
\begin{enumerate}
    \item If $B\geq 0$, then $\Diff_S(\Ff,B,\Mm,t)\geq 0$.
    \item $\Diff_S(\Ff,B,\Mm,\cdot)$ is a $\mathbb Q$-affine function.
    \item $\Diff_S(\Ff,B,\Mm,t)=\Diff_S(\Ff,0,0,t)+B|_S+\left(\Mm_X|_S-\Mm^S_S\right)$.
    \item $\Diff_S(\Ff,\cdot,\cdot,t)$ is a $\mathbb Q$-affine function for any $t\in\mathbb Q$.
    \item We have
    $$\Diff_S(\Ff,B^{\ninv}+(1-t)B^{\inv},\Mm,t)=t\Diff_S(\Ff,B^{\ninv},\Mm,1)+(1-t)\Diff_S\left(\Ff,B,\Mm,0\right).$$
\end{enumerate}
For any adjoint foliated structure
$$\Aa/U=\left(X,\Ff,B+\left(t\epsilon_{\Ff}\left(\widetilde{S}\right)+(1-t)\right)\widetilde{S},\Mm,t\right)\Big/U$$
such that $\widetilde{S}$ is not an irreducible component of $B$, we denote by 
$$\Diff_S(\Aa):=\Diff_S(\Ff,B,\Mm,t)$$
and say that
$$\Aa_S/U:=(S,\Ff_S:=\Ff|_S,\Diff_S(\Aa),\Mm^S:=\Mm|_S,t)/U$$
is the adjoint foliated structure induced by adjunction
$$K_{\Aa_S}:=K_{\Aa}|_S.$$
For any normalized foliated structure
$$\Bb/U=(X,\Ff,B+\widetilde{S},\Mm)(t)/U$$
such that $t<1$ and $\widetilde{S}$ is not an irreducible component of $B$, we denote by
$$\Bb_S:=\exp(\Aa_S)$$
where $\Aa_S/U$ is the adjoint foliated structure induced by adjunction $K_{\Aa_S}:=K_{\log\Bb}|_S$, and say that $\Bb_S/U$ is the normalized 
foliated structure induced by adjunction
$$K_{\Bb_S}:=K_{\Bb}|_S.$$
\end{prodef}
\begin{proof}
We first assume that $X$ is $\mathbb Q$-factorial.

(1) 
It follows from \cite[Proposition-Definition~3.7]{CS25a} that $\Diff(\Ff,B)\geq 0$ if $B\geq 0$. By Lemma~\ref{lem: bnef adjunction effectivity}, $\Mm_X|_S-\Mm^S_S\geq 0$, hence $\Diff_S(\Ff,B,\Mm,1)\geq 0$. By the usual adjunction formula for generalized pairs (cf.~\cite[Definition~4.7]{BZ16}), $\Diff_S(\Ff,B,\Mm,0)\geq 0$, so $\Diff_S(\Ff,B,\Mm,t)\geq 0$.

(2) By definition.

(3) By \cite[Proposition-Definition~3.7]{CS25a}, $\Diff(\Ff,B)=\Diff(\Ff,0)+B|_S$, so $$\Diff_S(\Ff,B,\Mm,1)=\Diff_S(\Ff,0,0,1)+B|_S+\left(\Mm_X|_S-\Mm^S_S\right).$$ By the usual adjunction formula for generalized pairs, $\Diff_S(\Ff,B,\Mm,0)=\Diff_S(\Ff,0,0,0)+B|_S+\left(\Mm_X|_S-\Mm^S_S\right)$. This implies (3).

(4)(5) follow from (3).

Now for the general case, let $X^0:=X\backslash Z$, $\widetilde S^0:=X^0\cap\widetilde{S}$, and $S^0:=\nu^{-1}(\widetilde{S})$. By the $\mathbb Q$-factorial case, there exists a canonically defined function
$$\Diff_{S^0}(\Ff|_{X^0},\cdot,\cdot,\cdot):\Weil_{\mathbb R}(X^0)_{\widetilde{S}^0}\times\bNef(X^0/X^0)\times [0,1]\rightarrow\Weil_{\mathbb R}(S^0)$$
Let $i: \Weil_{\mathbb R}(S^0)\rightarrow \Weil_{\mathbb R}(S)$ be the natural isomorphism. Then we may define
$$\Diff_{S}(\Ff,B,\Mm,t):=i_*\Diff_{S^0}(\Ff|_{X^0},B|_{X^0},\Mm|_{X^0},t).$$
All the properties (1-5) follow by construction.
\end{proof}

For the reader's convenience and for the simplicity of the statements, we will consider the following set-up in the rest of this section:

\begin{setup}\label{setup: pa no m}
$X$, $S$, $\Ff$, $B_i$, $C_k$, $\Mm_j$, $\Ff_S$, $\Mm^S_j$, $B(\cdot)$, $B(\cdot,\cdot)$, $\Mm(\cdot)$, $\Aa(\cdot,\cdot)$, $K(\cdot,\cdot)$, $B_S(\cdot,\cdot)$, $\Aa_S(\cdot,\cdot)$, $K_S(\cdot,\cdot)$ are as follows:
\begin{enumerate}
\item $(X,S)$ is $\mathbb Q$-factorial plt and $S$ is a prime divisor.
\item $\Ff$ is a foliation on $X$.
\item $S,B_1,\dots,B_l,C_1,\dots,C_m$ are distinct prime divisors on $X$.
\item $B_1,\dots,B_l$ are not $\Ff$-invariant and $C_1,\dots,C_m$ are $\Ff$-invariant.
\item $\Mm_1,\dots,\Mm_n$ are nef$/X$ $\bb$-Cartier $\bb$-divisors on $X$.
\item $\Ff_S:=\Ff|_S$ and $\Mm^S_j:=\Mm_j|_S$ for any $j$.
\item For any prime divisor $D$ on $S$, we denote by $r_D$ the local Cartier index of $X$ near the generic point of $D$,
$$p_{i,D}:=r_D\mult_{D}(B_i|_S),q_{k,D}:=r_D\mult_{D}(C_k|_S),\lambda_{j,D}:=r_D\mult_D\left(\Mm_{j,X}|_S-\Mm^S_{j,S}\right),$$
and by Lemma~\ref{lem: CS25 3.13(4) general case}, we may let $a_D$ be the non-negative integer such that
$$(K_{\Ff}+\epsilon_{\Ff}(S)S)|_S=K_{\Ff_S}+\frac{a_D+\epsilon_{\Ff_S}(D)(r_D-1)}{r_D}D$$
near the generic point of $D$.
\item For any $\bm{v}:=(b_1,\dots,b_l,c_1,\dots,c_m,\mu_1,\dots,\mu_n)\in\mathbb R^{l+m+n}$ and any $t\in [0,1]$:
\begin{enumerate}
\item $B(\bm{v}):=S+\sum_{i=1}^lb_iB_i+\sum_{k=1}^mc_kC_k$.
\item 
$$B(\bm{v},t):=B(\bm{v})^{\ninv}+(1-t)B(\bm{v})^{\inv}=(t\epsilon_{\Ff}(S)+(1-t))S+\sum_{i=1}^lb_iB_i+(1-t)\sum_{k=1}^mc_kC_k.$$ 
\item $\Mm(\bm{v}):=\sum_{j=1}^n\mu_j\Mm_j$.
\item $\Aa(\bm{v},t):=(X,\Ff,B(\bm{v},t),\Mm(\bm{v}),t)$ and $K(\bm{v},t):=K_{\Aa(\bm{v},t)}$.
\item $B_S(\bm{v},t):=\Diff_S(\Aa(\bm{v},t))$.
\item $\Aa_S(\bm{v},t):=\Aa(\bm{v},t)|_S$ and $K_S(\bm{v},t):=K_{\Aa_S(\bm{v},t)}$.
\end{enumerate}
\end{enumerate}
\end{setup}

\begin{prop}\label{prop: explicit formula}
   Notation and conditions as in Set-up~\ref{setup: pa no m}. Let $D$ be a prime divisor on $S$. Let $r:=r_D$, $\epsilon:=\epsilon_{\Ff_S}(D)$, $p_i:=p_{i,D},q_k:=q_{k,D},\lambda_j:=\lambda_{j,D}$, and $a:=a_D$. Then for any $\bm{v}=(b_1,\dots,b_l,c_1,\dots,c_m,\mu_1,\dots,\mu_n)\in\mathbb R^{l+m+n}$ and $t\in [0,1]$, we have
   $$\mult_DB_S(\bm{v},t)=\frac{1}{r}\left(t(a+\epsilon(r-1))+\sum_{i=1}^lp_ib_i+(1-t)\left(r-1+\sum_{k=1}^mq_kc_k\right)+\sum_{j=1}^n\lambda_j\mu_j\right).$$
\end{prop}
\begin{proof} 
By Lemma~\ref{lem: CS25 3.13(4) general case}, near the generic point of $D$, we have
$$(K_{X}+S)|_S=K_S+\frac{r-1}{r}D$$
and
$$(K_{\Ff}+\epsilon_{\Ff}(S)S)|_S=K_{\Ff_S}+\frac{a+\epsilon(r-1)}{r}D$$
for some non-negative integer $a$. By Proposition-Definition~\ref{prodef: CS25 3.7 r gpair}~(3), we have
\begin{align*}
    B_S(\bm{v},t)&=\Diff_S(\Ff,0,0,t)+\sum_{i=1}^lb_iB_i|_S+(1-t)\sum_{k=1}^mc_kC_k|_S+\sum_{j=1}^n\mu_j\left(\Mm_{j,X}|_S-\Mm_{j,S}^S\right)\\
    &=\left(t\frac{a+\epsilon(r-1)}{r}+(1-t)\frac{r-1}{r}+\sum_{i=1}^lb_i\frac{p_i}{r}+(1-t)\sum_{k=1}^mc_k\frac{q_k}{r}+\sum_{j=1}^n\mu_j\frac{\lambda_j}{r}\right)D
\end{align*}
near the generic point of $D$. The proposition follows.
\end{proof}

\subsection{Adjunction for structures with DCC coefficients}

\begin{thm}\label{thm: dcc precise adjunction}
Let $\Ii\subset [0,+\infty)$ be a DCC set and $\Ii_0\subset [0,+\infty)$ a finite set. Then there exist a DCC set $\Ii'\subset [0,1]$ and a finite set $\Ii_0'$ depending only on $\Ii$ and $\Ii_0$ satisfying the following.

Notation and conditions as in Set-up~\ref{setup: pa no m}. Assume that $\bm{v},t\in\Ii$. Let $D$ be a prime divisor on $S$. Let $r:=r_D$, $\epsilon:=\epsilon_{\Ff_S}(D)$, $p_i:=p_{i,D},q_k:=q_{k,D},\lambda_j:=\lambda_{j,D}$, and $a:=a_D$. Assume that $\mult_DB_S(\bm{v},t)\leq t\epsilon+(1-t)$.
\begin{enumerate}
    \item Assume that $\epsilon=0$ and $t<1$. Then:
    \begin{enumerate}
        \item We have
            $\frac{1}{1-t}\mult_DB_S(\bm{v},t)\in\Ii'.$
        \item Assume that $\frac{1}{1-t}\mult_DB_S(\bm{v},t)\in\Ii_0$. Then for any $i,k,j$,
        $$ta,tp_ib_i,q_kc_k,t\lambda_j\mu_j\in\Ii_0'.$$
      In particular, either $t\in\Ii_0'$, or $s\mapsto\frac{1}{1-s}\mult_DB_S(\bm{v},s)$ is a constant function.
    \end{enumerate}
\item Assume that $\Ff$ is algebraically integrable, $\Aa(\bm{v},t)$ is lc, and $\epsilon=1$. Then:
  \begin{enumerate}
        \item We have
            $\mult_DB_S(\bm{v},t)\in\Ii'.$
        \item Assume that $\mult_DB_S(\bm{v},t)\in\Ii_0$.
        \begin{enumerate}
            \item If $a=0$, then $q_k=0$ and $p_ib_i,\lambda_j\mu_j\in\Ii_0'$.
            \item If $a>0$, then for any $i,j$,
            $$t(a-1),p_ib_i,\lambda_j\mu_j\in\Ii_0',$$
            and either $t=1$ or $q_kc_k\in\Ii_0'$ for any $k$.
        \end{enumerate}
In particular, either $t\in\Ii_0'$, or $s\mapsto\mult_DB_S(\bm{v},s)$ is a constant function.
    \end{enumerate}
\end{enumerate}
\end{thm}
\begin{proof}
Write $\bm{v}=(b_1,\dots,b_l,c_1,\dots,c_m,\mu_1,\dots,\mu_n)$. By Proposition~\ref{prop: explicit formula},
$$\mult_DB_S(\bm{v},s)=\frac{1}{r}\left(s(a+\epsilon(r-1))+\sum_{i=1}^lp_ib_i+(1-s)\left(r-1+\sum_{k=1}^mq_kc_k\right)+\sum_{j=1}^n\lambda_j\mu_j\right)$$
for any $s\in [0,1]$.

(1) In this case,
$$1\geq \frac{1}{1-t}\mult_DB_S(\bm{v},t)=\frac{r-1+\frac{t}{1-t}a+\frac{1}{1-t}\left(\sum_{i=1}^lp_ib_i+\sum_{j=1}^n\lambda_j\mu_j\right)+\sum_{k=1}^mq_kc_k}{r}.$$
The rest are elementary computations.

(2) We first assume that $a=0$. In this case, by Lemma~\ref{lem: CS25 3.13(4) general case}~(3)(4) when $S$ is $\Ff$-invariant and by Lemma~\ref{lem: restriction ii} when $S$ is not $\Ff$-invariant, we see that $S$ is the only possible $\Ff$-invariant divisor which contains $D$, so $q_k=0$ for any $k$. Therefore,
$$1\geq\mult_DB_S(\bm{v},t)=\frac{1}{r}\left(r-1+\sum_{i=1}^lp_ib_i+\sum_{j=1}^n\lambda_j\mu_j\right)$$
and the rest are elementary computations. 

Now we assume that $a>0$. Then $a\geq 1$. In this case, since $\Aa(\bm{v},t)$ is lc, by \cite[Proposition~3.3]{Cas+24}, $\Aa(\bm{v},0)$ is lc, so
$$1\geq\mult_DB_S(\bm{v},0)=\frac{1}{r}\left(r-1+\sum_{i=1}^lp_ib_i+\sum_{k=1}^mq_kc_k+\sum_{j=1}^n\lambda_j\mu_j\right).$$
Therefore, $\sum_{k=1}^mq_kc_k\leq 1\leq a$, so
$$1\geq\mult_DB_S(\bm{v},t)=\frac{1}{r}\left(r+t(a-1)-(1-t)\left(1-\sum_{k=1}^mq_kc_k\right)+\sum_{i=1}^lp_ib_i+\sum_{j=1}^n\lambda_j\mu_j\right)$$
and the rest are elementary computations. 
\end{proof}

\begin{proof}[Proof of Theorem~\ref{thm: dcc adjunct to dcc}]
We may assume that $0,1\in\Ii$. By \cite[Theorem~6.0.1]{CHLX23}, we may assume that $t<1$. If $\Aa$ is $\mathbb Q$-factorial qdlt, then $\Aa_S$ is lc by \cite[Theorem~1.8]{Cas+24}, and the theorem follows from Theorem~\ref{thm: dcc precise adjunction}. For the general case, by \cite[Theorem~1.9]{Cas+24}, we may let $h\colon \Aa'\rightarrow\Aa$ be a $\mathbb Q$-factorial qdlt modification of $\Aa$, $S':=h^{-1}_*\widetilde{S}$, and $h_S: S'\rightarrow S$ the induced birational morphism. Let $\Aa_{S'}$ be the adjoint foliated structure induced by adjunction $K_{\Aa_{S'}}:=K_{\Aa'}|_{S'}$. Then we have
$$K_{\Aa_S}=(h_S)_*K_{\Aa_{S'}}.$$
The general case now follows from the $\mathbb Q$-factorial qdlt case.
\end{proof}

\begin{rem}
The proof of Theorems~\ref{thm: dcc precise adjunction} and~\ref{thm: dcc adjunct to dcc} indicates that, if $\Ii=\overline{\Ii}\subset\mathbb Q$, then we may choose $\Ii'$ so that $\Ii'=\overline{\Ii'}\subset\mathbb Q$. We do not need this fact in the paper, but it is useful for future applications.
\end{rem}

\begin{lem}\label{lem: non-trivial of b-divisors}
Let $\pi: X\rightarrow U$ be a contraction between normal quasi-projective varieties and $\Mm$ a nef$/U$ $\bb$-divisor on $X$. Let $F$ be a general fiber of $\pi$. Assume that $\Mm_X|_F\equiv 0$. Then $\Mm|_F\equiv \bm{0}$. 
\end{lem}
\begin{proof}
We may assume that $\dim X>\dim U$. Let $f: Y\rightarrow X$ be a projective birational morphism such that $\Mm$ descends to $Y$. Since $F$ is a general fiber of $\pi$, there exists a general fiber $F_Y$ of $\pi\circ f: Y\rightarrow U$ 
such that $f(F_Y)=F$. Let $g: F_Y\rightarrow F$ be the induced birational morphism. We have
$$g_*(\Mm_Y|_{F_Y})=\Mm_X|_F\equiv 0.$$
Since $\Mm_Y|_{F_Y}$ is nef$/U$, $\Mm_Y|_{F_Y}$ is nef$/F$. By the negativity lemma, we have
$$\Mm_Y|_{F_Y}=g^*(\Mm_X|_F)-E\equiv -E$$
for some $E\geq 0$. Thus $E=0$, hence 
$$\Mm_Y|_{F_Y}=g^*(\Mm_X|_F)\equiv 0.$$ 
Since $\Mm|_F=\overline{\Mm_Y|_{F_Y}}$, it follows that $\Mm|_F\equiv\bm{0}$.
\end{proof}

\begin{proof}[Proof of Theorem~\ref{thm: finite inversion of adjunction to finite 2}]
Since $S$ is not a component of $D$ and is an lc place of $\Bb'$,
$S$ is a component of $\lfloor B\rfloor$.
We may write $$B=S+\sum b_{i}B_{i}+\sum c_{l}C_{l}$$
and
$$D^{\ninv}=\sum d_{i}B_{i},D^{\inv}=\sum e_{l}C_{l}$$
where $b_{i},c_{l},d_{i},e_{l}\geq 0$ and $S,B_{i},C_{l}$ are distinct prime divisors.

For any prime divisor $P$ on $S$, let $r_{P}$ be the local Cartier index of $X$ near the generic point of $P$,
$p_{i,P}:=r_{P}\mult_{P}(B_{i}|_{S})$, $q_{l,P}:=r_{P}\mult_{P}(C_{l}|_{S})$, $\lambda_{j,P}:=r_{P}\mult_{P}(\Mm_{j,X}|_{S}-\Mm_{j,S}^S)$, $\xi_{k,P}:=r_{P}\mult_{P}(\Nn_{k,X}|_{S}-\Nn_{k,S}^S)$, and $a_{P}$ the unique non-negative integer such that
$$(K_{\Ff}+\epsilon_{\Ff}(S)S)|_{S}=K_{\Ff_{S}}+\frac{a_{P}+\epsilon_{\Ff_{S}}(P)(r_{P}-1)}{r_{P}}P$$
near the generic point of $P$ and whose existence is guaranteed by Lemma~\ref{lem: CS25 3.13(4) general case}. Then by Proposition~\ref{prop: explicit formula}, we have
\begin{align*}
    &(tK_{\Ff}+(1-t)K_{X}+(B+D)^{\ninv}+(1-t)(B+D)^{\inv}+\Mm_{X}+\Nn_{X})|_{S}\\
    =&tK_{\Ff_{S}}+(1-t)K_{S}\\
    +&\sum_{P}\frac{1}{r_{P}}\Bigg(t(a_{P}+\epsilon_{\Ff_{S}}(P)(r_{P}-1))+\sum p_{i,P}(b_{i}+d_{i})\\
    +&(1-t)\left(r_{P}-1+\sum q_{l,P}(c_{l}+e_{l})\right)+\sum \lambda_{j,P}\mu_{j}+\sum\xi_{k,P}\nu_k\Bigg)P\\
    +&\sum \mu_{j}\Mm^S_{j,S}+\sum \nu_{k}\Nn^S_{k,S}.
\end{align*}
We have the following.
\begin{itemize}
\item By our assumption (9), if $d_i\neq0$ (resp. $e_j\neq0$), then $B_i|_S\neq0$ (resp. $C_j|_S\neq0$) over the generic point of $Z$, hence $p_{i,P}\neq0$ (resp. $q_{j,P}\neq0$) for some $P$ that is horizontal$/Z$.
\item By Theorem~\ref{thm: dcc precise adjunction}, the following set
$$\{a_P,p_{i,P}(b_i+d_i),q_{l,P}(c_l+e_l),\xi_{k,P}\nu_k\mid P\text{ is horizontal}/Z\}$$
is a subset of a finite set depending only on $\Ii$ and $\Ii_0$.
\item For any $k$ such that $(\Nn_{k,X}|_S)|_F\not\equiv 0$, we have
$$\sum_P \frac{\xi_{k,P}\nu_k}{r_P}P|_F+\nu_k\Nn_{k,S}^S|_F=\left(\sum_P \frac{\xi_{k,P}\nu_k}{r_P}P+\nu_k\Nn_{k,S}^S\right)\Bigg|_F=(\Nn_{k,X}|_S)|_F\not\equiv0,$$
so either $\xi_{k,P}\neq0$ for some $P$ that is horizontal$/Z$, or $\Nn_{k,S}^S|_F\not\equiv0$ for some $k$. In the latter case, $\Nn_{k}^S|_F\not\equiv\bm{0}$ by Lemma~\ref{lem: non-trivial of b-divisors}, so $\nu_k\in\Ii_0$.
\end{itemize}
These indicate that, for any $i$ (resp. $l,k$) such that $d_i\neq0$ (resp. $e_l\neq0$, $(\Nn_{k,X}|_S)|_F\not\equiv0$), $d_i$ (resp. $e_l,\nu_k$) belong to a finite set. The theorem follows.
\end{proof}

\section{Global ACC to local ACC}\label{sec: global to local}

We now prove the ACC for interpolated lc thresholds (Theorem~\ref{thm: main ACC}). We do so by induction, simultaneously with the global ACC (Theorem~\ref{thm: global ACC main}). For the induction, we consider the following more general versions of these two theorems.

\begin{thm}[ACC for interpolated lc thresholds, general version]\label{thm: acc lct general version}
Let $d$ be a positive integer and $\Ii\subset [0,+\infty)$ a DCC set. Then there exists a positive real number $\tau=\tau_{\ilct}(d,\Ii)\in (0,1)$ depending only on $d$ and $\Ii$ satisfying the following. 

Let $\Bb=(X,\Ff,B,\Mm)(t)/U$ and $\Bb'=(X,\Ff,B',\Mm')(t')/U$ be two algebraically integrable normalized foliated structures of dimension $d$. Assume that
\begin{enumerate}
\item $\Bb/U\in\Ii$,
\item $\Bb/U\geq\Bb'/U\geq\Bb^{1-\tau}/U$, and
\item $\Bb'$ is lc.
\end{enumerate}
Then $\Bb$ is lc.
\end{thm}

\begin{thm}[Global ACC]\label{thm: global acc general}
Let $d$ be a positive integer and $\Ii\subset [0,1]$ a DCC set. Then there exists a finite set $\Ii_0=\Ii_0(d,\Ii)$ depending only on $d$ and $\Ii$ satisfying the following. 

Assume that $\Bb/U=(X,\Ff,B,\Mm)(t)/U$ is an lc algebraically integrable normalized foliated structure of dimension $d$ such that the associated morphism $\pi: X\rightarrow U$ is a contraction. Let $F$ be a general fiber of $\pi$. Assume that
\begin{itemize}
\item $K_{\Bb}|_F\equiv0$, and
\item $B,t\in\Ii$ and $\Mm=\sum\mu_j\Mm_j$, where each $\mu_j\in\Ii$ and each $\Mm_j$ is nef$/U$ $\bb$-Cartier.
\end{itemize}
Then:
\begin{enumerate}
    \item $B^{\ninv}|_F\in\Ii_0$, and $B|_F\in\Ii_0$ if $t<1$.
    \item $\{\mu_j\mid \Mm_j|_{F}\not\equiv\bm{0}\}\subset\Ii_0$.
    \item Let $\Bb(s):=(X,\Ff,B,\Mm)(s)$ for any $s\in [0,1]$. Assume that $K_{\Bb(s)}$ is $\mathbb R$-Cartier for any $s\in [0,1]$. Then either $t\in\Ii_0$, or $K_{\Bb(s)}|_F\equiv 0$ for any $s\in [0,1]$.
\end{enumerate}
\end{thm}

In the rest of this section, we prove Theorem~\ref{thm: acc lct general version} assuming Theorem~\ref{thm: global acc general} in lower dimensions; that is, Theorem~\ref{thm: global acc general} in dimension $\leq d-1$ implies Theorem~\ref{thm: acc lct general version} in dimension $d$.

\subsection{Log canonicity under order}\label{subsec: log canonicity under order} Recall that we have defined an order $\geq$ between normalized foliated structures (Definition~\ref{defn: order of nfs}). In this subsection, we show that log canonicity of algebraically integrable normalized foliated structures is preserved under this order (Lemma~\ref{lem: lc preserved under inequality}).

\begin{lem}\label{lem: neg lem for generalized boundary}
    Let $X$ be a normal quasi-projective variety, $B\geq 0$ an $\mathbb R$-divisor on $X$, and $\Mm$ a nef$/X$ $\bb$-divisor on $X$, such that $B+\Mm_X$ is $\mathbb R$-Cartier. Let $h\colon X'\rightarrow X$ be a projective birational morphism and $B':=h^{-1}_*B$. Then
    $$h^*(B+\Mm_X)\geq B'+\Mm_{X'}.$$
\end{lem}
\begin{proof}
We may assume that $h$ is a resolution of $X$ and $\Mm$ descends to $X'$, and write
$$h^*(B+\Mm_X)=B'+\Mm_{X'}+E.$$
Then $E$ is exceptional$/X$. Since $B'\geq 0$, for any irreducible component $D$ of $E$ and any very general exceptional$/X$ irreducible curve $C\subset D$, we have $B'\cdot C\geq 0$. Since $\Mm_{X'}$ is nef$/X$, $\Mm_{X'}\cdot C\geq 0$. Thus $E\cdot C\leq 0$. By \cite[Lemma~3.3]{Bir12}, it follows that $E\geq 0$.
\end{proof}

\begin{lem}\label{lem: lc preserved under inequality}
Let $\Bb/U$ be an algebraically integrable normalized foliated structure and $\Bb'/U$ an algebraically integrable sub-normalized foliated structure, such that $\Bb/U\geq\Bb'/U$ and $\Bb/U$ is lc. Let $t$ and $t'$ be the parameters of $\Bb$ and $\Bb'$, respectively. Assume that either $t<1$ or $t=t'=1$. Then:
\begin{enumerate}
    \item $\Bb'$ is lc.
    \item Assume that $t'>0$ or $t=t'=0$. Then any lc place of $\Bb'$ is an lc place of $\Bb$.
\end{enumerate}
\end{lem}
\begin{proof}
We write $\Bb=(X,\Ff,B,\Mm)(t)$ and $\Bb'=(X,\Ff,B',\Mm')(t')$. 

We first prove the lemma when $t=t'$. In this case, $(X,\Ff,B,\Mm)(t)$ is lc, $(B-B')^{\ninv}+(1-t)(B-B')^{\inv}\geq 0$, $\Mm-\Mm'$ is nef$/U$, and $(B-B')^{\ninv}+(1-t)(B-B')^{\inv}+(\Mm_X-\Mm'_X)$ is $\mathbb R$-Cartier. By Lemma~\ref{lem: neg lem for generalized boundary},
$$a(D,\Bb)\leq a(D,\Bb')$$
for any prime divisor $D$ over $X$, and the lemma follows. Therefore, in the following, we may assume that $t>t'$. 

We first prove (1). We have $1>t>t'$. The proof is very similar to \cite[Proof of Proposition~3.3]{Cas+24}. We have $t=s\cdot 1+(1-s)\cdot t'$ where $s=\frac{t-t'}{1-t'}$. Write
$$\Aa:=\log\Bb:=(X,\Ff,\Delta,\Mm,t),\quad \Aa':=\log\Bb':=(X,\Ff,\Delta',\Mm',t').$$
Let $s:=\frac{t-t'}{1-t'}$. Then $t=s\cdot 1+(1-s)\cdot t'$. Let $\Delta''$ be the unique $\mathbb R$-divisor such that $s\Delta''+(1-s)\Delta'=\Delta$ and let $\Mm''$ be the unique $\bb$-divisor such that $s\Mm''+(1-s)\Mm'=\Mm$. Since $\Bb/U\geq\Bb'/U$,
$$\Mm''-\Mm=\frac{1-s}{s}(\Mm-\Mm')$$
is nef$/U$. Hence $\Mm''$ is nef$/U$ as well. Since
$$\Delta^{\ninv}+\frac{1}{1-t}\Delta^{\inv}\geq\Delta'^{\ninv}+\frac{1}{1-t'}\Delta'^{\inv},$$
we have 
$$\Delta''=\frac{1}{s}(\Delta-(1-s)\Delta')\geq 0\text{ and }\Delta''^{\ninv}\geq\Delta^{\ninv}.$$
We let $\Aa'':=(X,\Ff,\Delta'',\Mm'')$ and let $h\colon (Y,\Ff_Y,B_Y'',\Mm'';G)\rightarrow (X,\Ff,\Delta'',\Mm'')$ be a $\mathbb Q$-factorial proper ACSS modification of $(X,\Ff,\Delta'',\Mm'')$ (cf.~\cite[Theorem~8.2.2]{CHLX23}). We have
$$B_Y''=h^{-1}_*(\Delta''^{\ninv}\wedge\Supp \Delta''^{\ninv})+\Exc(h)^{\ninv},$$
and for any $h$-exceptional prime divisor $E$, we have
$$a(E,\Aa'')\leq -\epsilon_{\Ff}(E),$$
since $h$ is an ACSS modification of $\Aa''$ and only extracts nklt places of $\Aa''$. 
Since $\Aa$ is lc, we have
$$a(E,\Aa)\geq -t\epsilon_{\Ff}(E)-(1-t).$$
Since $\Aa=s\Aa''+(1-s)\Aa'$, we have
$$a(E,\Aa)=sa(E,\Aa'')+(1-s)a(E,\Aa').$$
Thus
$$a(E,\Aa')\geq\frac{-t\epsilon_{\Ff}(E)-(1-t)+s\epsilon_{\Ff}(E)}{1-s}=-t'\epsilon_{\Ff}(E)-(1-t')$$
and equality holds if and only if $a(E,\Aa'')=-\epsilon_{\Ff}(E)$ and $a(E,\Aa)=-t\epsilon_{\Ff}(E)-(1-t)$. Let 
$$\Bb_{Y}':=(Y,\Ff_Y,B_{Y}':=h^{-1}_*B'+\Exc(h),\Mm')(t'),\Bb_Y:=(Y,\Ff_Y,B_{Y}:=h^{-1}_*B+\Exc(h),\Mm)(t).$$
Then we have 
$$K_{\Bb_{Y}'}\geq h^*K_{\Bb'}\quad \text{ and }\quad K_{\Bb_Y}\geq h^*K_{\Bb}.$$
\begin{claim}\label{claim: by+g and by'}
$B_Y''+G\geq B_{Y}\geq B_Y'$.
\end{claim}
\begin{proof}
    For any prime divisor $D$ that is an irreducible component of $B_Y$, we distinguish four cases.
    \medskip

    \noindent\textbf{Case 1.} $D$ is an irreducible component of $\Exc(h)^{\inv}$. Then
    $$\mult_DB_{Y}=1=\mult_DG=\mult_D(B_Y''+G).$$
    
    \medskip

    \noindent\textbf{Case 2.} $D$ is an irreducible component of $\Exc(h)^{\ninv}$. Then
    $$\mult_DB_{Y}=1=\mult_DB_Y''=\mult_D(B_Y''+G).$$

    \medskip

    \noindent\textbf{Case 3.} $D$ is an irreducible component of $h^{-1}_*B$ that is $\Ff_Y$-invariant. Then
$$\mult_D(B_Y''+G)=\mult_DG=1\geq\mult_{h_*D}B=\mult_DB_{Y}.$$

    \medskip

    \noindent\textbf{Case 4.} $D$ is an irreducible component of $h^{-1}_*B$ that is not $\Ff_Y$-invariant. Then
    $$\mult_D(B_Y''+G)=\mult_DB_Y''=\min\{\mult_{h_*D}\Delta''^{\ninv},1\}\geq\mult_{h_*D}\Delta^{\ninv}=\mult_DB_{Y}.$$
    The claim follows.
\end{proof}
\noindent\textit{Proof of Lemma~\ref{lem: lc preserved under inequality} continued.} We let
$$\Bb_Y(s):=(Y,\Ff_Y,B_Y''+G,\Mm'')(s)$$
for any $s\in [0,1]$. Then $\Bb_Y(s)$ is $\mathbb Q$-factorial lc for any $s\in [0,1]$. By Claim~\ref{claim: by+g and by'}, $\Bb_Y(t')/U\geq\Bb_Y'/U$. By the $t=t'$ case, $\Bb_{Y}'$ is lc, hence $\Bb'$ is lc. This implies (1).

Now we prove (2). We let $g: V\rightarrow X$ be a $\mathbb Q$-factorial qdlt modification of $\Bb$ and let $\Bb_V:=g^*\Bb:=(V,\Ff_V,B_V,\Mm)(t)$. Then $B_V=g^{-1}_*B+\Exc(g)$. By (1), we have that $\Bb_V(s):=(V,\Ff_V,B_V,\Mm)(s)$ is lc for any $s\in [0,t]$. Since
$$\Bb_V(s)=\Bb_V^{\frac{s}{t}}\cdot\Bb_V(0)^{\frac{t-s}{t}},$$
for any $s\in (0,t]$, any lc place of $\Bb_V(s)$ is an lc place of $\Bb_V$. Since $B\geq B'$, $B_V\geq g^{-1}_*B'+\Exc(g)$, so any lc place of $\Bb'$ is an lc place of $\Bb_V(t')$ and hence an lc place of $\Bb_V$, hence an lc place of $\Bb$. (2) follows.
\end{proof}

\subsection{Reduction to the klt case}

In this subsection, we reduce Theorem~\ref{thm: acc lct general version} to the case when $X$ is $\mathbb Q$-factorial klt and $\Bb'=\Bb^s$ for some $s>0$.

\begin{lem}\label{lem: acc reduce to klt}
    Assume that Theorem~\ref{thm: acc lct general version} holds in dimension $d$ when $X$ is $\mathbb Q$-factorial klt. Then Theorem~\ref{thm: acc lct general version} holds in dimension $d$.
\end{lem}
\begin{proof}
Pick $\tau$ so that Theorem~\ref{thm: acc lct general version} holds in dimension $d$ when $X$ is $\mathbb Q$-factorial klt. We may assume that $\tau<\frac{1}{2}$. By Lemma~\ref{lem: construct smaller bb}, there exists $\Bb''$ such that
$$\Bb/U\geq\Bb'/U\geq\Bb''/U\geq\Bb^{\frac{1}{2}}/U$$
and a real number $\mu\in (0,1)$ such that $\Bb'=\Bb^\mu\cdot\Bb''^{1-\mu}$. Let $t''$ be the parameter of $\Bb''$. Then either $t=t'=t''=1$, or $t=t'=t''=0$, or $0<t''\leq t'<1$. 

Let $\Bb(s):=\Bb^s\cdot\Bb''^{1-s}$ for any $s\in [0,1]$.
Let $h\colon Y\rightarrow X$ be a $\mathbb Q$-factorial qdlt modification (resp. $\mathbb Q$-factorial ACSS modification) of $\Bb''$ if $t''<1$ (resp. $t''=1$) and let $\Bb_Y(s):=(h^{-1}_*\Bb(s),\Exc(h))$ for any $s\in [0,1]$. Since $\Bb(\mu)$ is lc, by Definition-Lemma~\ref{deflem: nfs geometric average}, $\Bb(s)$ is lc for any $s\in [0,\mu]$. By Lemma~\ref{lem: lc preserved under inequality}, any lc place of $\Bb''=\Bb(0)$ is an lc place of $\Bb(s)$ for any $s\in [0,\mu]$, so $\Bb_Y(s)=h^*\Bb(s)$ for any $s\in [0,\mu]$. Since 
$$s\mapsto a(E,\Bb(s))$$
is an affine function for any prime divisor $E$ over $X$, we have $\Bb_Y(s)=h^*\Bb(s)$ for any $s\in [0,1]$. We have $\Bb_Y(1)/U\in\Ii\cup\{1\}$,
$\Bb_Y(1)/U\geq\Bb_Y(\mu)/U\geq\Bb_Y(1)^{1-\tau}/U$, and $\Bb_Y(\mu)$ is lc. By our assumption, $\Bb_Y(1)$ is lc, hence $\Bb(1)$ is lc. The lemma follows.
\end{proof}

\begin{lem}\label{lem: reduce to one parameter acc}
Fix $t_0, t'_0 \in [0, 1]$ such that $t'_0 = (1-\tau)t_0$.
    Assume that Theorem~\ref{thm: acc lct general version} holds in dimension $d$ under the following additional hypotheses:
    \begin{itemize}
        \item 
    $X$ is $\mathbb Q$-factorial klt, 
    \item $B'=(1-\tau)B$; \item $\Mm'=(1-\tau)\Mm$; and 
    \item the parameter of $\mathfrak B$ is $t_0$ and the parameter of $\mathfrak B'$ is $t'_0$. 
    \end{itemize}Then Theorem~\ref{thm: acc lct general version} holds in dimension $d$.
\end{lem}
\begin{proof}
By Lemma~\ref{lem: acc reduce to klt}, we may assume that $X$ is $\mathbb Q$-factorial klt. Since $\Bb'/U\geq\Bb^{1-\tau}/U$ and $\Bb'/U$ is lc, by Lemma~\ref{lem: lc preserved under inequality}, $\Bb^{1-\tau}$ is lc when $t'<1$, and $(X,\Ff,(1-\tau)B,(1-\tau)\Mm)(1)$ is lc when $t'=t=1$.

If $t'=t=1$, then we consider
$$\mu:=\lct(X,\Ff,0;\bm{0};B^{\ninv},\Mm).$$
We have $\mu\geq 1-\tau$. Since $B^{\ninv}\in\Ii$ and $\Mm\in\bNef(X/X,\Ii)$, by \cite[Theorem~2.4.4]{CHLX23}, we may choose $\tau$ such that if $\mu\geq 1-\tau$, then $\mu\geq 1$. In particular, $(X,\Ff,B^{\ninv},\Mm)$ is lc, hence $\Bb$ is lc. Thus we may assume that $t'<1$, and the lemma follows from our assumption. 
\end{proof}

\begin{lem}\label{lem: klt bs}
Let $\Bb/U:=(X,\Ff,B,\Mm)(t)/U$ be an algebraically integrable normalized foliated structure such that $X$ is $\mathbb Q$-factorial klt and let 
$$\mu:=\sup\{s\in [0,1]\mid \Bb^s\text{ is lc}\}.$$
Assume that $\mu t<1$. Then $\mu>0$ and $\Bb^{s}$ is klt for any $s\in [0,\mu)$.
\end{lem}
\begin{proof}
Since $X$ is klt and klt is an open condition, it follows that $\mu>0$. Since $\mu t<1$, by Lemma~\ref{lem: lc preserved under inequality}, $\Bb^s$ is lc for any $s\in [0,\mu)$. If $\Bb^s$ is not klt 
for some $s\in [0,\mu)$, then $s>0$ and there exists an lc place $E$ of $\Bb^s$. By Lemma~\ref{lem: lc preserved under inequality}, $E$ is an lc place of $\Bb^r$ for any $r\in [s,\mu]$. Since $r\mapsto a(E,\Bb^r)$ is a quadratic function,
$E$ is an lc place of $\Bb^r$ for any $r\in [0,\mu]$. In particular, $E$ is an lc place of $\Bb^0$ and hence an lc place of $X$. This is not possible.
\end{proof}

\begin{lem}\label{lem: extract unique lc place}
Let $\Bb/U:=(X,\Ff,B,\Mm)(t)/U$ be an algebraically integrable normalized foliated structure such that $X$ is $\mathbb Q$-factorial klt and let 
$$\mu:=\sup\{s\in [0,1]\mid \Bb^s\text{ is lc}\}.$$
Assume that $\mu t<1$ and that any lc place of $\Bb^{\mu}$ is not on $X$. Then there exists a projective birational morphism $g: Y\rightarrow X$ satisfying the following.
\begin{enumerate}
    \item $g$ is a divisorial contraction of a prime divisor $E$ and $Y$ is $\mathbb Q$-factorial. In particular, $\rho(Y/X)=1$.
    \item $E$ is an lc place of $\Bb^{\mu}$.
    \item $\Bb_Y(s):=(g^{-1}_*\Bb^s,E)$ is plt for any $s\in [0,\mu)$. In particular, $E$ is normal.
\end{enumerate}
\end{lem}
\begin{proof}
Let $E_0$ be an lc place of $\Bb^{\mu}$. Let $f: W\rightarrow X$ be a foliated log resolution of $\Bb$ such that $E_0$ is on $W$. Let $E_0,\dots,E_n$ be all the $f$-exceptional prime divisors and let $\Bb_W(s):=(f^{-1}_*\Bb^s,\Exc(f))$. Then we have
$$K_{\Bb_W(s)}=f^*K_{\Bb^s}+\sum_{i=0}^nb(i,s)E_i\sim_{\mathbb R,X}\sum_{i=0}^nb(i,s)E_i$$
for some real numbers $b(i,s)$. By Lemma~\ref{lem: klt bs}, $\Bb^s$ is klt for any $s\in [0,\mu)$, hence $b(i,s)>0$ for any $i$ and any $s\in [0,\mu)$. Since $\Bb^\mu$ is lc and $E_0$ is an lc place of $\Bb^{\mu}$,
$b(i,\mu)\geq 0$ and $b(0,\mu)=0$. 

We run a $K_{\Bb_W(\mu)}$-MMP$/X$ with scaling of an ample divisor. By \cite[Theorem~3.6]{Cas+25b}, this MMP terminates with a good minimal model $\Bb_V(\mu)/X$ of $\Bb_W(\mu)/X$ such that $K_{\Bb_V(\mu)}\sim_{\mathbb R,X}0$ and the divisors contracted by the induced birational map $\phi\colon W\dashrightarrow V$ are exactly the $E_i$ such that $b(i,\mu)>0$. In particular, $E_0$ is not contracted by $\phi$. Possibly reordering indices, we may assume that $E_i$ is not contracted by $\phi$ if and only if $i\leq m$ for some $m\leq n$. We let $E_{i,V}$ be the images of $E_i$ on $V$ for any $0\leq i\leq n$ and let $\pi: V\rightarrow X$ be the induced morphism. Let $\Bb_V(s):=\phi_*\Bb_W(s)$ for any $s\in [0,1]$. Then $\Bb_V(\mu)=\pi^*\Bb^{\mu}$ is $\mathbb Q$-factorial qdlt. 

Since $\phi$ is also a sequence of steps of a $K_{\Bb_W(s)}$-MMP$/U$ for any $0\leq \mu-s\ll 1$, $\Bb_V(s)$ is $\mathbb Q$-factorial qdlt for any $0\leq \mu-s\ll 1$. Now pick $s_0$ such that $0<\mu-s_0\ll 1$ and run a $K_{\Bb_V(s_0)}$-MMP$/X$. Note that
$$K_{\Bb_V(s_0)}\sim_{\mathbb R,X}\sum_{i=0}^mb(i,s_0)E_{i,V}.$$
Since $b(i,s_0)>0$ for any $i$, by \cite[Theorem~3.6]{Cas+25b} and since $X$ is $\mathbb Q$-factorial klt, this MMP terminates with $\Bb^{s_0}/X$ as a good minimal model of $\Bb_V(s_0)/X$, and the last step of this MMP is a divisorial contraction $g: Y\rightarrow X$. We let $\Bb_Y(s)$ be the pushforward of $\Bb_V(s)$ on $Y$ for any $s\in [0,1]$ and let $E$ be the unique prime $g$-exceptional divisor. Then we have $\Bb_Y(s)=(g^{-1}_*\Bb^s,E)$ for any $s\in [0,1]$.

We show that $g$ satisfies our requirements. 

(1) follows from our construction. 

Since $\pi$ only extracts lc places of $\Bb^{\mu}$, $g$ only extracts lc places of $\Bb^{\mu}$. This implies (2).

To prove (3), note that by our construction, $\Bb_Y(s_0)$ is $\mathbb Q$-factorial qdlt and $E$ is the only lc place of $\Bb_Y(s_0)$, so $\Bb_Y(s_0)$ is $\mathbb Q$-factorial plt. By the definition of qdlt \cite[Definition~3.1]{Cas+24}, $\Bb_Y(0)$ is plt. Since $E$ is an lc place of $\Bb_Y(0)$, $E$ is normal. For any $s\in [0,\mu)$, since $\Bb_Y(\mu)\geq\Bb_Y(s)$, $\Bb_Y(s)$ is lc. If $\Bb_Y(s)$ is not plt for some $s\in (0,\mu)$, then there exists an lc place $F\neq E$ of $\Bb_Y(s)$. By Lemma~\ref{lem: lc preserved under inequality}, $F$ is an lc place of $\Bb_Y(r)$ for any $r\in [s,\mu]$. Since $r\mapsto a(F,\Bb_Y(r))$ is a quadratic function, 
$F$ is an lc place of $\Bb_Y(r)$ for any $r\in [0,\mu]$. Thus $F$ is an lc place of $\Bb_Y(s_0)$, which is not possible. This implies (3).
\end{proof}

\begin{thm}\label{thm: global to local}
   Let $d\geq 2$ be an integer. Assume Theorem~\ref{thm: global acc general} holds in dimension $d-1$. Then Theorem~\ref{thm: acc lct general version} holds in dimension $d$.
\end{thm}
\begin{proof}
Suppose that Theorem~\ref{thm: acc lct general version} does not hold in dimension $d$. By Lemma~\ref{lem: reduce to one parameter acc}, there exist a strictly increasing sequence of real numbers $s_i$ such that $s_i\in (0,1)$ for each $i$, $\lim_{i\rightarrow+\infty}s_i=1$, and for each $i$, there exists an algebraically integrable normalized foliated structure 
$$\Bb_i/U_i:=(X_i,\Ff_i,B_i,\Mm_i)(t_i)/U_i$$
of dimension $d$, such that $X_i$ is $\mathbb Q$-factorial klt, $\Bb_i/U_i\in\Ii$, $\Bb_i^{s_i}$ is lc, but $\Bb_i$ is not lc. Since lc is a closed condition, possibly replacing $s_i$ with
$$\mu_i:=\sup\{s\mid \Bb_i^s\text{ is lc}\},$$
we may assume that $s_i=\mu_i<1$ for any $i$.

Since $\mu_i$ is strictly increasing and $B_i\in\Ii$, possibly passing to a subsequence, we have that $\lfloor\mu_iB_i\rfloor=0$ for any $i$. By Lemma~\ref{lem: extract unique lc place}, for each $i$, there exists a projective birational morphism $g_i: Y_i\rightarrow X_i$ satisfying the following.
\begin{enumerate}
    \item $g_i$ is a divisorial contraction of a prime divisor $E_i$ and $Y_i$ is $\mathbb Q$-factorial.
    \item $E_i$ is an lc place of $\Bb_i^{\mu_i}$.
    \item $\Bb_{Y_i}(s):=(g_{i,*}^{-1}\Bb_i^s,E_i)$ is $\mathbb Q$-factorial plt for any $s\in [0,\mu_i)$ and $E_i$ is normal.
\end{enumerate}
Let $W_i$ be the normalization of $\Center_{X_i}E_i$, $E_i\rightarrow W_i$ the induced projective morphism, 
$E_i\xrightarrow{\pi_i}T_i\rightarrow W_i$ the Stein factorization of $E_i\rightarrow W_i$, and $F_i$ a general fiber of $\pi_i$. Since $\Bb_{Y_i}(0)$ is plt, $(Y_i,E_i)$ is $\mathbb Q$-factorial plt, so $E_i$ is potentially klt. 

We let $\Cc_i:=(X_i,\Ff_i,\mu_iB_i,\mu_i\Mm_i)(0)$. Then $\Bb^{\mu_i}/U_i\geq\Cc_i/U_i$. By Lemma~\ref{lem: lc preserved under inequality}, $\Cc_i$ is lc. By \cite[Theorem~1.5]{BZ16} and since $X_i$ is $\mathbb Q$-factorial klt, possibly passing to a subsequence, we may assume that $\Cc_i$ is klt for any $i$. In particular, we may assume that $t_i>0$ for each $i$. We let 
$$\Cc_i(s):=(X_i,\Ff_i,\mu_iB_i,\mu_i\Mm_i)(st_i)$$
for any $s\in [0,1]$. Then $\Cc_i(0)=\Cc_i$ is klt and $\Cc_i(\mu_i)=\Bb_i^{\mu_i}$ is lc but not klt, so $\Cc_i(s)$ is klt for any $s\in [0,\mu_i)$. Let $\Cc_{Y_i}(s):=(g^{-1}_{i,*}\Cc_i(s),E_i)$ for any $s\in [0,1]$. Then $\Bb_{Y_i}(\mu_i)/U_i\geq\Cc_{Y_i}(s)/U_i$ for any $s\in [0,\mu_i]$. By Lemma~\ref{lem: lc preserved under inequality}, $\Cc_{Y_i}(s)$ is lc for any $s\in [0,\mu_i]$.

Since $\rho(Y_i/X_i)=1$, $K_{\Cc_{Y_i}(s)}$ is anti-ample$/X_i$ for any $s\in [0,\mu_i)$. We let $\Cc_{E_i}(s)$ be the normalized foliated structure induced by adjunction
$$K_{\Cc_{E_i}(s)}:=K_{\Cc_{Y_i}(s)}|_{E_i}$$
for any $s\in [0,1)$. Then $K_{\Cc_{E_i}(s)}$ is anti-ample$/T_i$ for any $s\in (0,\mu_i)$ and $K_{\Cc_{E_i}(\mu_i)}\sim_{\mathbb R,T_i}0$. We may write
$$\Cc_{E_i}(s)=(E_i,\Ff_{E_i},B_{E_i}(s),\mu_i\Mm_i|_{E_i},st_i)$$
for any $s\in [0,\mu_i]$. By Theorem~\ref{thm: dcc adjunct to dcc}, there exists a DCC set $\Ii'$ depending only on $\Ii$ and $\{\mu_i\}$, such that $B_{E_i}(\mu_i)\in\Ii'$. By Theorem~\ref{thm: global acc general}~(1) in dimension $\leq d-1$, there exists a finite set $\Ii_0$ depending only on $d$, $\Ii$, and $\{\mu_i\}$, such that $B_{E_i}(\mu_i)\in\Ii_0$ over the generic point of $W_i$. By Theorem~\ref{thm: dcc precise adjunction}, there exists a finite set $\Ii_0'$, such that either $\mu_it_i\in\Ii_0'$, or $s\mapsto B_{E_i}(s)$ is a constant function over the generic point of $W_i$. Since $t_i>0$ and $\mu_i$ is strictly increasing, possibly passing to a subsequence, we may assume that $s\mapsto B_{E_i}(s)$ is a constant function over the generic point of $W_i$. By Theorem~\ref{thm: global acc general}~(2) in dimension $\leq d-1$, $\mu_it_i$ belongs to a finite set. This is not possible.
\end{proof}

\subsection{A variation of local ACC} The ACC for interpolated lc thresholds, Theorem~\ref{thm: acc lct general version}, is in the spirit of \cite[Theorem~A]{HMX14} as well as the inverse of stability (cf.~\cite[Section~12]{Sho20}). On the other hand, there is no ``threshold'' in the statement of Theorem~\ref{thm: acc lct general version}. We provide the following version of the ACC for interpolated lc thresholds whose statement is more precise and closer to the original ACC for lc thresholds \cite[Theorem~1.1]{HMX14}. We also need this variation in the inductive proof of the global ACC.

\begin{thm}[ACC for interpolated lc thresholds, precise version]\label{thm: ACC precise}
     Let $d$ be a positive integer and $\Ii\subset [0,+\infty)$ a DCC set. Then there exists an ACC set $\Ii':=\lct_{\aif}(d,\Ii)\subset [0,+\infty)$ depending only on $d$ and $\Ii$ satisfying the following.

     Let $\Bb/U:=(X,\Ff,B,\Mm)(t)/U$ be an lc algebraically integrable normalized foliated structure of dimension $d$ such that $\Bb/U\in\Ii$. Let $D\in\Ii$ be an $\mathbb R$-divisor, $\Nn\in\bNef(X/U,\Ii)$ a $\bb$-divisor on $X$, and $\mu\in\Ii$. Let
     $$\Bb(s):=(X,\Ff,B+sD,\Mm+s\Nn)(t+s\mu)$$
and assume that $K_{\Bb(s)}$ is $\mathbb R$-Cartier for any $s$. Then
     $$\sup\{s\geq 0\mid \Bb(s)\text{ is lc}\}\in\Ii'.$$
\end{thm}

We remark that when $\mu=0$ and $t=1$, Theorem~\ref{thm: ACC precise} was proven in \cite[Theorem~2.4.4]{CHLX23}, \cite[Theorem~1.1]{DLM23}.

\begin{lem}\label{lem: acc implies precise acc}
Assume Theorem~\ref{thm: acc lct general version} holds in dimension $d$. Then Theorem~\ref{thm: ACC precise} holds in dimension $d$.
\end{lem}
\begin{proof}
    Suppose that the theorem does not hold. Then for any $i\in\mathbb N^+$, there exists an algebraically integrable normalized foliated structure 
$$\Bb_i/U_i:=(X_i,\Ff_i,B_i,\Mm_i)(t_i)/U_i$$
of dimension $d$ such that $\Bb_i/U_i\in\Ii$, an $\mathbb R$-divisor $D_i\in\Ii$ on $X_i$, $\Nn_i\in\bNef(X_i/U_i,\Ii)$, $\mu_i\in\Ii$, 
$$\Bb_i(s):=(X_i,\Ff_i,B_i+sD_i,\Mm_i+s\Nn_i)(t_i+s\mu_i)$$
for any $s\in [0,+\infty)$, and
$$s_i:=\sup\{s\geq 0\mid \Bb_i(s)\text{ is lc}\},$$
such that $t_i+s_i\mu_i\in [0,1]$, $\Bb_i(s)$ is $\mathbb R$-Cartier for any $s\in [0,+\infty)$, and $s_i$ is strictly increasing. Let $r:=\lim_{i\rightarrow+\infty}s_i$. 

If $r<+\infty$, then there exists a DCC set $\Ii_1\subset [0,+\infty)$ such that $\Bb_i(r)/U_i\in\Ii_1$ for any $i$. Possibly passing to a subsequence, we may assume that
$$\Bb_i(r)/U_i\geq\Bb_i(s_i)/U_i\geq\Bb_i(r)^{1-\tau_1}/U_i$$
for any $i$, where $\tau_1=\tau_{\ilct}(d,\Ii_1)$ is the value defined as in Theorem~\ref{thm: acc lct general version} in dimension $d$. By Theorem~\ref{thm: acc lct general version} in dimension $d$, $\Bb_i(r)$ is lc, a contradiction. 

If $r=+\infty$, then possibly passing to a subsequence, we may assume that $\mu_i=0$ for each $i$. If $t_i=1$ for infinitely many $i$, then we are done by \cite[Theorem~2.4.4]{CHLX23}, so possibly passing to a subsequence, we may assume that $t_i<1$ for each $i$. Thus we may assume that $D_i=0$ for each $i$. We have that $\Nn_{i,X_i}$ is $\mathbb R$-Cartier for any $i$. There exists a DCC set $\Ii_2\subset [0,+\infty)$ such that 
$$\Bb_i(s_i+1)/U_i\in\Ii_2$$
for any $i$. Possibly passing to a subsequence, we may assume that
$$\Bb_i(s_i+1)/U_i\geq\Bb_i(s_i)/U_i\geq\Bb_i(s_i+1)^{1-\tau_2}/U_i$$
for any $i$, where $\tau_2=\tau_{\ilct}(d,\Ii_2)$ is the value defined as in Theorem~\ref{thm: acc lct general version} in dimension $d$. By Theorem~\ref{thm: acc lct general version} in dimension $d$, $\Bb_i(s_i+1)$ is lc, a contradiction.
\end{proof}

\section{Local ACC to global ACC}\label{sec: local to global}

The goal of this section is to prove the global ACC (Theorem~\ref{thm: global acc general}) assuming the local ACC (Theorem~\ref{thm: acc lct general version}) of the same dimension.

\subsection{The special cases}

In this subsection, we prove some special cases of Theorem~\ref{thm: global acc general} unconditionally.

\begin{thm}\label{thm: global acc t=1}
    Theorem~\ref{thm: global acc general} holds when $t=1$.
\end{thm}
\begin{proof}
We may assume that $1\in\Ii_0$ and Theorem~\ref{thm: global acc general}~(3) automatically holds in this case, so we only need to prove Theorem~\ref{thm: global acc general}~(1-2). Since $t=1$, we may assume that $B^{\inv}=0$. Let $\Aa:=(X,\Ff,B,\Mm)$.
    
    Let $h\colon X'\rightarrow X$ be a $\mathbb Q$-factorial ACSS modification of $\Aa$ associated with the contraction $f: X'\rightarrow Z$ and let $\Aa':=h^*\Aa=:(X',\Ff',B',\Mm)$. Possibly shrinking $U$, we may assume that each irreducible component of $B'$ dominates $U$, and $\Mm_j|_{F}\not\equiv\bm{0}$ for each $j$, where $F$ is a general fiber of $\pi\circ h$. By Lemma~\ref{lem: non-trivial of b-divisors}, $\Mm_{j,X'}|_{F}\not\equiv 0$ for each $j$. Let $D$ be either an irreducible component of $B'$ or $\Mm_{j,X'}$ for some $j$. Pick $0<\epsilon\ll 1$. Then
$$K_{\Aa'}-\epsilon D\equiv_U-\epsilon D$$
is not pseudo-effective$/U$.
By \cite[Lemma~9.1.4, Theorem~9.3.4]{CHLX23}, we may run a $(K_{\Aa'}-\epsilon D)$-MMP$/U$ with scaling of an ample divisor which terminates with a Mori fiber space$/U$ $g: X''\rightarrow T$ associated with a birational map $\phi\colon X\dashrightarrow X''$, such that this MMP is also an MMP$/Z$. In particular, $g$ is a contraction$/Z$. Let $\Aa'':=\phi_*\Aa':=(X'',\Ff'',B'',\Mm)$ and $D'':=\phi_*D$. Since $K_{\Aa'}\sim_{\mathbb R,U}0$, this MMP is $K_{\Aa'}$-trivial, and $D''$ dominates $T$. Let $L$ be a general fiber of $g$, $B_L:=B''|_L$, $\Mm_j^L:=\Mm_j|_L$, and $D_L:=D''|_L$. Then $K_{\Ff''}|_L=K_L=K_{X''}|_L$, so we have
$$K_L+B_L+\sum\mu_j\Mm_{j,L}^L\equiv 0,$$
$D_L$ is either a component of $B_L$ or $D_L=\Mm_{j,L}^L$ is ample for some $j$, and $(L,B_L,\sum\mu_j\Mm_j^L)$ is lc. 

If $D$ is a component of $B'$, then by \cite[Theorem~1.6]{BZ16}, $\mult_DB'=\mult_{D_L}B_L$ belongs to a finite set depending only on $d$ and $\Ii$. If $D=\Mm_{j,X'}$ for some $j$, then since $\Mm_{j,L}^L$ is ample, $\Mm_j^L\not\equiv\bm{0}$. By \cite[Theorem~1.6]{BZ16}, $\mu_j$ belongs to a finite set depending only on $d$ and $\Ii$. This implies Theorem~\ref{thm: global acc general} when $t=1$.
\end{proof}

\begin{thm}\label{thm: relative global ACC gpair}
      Theorem~\ref{thm: global acc general} holds when $t\in\{0,1\}$.
\end{thm}
\begin{proof}
By applying Theorem~\ref{thm: global acc t=1} to $(X,T_X,B,\Mm)(1)$, we obtain Theorem~\ref{thm: global acc general}~(1-2) when $t=0$. Theorem~\ref{thm: global acc general}~(3) is automatic.
\end{proof}

\begin{lem}\label{lem: no part 2 case}
Theorem~\ref{thm: global acc general} holds if $K_{\Bb(s)}$ is $\mathbb R$-Cartier and $K_{\Bb(s)}|_F\equiv 0$ for some $s\neq t$.
\end{lem}
\begin{proof}
We have that $K_{\Bb(s)}$ is $\mathbb R$-Cartier and $K_{\Bb(s)}|_F\equiv 0$ for all $s\in [0,1]$. Thus we only need to prove Theorem~\ref{thm: global acc general}~(1-2). By Theorem~\ref{thm: relative global ACC gpair}, we may assume that $t\in (0,1)$. By \cite[Proposition~3.3]{Cas+24}, $\Bb(0)$ is lc. Since $K_{\Bb(0)}|_F\equiv 0$, Theorem~\ref{thm: global acc general}~(1-2) follows from Theorem~\ref{thm: relative global ACC gpair} applied to $\Bb(0)$.
\end{proof}

\subsection{Reduction to relative Picard number one}

In this subsection, we reduce Theorem~\ref{thm: global acc general} to the case when $\rho(X/U)=1$ and $X$ is of Fano type over $U$. We first consider the following set-up which will be used consistently in the rest of the proof of Theorem~\ref{thm: global acc general}.

\begin{setup}\label{setup: global acc}
We set
$d,\Ii,\pi,X,U,F,d_F,\Bb,\Ff,B,\Mm,t,\Bb(\cdot),n,\{\mu_j\}_{j=1}^n,\{\Mm_j\}_{j=1}^n$ satisfying the following conditions.
    \begin{enumerate}
    \item $d$ is a positive integer and $\Ii\subset [0,+\infty)$ is a set.
    \item $\pi: X\rightarrow U$ is a contraction between normal quasi-projective varieties, $\dim X=d>\dim U$, $X$ is of Fano type over $U$, and $\rho(X/U)=1$.
    \item $F$ is a general fiber of $\pi$ and $d_F:=\dim F$.
    \item $\Bb/U=(X,\Ff,B,\Mm)(t)/U$ is a $\mathbb Q$-factorial lc algebraically integrable adjoint foliated structure such that $X$ is klt and $t\in\Ii\cap (0,1)$.
    \item $\Bb(s):=(X,\Ff,B,\Mm)(s)$.
    \item $B\in\Ii$ and any irreducible component of $B$ is horizontal$/U$.
    \item $\Mm=\sum_{j=1}^n\mu_j\Mm_j$ where each $\mu_j\in\Ii$, $\Mm_j$ is nef$/U$ $\bb$-Cartier, and $\Mm_{j,X}$ is ample$/U$.
    \item $K_{\Bb(t)}\sim_{\mathbb R,U}0$.
    \end{enumerate}
\end{setup}

\begin{thm}\label{thm: picard number one case}
Let $d$ be a positive integer and $\Ii\subset [0,+\infty)$ a DCC set. Then there exists a finite set $\Ii_0\subset\Ii$ depending only on $d$ and $\Ii$ satisfying the following.

Notation and conditions as in Set-up~\ref{setup: global acc}. Then:
\begin{enumerate}
    \item $B,\mu_1,\dots,\mu_n\in\Ii_0$.
    \item Assume that $K_{\Bb(s)}\not\sim_{\mathbb R,U}0$ for some $s\neq t$. Then $t\in\Ii_0$.
\end{enumerate}
\end{thm}
We will later show that Theorem~\ref{thm: global acc general} in dimension $\leq d-1$ and Theorem~\ref{thm: acc lct general version} in dimension $\leq d$ imply Theorem~\ref{thm: picard number one case}. See Theorem~\ref{thm: acc to global acc} for details. Here, we first show that Theorem~\ref{thm: picard number one case} in dimension $d$ implies Theorem~\ref{thm: global acc general} in dimension $d$:

\begin{thm}\label{thm: picard one imply general global acc}
Theorem~\ref{thm: picard number one case} in dimension $d$ implies Theorem~\ref{thm: global acc general} in dimension $d$.
\end{thm}
\begin{proof}
Notation and conditions as in Theorem~\ref{thm: global acc general}. By Theorem~\ref{thm: relative global ACC gpair} we may assume that $t\in (0,1)$. Possibly replacing $\Ii$ by $\Ii\cup\{1\}$, we may assume that $1\in\Ii$. By Lemma~\ref{lem: no part 2 case}, we may assume that one of the following cases holds:
\begin{itemize}
    \item[(\textbf{Case 1})] $K_{\Bb(s)}$ is not $\mathbb R$-Cartier for any $s\neq t$. In this case we only need to prove (1-2).
    \item[(\textbf{Case 2})] $K_{\Bb(s)}$ is $\mathbb R$-Cartier for any $s\in [0,1]$, and $K_{\Bb(s)}|_F\not\equiv 0$ for any $s\neq t$.
\end{itemize}
We let $h\colon X'\rightarrow X$ be a $\mathbb Q$-factorial qdlt modification of $\Bb=\Bb(t)$, and we let $\Bb'(s):=(h^{-1}_*\Bb(s),\Exc(h))$ for any $s\in [0,1]$. Then $\Bb'(t)=h^*\Bb(t)$, and we may write 
$$\Bb'(s)=(X',\Ff':=h^{-1}\Ff,B':=h^{-1}_*B+\Exc(h),\Mm)(s)$$
for any $s\in [0,1]$. Let $\pi':=\pi\circ h$ and $F'$ a fiber of $\pi'$ such that $h(F')=F$. Then $F'$ is a general fiber of $\pi'$ and $\Bb'(t)$ is $\mathbb Q$-factorial qdlt. Moreover, if we are in \textbf{Case 2}, then
$$(h|_{F'})_*(K_{\Bb'(s)}|_{F'})=K_{\Bb(s)}|_F\not\equiv 0,$$
hence $K_{\Bb'(s)}|_{F'}\not\equiv 0$.

We first prove the theorem when $K_{\Bb'(s)}|_{F'}\equiv 0$ for some $s\neq t$. Under this assumption, we are in \textbf{Case 1}, so we only need to prove (1-2). By Lemma~\ref{lem: no part 2 case} applied to $\Bb'(t)$, there exists a finite set $\Ii_0$ depending only on $d$ and $\Ii$ such that 
$$B'|_{F'}\in\Ii_0\quad \text{and}\quad \{\mu_j\mid \Mm_j|_{F'}\not\equiv\bm{0}\}\subset\Ii_0.$$
Since 
$$B|_F=(h|_{F'})_*B'|_{F'}\quad \text{and}\quad \Mm_j|_F=\Mm_j|_{F'},$$
we obtain (1-2).

In the following, we may assume that $K_{\Bb'(s)}|_{F'}\not\equiv 0$ for any $s\neq t$. By \cite[Theorem~5.1]{Cas+25a}, $K_{\Bb'(1)}$ is pseudo-effective$/U$ but $K_{\Bb'(1)}|_{F'}\not\equiv 0$. Therefore, $K_{\Bb'(s)}|_{F'}$ is not pseudo-effective for any $s<t$, so $K_{\Bb'(s)}$ is not pseudo-effective$/U$ for any $s<t$. Since $X'$ is $\mathbb Q$-factorial, possibly removing the $\Mm_j$ such that $\Mm_j|_{F'}\equiv\bm{0}$, we may assume that $\Mm_j|_{F'}\not\equiv\bm{0}$ for any $j$. Fix $0<\epsilon\ll 1$. We define
$$\Cc_D:=(\Bb'(t),-\epsilon D),\quad \Cc_j:=(\Bb'(t),-\epsilon\Mm_j),\quad \text{ and }\quad \Cc':=\Bb'(t-\epsilon)$$
for any irreducible component $D$ of $B'$ and $1\leq j\leq n$. Then
$$K_{\Cc_D}|_{F'}\equiv -\epsilon(D^{\ninv}+(1-t)D^{\inv})|_{F'},\quad K_{\Cc_j}|_{F'}\equiv -\epsilon(\Mm_j|_F)_{F'},\quad \text{and}\quad K_{\Cc'}|_{F'}$$
are not pseudo-effective. By \cite[Theorem~2.1.5]{Cas+25a}, we may run a $K_{\Cc_D}$-MMP$/U$ (resp. $K_{\Cc_j}$-MMP$/U$, $K_{\Cc'}$-MMP$/U$) with scaling of an ample divisor $\phi_D: X'\dashrightarrow X_D$ (resp. $\phi_j: X'\dashrightarrow X_j$, $\phi\colon X'\dashrightarrow Y$) which terminates with a Mori fiber space $g_D: X_D\rightarrow T_D$ (resp. $g_j: X_j\rightarrow T_j$, $g: Y\rightarrow T$). Let $\Bb_D(s):=\phi_{D,*}\Bb'(s)$ (resp. $\Bb_j(s):=\phi_{j,*}\Bb'(s)$, $\Bb_Y(s):=\phi_{*}\Bb'(s)$) for any $s\in [0,1]$ and let $\Cc_D':=\phi_{D,*}\Cc_D$ (resp. $\Cc_j':=\phi_{j,*}\Cc_j$, $\Cc_Y':=\phi_{*}\Cc'$). Since $X$ is $\mathbb Q$-factorial klt, $X_D,X_j,Y$ are $\mathbb Q$-factorial klt, and $\rho(X_D/T_D)=\rho(X_j/T_j)=\rho(Y/T)=1$. Since $K_{\Cc_D'}$ (resp. $K_{\Cc_j'}$, $K_{\Cc_Y'}$) is anti-ample$/T_D$ (resp. anti-ample$/T_j$, anti-ample$/T$), by \cite[Theorem~9.2]{Cas+25a}, $X_D$ (resp. $X_j,Y$) is of Fano type over $T_D$ (resp. $T_j,T$). 

We let $B_D$ (resp. $B_j,B_Y$) be the horizontal$/T_D$ (resp. horizontal$/T_j$, horizontal$/T$) part of $\phi_{D,*}B'$ (resp. $\phi_{j,*}B'$, $\phi_{*}B'$), $\Lambda_D$ (resp. $\Lambda_j,\Lambda$) the subset of $\{1,2,\dots,n\}$ such that $k\in\Lambda_D$ (resp. $\Lambda_j,\Lambda$) if and only if $\Mm_{k,X_D}$ (resp. $\Mm_{k,X_j}$, $\Mm_{k,Y}$) is ample$/T_D$ (resp. ample$/T_j$, ample$/T$), and let
$$\Nn^D\text{(resp. }\Nn^j,\Nn^Y\text{)}:=\sum_{k\mid k\in\Lambda_D\text{(resp. }k\in\Lambda_j,k\in\Lambda\text{)}}\mu_k\Mm_k,$$
and let
$$\widetilde{\Bb_D}(s):=\left(X_D,\Ff_D:=\phi_{D,*}\Ff',B_D,\Nn^D\right)(s),\quad \widetilde{\Bb_j}(s):=\left(X_j,\Ff_j:=\phi_{j,*}\Ff',B_j,\Nn^j\right)(s),$$
and 
$$\widetilde{\Bb_Y}(s):=\left(Y,\Ff_Y:=\phi_{*}\Ff',B_Y,\Nn^Y\right)(s)$$
for any $s\in [0,1]$. Since $\phi_D$ (resp. $\phi_j$, $\phi$) is $K_{\Bb'(t)}$-trivial, $\Bb_D(t)$ (resp. $\Bb_j(t)$, $\Bb_Y(t)$) is lc. Since $\Bb_D(t)/U\geq \widetilde{\Bb_D}(t)/U$ (resp. $\Bb_j(t)/U\geq \widetilde{\Bb_j}(t)/U$, $\Bb_Y(t)/U\geq \widetilde{\Bb_Y}(t)/U$) and $t<1$, by Lemma~\ref{lem: lc preserved under inequality}, $\widetilde{\Bb_D}(t)$ (resp. $\widetilde{\Bb_j}(t)$, $\widetilde{\Bb_Y}(t)$) is lc. Since $\rho(X_D/T_D)=\rho(X_j/T_j)=\rho(Y/T)=1$, we have that
$$K_{\widetilde{\Bb_D}(t)}\sim_{\mathbb R,T_D}K_{\Bb_D(t)}\sim_{\mathbb R,T_D}0,\quad K_{\widetilde{\Bb_j}(t)}\sim_{\mathbb R,T_j}K_{\Bb_j(t)}\sim_{\mathbb R,T_j}0,\quad K_{\widetilde{\Bb_Y}(t)}\sim_{\mathbb R,T}K_{\Bb_Y(t)}\sim_{\mathbb R,T}0.$$
Moreover, since $g_D$ (resp. $g_j,g$) is a $(-\phi_{D,*}D)$-Mori fiber space (resp. $(-\Mm_{j,X_j})$-Mori fiber space, $K_{\Bb_Y(0)}$-Mori fiber space), $\phi_{D,*}D$ is horizontal$/T_D$ (resp. $\Mm_{j,X_j}$ is ample$/T_j$, $K_{\Bb_Y(s)}\not\sim_{\mathbb R,T}0$ for any $s\neq t$). By Theorem~\ref{thm: picard number one case}, there exists a finite set depending only on $d$ and $\Ii$, such that $\mult_{\phi_{D,*}D}B_D\in\Ii_0$ (resp. $\mu_j\in\Ii_0$, $t\in\Ii_0$). The theorem follows.
\end{proof}

\subsection{The bounded fiber case}

The previous subsection reduced Theorem~\ref{thm: global acc general} to Theorem~\ref{thm: picard number one case}. In this subsection, we prove Theorem~\ref{thm: picard number one case} under the extra assumption that $F$ belongs to a bounded family.

\begin{prop}\label{prop: when fiber bounded}
Let $d$ be a positive integer, $\Ii\subset [0,+\infty)$ a DCC set, and $\mathcal{P}$ a bounded family of varieties. Then there exists a finite set $\Ii_0\subset\Ii$ depending only on $d$, $\Ii$, and $\mathcal{P}$ satisfying the following.

Notation and conditions as in Set-up~\ref{setup: global acc}. Assume that $F\in\mathcal{P}$ or $\dim F=1$. Then:
\begin{enumerate}
    \item $B,\mu_1,\dots,\mu_n\in\Ii_0$.
    \item Assume that $K_{\Bb(s)}\not\sim_{\mathbb R,U}0$ for some $s\neq t$. Then $t\in\Ii_0$.
\end{enumerate}
\end{prop}
\begin{proof}
We may assume that $1\in\Ii$ and $\mathbb P^1\in\mathcal{P}$. If $\dim F=1$, then since $X$ is of Fano type over $U$, we have $F=\mathbb P^1\in\mathcal{P}$. Thus we may assume that $F\in\mathcal{P}$. By Lemma~\ref{lem: no part 2 case}, we may assume that $K_{\Bb(s)}\not\sim_{\mathbb R,U}0$ for any $s\neq t$.

Let $d_F:=\dim F>0$. Since $F\in\mathcal{P}$, there exists a positive integer $r$ depending only on $\mathcal{P}$ and a very ample Cartier divisor $H$ on $F$ such that $H^{d_F}\leq r$ and $K_F\cdot H^{d_F-1}\geq -r$. By \cite[Theorem~5.1]{Cas+25a}, $K_{\Bb(1)}$ is pseudo-effective$/U$. Since $t<1$, $K_{\Bb(1)}|_F\not\sim_{\mathbb R}0$, so $K_{\Bb(1)}$ is ample$/U$ and $K_{\Bb(1)}|_F\cdot H^{d_F-1}>0$.

We write $B=\sum_{i=1}^l b_iB_i+\sum_{k=1}^m c_kC_k$ where $B_i,C_k$ are the irreducible components of $B$, each $B_i$ is not $\Ff$-invariant and each $C_k$ is $\Ff$-invariant. We let $\Mm_j^F:=\Mm_j|_F$ for any $j$. We denote by
$$p_i:=B_i|_F\cdot H^{d_F-1},\quad q_k:=C_k|_F\cdot H^{d_F-1},\quad \text{and}\quad  \lambda_j:=\Mm_{j,F}^F\cdot H^{d_F-1}$$
for any $i,k,j$, $a:=-K_F\cdot H^{d_F-1}$, and $e:=K_{\Ff}|_F\cdot H^{d_F-1}$. Then $p_i,q_k,\lambda_j$ are positive integers and $a$ is an integer in $[0,r]$. We let $\mu$ be the unique real number such that $K_X+\mu K_{\Bb(1)}\sim_{\mathbb R,U}0$. Then $K_X+\mu K_{\Bb(1)}$ is pseudo-effective$/U$, so by \cite[Theorem~5.1]{Cas+25a}, $K_{\Ff}+\mu K_{\Bb(1)}$ is pseudo-effective$/U$. Thus
$$e=K_{\Ff}|_F\cdot H^{d_F-1}=(K_{\Ff}+\mu K_{\Bb(1)})|_F\cdot H^{d_F-1}-(K_{X}+\mu K_{\Bb(1)})|_F\cdot H^{d_F-1}+K_F\cdot H^{d_F-1}\geq -a.$$
Therefore, $e$ is an integer in $[-a,+\infty)\subset [-r,+\infty)$. Since $K_{\Bb(t)}|_F\cdot H^{d_F-1}=0$, we have
$$a=\sum_{i=1}^lp_ib_i+\sum_{k=1}^mq_kc_k+\sum_{j=1}^n\lambda_j\mu_j+\frac{t}{1-t}\left(e+\sum_{i=1}^lp_ib_i+\sum_{j=1}^n\lambda_j\mu_j\right).$$
Moreover,
$$e+\sum_{i=1}^lp_ib_i+\sum_{j=1}^n\lambda_j\mu_j=K_{\Bb(1)}|_F\cdot H^{d_F-1}>0.$$
Since $\Ii$ is a DCC set, the proposition follows.
\end{proof}

\subsection{The non-klt case}

In this subsection, we prove Theorem~\ref{thm: picard number one case} under the extra assumption that $\Bb$ is not klt over the generic point of $U$. We first prove the case when there exists an lc place on $X$ (Lemma~\ref{lem: when exist lc place}) and then prove the general case (Lemma~\ref{lem: when exist lc center}).

\begin{lem}\label{lem: when exist lc place}
Let $d$ be a positive integer and $\Ii\subset [0,+\infty)$ a DCC set. Assume that Theorem~\ref{thm: global acc general} holds in dimension $\leq d-1$. Then there exists a finite set $\Ii_0\subset\Ii$ depending only on $d$ and $\Ii$ satisfying the following.

Notation and conditions as in Set-up~\ref{setup: global acc} and assume that $\lfloor B\rfloor\neq0$. Then:
\begin{enumerate}
    \item $B,\mu_1,\dots,\mu_n\in\Ii_0$.
    \item Assume that $K_{\Bb(s)}\not\sim_{\mathbb R,U}0$ for some $s\neq t$. Then $t\in\Ii_0$.
\end{enumerate}
\end{lem}
\begin{proof}
By Proposition~\ref{prop: when fiber bounded} we may assume that $\dim F \geq 2$. By Lemma~\ref{lem: no part 2 case}, we may assume that $K_{\Bb(s)}\not\sim_{\mathbb R,U}0$ for any $s\neq t$.

Let $\widetilde S$ be an irreducible component of $\lfloor B\rfloor$. Since all the irreducible components of $B$ are horizontal$/U$ and ample$/U$, $\widetilde S$ is horizontal$/U$ and ample$/U$. Let $S$ be the normalization of $\widetilde S$, $\Mm_j^S:=\Mm_j|_S$ for any $j$, and $\Mm^S:=\Mm|_S$. Since $d_F\geq 2$, we have the following.
\begin{itemize}
 \item $\Mm_{j,X}|_S$ is ample$/U$, hence $\Mm_{j,X}|_S\not\equiv_U0$ for any $j$.
 \item For any irreducible component $L\neq S$ of $B$, $L|_S\geq 0$ is ample$/U$, hence $L|_S\neq0$ over the generic point of $U$.
\end{itemize}
Let $\Cc(s):=\left(X,\Ff,\widetilde{S},\bm{0}\right)(s)$ for any $s\in [0,1]$. Then $\Bb(s)/U\geq\Cc(s)/U$ for any $s\in [0,1]$. Since $\Bb(t)$ is lc, by Lemma~\ref{lem: lc preserved under inequality}, $\Bb(s)$ and $\Cc(s)$ are lc for any $s\in [0,t]$. We let $h\colon X'\rightarrow X$ be a $\mathbb Q$-factorial qdlt modification of $\Cc\left(\frac{t}{2}\right)$ and let $$\Bb'(s):=(h^{-1}_*\Bb(s),\Exc(h)),\quad \Cc'(s):=(h^{-1}_*\Cc(s),\Exc(h))$$ for any $s\in [0,1]$. By Lemma~\ref{lem: lc preserved under inequality} and linearity of discrepancies, any prime divisor extracted by $h$ is an lc place of $\Cc(s)$ and $\Bb(s)$ for any $s\in [0,t]$. In particular, $\Bb'(s)=h^*\Bb(s)$ and $\Cc'(s)=h^*\Cc(s)$ for any $s\in [0,1]$. Let $S':=h^{-1}_*\widetilde{S}$. Since $\Cc\left(\frac{t}{2}\right)$ is $\mathbb Q$-factorial qdlt, $(X',S'+\Exc(h))$ is $\mathbb Q$-factorial qdlt. In particular, $S'$ is normal. 

Let $h_S: S'\rightarrow S$ be the induced birational morphism. Write $\Bb'(s):=(X',\Ff',B',\Mm)(s)$ where $B'=h^{-1}_*B+\Exc(h)$. Then for any irreducible component $L'\neq S'$ of $B'$ that is not an irreducible component of $\Exc(h)$, $h_*L'$ is an irreducible component of $B$. Since $(X,B,\Mm)$ is lc, $h_*L'$ does not contain any lc center of $(X,S)$. Since $h$ only extracts lc places of $(X,S)$,
$$L'|_{S'}=h_S^*\left((h_*L')|_{S}\right)$$
is big$/U$ and nef$/U$. Moreover, for any $j$,
$$\Mm_{j,X'}|_{S'}=h_S^*\left(\Mm_{j,X}|_S\right)=h_S^*h_{S,*}\left(\Mm_{j,X'}|_{S'}\right)$$
is big$/U$ and nef$/U$. Since $d_F\geq 2$, $\dim S'>\dim U$. Therefore, at least one component of $L'|_{S'}$ is horizontal$/U$, and $\Mm_{j,X'}|_{S'}\not\equiv 0$ over the generic point of $U$.

Let $\Aa'(s):=\log\Bb'(s)$ for any $s\in [0,1]$. Let $\Aa_{S'}(s)$ be the adjoint foliated structure induced by adjunction
$$K_{\Aa_{S'}(s)}:=K_{\Aa'(s)}|_{S'}$$
for any $s\in [0,1]$. Then $\Aa_{S'}(t)$ is lc and $K_{\Aa_{S'}(t)}\sim_{\mathbb R,U}0$. Since $\rho(X/U)=1$ and $K_{\Bb(1)}\not\sim_{\mathbb R,U}0$, by \cite[Theorem~5.1]{Cas+25a}, $K_{\Bb(1)}$ is ample$/U$. Thus
$$K_{\Aa_{S'}(1)}=\left(h^*K_{\Bb_S(1)}\right)|_{S'}=h_S^*(K_{\Bb_S(1)}|_S)$$ is big$/U$ and nef$/U$. Since $\dim S'>\dim U$, $K_{\Aa_{S'}(1)}\not\equiv 0$ over the generic point of $U$.

Let $S'\xrightarrow{\beta} Z\rightarrow U$ be the Stein factorization of $S'\rightarrow U$. Write
$$\Aa_{S'}(s):=\left(S',\Ff_{S'},B_{S'}(s)^{\ninv}+(1-s)B_{S'}(s)^{\inv},\sum \mu_j\Mm_j|_{S'},s\right)$$
for any $s\in [0,1]$. By Theorem~\ref{thm: dcc adjunct to dcc}, there exists a DCC set $\Ii_1$ depending only on $\Ii$ such that $B_{S'}(t)\in\Ii_1$. Let
$$\Bb_{S'}:=\left(S',\Ff_{S'},B_{S'}(t),\sum \mu_j\Mm_j|_{S'}\right)(t).$$
Then
\begin{itemize}
    \item $K_{\Bb_{S'}}=K_{\Aa_{S'}(t)}\sim_{\mathbb R,Z}0$,
    \item $B_{S'}(t)\in\Ii_1$, $t\in\Ii\cap (0,1)$, each $\mu_j\in\Ii$, and each $\Mm_j|_{S'}$ is nef$/U$ $\bb$-Cartier.
\end{itemize}
By Theorem~\ref{thm: global acc general}~(1-2) in dimension $d-1$, there exists a finite set $\Ii_2$ depending only on $d$ and $\Ii$, such that for any general fiber $G$ of $\beta$, we have 
$$B_S(t)|_G\in\Ii_2\quad \text{and}\quad \{\mu_j\mid (\Mm_j|_{S'})|_G\not\equiv\bm{0}\}\subset\Ii_2.$$ Since $L'|_{S'}\neq0$ over the generic point of $Z$ for any irreducible component $L'$ of $B'$ and $(\Mm_{j,X}|_{S'})|_G\not\equiv 0$ for any $j$, by Theorem~\ref{thm: finite inversion of adjunction to finite 2}, there exists a finite set $\Ii_0$ depending only on $d$ and $\Ii$ such that $B',\mu_1,\dots,\mu_n\in\Ii_0$. In particular, $B=h_*B'\in\Ii_0$. By Theorem~\ref{thm: dcc precise adjunction}, either $t$ belongs to a finite set depending only on $d$ and $\Ii$, or $B_{S'}(t)=B_{S'}(s)$ for any $s\in [0,1)$. In the latter case, we let 
$$\Bb_{S'}(s):=\left(S',\Ff_{S'},B_{S'}(t),\sum \mu_j\Mm_j|_{S'}\right)(s)$$
for any $s\in [0,1]$. Then
$$K_{\Bb_{S'}(t)}\sim_{\mathbb R,Z}0\quad \text{ and }\quad K_{\Bb_{S'}(1)}=K_{\Aa_{S'}(1)}\not\equiv 0\text{ over the generic point of }Z.$$
By Theorem~\ref{thm: global acc general}~(3) in dimension $d-1$, $t$ belongs to a finite set depending only on $d$ and $\Ii$. The lemma follows.
\end{proof}

\begin{lem}\label{lem: when exist lc center}
Let $d$ be a positive integer and $\Ii\subset [0,+\infty)$ a DCC set. Assume that Theorem~\ref{thm: global acc general} holds in dimension $\leq d-1$ and Theorem~\ref{thm: acc lct general version} holds in dimension $\leq d$. Then there exists a finite set $\Ii_0\subset\Ii$ depending only on $d$ and $\Ii$ satisfying the following.

Notation and conditions as in Set-up~\ref{setup: global acc} and assume that $\Bb$ is not klt over the generic point of $U$. Then:
\begin{enumerate}
    \item $B,\mu_1,\dots,\mu_n\in\Ii_0$.
    \item Assume that $K_{\Bb(s)}\not\sim_{\mathbb R,U}0$ for some $s\neq t$. Then $t\in\Ii_0$.
\end{enumerate}
\end{lem}
\begin{proof}
We may assume that $1\in\Ii$. By Lemma~\ref{lem: no part 2 case}, we may assume that $K_{\Bb(s)}\not\sim_{\mathbb R,U}0$ for any $s\neq t$. By condition (4) of Set-up~\ref{setup: global acc}, $t<1$.

Let $E$ be an lc place of $\Bb$ such that $V:=\Center_XE$ is horizontal$/U$. 
By Lemma~\ref{lem: when exist lc place}, we may assume that $E$ is exceptional$/X$. Since $X$ is klt, by \cite[Theorem~2.2.3]{Cas+25a}, there exists a projective birational morphism $h\colon X'\rightarrow X$ such that $X'$ is $\mathbb Q$-factorial and $\Exc(h)=E$. We write
$$\Bb'(s):=(h^{-1}_*\Bb(s),E):=(X',\Ff',B',\Mm)(s)$$
for any $s\in [0,1]$ and let $\Bb':=\Bb'(t)$.

By Lemma~\ref{lem: acc implies precise acc} and Theorem~\ref{thm: acc lct general version}, Theorem~\ref{thm: ACC precise} holds in dimension $d$. Let $\Ii':=\lct_{\aif}(d,\Ii\cup\mathbb N)$ be the ACC set as in Theorem~\ref{thm: acc lct general version} and let $\Ii_1:=\Ii\cap\Ii'$. Then $\Ii_1$ is a finite set. By Theorem~\ref{thm: ACC precise} in dimension $d$, we have the following.
\begin{itemize}
    \item For any irreducible component $D$ of $B$, either $\mult_DB\in\Ii_1$ or $V\not\subset\Supp D$. In the latter case, $h^*D=h^{-1}_*D$ is big$/U$.
    \item For any $j$, either $\mu_j\in\Ii_1$, or $h^*\Mm_{j,X}=\Mm_{j,X'}$. In the latter case, $\Mm_{j,X'}$ is big$/U$.
    \item Either $t\in\Ii_1$, or $\Bb(s)$ is lc for some $s>t$. In the latter case, by Lemma~\ref{lem: lc preserved under inequality} and linearity of discrepancies, we have that $E$ is an lc place of $\Bb(s)$ for any $s\in [0,1]$, and $\Bb'(s)=h^*\Bb(s)$ for any $s\in [0,1]$, hence $K_{\Bb'(s)}$ is big$/U$ or anti-big$/U$ for any $s\in [0,1]$ such that $s\neq t$.
\end{itemize}
We run a $(K_{\Bb'}-\delta E)$-MMP$/U$ with scaling of an ample divisor for some $0<\delta\ll 1$. Since $K_{\Bb'}=h^*K_{\Bb}\sim_{\mathbb R,U}0$, this is a $(-E)$-MMP$/U$ so $E$ is not contracted by this MMP. By \cite[Theorem~2.1.5]{Cas+25a}, the MMP terminates with a Mori fiber space$/U$ $g: X''\rightarrow T$ associated with a birational map$/U$ $\phi\colon X'\dashrightarrow X''$. Let $E'':=\phi_*E$, then $E''$ is ample$/T$. Let $\Bb''(s):=\phi_*\Bb'(s)$ for any $s\in [0,1]$. Since $\phi$ is $K_{\Bb'}$-trivial, $\Bb''(t)$ is lc. Since $X'$ is $\mathbb Q$-factorial klt, $X''$ is $\mathbb Q$-factorial klt.

Let $B''$ be the horizontal$/T$ part of $\phi_*B'$ and let $$\Nn:=\sum_{j\mid \Mm_{j,X''}\not\sim_{\mathbb R,T}0}\mu_j\Mm_j.$$ 
Let
$$\Cc(s):=(X'',\Ff'':=\phi_*\Ff',B'',\Nn)(s)$$
for any $s\in [0,1]$. Then $\Bb''(s)\geq\Cc(s)$ for any $s\in [0,1]$. Since $\Bb''(t)$ is lc, by Lemma~\ref{lem: lc preserved under inequality}, $\Cc(t)$ is lc. Since $\rho(X''/T)=1$,
$$K_{\Cc(s)}\sim_{\mathbb R,T}K_{\Bb''(s)}$$
for any $s\in [0,1]$. In particular, 
$$K_{\Cc(t)}\sim_{\mathbb R,T}K_{\Bb''(t)}\sim_{\mathbb R,T}0.$$

For any irreducible component $D$ of $B$ (resp. any $j$, $t$), either $\mult_DB\in\Ii_1$ (resp. $\mu_j\in\Ii_1$, $t\in\Ii_1$), or $h^{-1}_*D$ is big$/U$ (resp. $\Mm_{j,X'}$ is big$/U$, $K_{\Bb'(s)}$ is big$/U$ or anti-big$/U$ for any $s\in [0,1]$ such that $s\neq t$). In the latter case, $\phi_*h^{-1}_*D$ is big$/U$ (resp. $\Mm_{j,X''}$ is big$/U$, $K_{\Bb''(s)}$ is big$/U$ or anti-big$/U$ for any $s\in [0,1]$ such that $s\neq t$), hence $\phi_*h^{-1}_*D$ is ample$/T$ (resp.  $\Mm_{j,X''}$ is ample$/T$, $K_{\Bb''(s)}$ is ample$/T$ or anti-ample$/T$ for any $s\in [0,1]$ such that $t\neq s$), so $\phi_*h^{-1}_*D$ is an irreducible component of $B''$ (resp. $\Mm_{j,X''}\not\sim_{\mathbb R,T}0$, $K_{\Bb''(s)}\not\sim_{\mathbb R,T}0$ for any $s\in [0,1]$ such that $t\neq s$). There are two cases.

\medskip

\noindent\textbf{Case 1.} We have $K_{\Cc(s)}\not\sim_{\mathbb R,T}0$ for some $s\neq t$. Then, since $\rho(X''/T)=1$, we have that $K_{\Cc(s)}\not\sim_{\mathbb R,T}0$ for any $s\neq t$. By Lemma~\ref{lem: when exist lc place} applied to $\Cc(t)$ and $g: X''\rightarrow T$, there exists a finite set $\Ii_2$ depending only on $d$ and $\Ii$ such that $B'',t\in\Ii_2$ and $\mu_j\in\Ii_2$ for any $j$ such that $\Mm_{j,X''}\not\sim_{\mathbb R,T}0$. Thus $B,t,\mu_1,\dots,\mu_n\in\Ii_1\cup\Ii_2$ and we are done.

\medskip

\noindent\textbf{Case 2.} We have $K_{\Cc(s)}\sim_{\mathbb R,T}0$ for some $s\neq t$, hence $K_{\Cc(s)}\sim_{\mathbb R,T}0$ for any $s\neq t$. Thus $t\in\Ii_1$. By Lemma~\ref{lem: no part 2 case}, there exists a finite set $\Ii_2$ depending only on $d$ and $\Ii$ such that $B'',t\in\Ii_2$ and $\mu_j\in\Ii_2$ for any $j$ such that $\Mm_{j,X''}\not\sim_{\mathbb R,T}0$. Thus $B,t,\mu_1,\dots,\mu_n\in\Ii_1\cup\Ii_2$ and we are done.
\end{proof}

\subsection{Constructing gaps}

In the statement of the global ACC (Theorem~\ref{thm: global acc general}), we need to exclude the possibility that $t$ approaches $1$. Yet in this case, the foliated version of BAB \cite[Theorem~B]{Cas+25a} cannot be applied to the proof. To resolve this issue, we need the following technical lemma which essentially allows us to replace $t$ by a number that is bounded away from $1$.

\begin{lem}\label{lem: one lc place one close to one}
Let $d$ be a positive integer and $\Ii_1\subset [0,+\infty)$ a DCC set. Assume that Theorem~\ref{thm: global acc general} holds in dimension $\leq d-1$ and Theorem~\ref{thm: acc lct general version} holds in dimension $\leq d$. Then there exists a real number $\tau=\tau_{\aux}(d,\Ii_1)\in (0,1)$ depending only on $d$ and $\Ii_1$ satisfying the following.

Let $\Ii:=\Ii_1\cup [1-\tau,1]$. Notation and conditions as in Set-up~\ref{setup: global acc}. Further assume that \begin{enumerate}
\item $K_{\Bb(s)}\not\sim_{\mathbb R,U}0$ for some $s\neq t$, 
    \item $\mu_j\in\Ii_1$ for any $j$,
    \item $S$ is an irreducible component of $\lfloor B\rfloor$, and
    \item for any irreducible component $D\neq S$
    of $B$, $\mult_DB\in\Ii_1$.
\end{enumerate}
Then $\mult_LB\not\in (1-\tau,1)$ and $t\not\in (1-\tau,1)$.
\end{lem}
\begin{proof}
Assume that the lemma does not hold. Then for any positive integer $i$, there exists a $\mathbb Q$-factorial lc algebraically integrable normalized foliated structure 
$$\Bb_i/U_i:=\left(X_i,\Ff_i,B_i,\Mm^i\right)(t_i)$$
satisfying the following:
\begin{enumerate}
\item $\pi_i: X_i\rightarrow U_i$ is a contraction between normal quasi-projective varieties, $\dim X_i=d>\dim U_i$, $X_i$ is of Fano type over $U_i$, and $\rho(X_i/U_i)=1$. 
\item $X_i$ is klt and $t_i\in (0,1)$.
\item $\Bb_i(s):=\left(X_i,\Ff_i,B_i,\Mm^i\right)(s)$ for any $s\in [0,1]$.
\item Any irreducible component of $B_i$ is horizontal$/U_i$, $L_i$ is an irreducible component of $B_i$, and for any irreducible component $D\neq L_i$ of $B_i$, $\mult_DB_i\in\Ii_1$.   
\item $S_i$ is an irreducible component of $\lfloor B_i\rfloor$.
\item $\Mm^i=\sum_{j=1}^{n_i}\mu_{i,j}\Mm^i_j$ where each $\mu_{i,j}\in\Ii_1$, $\Mm^i_j$ is nef$/U_i$ $\bb$-Cartier, and $\Mm^i_{j,X_i}$ is ample$/U_i$.
\item $K_{\Bb_i(t_i)}\sim_{\mathbb R,U_i}0$ and $K_{\Bb_i(s)}\not\sim_{\mathbb R,U_i}0$ for any $s\neq t_i$,
\item $r_i:=\mult_{L_i}B_i\in\Ii_1\cup [1-1/i,1]$ and $t_i\in\Ii_1\cup [1-1/i,1]$, and
\item either $r_i\in (1-1/i,1)$ or $t_i\in (1-1/i,1)$.
\end{enumerate}
Possibly passing to a subsequence, we may assume that $r_i\in (1-1/i,1)$ for any $i$ or $r_i\in\Ii_1$ for any $i$, and $t_i\in (1-1/i,1)$ for any $i$ or $t_i\in\Ii_1$ for any $i$. Thus $\Ii_2:=\Ii_1\cup\{r_i\}_{i=1}^{+\infty}\cup\{t_i\}_{i=1}^{+\infty}$ is a DCC set. By applying Lemma~\ref{lem: when exist lc center} to the DCC set $\Ii_2$, we deduce that there exists a finite set $\Ii_0$ depending only on $d,\Ii$, and $\{r_i\}$ such that $r_i\in\Ii_0,t_i\in\Ii_0$ for any $i$. This is not possible as either $r_i\in (1-1/i,1)$ or $t_i\in (1-1/i,1)$ for any $i$.
\end{proof}

\subsection{Relative two-ray game and inductive proof of Theorem~\ref{thm: global acc general}}

We need the following two-ray game lemma which can be seen as a relative version of \cite[Lemma~9.1]{MP04}. With this and all the preparations above, we are able to prove Theorem~\ref{thm: global acc general} by induction on the dimension (Theorem~\ref{thm: acc to global acc}).

\begin{lem}\label{lem: relative two-ray}
    Let $\pi: X\rightarrow U$ be a contraction between normal quasi-projective varieties such that $X$ is $\mathbb Q$-factorial, $\rho(X/U)=2$, and $\dim X>\dim U$. Let $D$ be a big$/U$ $\mathbb R$-divisor on $X$. Let $\phi_i: X_i\xrightarrow{\psi_i} Z_i\xleftarrow{\psi_i^+} X_{i+1}$, $0\leq i\leq n-1$, $X_0:=X$ be a sequence of $G_i$-flips$/U$ for some $\mathbb R$-Cartier $\mathbb R$-divisors $G_i$ on $X_i$ 
    such that each $X_i$ is $\mathbb Q$-factorial, and let $D_i$ be the image of $D$ on $X_i$ for each $i$. Assume that for each $0\leq i\leq n$, $\rho(X_i/U)$ is spanned by extremal rays $R_i,S_i$ satisfying the following:
    \begin{enumerate}
        \item $D\cdot R_0\leq 0$.
        \item For any $0\leq i\leq n-1$, $S_i$ is contracted by $\psi_i$ and $R_{i+1}$ is contracted by $\psi_i^+$.
    \end{enumerate}
    Then for any $0\leq i\leq n$, $D_i\cdot S_i>0$, and for any $1\leq i\leq n$, $D_i\cdot R_i<0$.
\end{lem}
\begin{proof}
Assume that $D_i\cdot R_i\leq 0$. If $D_i\cdot S_i\leq 0$, then $D_i$ is anti-nef$/U$ by Kleiman's criterion. Since $D$ is big$/U$, $D_i$ is big$/U$, so $0=D_i+(-D_i)$ is big$/U$. This is not possible as $\dim X>\dim U$. Thus $D_i\cdot S_i>0$, hence $D_{i+1}\cdot R_{i+1}<0$. We are done by induction on $i$ as we already have $D_0\cdot R_0\leq 0$.
\end{proof}

\begin{lem}\label{lem: fano type after extract divisor}
    Let $\Bb/U=(X,\Ff,B,\Mm)(t)$ be a klt algebraically integrable normalized foliated structure. Assume that $X$ is of Fano type over $U$ and $-K_{\Bb}$ is nef$/U$. Let $h\colon X'\rightarrow X$ be a projective birational morphism such that $a(E,\Bb)<0$ for any $h$-exceptional prime divisor $E$. Then $X'$ is of Fano type over $U$.
\end{lem}
\begin{proof}
    There exists a klt pair $(X,\Delta)$ with $-(K_X+\Delta)$ ample$/U$. Let $\Bb':=(X,\Ff,\Delta,\bm{0})(0)$. Then for any $0<\epsilon\ll 1$, the structure $\Cc:=\Bb'^{\epsilon}\cdot\Bb^{1-\epsilon}$ is klt, $-K_{\Cc}$ is ample$/U$, and $a(E,\Cc)<0$ for any $h$-exceptional prime divisor $E$. Let $\Cc':=h^*\Cc$. Then $\Cc'$ is a klt algebraically integrable normalized foliated structure, and $-K_{\Cc'}$ is nef$/U$ and big$/U$. By \cite[Theorem~9.2]{Cas+25a}, $X'$ is of Fano type over $U$.
\end{proof}

\begin{thm}\label{thm: acc to global acc}
Assume that Theorem~\ref{thm: global acc general} holds in dimension $\leq d-1$ and Theorem~\ref{thm: acc lct general version} holds in dimension $\leq d$. Then Theorems~\ref{thm: global acc general} and~\ref{thm: picard number one case} hold in dimension $d$.
\end{thm}
\begin{proof}
By Theorem~\ref{thm: picard one imply general global acc}, we only need to prove that Theorem~\ref{thm: picard number one case} holds in dimension $d$. We argue by induction on $d_F = \dim F$. When $\dim F=1$, Theorem~\ref{thm: picard number one case} follows from Proposition~\ref{prop: when fiber bounded}. Fix $d'>1$ and assume that the theorem holds for $d_F<d'$. In what follows, we argue by contradiction.

\medskip

\noindent\textbf{Step 1.} In this step, we construct a sequence of normalized foliated structures and auxiliary parameters $t_i'$.

Suppose that the theorem does not hold. By Theorem~\ref{thm: picard one imply general global acc}, there exists a DCC set $\Ii\subset [0,+\infty)$ and the following data for any positive integer $i$:
\begin{enumerate}
\item $\pi_i: X_i\rightarrow U_i$ is a contraction between normal quasi-projective varieties, $\dim X_i=d>\dim U_i$, $X_i$ is of Fano type over $U_i$, and $\rho(X_i/U_i)=1$. 
\item $F_i$ is a general fiber of $\pi_i$ and $d_F=\dim F_i$.
\item $\dim X_i=d$ and $\dim X_i-\dim U_i=d'$.
\item $\Bb_i/U_i:=(X_i,\Ff_i,B_i,\Mm^i)(t_i)/U_i$ is a $\mathbb Q$-factorial lc algebraically integrable normalized foliated structure such that $X_i$ is klt and $t_i\in\Ii\cap (0,1)$.
\item $\Bb_i(s):=(X_i,\Ff_i,B_i,\Mm^i)(s)$ for any $s\in [0,1]$.
\item $B_i=\sum_{j=1}^{m_i}b_{i,j}B_{i,j}$, where $B_{i,j}$ are the irreducible components of $B_i$ and $b_{i,j}\in\Ii$ for any $i,j$, and each $B_{i,j}$ is horizontal$/U_i$,
\item $\Mm^i=\sum_{j=1}^{n_i}\mu_{i,j}\Mm^i_j$, where each $\mu_{i,j}\in\Ii$, $\Mm^i_j$ is nef$/U_i$ $\bb$-Cartier, and $\Mm^i_{j,X_i}$ is ample$/U_i$.
\item $K_{\Bb_i(t_i)}\sim_{\mathbb R,U_i}0$ for any $i$.
\item One of the following conditions holds:
    \begin{enumerate}
        \item $b_{i,1}$ is strictly increasing and $\lim_{i\rightarrow+\infty}b_{i,1}=1$.
        \item $b_{i,1}$ is strictly increasing and $\lim_{i\rightarrow+\infty}b_{i,1}<1$.
        \item $\mu_{i,1}$ is strictly increasing, and $\{b_{i,j}\}$ is a finite set.
        \item $t_i$ is strictly increasing, and $\{b_{i,j},\mu_{i,j}\}$ is a finite set.
    \end{enumerate}
\end{enumerate}
By \cite[Theorem~5.1]{Cas+25a}, $K_{\Bb_i(s)}$ is ample$/U_i$ for any $s>t_i$ and $K_{\Bb_i(s)}$ is anti-ample$/U_i$ for any $s<t_i$. Possibly replacing $\Ii$, we may assume that $\mathbb N\subset\Ii$ and $\Ii=\overline{\Ii}$. Possibly passing to a subsequence, we may assume that $t_i$ is increasing. By Lemma~\ref{lem: no part 2 case}, possibly passing to a subsequence, we may assume that $K_{\Bb_i(s)}\not\sim_{\mathbb R,U_i}0$ for any $s\neq t_i$.

By the inductive hypothesis on $d_F$, we may assume that there exists a positive real number $\tau(d,\Ii,d')$ depending only on $d$, $\Ii$, and $d'$ satisfying the following: for any data as in Set-up~\ref{setup: global acc}, if $d_F<d'$, then $t<1-\tau(d,\Ii,d')$. We define
$$\tau:=\frac{1}{2}\min\{\tau_{\ilct}(d,\Ii),\tau_{\aux}(d,\Ii),\tau(d,\Ii,d')\}$$
where $\tau_{\ilct}(d,\Ii)$ is the positive real number as in Theorem~\ref{thm: acc lct general version} in dimension $d$ and $\tau_{\aux}(d,\Ii)$ is the positive real number as in Lemma~\ref{lem: one lc place one close to one}. We define
$$t_i':=\min\{t_i,1-\tau\}$$
for any $i$. Possibly passing to a subsequence, one of the following cases holds:
\begin{itemize}
    \item (\textbf{Case 1}) $t_i'=1-\tau<t_i$ for any $i$.
    \item (\textbf{Case 2}) $t_i'=t_i\leq 1-\tau$ for any $i$.
\end{itemize}
By Lemma~\ref{lem: when exist lc center}, we may assume that $\Bb_i$ is klt over the generic point of $U_i$ for each $i$. Possibly shrinking $U_i$, we may assume that $\Bb_i$ is klt for each $i$.

\medskip

\noindent\textbf{Step 2.} In this step, we construct a divisor $E_i$ with small discrepancy with respect to $\Bb_i$, and construct two birational models $X_i'$ and $X_i''$ of $X_i$. We let 
$$\Bb_i^o:=(X_i,\Ff_i,0,\bm{0})(t_i')$$
for each $i$ and define
$$a_i:=\inf\{a(E,\Bb_i^o)+t_i'\epsilon_{\Ff_i}(E)+(1-t_i')\mid\Center_{X_i}E\text{ dominates }U_i\}.$$
Possibly passing to a subsequence, we may assume that $a_i$ is either increasing or decreasing, and we may let $a:=\lim_{i\rightarrow+\infty}a_i$. Possibly passing to a subsequence, we may assume that $a_i\geq\frac{a}{2}$ for any $i$. 

Suppose that $a>0$. Then $\Bb_i^o$ is $\frac{a}{2}$-lc for any $i$. By \cite[Proposition~9.1]{Cas+25a}, $X_i$ is $\frac{a}{2}$-lc over the generic point of $U_i$. Since $\rho(X_i/U_i)=1$ and $X_i$ is of Fano type over $U_i$, $F_i$ is an $\frac{a}{2}$-lc Fano variety. By \cite[Theorem~1.1]{Bir21a}, $F_i$ belongs to a bounded family, contradicting Proposition~\ref{prop: when fiber bounded}. Therefore, $a=0$. Possibly passing to a subsequence, we may assume that $a_i<\tau$ for any $i$. 

Since $\Bb_i$ is klt, by Lemma~\ref{lem: lc preserved under inequality}, $\Bb_i(t_i')$ is klt, hence $\Bb_i^o$ is klt. By \cite[Lemma~2.13]{CLW26}, there are only finitely many prime divisors $E$ over $X_i$ such that $a(E,\Bb_i^o)<0$. Since $t_i'\leq 1-\tau$, there are only finitely many prime divisors $E$ over $X_i$ such that
$$a(E,\Bb_i^o)+t_i'\epsilon_{\Ff_i}(E)+(1-t_i')<\tau,$$
so
$$a_i=\min\{a(E,\Bb_i^o)+t_i'\epsilon_{\Ff_i}(E)+(1-t_i')\mid\Center_{X_i}E\text{ dominates }U_i\}$$ 
for each $i$. In particular, we may let $E_i$ be a prime divisor over $X_i$ such that $\Center_{X_i}E_i$ dominates $U_i$ and 
$$a(E_i,\Bb_i^o)+t_i'\epsilon_{\Ff_i}(E_i)+(1-t_i')=a_i.$$
Then $a(E_i,\Bb_i^o)<0$. Since the boundary of $\Bb_i^o$ is $0$, $E_i$ is exceptional$/X_i$. By \cite[Theorem~2.2.3]{Cas+25a}, there exists a projective birational morphism $h_i: X_i'\rightarrow X_i$ such that $X_i'$ is $\mathbb Q$-factorial and $\Exc(h_i)=E_i$. We write
$$\Bb_i':=h_i^*\Bb_i(t_i'):=(X_i',\Ff_i',B_i'+e_iE_i,\Mm^i)(t_i')$$
and let 
$$\Bb_i'(s):=(X_i',\Ff_i',B_i'+e_iE_i,\Mm^i)(s)$$
for any $s\in [0,1]$, where $B_i':=h^{-1}_{i,*}B_i$. Then we have
$$e_i=-\frac{a(E_i,\Bb_i(t_i'))}{t_i'\epsilon_{\Ff_i}(E_i)+(1-t_i')}\geq -\frac{a(E_i,\Bb_i^o)}{t_i'\epsilon_{\Ff_i}(E_i)+(1-t_i')}=1-\frac{a_i}{t_i'\epsilon_{\Ff_i}(E_i)+(1-t_i')}.$$
Since $t_i'\leq 1-\tau$ and $a=0$, possibly passing to a subsequence, we may assume that $e_i$ is strictly increasing and $\lim_{i\rightarrow+\infty}e_i=1$.

Since $t_i'\leq t_i$, $-K_{\Bb_i(t_i')}$ is nef$/U_i$. By Lemma~\ref{lem: fano type after extract divisor}, $X_i'$ is of Fano type over $U_i$. Since $E_i$ dominates $U_i$ and $\dim X_i>\dim U_i$, $-E_i$ is not pseudo-effective$/U_i$. Thus we may run a $(-E_i)$-MMP$/U_i$ with scaling of an ample divisor which terminates with a Mori fiber space $g_i: X_i''\rightarrow T_i$. We let $\phi_i: X_i'\dashrightarrow X_i''$ be the induced birational map and let
$$E_i'':=\phi_{i,*}E_i\quad \text{and}\quad \Bb_i''(s):=\phi_{i,*}\Bb_i'(s):=(X_i'',\Ff_i'',B_i''+e_iE_i'',\Mm^i)(s)$$
for any $s\in [0,1]$. Then $g_i$ is a $(-E_i'')$-Mori fiber space$/U_i$, so $E_i''$ dominates $T_i$. Since $\Bb_i'(t_i')$ is klt and $-K_{\Bb_i'(t_i')}$ is semi-ample$/U_i$, $\Bb_i''(t_i')$ is klt. Moreover:
\begin{itemize}
    \item If we are in \textbf{Case 1}, then $K_{\Bb_i'(t_i')}$ is anti-big$/U_i$, hence $K_{\Bb_i''(t_i')}$ is anti-big$/U_i$, and so $K_{\Bb_i''(t_i')}$ is anti-ample$/T_i$.
    \item If we are in \textbf{Case 2}, then $K_{\Bb_i'(t_i)}\sim_{\mathbb R,U_i}0$, so $K_{\Bb_i''(t_i)}\sim_{\mathbb R,T_i}0.$
\end{itemize}
We let $B_i''^o$ be the horizontal$/T_i$ part of $B_i''$,
$$\Nn^i:=\sum_{j\mid\Mm^i_{j,X''}\not\sim_{\mathbb R,T_i}0}\mu_{i,j}\Mm^i_{j},$$
and let
$$\Dd_i(\mu,s):=(X_i'',\Ff_i'',B_i''^o+\mu E_i'',\Nn^i)(s)$$
for any $\mu,s\in [0,1]$. Then we have
$$K_{\Dd_i(e_i,s)}\sim_{\mathbb R,T_i}K_{\Bb_i''(s)}$$
for any $s\in [0,1]$. Since $\Bb_i''(t_i')$ is klt, $\Dd_i(e_i,t_i')$ is klt. Since $\tau<\tau_{\ilct}(d,\Ii)$ and $\lim_{i\rightarrow+\infty}e_i=1$, possibly passing to a subsequence and applying Lemma~\ref{lem: lc preserved under inequality}, we may assume that for any $\mu\in [0,1]$ and any $s\in [0,t_i]$, $\Dd_i(\mu,s)$ is lc. Moreover, if we are in \textbf{Case 1}, then $\Dd_i(\mu,s)$ is lc for any $\mu,s\in [0,1]$.

\medskip

\noindent\textbf{Step 3.} We prove the theorem when $b_{i,1}$ is strictly increasing and $\lim_{i\rightarrow+\infty}b_{i,1}=1$. In the rest of this step, we assume that $\lim_{i\rightarrow+\infty}b_{i,1}=1$.

We let $B'_{i,1}:=h_{i,*}^{-1}B_{i,1}$, $B''_{i,1}:=\phi_{i,*}B'_{i,1}$, $\Delta_i:=B_i-b_{i,1}B_{i,1}$, $\Delta_i':=B_i'-b_{i,1}B'_{i,1}$, and $$\Delta_i'':=\phi_{i,*}\Delta_i'=B_i''-b_{i,1}B''_{i,1}.$$
Then we have
\begin{align*}
\Bb_i'(s) &= (X_i',\Ff_i',b_{i,1}B'_{i,1}+\Delta_i'+e_iE_i',\Mm^i)(s),\\
\Bb_i''(s) &= (X_i'',\Ff_i'',b_{i,1}B''_{i,1}+\Delta_i''+e_iE_i'',\Mm^i)(s)
\end{align*}
for any $s\in [0,1]$. We let
$$\Cc_i(s):=(X_i,\Ff_i,B_{i,1}+\Delta_i,\Mm^i)(s),\quad \Cc_i'(s):=(X_i',\Ff_i',B_{i,1}'+\Delta_i'+E_i,\Mm^i)(s),$$
and
$$\Cc_i''(s):=(X_i'',\Ff_i'',B_{i,1}''+\Delta_i''+E_i'',\Mm^i)(s)$$
for any $s\in [0,1]$. Since $$\lim_{i\rightarrow+\infty}b_{i,1}=\lim_{i\rightarrow+\infty}e_i=1$$ 
and $\tau<\tau_{\ilct}(d,\Ii)$, and $\Bb_i(t_i'),\Bb_i'(t_i'),\Bb_i''(t_i')$ are klt, by definition of $\tau$ and by Theorem~\ref{thm: acc lct general version}, possibly passing to a subsequence, we have that
\begin{itemize}
    \item $\Cc_i(s),\Cc_i'(s),\Cc_i''(s)$ are lc for any $s\in [1-\tau,1]=[1-t_i',1]$ if we are in \textbf{Case 1}, and
    \item $\Cc_i(t_i'),\Cc_i'(t_i'),\Cc_i''(t_i')$ are lc if we are in \textbf{Case 2}. 
\end{itemize}
In particular, $\Cc_i(t_i),\Cc_i'(t_i),\Cc_i''(t_i)$ are lc. Since $K_{\Bb_i(t_i)}\sim_{\mathbb R,U_i}0$, $B_{i,1}$ is ample$/U_i$, and $\Bb_i(s)=(\Cc_i(s),(b_{i,1}-1)B_{i,1})$ for any $s\in [0,1]$, $K_{\Cc_i(t_i)}$ is ample$/U_i$. Since $h_i$ contracts $E_i$ and $\Cc_i(t_i)$ is lc, we have that
$$K_{\Cc_i'(t_i)}\geq h_i^*K_{\Cc_i(t_i)}.$$
Therefore, $K_{\Cc_i'(t_i)}$ is big$/U_i$. Thus $K_{\Cc_i''(t_i)}$ is big$/U_i$. Since $\rho(X_i''/T_i)=1$, $K_{\Cc_i''(t_i)}$ is ample$/T_i$. 
\begin{claim}\label{claim: no divisorial contraction}
    Possibly passing to a subsequence, we have that $B_{i,1}'$ is not contracted by $\phi_i$ for any $i$.
\end{claim}
\begin{proof}
    Suppose the claim does not hold. Since $\rho(X_i'/U_i)=2$, we may write $\phi_i=\psi_i\circ\varphi_i$ where $\varphi_i: X_i'\dashrightarrow Y_i$ is a sequence of flips and $\psi_i: Y_i\rightarrow X_i''$ is the divisorial contraction of $B_{Y_i,1}:=(\varphi_i)_*B_{i,1}'$. Let 
    $$\Bb_{Y_i}(s):=\varphi_{i,*}\Bb_i'(s) \quad \text{and}\quad \Cc_{Y_i}(s):=\varphi_{i,*}\Cc_i'(s)$$
    for any $s\in [0,1]$. Since $\Bb_i'(t_i')$ is klt and $-K_{\Bb_i'(t_i')}$ is semi-ample$/U_i$, $\Bb_{Y_i}(t_i')$ is klt. Since $\tau<\tau_{\ilct}(d,\Ii)$, by Theorem~\ref{thm: acc lct general version}, possibly passing to a subsequence, we have that $\Cc_{Y_i}(s)$ is lc for any $t_i'\leq s\leq t_i$. Since $K_{\Cc_i'(t_i)}$ is big$/U_i$, $K_{\Cc_{Y_i}(t_i)}$ is big$/U_i$. Since $\psi_i$ contracts $B_{Y_i,1}$ and $\Cc_i''(t_i)$ is lc, we have
    $$K_{\Cc_{Y_i}(t_i)}\geq \psi_i^*K_{\Cc_i''(t_i)}.$$
    Let $R_i$ be the unique extremal ray of $\overline{NE}(X_i'/X_i)$ and let $S_i$ be the unique extremal ray of $\overline{NE}(Y_i/X_i'')$. By the negativity lemma, $K_{\Cc_i'(t_i)}\cdot R_i\leq 0$ and $K_{\Cc_{Y_i}(t_i)}\cdot S_i\leq 0$. This contradicts Lemma~\ref{lem: relative two-ray}.
\end{proof}

\begin{claim}\label{claim: no trivial case}
    Possibly passing to a subsequence, we have that $B_{i,1}''$ is ample$/T_i$.
\end{claim}
\begin{proof}
Suppose the claim does not hold. By Claim~\ref{claim: no divisorial contraction}, possibly passing to a subsequence, we may assume that $B_{i,1}''\neq0$ but $B_{i,1}''$ is vertical$/T_i$ for any $i$. Since $B_{i,1}$ dominates $U_i$, $B_{i,1}''$ dominates $U_i$. Thus $\dim T_i>\dim U_i$, so $$\dim X_i-\dim T_i<\dim X_i-\dim U_i=d'.$$
We have
$$K_{\Dd_i(1,s)}\sim_{\mathbb R,T_i}K_{\Cc_i''(s)}$$
for any $s\in [0,1]$. Let $\Ii':=\Ii\cup\{e_i\}_{i=1}^{+\infty}$, then $\Ii'$ is a DCC set. 

By \textbf{Step 2}, if we are in \textbf{Case 2}, then $K_{\Dd_i(t_i,s)}\sim_{\mathbb R,T_i}K_{\Bb_i''(t_i)}\sim_{\mathbb R,T_i}0$. By the inductive hypothesis on $d_F$, $\{e_i\}$ is a finite set, a contradiction. Therefore, we are in \textbf{Case 1}, and we have $t_i'=1-\tau$ for each $i$. By \textbf{Step 2}, $K_{\Dd_i(e_i,1-\tau)}\sim_{\mathbb R,T_i}K_{\Bb_i''(t_i')}$ is anti-ample$/T_i$.

If $K_{\Dd_i(1,1-\tau)}$ is ample$/T_i$, then there exists $e_i'\in (e_i,1)$ such that $K_{\Dd_i(e_i',1-\tau)}\sim_{\mathbb R,T_i}0$. By the inductive hypothesis on $d_F$, $\{e_i'\}$ belongs to a finite set, a contradiction. Thus $K_{\Dd_i(1,1-\tau)}$ is not ample$/T_i$. Since $K_{\Dd_i(1,t_i)}\sim_{\mathbb R,T_i}K_{\Cc_i''(t_i)}$ is ample$/T_i$, there exists $t_i''\in [1-\tau,t_i)$ such that $K_{\Dd_i(1,t_i'')}\sim_{\mathbb R,T_i}0$. Since $E_i''$ is an irreducible component of $\lfloor B_i''^o+E_i''\rfloor$ and $\tau<\tau_{\aux}(d,\Ii)$, by Lemma~\ref{lem: one lc place one close to one}, $t_i''\not\in (1-\tau_{\aux}(d,\Ii),1)$. Thus
$$1-\tau\leq t_i'\leq 1-\tau_{\aux}(d,\Ii)<1-\tau,$$
a contradiction.
\end{proof}
\noindent\textit{Proof of Theorem~\ref{thm: acc to global acc} continued.} By Claims~\ref{claim: no divisorial contraction} and~\ref{claim: no trivial case}, we may assume that $B_{i,1}''$ is ample$/T_i$ for each $i$. We let $\Delta_i''^o$ be the horizontal$/T_i$ part of $\Delta_i''$ and let
$$\Dd_i(\lambda,\mu,s):=\left(X_i'',\Ff_i'',\Delta_i''^o+\lambda B_{i,1}''+\mu E_i'',\Nn^i\right)(s)$$
for any $\lambda,\mu,s\in [0,1]$. Then we have
$$K_{\Dd_i(b_{i,1},e_i,s)}\sim_{\mathbb R,T_i}K_{\Bb_i''(s)}\text{ and }K_{\Dd_i(1,1,s)}\sim_{\mathbb R,T_i}K_{\Cc_i''(s)}$$
for any $s\in [0,1]$. By \textbf{Step 2}, if we are in \textbf{Case 1}, then $K_{\Dd_i(b_{i,1},e_i,t_i')}$ is anti-ample$/T_i$, and if we are in \textbf{Case 2}, then $K_{\Dd_i(b_{i,1},e_i,t_i')}\sim_{\mathbb R,T_i}0$. Since $\tau<\tau_{\ilct}(d,\Ii)$, possibly passing to a subsequence and applying Lemma~\ref{lem: lc preserved under inequality}, we may assume that for any $\lambda\in [0,1],\mu\in [0,1]$, and $s\in [0,t_i]$, $\Dd_i(\lambda,\mu,s)$ is lc. Moreover, if we are in \textbf{Case 1}, then $\Dd_i(\lambda,\mu,s)$ is lc for any $\lambda,\mu,s\in [0,1]$.

Since $\rho(X_i''/T_i)=1$, we have that
$$B_{i,1}''^{\ninv}+(1-t_i')B_{i,1}''^{\inv}\sim_{\mathbb R,T_i}c_i(E_i''^{\ninv}+(1-t_i')E_i''^{\inv})$$
for some positive real number $c_i$. Possibly passing to a subsequence, we may assume that $c_i\leq 1$ for each $i$ or $c_i>1$ for each $i$. 

If $c_i\leq 1$ (resp. $c_i>1$), then we let $e_i':=e_i-\lambda_i(1-b_{i,1})$ (resp. $b_i':=b_i-\frac{1}{\lambda_i}(1-e_i)$). Possibly passing to a subsequence, we may assume that $e_i'\in (0,1)$ (resp. $b_i'\in (0,1)$) is strictly increasing and $\lim_{i\rightarrow+\infty}e_i'=1$ (resp. $\lim_{i\rightarrow+\infty}b_i'=1$). 

We have
$$K_{\Dd_i(1,e_i',t_i')}\sim_{\mathbb R,T_i}K_{\Dd_i(b_i,e_i,t_i')}\sim_{\mathbb R,T_i}K_{\Bb_i''(t_i')}$$
$$\text{(resp. }K_{\Dd_i(b_i',1,t_i')}\sim_{\mathbb R,T_i}K_{\Dd_i(b_i,e_i,t_i')}\sim_{\mathbb R,T_i}K_{\Bb_i''(t_i')}\text{)}.$$

If we are in \textbf{Case 2}, then $K_{\Dd_i(1,e_i',t_i')}\sim_{\mathbb R,T_i}0$ (resp. $K_{\Dd_i(b_i',1,t_i')}\sim_{\mathbb R,T_i}0$). If $K_{\Dd_i(1,e_i',s)}\sim_{\mathbb R,T_i}0$ (resp. $K_{\Dd_i(b_i',1,s)}\sim_{\mathbb R,T_i}0$) for any $s\in [0,1]$, then by Lemma~\ref{lem: no part 2 case}, $\{e_i'\}$ (resp. $\{b_i'\}$) is a finite set, a contradiction. Thus $K_{\Dd_i(1,e_i',s)}\not\sim_{\mathbb R,T_i}0$ (resp. $K_{\Dd_i(b_i',1,s)}\not\sim_{\mathbb R,T_i}0$) for any $s\in [0,1]$ such that $s\neq t$. By Lemma~\ref{lem: one lc place one close to one}, $e_i'<1-\tau_{\aux}(d,\Ii)$ (resp. $b_i'<1-\tau_{\aux}(d,\Ii)$) for any $i$. This is not possible as $\lim_{i\rightarrow+\infty}e_i'=1$ (resp. $\lim_{i\rightarrow+\infty}b_i'=1$). Therefore, we are in \textbf{Case 1}, and we have $t_i'=1-\tau$ and $K_{\Dd_i(1,e_i',1-\tau)}$ (resp. $K_{\Dd_i(b_i',1,1-\tau)}$) is anti-ample$/T_i$. 

If $K_{\Dd_i(1,1,1-\tau)}$ is ample$/T_i$ for infinitely many $i$, then possibly passing to a subsequence, we may assume that $K_{\Dd_i(1,1,1-\tau)}$ is ample$/T_i$ for each $i$. Then there exists $e_i''\in (e_i',1)$ (resp. $b_i''\in (b_i',1)$) such that $K_{\Dd_i(1,e_i'',1-\tau)}\sim_{\mathbb R,T_i}0$ (resp. $K_{\Dd_i(b_i'',1,1-\tau)}\sim_{\mathbb R,T_i}0$). Possibly passing to a subsequence, we may assume that $e_i''$ (resp. $b_i''$) is strictly increasing. If $K_{\Dd_i(1,e_i'',s)}\sim_{\mathbb R,T_i}0$ (resp. $K_{\Dd_i(b_i'',1,s)}\sim_{\mathbb R,T_i}0$) for any $s\in [0,1]$ for $i\in\Lambda$ for an infinite set $\Lambda\subset\mathbb N^+$, then we get a contradiction to Lemma~\ref{lem: no part 2 case} as $\{e_i''\}_{i\in\Lambda}$ (resp.  $\{b_i''\}_{i\in\Lambda}$) is not a finite set. Thus possibly passing to a subsequence, we may assume that $K_{\Dd_i(1,e_i'',s)}\not\sim_{\mathbb R,T_i}0$ (resp. $K_{\Dd_i(b_i'',1,s)}\not\sim_{\mathbb R,T_i}0$) for any $s\in [0,1]$ such that $s\neq1-\tau$. By Lemma~\ref{lem: one lc place one close to one}, $e_i''<1-\tau_{\aux}(d,\Ii\cup\{1-\tau\})$ (resp. $b_i''<1-\tau_{\aux}(d,\Ii\cup\{1-\tau\})$) for any $i$. This is not possible as $\lim_{i\rightarrow+\infty}e_i''=1$ (resp. $\lim_{i\rightarrow+\infty}b_i''=1$). 

Therefore, possibly passing to a subsequence, we may assume that $K_{\Dd_i(1,1,1-\tau)}$ is not ample$/T_i$ for any $i$. Since $K_{\Dd_i(1,1,t_i)}\sim_{\mathbb R,T_i}K_{\Cc_i''(t_i)}$ is ample$/T_i$, there exists $t_i''\in [1-\tau,t_i)$ such that $K_{\Dd_i(1,1,t_i'')}\sim_{\mathbb R,T_i}0$. By Lemma~\ref{lem: one lc place one close to one}, $t_i''\not\in (1-\tau_{\aux}(d,\Ii),1)$. Thus
$$1-\tau\leq t_i'\leq 1-\tau_{\aux}(d,\Ii)<1-\tau,$$
a contradiction.

\medskip

\noindent\textbf{Step 4.} We conclude the proof in this step. 

Suppose that we are in \textbf{Case 2}. Then $t_i=t_i'$ and $$K_{\Dd_i(e_i,t_i)}\sim_{\mathbb R,T_i}K_{\Bb_i''(t_i)}\sim_{\mathbb R,T_i}0.$$ 
We get a contradiction to \textbf{Step 3} as the coefficients of $B_i''^o+e_iE_i''$ 
belong to the DCC set $\Ii\cup\{e_i\}_{i=1}^{+\infty}$, $e_i$ is strictly increasing, and $\lim_{i\rightarrow+\infty}e_i=1$. Therefore, we are in \textbf{Case 1}, and $t_i'=1-\tau$ for any $i$. 

Possibly passing to a subsequence, we may assume that $K_{\Dd_i(1,1-\tau)}$ is ample$/T_i$ for each $i$ or $K_{\Dd_i(1,1-\tau)}$ is not ample$/T_i$ for each $i$. If $K_{\Dd_i(1,1-\tau)}$ is ample$/T_i$ for each $i$, then there exists $e_i'\in (e_i,1)$ such that $K_{\Dd_i(e_i',1-\tau)}\sim_{\mathbb R,T_i}0$. Possibly passing to a subsequence, we may assume that $e_i'$ is strictly increasing. Then $K_{\Dd_i(e_i',1-\tau)}\sim_{\mathbb R,T_i}0$. Since the coefficients of $B_i''^o+e_i'E_i''$ and $1-\tau$ belong to the DCC set $\Ii\cup\{e_i'\}_{i=1}^{+\infty}\cup\{1-\tau\}$, $e_i'$ is strictly increasing and $\lim_{i\rightarrow+\infty}e_i'=1$, we get a contradiction to \textbf{Step 3}. Therefore, $K_{\Dd_i(1,1-\tau)}$ is not ample$/T_i$ for each $i$. 

Since $t_i>1-\tau$, $\tau<\tau_{\ilct}(d,\Ii)$, and $\Bb_i(t_i)$ is lc, $\Bb_i(1)$ is lc. Let
$$\Cc_i':=(X_i',\Ff_i',B_i'+E_i,\Mm^i)(1),$$
then $K_{\Cc_i'}\geq h^*K_{\Bb_i(1)}$. Since $K_{\Bb_i(1)}$ is ample$/U_i$, $K_{\Cc_i'}$ is big$/U_i$. Thus $\phi_{i,*}K_{\Cc_i'}$ is big$/U_i$, hence
$$K_{\Dd_i(1,1)}\sim_{\mathbb R,T_i}\phi_{i,*}K_{\Cc_i'}$$
is ample$/T_i$. Therefore, there exists $t_i''\in [1-\tau,1)$ such that $K_{\Dd_i(1,t_i'')}\sim_{\mathbb R,T_i}0$. By Lemma~\ref{lem: one lc place one close to one}, $t_i''\not\in (1-\tau_{\aux}(d,\Ii),1)$, so
$$1-\tau\leq t_i''\leq 1-\tau_{\aux}(d,\Ii)<1-\tau,$$
a contradiction.
\end{proof}

\section{Proof of the local ACC and the global ACC}\label{sec: proof of acc and global acc}

In this section, we prove the ACC (Theorem~\ref{thm: main ACC}) and the global ACC (Theorem~\ref{thm: global ACC main}).

\begin{lem}\label{lem: acc dim 1}
     Theorem~\ref{thm: acc lct general version} holds when $d=1$.
\end{lem}
\begin{proof}
Since $\dim X=1$, $\Ff=T_X$ or $\Ff=0$, and we may assume that $\Mm=\Mm'=\bm{0}$.

If $\Ff=T_X$, then $(X,B')$ is lc, hence $(X,(1-\tau)B)$ is lc. By \cite[Theorem~1.1]{HMX14}, we may choose $\tau$ depending only on $\Ii$ such that $(X,B)$ is lc, hence $\Bb$ is lc. If $\Ff=0$, then $B^{\ninv}=0$. In this case, if $t=1$, then $B^{\ninv}+(1-t)B^{\inv}=0$. Since $\Ff$ is lc, $\Bb$ is lc. If $\Ff=0$ and $t<1$, then $(X,B')$ is lc hence $(X,(1-\tau)B)$ is lc. By \cite[Theorem~1.1]{HMX14}, we may choose $\tau$ depending only on $\Ii$ such that $(X,B)$ is lc, hence $\Bb$ is lc.
\end{proof}

\begin{proof}[Proofs of Theorems~\ref{thm: acc lct general version} and~\ref{thm: global acc general}]
They follow from Lemma~\ref{lem: acc dim 1}, Theorem~\ref{thm: global to local}, and Theorem~\ref{thm: acc to global acc}.
\end{proof}

\begin{proof}[Proof of Theorem~\ref{thm: ACC precise}]
It follows from Theorem~\ref{thm: acc lct general version} and Lemma~\ref{lem: acc implies precise acc}.
\end{proof}

\begin{proof}[Proof of Theorem~\ref{thm: main ACC}]
It is a special case of Theorem~\ref{thm: ACC precise} by considering $\Ii=\{0,1\}$, $B=D=0$, $\Mm=\Nn=\bm{0}$, $t=0$, and $\mu=1$.
\end{proof}

\begin{proof}[Proof of Theorem~\ref{thm: global ACC main}]
Suppose that the theorem does not hold. Then there exists a strictly increasing sequence of real numbers $t_i\in (0,1)$ and projective lc algebraically integrable normalized foliated structures $\Bb_i:=(X_i,\Ff_i)(t_i)$ of dimension $d$ such that $K_{\Bb_i}\equiv 0$. We let $\Bb_i(s):=(X_i,\Ff_i)(s)$ for any $s\in [0,1]$. We let $h_i: X_i'\rightarrow X_i$ be a $\mathbb Q$-factorial qdlt modification of $\Bb_i$ and let $\Bb_i'(s):=(h_{i,*}^{-1}\Bb_i(s),\Exc(h_i))$ for any $s\in [0,1]$. By Theorem~\ref{thm: global acc general}, possibly passing to a subsequence, we have that $K_{\Bb_i'(s)}\equiv 0$ for any $s\in [0,1]$. In particular, $K_{\Bb_i'(0)}=K_{X_i'}+\Exc(h_i)\equiv 0$. Since $t_i\in (0,1)$, by Lemma~\ref{lem: lc preserved under inequality}, $\Bb_i'(0)$ is lc. By \cite[Theorem~1.4]{Gon11}, $K_{\Bb_i'(0)}\sim_{\mathbb R}0$. Thus $K_{\Bb_i(0)}=K_{X_i}$ is $\mathbb R$-Cartier and $K_{X_i}\sim_{\mathbb R}0$. Thus $K_{X_i}\equiv K_{\Ff_i}\equiv 0$, a contradiction.
\end{proof}

\begin{rem}
Under the setting of Theorem~\ref{thm: global ACC main}, we actually have that $tK_{\Ff}+(1-t)K_X\sim_{\mathbb R}0$, i.e.\ abundance holds for numerical dimension zero lc algebraically integrable adjoint foliated structures. The proof of this result requires the canonical bundle formula for adjoint foliated structures, and we leave it to a future work.
\end{rem}

\section{ACC type results}\label{sec: acc type results}

In this section, we prove several variations of the ACC for lc thresholds and of the global ACC, namely Theorems~\ref{thm: pet main}, \ref{thm: pet gap main}, \ref{thm: r-complementary main}, and \ref{thm: fano spectrum main}.

\subsection{ACC for the Fano spectrum}
In this subsection, we prove the ACC for the Fano spectrum (Theorem~\ref{thm: fano spectrum main}). It follows from the following more general theorem.

\begin{thm}\label{thm: fano spectrum general}
    Let $d$ be a positive integer and $\Ii\subset [0,+\infty)$ a DCC set. Then there exists an ACC set $\Ii'\subset [0,+\infty)$ depending only on $d$ and $\Ii$ 
  satisfying the following. 
    
    Assume that $\Bb$ is a projective lc algebraically integrable normalized foliated structure of dimension $d$ such that $\Bb\in\Ii$ and $-K_{\Bb}$ is nef but not numerically trivial. Let $H$ be a Cartier divisor and $r$ a positive real number such that $-K_{\Bb}\sim_{\mathbb R}rH$. Then $r\in\Ii'$.
\end{thm}

\begin{proof}
We consider $\Bb':=(\Bb,r\overline{H})$. Then $\Bb'$ is lc and $K_{\Bb'}\equiv 0$. By Theorem~\ref{thm: global acc general}, if $r$ belongs to a DCC set $\Ii''$, then $r$ belongs to a finite set depending only on $d$, $\Ii$, and $\Ii''$. Thus $r$ belongs to an ACC set. 
\end{proof}

\begin{proof}[Proof of Theorem~\ref{thm: fano spectrum main}]
    It is a special case of Theorem~\ref{thm: fano spectrum general}.
\end{proof}

\subsection{ACC for pseudo-effective thresholds}\label{subsec: acc pet}
In this subsection, we prove Theorems~\ref{thm: pet main} and~\ref{thm: pet gap main}. They follow from the following more general theorems (Theorems~\ref{thm: pet general} and~\ref{thm: big gap general}).

\begin{prop}\label{prop: global acc in bb}
    Let $d$ be a positive integer and $\Ii\subset [0,+\infty)$ a DCC set. Then there exists a positive real number $\tau=\tau_{\iglct}(d,\Ii)$ depending only on $d$ and $\Ii$ satisfying the following. Let $\Bb/U$ and $\Bb'/U$ be two algebraically integrable normalized foliated structures of dimension $d$. Assume that
    \begin{enumerate}
    \item $X$ is the ambient variety of $\Bb$ and $\Bb'$ and the associated morphism $X\rightarrow U$ is a contraction, $X$ is $\mathbb Q$-factorial of Fano type over $U$, and $\rho(X/U)=1$,
        \item $\Bb/U\in\Ii$,
        \item $\Bb/U\geq\Bb'/U\geq\Bb^{1-\tau}/U$, and
        \item $\Bb'$ is lc and $K_{\Bb'}\sim_{\mathbb R,U}0$.
    \end{enumerate}
    Then $K_{\Bb''}\sim_{\mathbb R,U}0$ for any $\Bb/U\geq\Bb''/U\geq\Bb'/U$.
\end{prop}
\begin{proof}
Suppose that the proposition does not hold. Then there exist a strictly decreasing sequence of real numbers $\tau_i$ such that $\tau_i\in (0,1)$ for each $i$, $\lim_{i\rightarrow+\infty}\tau_i=0$, and for each $i$, we have two algebraically integrable normalized foliated structures 
       $$\Bb_i/U_i:=(X_i,\Ff_i,B_i,\Mm_i)(t_i)/U_i,\quad \Bb_i'/U_i:=(X_i,\Ff_i,B_i',\Mm_i')(t_i')/U_i$$
      of dimension $d$, such that
       \begin{itemize}
       \item $X_i$ is $\mathbb Q$-factorial of dimension $d$, the associated morphism $X_i\rightarrow U_i$ is a contraction, $X_i$ is of Fano type over $U_i$, and $\rho(X_i/U_i)=1$,
        \item $\Bb_i/U_i\in\Ii$,
        \item $\Bb_i/U_i\geq\Bb_i'/U_i\geq\Bb_i^{1-\tau_i}/U_i$, 
        \item $\Bb_i'$ is lc and $K_{\Bb_i'}\sim_{\mathbb R,U_i}0$, and
        \item $K_{\Bb_i''}\not\sim_{\mathbb R,U_i}0$ for some $\Bb_i\geq\Bb_i''\geq\Bb_i'$.
       \end{itemize}
Possibly passing to a subsequence, we may assume that $\tau_i<\tau_{\ilct}(d,\Ii)$ for each $i$, where $\tau_{\ilct}(d,\Ii)$ is the value defined as in Theorem~\ref{thm: acc lct general version}. Write 
$$\Bb_i'':=(X_i,\Ff_i,B_i'',\Mm_i'')(t_i'').$$
Since $X_i$ is $\mathbb Q$-factorial, $\rho(X_i/U_i)=1$, and $\Bb_i/U_i\geq\Bb_i'/U_i\geq\Bb_i^{1-\tau_i}/U_i$, there exist real numbers $a_i,b_i,c_i,d_i\in [1-\tau_i,1]$ such that $B_i'^{\ninv}+\Mm'_{i,X_i}\sim_{\mathbb R,U_i}a_i(B_i^{\ninv}+\Mm_{i,X_i})$, $B_i'^{\inv}\sim_{\mathbb R,U_i}b_iB_i^{\inv}$, $B_i''^{\ninv}+\Mm''_{i,X_i}\sim_{\mathbb R,U_i}c_i(B_i^{\ninv}+\Mm_{i,X_i})$, and $B_i''^{\inv}\sim_{\mathbb R,U_i}d_iB_i^{\inv}$. Let
$$\Bb_i(a,b,s):=(X_i,\Ff_i,aB_i^{\ninv}+bB_i^{\inv},a\Mm_i)(s)$$
for any $a,b,s\in [0,1]$. Then
$$K_{\Bb_i(a_i,b_i,t_i')}\sim_{\mathbb R,U_i}K_{\Bb_i'}\sim_{\mathbb R,U_i}0\quad \text{and}\quad K_{\Bb_i(c_i,d_i,t_i'')}\sim_{\mathbb R,U_i}K_{\Bb_i''}\not\sim_{\mathbb R,U_i}0.$$
For any $a,b\in [1-\tau_i,1]$ and $s\in [(1-\tau_i)t_i,t_i]$, since $\Bb_i'$ is lc and $$\Bb_i(a,b,s)/U_i\geq\Bb_i'/U_i\geq\Bb_i(a,b,s)^{1-\tau_i}/U_i$$
 by Theorem~\ref{thm: acc lct general version}, $\Bb_i(a,b,s)$ is lc. By Theorem~\ref{thm: global acc general}~(1-2), possibly passing to a subsequence, we have the following for all $i$:
\begin{itemize}
    \item Either $a_i=1$ or $B_i^{\ninv}\sim_{\mathbb R,U_i}0$. Thus $a_iB_i^{\ninv}\sim_{\mathbb R,U_i}B_i^{\ninv}$.
    \item Either $a_i=1$ or $\Mm_i\equiv_{U_i}\bm{0}$. Thus $a_i\Mm_{i,X_i}\sim_{\mathbb R,U_i}\Mm_{i,X_i}$.
    \item We have $b_i=1$, or $t_i=t_i'=1$, or  $B_i^{\inv}\sim_{\mathbb R,U_i}0$. Thus $$(1-t_i')b_iB_i^{\inv}\sim_{\mathbb R,U_i}(1-t_i')B_i^{\inv}.$$ 

\end{itemize}
Let
$$\Bb_i(s):=\Bb_i(1,1,s)=(X_i,\Ff_i,B_i,\Mm_i)(s)$$
for any $s\in [0,1]$. Then 
$$K_{\Bb_i(t_i')}\sim_{\mathbb R,U_i}K_{\Bb_i(a_i,b_i,t_i')}\sim_{\mathbb R,U_i}0\quad \text{and}\quad K_{\Bb_i(t_i'')}\sim_{\mathbb R,U_i}K_{\Bb_i(c_i,d_i,t_i'')}\not\sim_{\mathbb R,U_i}0.$$
By Theorem~\ref{thm: global acc general}~(3), $t_i'$ belongs to a finite set. Since
$$t_i\geq t_i''\geq t_i'\geq (1-\tau_i)t_i$$
and $t_i\in\Ii$, possibly passing to a subsequence, we have that $t_i'=t_i''=t_i$. However,
$$K_{\Bb_i(t_i')}\sim_{\mathbb R,U_i}0\not\sim_{\mathbb R,U_i}K_{\Bb_i(t_i'')},$$
which is not possible.
\end{proof}

\begin{lem}\label{lem: pe preserved under order}
    Let $\Bb/U$ and $\Bb'/U$ be two $\mathbb Q$-factorial normalized foliated structures such that $\Bb/U\geq\Bb'/U$. Assume that $\Bb'$ is lc and $K_{\Bb'}$ is pseudo-effective$/U$. Then $K_{\Bb}$ is pseudo-effective$/U$.
\end{lem}
\begin{proof}
    Write $\Bb=(X,\Ff,B,\Mm)(t)$ and $\Bb'=(X,\Ff,B',\Mm')(t')$.
    Let $\Bb'':=(X,\Ff,B',\Mm')(t)$. Since $t\geq t'$, by \cite[Theorem~5.1]{Cas+25a}, $K_{\Bb''}$ is pseudo-effective$/U$. Since $B\geq B'$ and $\Mm-\Mm'$ is nef$/U$, $K_{\Bb}-K_{\Bb''}$ is pseudo-effective$/U$. Thus  $K_{\Bb}$ is pseudo-effective$/U$.
\end{proof}

\begin{thm}\label{thm: pet general}
    Let $d$ be a positive integer and $\Ii\subset [0,+\infty)$ a DCC set. Then there exists a positive real number $\tau=\tau_{\pet}(d,\Ii)$ depending only on $d$ and $\Ii$ satisfying the following. 

    Let $\Bb/U$ and $\Bb'/U$ be two algebraically integrable normalized foliated structures of dimension $d$. Assume that
    \begin{enumerate}
        \item $\Bb/U\in\Ii$,
        \item $\Bb/U\geq\Bb'/U\geq\Bb^{1-\tau}/U$,
        \item $\Bb'$ is lc and $K_{\Bb'}$ is not pseudo-effective$/U$, and
        \item $\Bb''/U$ is a normalized foliated structure such that $$\Bb^{\epsilon}\cdot\Bb'^{1-\epsilon}/U\geq\Bb''/U\geq\Bb'/U$$
        for some $\epsilon\in (0,1)$.
    \end{enumerate}
    Then $K_{\Bb''}$ is not pseudo-effective$/U$.
\end{thm}
\begin{proof}
Assume that the theorem does not hold. Then there exist a strictly decreasing sequence of real numbers $\tau_i$ such that $\tau_i\in (0,1)$ for each $i$, $\lim_{i\rightarrow+\infty}\tau_i=0$, and for each $i$, we have two algebraically integrable normalized foliated structures 
       $$\Bb_i/U_i:=(X_i,\Ff_i,B_i,\Mm_i)(t_i)/U_i,\quad \Bb_i'/U_i:=(X_i,\Ff_i,B_i',\Mm_i')(t_i')/U_i$$
      of dimension $d$, such that
       \begin{itemize}
           \item $\Bb_i/U_i\in\Ii$,
           \item        $\Bb_i/U_i\geq\Bb_i'/U_i\geq\Bb_i^{1-\tau_i}/U_i$, 
           \item $\Bb_i'$ is lc and $K_{\Bb_i'}$ is not pseudo-effective$/U_i$, and
           \item $K_{\Bb_i''}$ is pseudo-effective$/U_i$ for some
           $\Bb_i''=(X_i,\Ff_i,B_i'',\Mm_i'')(t_i'')$
           such that
           $$\Bb_i^{\epsilon_i}\cdot\Bb_i'^{1-\epsilon_i}/U_i\geq\Bb_i''/U_i\geq\Bb_i'/U_i$$
           for some $\epsilon_i\in (0,1)$.
       \end{itemize}
       We denote by $\Bb_i(s):=\Bb_i^{s}\cdot\Bb_i'^{1-s}$ for any $s\in [0,1]$. By Lemma~\ref{lem: pe preserved under order}, possibly replacing $\Bb_i''$ with $\Bb_i(\epsilon_i)$, we may assume that $\Bb_i''=\Bb_i(\epsilon_i)$.
    
    We first prove the case when $X_i$ is $\mathbb Q$-factorial klt for all $i$. By \cite[Theorem~2.1.5]{Cas+25a}, we may run a $K_{\Bb_i'}$-MMP$/U_i$ with scaling of an ample divisor which terminates with a Mori fiber space$/U_i$ $f_i: Y_i\rightarrow Z_i$ of $\Bb_i'/U_i$ associated with a birational map $\phi_i: X_i\dashrightarrow Y_i$. Let $\Bb_{Y_i},\Bb'_{Y_i},\Bb_{Y_i}(s)$ be the images of $\Bb_i,\Bb_i',\Bb_i(s)$ on $Y_i$ respectively for any $s\in [0,1]$. Then $\Bb'_{Y_i}$ is lc. By Theorem~\ref{thm: acc lct general version}, possibly passing to a subsequence, we have that $\Bb_{Y_i}(s)$ is lc for any $s\in [0,1]$. Since $K_{\Bb_i(\epsilon_i)}$ is pseudo-effective$/U_i$, $K_{\Bb_{Y_i}(\epsilon_i)}$ is pseudo-effective$/U_i$, hence $K_{\Bb_{Y_i}(\epsilon_i)}$ is nef$/Z_i$. Since $\rho(Y_i/Z_i)=1$ and $K_{\Bb_{Y_i}'}$ is anti-ample$/Z_i$, there exists $0<\eta_i\leq\epsilon_i<1$ such that $K_{\Bb_{Y_i}(\eta_i)}\sim_{\mathbb R,Z_i}0$. By Proposition~\ref{prop: global acc in bb}, $K_{\Bb_{Y_i}(s)}\sim_{\mathbb R,Z_i}0$ for any $s\in [\eta_i,1]$. We may write 
    $$K_{\Bb_{Y_i}(s)}=\sum_{j=1}^{m_i}\mu_{i,j}(s)D_{i,j}$$
    where $D_{i,j}$ are integral divisors and $\mu_{i,j}$ are quadratic functions. Therefore, $K_{\Bb_{Y_i}(s)}\sim_{\mathbb R,Z_i}0$ for any $s\in [0,1]$. In particular, 
    $$K_{\Bb'_{Y_i}}=K_{\Bb_{Y_i}(0)}\sim_{\mathbb R,Z_i}0.$$
    This is not possible as $f_i$ is a $K_{\Bb'_{Y_i}}$-Mori fiber space.
    
    Now we prove the general case. Let $h_i: \Bb'_{Y_i}\rightarrow\Bb_i$ be a $\mathbb Q$-factorial qdlt modification (resp. ACSS modification) of $\Bb_i'$ if $t_i'<1$ (resp. $t_i'=1$) for each $i$, $\Bb_{Y_i}:=(h_{i,*}^{-1}\Bb_i,\Exc(h_i))$, and $\Bb_{Y_i}(s):=(h_{i,*}^{-1}\Bb_i(s),\Exc(h_i))$ for any $s\in [0,1]$. By the case when the ambient variety is klt, possibly passing to a subsequence, we have that $K_{\Bb_{Y_i}(s)}$ is not pseudo-effective$/U_i$ for any $s\in (0,1)$. In particular, $K_{\Bb_{Y_i}(\epsilon_i)}$ is not pseudo-effective$/U_i$. By Theorem~\ref{thm: acc lct general version}, possibly passing to a subsequence, we have that $\Bb_i(\epsilon_i)$ is lc. By Lemma~\ref{lem: lc preserved under inequality}, any lc place of $\Bb_i'$ is an lc place of $\Bb_i(\epsilon_i)$, hence we have $K_{\Bb_{Y_i}(\epsilon_i)}=h_i^*K_{\Bb_i(\epsilon_i)}$. Thus $K_{\Bb_i(\epsilon_i)}=K_{\Bb_i''}$ is not pseudo-effective$/U_i$, a contradiction.   
\end{proof}

\begin{proof}[Proof of Theorem~\ref{thm: pet main}]
Suppose that the theorem does not hold. Then there exists a sequence of lc algebraically integrable foliations $\Ff_i$ on lc varieties $X_i$ of dimension $d$ such that
$$t_i:=\pet(X_i;\Ff_i)$$
is strictly increasing. Let $t:=\lim_{i\rightarrow+\infty}t_i$. Let $\tau:=\tau_{\pet}(d,\{t\})$ be as in Theorem~\ref{thm: pet general}. Possibly passing to a subsequence, we may assume that $t_i\geq\left(1-\frac{1}{2}\tau\right)t$ for any $i$. Let $\Bb_i:=(X_i,\Ff_i)(t)$, $\Bb_i':=(X_i,\Ff_i)(2t_i-t)$, and $\Bb_i'':=(X_i,\Ff_i)(t_i)$. By Theorem~\ref{thm: pet general}, $K_{\Bb_i''}$ is not pseudo-effective, a contradiction.
\end{proof}

\begin{thm}\label{thm: big gap general}
    Let $d$ be a positive integer and $\Ii\subset [0,+\infty)$ a DCC set. Then there exists a positive real number $\tau=\tau_{\gt}(d,\Ii)$ depending only on $d$ and $\Ii$ satisfying the following.

    Let $\Bb/U$ and $\Bb'/U$ be two algebraically integrable normalized foliated structures of dimension $d$. Assume that
    \begin{enumerate}
        \item $\Bb/U\in\Ii$ and $K_{\Bb}$ is big$/U$,
        \item $\Bb/U\geq\Bb'/U\geq\Bb^{1-\tau}/U$, and
        \item $\Bb'$ is lc.
    \end{enumerate}
    Then $K_{\Bb'}$ is big$/U$.
\end{thm}
\begin{proof}
We show that we may take $\tau=\frac{1}{2}\tau_{\pet}(d,\Ii)$ where $\tau_{\pet}(d,\Ii)$ is the value as in Theorem~\ref{thm: pet general}. Suppose that there exist $\Bb/U$ and $\Bb'/U$ so that (1-3) hold but $K_{\Bb'}$ is not big$/U$. By Lemma~\ref{lem: construct smaller bb}, there exists $\Bb''/U$ such that
$$\Bb''/U\geq\Bb^{1-2\tau}\quad \text{and}\quad  \Bb'=\Bb^{a}\cdot\Bb''^{1-a}$$
for some $a\in (0,1)$. We have
$$K_{\Bb'}=aK_{\Bb}+(1-a)K_{\Bb''}.$$
Since $K_{\Bb'}$ is not big$/U$ and $K_{\Bb}$ is big$/U$, $K_{\Bb''}$ is not pseudo-effective$/U$. Let $\Bb(s):=\Bb^s\cdot\Bb''^{1-s}$ for any $s\in [0,1]$. By Theorem~\ref{thm: pet general} and our construction of $\tau$, $K_{\Bb(s)}$ is not pseudo-effective for any $s\in [0,1)$. This is not possible as big$/U$ is an open condition and $K_{\Bb(1)}=K_{\Bb}$ is big$/U$.
\end{proof}

\begin{proof}[Proof of Theorem~\ref{thm: pet gap main}]
Let $\tau=\tau_{\gt}(d,\Ii)$. If $t<1$, then $X$ is lc by Lemma~\ref{lem: lc preserved under inequality}. Thus for any $s\in [(1-\tau)t,t]$, $(X,\Ff,s)$ is lc. The theorem follows from Theorem~\ref{thm: big gap general}.
\end{proof}

\subsection{ACC for \texorpdfstring{$\mathbb R$}{}-complementary thresholds}\label{subsec: acc rct}

In this subsection, we prove the ACC for $\mathbb R$-complementary thresholds (Theorem~\ref{thm: r-complementary main}), which follows from the following more general theorem.

\begin{thm}\label{thm: rct general}
Let $d$ be a positive integer and $\Ii\subset [0,+\infty)$ a DCC set. Then there exists a positive real number $\tau=\tau_{\Rct}(d,\Ii)$ depending only on $d$ and $\Ii$ satisfying the following. 

Assume that $\Bb/U$ and $\Bb'/U$ are two algebraically integrable normalized foliated structures of dimension $d$. Assume that
\begin{enumerate}
    \item $\Bb/U\in\Ii$,
    \item $\Bb/U\geq\Bb'/U\geq\Bb^{1-\tau}/U$, 
    \item $\Bb'/U$ is $\mathbb R$-complementary,
    \item the ambient variety of $\Bb,\Bb'$ is of Fano type over $U$, and
    \item the parameter $t$ of $\Bb$ is $<1$.
\end{enumerate}
Then $\Bb/U$ is $\mathbb R$-complementary.
\end{thm}
\begin{proof}
    Suppose that the theorem does not hold. Then there exists a strictly decreasing sequence of real numbers $\tau_i$ such that $\tau_i\in (0,1)$ for each $i$, $\lim_{i\rightarrow+\infty}\tau_i=0$, and for each $i$, we have two algebraically integrable normalized foliated structures 
       $$\Bb_i/U_i:=(X_i,\Ff_i,B_i,\Mm_i)(t_i)/U_i,\quad \Bb_i'/U_i:=(X_i,\Ff_i,B_i',\Mm_i')(t_i')/U_i$$
      of dimension $d$, such that
       \begin{itemize}
       \item $X_i$ is of Fano type over $U_i$,
           \item $\Bb_i/U_i\in\Ii$ and $t_i<1$,
           \item        $\Bb_i/U_i\geq\Bb_i'/U_i\geq\Bb_i^{1-\tau_i}/U_i$, 
           \item $\Bb_i'/U_i$ is $\mathbb R$-complementary, and
           \item $\Bb_i/U_i$ is not $\mathbb R$-complementary.
       \end{itemize}
       Possibly replacing $X_i$ with a small $\mathbb Q$-factorialization and replacing $\Bb_i,\Bb_i'$ with their pullbacks, we may assume that $X_i$ is $\mathbb Q$-factorial for each $i$. We run a $\left(-K_{\Bb_i}\right)$-MMP$/U_i$ with scaling of an ample divisor which terminates with either a good minimal model$/U_i$ $Y_i$, or a Mori fiber space$/U_i$ $f_i: Y_i\rightarrow Z_i$. Let $\Bb_{Y_i}$ and $\Bb_{Y_i'}$ be the images of $\Bb_i,\Bb_i'$ on $Y_i$ for each $i$. Then $\Bb_{Y_i}'/U_i$ is $\mathbb R$-complementary. In particular, $\Bb_{Y_i}'$ is lc. By Theorem~\ref{thm: acc lct general version}, possibly passing to a subsequence, we have that $\Bb_{Y_i}$ is lc. There are two cases.

       \medskip

       \noindent\textbf{Case 1.} There exists a $(-K_{\Bb_{Y_i}})$-Mori fiber space$/U_i$ $f_i: Y_i\rightarrow Z_i$. Let $\Bb_{Y_i}(s):=\Bb_{Y_i}^{s}\cdot\Bb_{Y_i}'^{1-s}$ for any $s\in [0,1]$. Since $-K_{\Bb'_{i}}$ is pseudo-effective$/U_i$, $-K_{\Bb'_{Y_i}}=-K_{\Bb_{Y_i}(0)}$ is pseudo-effective$/U_i$. Since $\rho(Y_i/Z_i)=1$ and $-K_{\Bb_{Y_i}}=-K_{\Bb_{Y_i}(1)}$ is anti-ample$/Z_i$, there exists $\epsilon_i\in [0,1)$ such that $K_{\Bb_{Y_i}(\epsilon_i)}\sim_{\mathbb R,Z_i}0$. Since $\lim_{i\rightarrow+\infty}\tau_i=0$ and
       $$\Bb_{Y_i}(\epsilon_i)/U_i\geq\Bb_{Y_i}'/U_i\geq\Bb_{Y_i}^{1-\tau_i}/U_i,$$
       by Proposition~\ref{prop: global acc in bb}, possibly passing to a subsequence, we have that $K_{\Bb_{Y_i}(s)}\sim_{\mathbb R,Z_i}0$ for any $s\in [\epsilon_i,1]$. In particular, $K_{\Bb_{Y_i}}=K_{\Bb_{Y_i}(1)}\sim_{\mathbb R,Z_i}0$. This is not possible as $f_i$ is a $(-K_{\Bb_{Y_i}})$-Mori fiber space.

       \medskip

       \noindent\textbf{Case 2.} $-K_{\Bb_{Y_i}}$ is semi-ample$/U_i$. In this case, since $t_i<1$, by \cite[Theorem~3.28]{Cas+25a}, we may pick $0\leq D_{Y_i}\sim_{\mathbb R,U_i}-K_{\Bb_{Y_i}}$ such that $(\Bb_{Y_i},D_{Y_i})$ is lc. Let $p_i: W_i\rightarrow X_i$ and $q_i: W_i\rightarrow Y_i$ be a resolution of indeterminacies of the induced birational map $\phi_i: X_i\dashrightarrow Y_i$, and let 
       $$K_{\Bb_i}+D_i:=p_{i,*}q_i^*(K_{\Bb_{Y_i}}+D_{Y_i}).$$
       Since $K_{\Bb_{Y_i}}+D_{Y_i}\sim_{\mathbb R,U_i}0$, $K_{\Bb_i}+D_i\sim_{\mathbb R,U_i}0$, and
       $$p_i^*(K_{\Bb_i}+D_i)=q_i^*(K_{\Bb_{Y_i}}+D_{Y_i}).$$
       Since $\phi_i$ is $\left(-K_{\Bb_i}\right)$-negative, we have
       $$p_i^*K_{\Bb_i}\leq q_i^*K_{\Bb_{Y_i}},$$
       hence
       $$p_i^*D_i\geq q_i^*D_{Y_i}\geq 0.$$
       Therefore, $D_i\geq 0$. Since $(\Bb_{Y_i},D_{Y_i})$ is lc, $(\Bb_i,D_i)$ is lc. Thus $(\Bb_i,D_i)/U_i$ is an $\mathbb R$-complement of $\Bb_i/U_i$, a contradiction.
\end{proof}

The following proposition is also useful.

\begin{prop}\label{prop: r complementary preserved under inequality}
Let $\Bb/U=(X,\Ff,B,\Mm)(t)/U$ be an algebraically integrable normalized adjoint foliated structure such that $X$ is of Fano type over $U$, $t<1$, and $\Bb/U$ is $\mathbb R$-complementary. Let $\Bb'/U$ be an algebraically integrable normalized adjoint foliated structure such that $\Bb/U\geq\Bb'/U$. Then $\Bb'/U$ is $\mathbb R$-complementary.
\end{prop}
\begin{proof}
Since $X$ is of Fano type over $U$, possibly replacing $X$ by a small $\mathbb Q$-factorialization, we may assume that $X$ is $\mathbb Q$-factorial. 

We first show that $-K_{\Bb'}$ is pseudo-effective$/U$. Write $\Bb'=(X,\Ff,B',\Mm')(t')$ and let $\Bb(s):=(X,\Ff,B,\Mm)(s)$ for any $s\in [0,1]$. Then $-K_{\Bb(t')}$ is not pseudo-effective$/U$. We run a $(-K_{\Bb(t')})$-MMP$/U$ which terminates with a Mori fiber space$/U$ $f: Y\rightarrow Z$ with induced birational map $\phi\colon X\dashrightarrow Y$. Let $\Bb_Y(s):=\phi_*\Bb(s)$ for any $s\in [0,1]$. Then $K_{\Bb_Y(t')}$ is ample$/Z$ and $-K_{\Bb_Y(t)}$ is pseudo-effective$/Z$, hence $K_{\Bb_Y(1)}$ is not pseudo-effective$/Z$ as $\dim Y>\dim Z$. This contradicts Lemma~\ref{lem: pe preserved under order}.

We run a $(-K_{\Bb'})$-MMP$/U$ $\phi\colon X\dashrightarrow Y$ which terminates with a model $Y$ such that $-K_{\Bb'_Y}$ is semi-ample$/U$, where $\Bb'_Y:=\phi_*\Bb'$. Let $\Bb_Y:=\phi_*\Bb$, then $\Bb_Y/U$ is $\mathbb R$-complementary and $\Bb_Y/U\geq\Bb'_Y/U$. By Lemma~\ref{lem: lc preserved under inequality}, $\Bb'_Y$ is lc. Since $t<1$, by \cite[Theorem~3.28]{Cas+25a}, we may pick $0\leq D_Y\sim_{\mathbb R,U}-K_{\Bb_Y'}$ such that $(\Bb_Y',D_Y)$ is lc. Let $p\colon W\rightarrow X$ and $q\colon W\rightarrow Y$ be a resolution of indeterminacies of $\phi$, and let 
       $$K_{\Bb'}+D:=p_{*}q^*(K_{\Bb_Y'}+D_Y).$$
       Since $K_{\Bb_Y'}+D_Y\sim_{\mathbb R,U}0$, $K_{\Bb'}+D\sim_{\mathbb R,U}0$, and $p^*(K_{\Bb'}+D)=q^*(K_{\Bb_Y'}+D_Y)$. Since $\phi$ is $(-K_{\Bb'})$-negative, $\phi$ is $D$-negative, hence $D\geq 0$. Thus $(\Bb',D)/U$ is an $\mathbb R$-complement of $\Bb'/U$.
\end{proof}

\begin{proof}[Proof of Theorem~\ref{thm: r-complementary main}]
    Suppose that the theorem does not hold. Then for each $i$, there exists a $\mathbb Q$-Gorenstein algebraically integrable foliation $\Ff_i$ on a Fano type variety $X_i$ of dimension $d$ such that $\Rct(X_i;\Ff_i)=t_i$ and $t_i$ is strictly increasing. Let $t:=\lim_{i\rightarrow+\infty}t_i$. We pick $t>s_i>t_i$ such that $s_i$ is strictly increasing. Then $(X_i,\Ff_i,s_i)$ does not have an $\mathbb R$-complement. Let $\Ii:=\{s_i\}_{i=1}^{+\infty}$ and let $\tau:=\tau_{\Rct}(d,\Ii)$ be the value as in Theorem~\ref{thm: rct general}. Possibly passing to a subsequence, we may assume that $t_i>(1-\tau)s_i$ for each $i$. This contradicts Theorem~\ref{thm: rct general}.
\end{proof}

\section{Effective birationality}\label{sec: eb}

The goal of this section is to prove the effective birationality theorem for pluricanonical systems of adjoint foliated structures, namely Theorem~\ref{thm: afs main}, which follows from the following more general theorem.

\begin{thm}\label{thm: afs eb}
    Let $d$ be a positive integer and $\Ii\subset [0,+\infty)$ a DCC set. Then there exist two positive integers $m_0=m_0(d,\Ii),n_0=n_0(d,\Ii)$ depending only on $d$ and $\Ii$ satisfying the following. Assume that
    $$\Bb:=\left(X,\Ff,B,\Mm=\sum\mu_j\Mm_j\right)\left(t\right)$$
    is a projective lc algebraically integrable normalized foliated structure of general type such that $\dim X=d,B,\mu_j\in\Ii,t\in\Ii$, each $\Mm_j$ is nef $\bb$-Cartier, and
    \begin{itemize}
        \item either $t<1$, or
        \item $t=1$ and $(X,B,\Mm)$ is lc.
    \end{itemize}
    Let $m$ and $n$ be two positive integers such that
    $$m\geq m_0\quad \text{and}\quad n\geq\min\left\{\left\lfloor\frac{mt}{1-t}\right\rfloor,mn_0\right\}.$$
    Here we consider $\left\lfloor\frac{mt}{1-t}\right\rfloor=+\infty$ if $t=1$. Then the linear system
    $$\left|mK_X+\lfloor mB\rfloor+\sum\lfloor m\mu_j\rfloor\Mm_{j,X}+nK_{\Ff}+\lfloor nB^{\ninv}\rfloor+\sum\lfloor n\mu_j\rfloor\Mm_{j,X}+L\right|$$
    defines a birational map for any pseudo-effective divisor $L$ on $X$. 
\end{thm}

Some easy lemmas are needed before we prove Theorem~\ref{thm: afs eb}.

\begin{lem}\label{lem: potentially birational for afs}
Let $\Bb=(X,\Ff,B,\Mm)(t)$ be a $\mathbb Q$-factorial projective lc algebraically integrable $\mathbb Q$-normalized foliated structure such that $t<1$. Let $\Bb(s):=(X,\Ff,B,\Mm)(s)$ for any $s\in [0,1]$. Let $m$ be a positive integer such that $mK_{\Bb}$ is potentially birational. Let $L$ be a pseudo-effective $\mathbb Q$-divisor on $X$. Then for any rational number $s\in [t,1)$ and any rational number 
$$m_s>\frac{m(1-t)}{1-s},$$
we have that $m_sK_{\Bb(s)}+L$ is potentially birational.
\end{lem}
\begin{proof}
We have
$$K_{\Bb(s)}=\frac{1-s}{1-t}K_{\Bb}+\frac{s-t}{1-t}K_{\Bb(1)}$$
so
$$m_sK_{\Bb(s)}+L=mK_{\Bb}+\left(\frac{m_s(1-s)}{1-t}-m\right)K_{\Bb}+\frac{s-t}{1-t}K_{\Bb(1)}+L=:mK_{\Bb}+D.$$
By \cite[Theorem~5.1]{Cas+25a}, $K_{\Bb(1)}$ is pseudo-effective. Since $m_s>\frac{m(1-t)}{1-s}$ and $K_{\Bb}$ is big, $D$ is big. The lemma follows from \cite[Lemma~2.15]{HL23b}.
\end{proof}

\begin{lem}\label{lem: birational to +pe birational for afs}
Let $\Bb:=(X,\Ff,B,\Mm)(t)$ be a $\mathbb Q$-factorial projective lc algebraically integrable $\mathbb Q$-normalized foliated structure such that $t<1$. Let $m_0$ be a positive rational number such that $m_0K_{\Bb}$ is potentially birational. Then for any integral pseudo-effective divisor $L$ and any integer $m$ such that $mK_{\Bb}$ is integral and 
    $$m>m_0+\frac{1}{1-t},$$
    we have that
    $$|mK_{\Bb}+L|$$
    defines a birational map.
\end{lem}
\begin{proof}
Let $\Bb(s):=(X,\Ff,B,\Mm)(s)$ for any $s\in [0,1]$. We have that 
$$mK_{\Bb}+L=K_{\Bb(0)}+(m-1)K_{\Bb\left(\frac{mt}{m-1}\right)}=K_X+\left(B+\Mm_X+L+(m-1)K_{\Bb\left(\frac{mt}{m-1}\right)}\right).$$
Since $B+\Mm_X+L$ is pseudo-effective and
$$m-1>\frac{m_0(1-t)}{1-\frac{mt}{m-1}},$$
by Lemma~\ref{lem: potentially birational for afs},
$$D:=B+\Mm_X+L+(m-1)K_{\Bb\left(\frac{mt}{m-1}\right)}$$
is potentially birational. Moreover, $D=mK_{\Bb}+L-K_X$ is an integral divisor. By \cite[Lemma~2.3.4]{HMX13}, 
$$|mK_{\Bb}+L|=|K_X+\lceil D\rceil|$$
defines a birational map.
\end{proof}

\begin{lem}\label{lem: elementary perturbation eb}
Let $p,n$ be two positive integers and $\Ii\subset [0,+\infty)$ a DCC set. Then there exists a positive integer $m_0:=m_0(p,n,\Ii)\geq n+2$ depending only on $p$, $n$, and $\Ii$ satisfying the following. Denote by $\gamma_p:=\frac{1}{p}\left\lfloor (p-1)\gamma\right\rfloor$ for any $\gamma\in\Ii$. Then for any $\gamma\in\Ii\backslash\{0\}$ and any real number $m\geq m_0$, we have the following.
    \begin{enumerate}
    \item $m(\gamma-\gamma_p)>1$ and $\lfloor m\gamma\rfloor>m\gamma_p$.
        \item Assume that $\gamma<1$. Then
        $$\frac{\left\lfloor\frac{m\gamma}{1-\gamma}\right\rfloor-\frac{n\gamma_p}{1-\gamma_p}}{m-n}>\frac{\gamma_p}{1-\gamma_p}.$$    \end{enumerate}
\end{lem}
\begin{proof}
We may assume that $\Ii\backslash\{0\}\neq\emptyset$ and we let $\mu:=\min\{\gamma\in\Ii\mid \gamma>0\}$.

(1) We have 
$$\gamma-\gamma_p\geq\frac{\gamma}{p}\geq\frac{\mu}{p},$$
so for any real number $m>\frac{p}{\mu}$ and $\gamma\in\Ii\backslash\{0\}$, we have that $m(\gamma-\gamma_p)>1$. We have
$$\lfloor m\gamma\rfloor>m\gamma-1>m\gamma_p.$$

(2) By (1), we have
$$m(\gamma-\gamma_p)>1>(1-\gamma)(1-\gamma_p).$$
This is equivalent to the inequality we desire.
\end{proof}

\begin{proof}[Proof of Theorem~\ref{thm: afs eb}]
We let $\tau_{\gt}(d,\Ii)$ be the positive real number as in Theorem~\ref{thm: big gap general} and let $\tau:=\frac{1}{2}\tau_{\gt}(d,\Ii)$. Let $p\geq 2$ be an integer depending only on $d$ and $\Ii$ such that 
$$\lfloor (p-1)\gamma\rfloor\geq (1-\tau)p\gamma$$ 
for any $\gamma\in\Ii$. Indeed, we may take 
$p\coloneqq \left\lceil\frac{1}{\tau\gamma_0}+\frac{1}{\tau}\right\rceil$, and we have
\begin{align*}
   \lfloor (p-1)\gamma\rfloor&=(p-1)\gamma-\{(p-1)\gamma\}>(p-1)\gamma-1\geq (p-1)\gamma-(\tau p-1)\gamma_0\\
   &\geq (p-1)\gamma-(\tau p-1)\gamma=(1-\tau)p\gamma
\end{align*}
for any $\gamma\in\Ii\backslash\{0\}$, while when $\gamma=0$ the inequality is obvious.

Let $h\colon X'\rightarrow X$ be a foliated log resolution of $\Bb$ and
$$\Bb':=(h^{-1}_*\Bb,\Exc(h)):=(X',\Ff',B',\Mm)(t).$$
We let $B'_p:=\frac{1}{p}\left\lfloor (p-1)B\right\rfloor$, $\Mm(p):=\frac{1}{p}\sum\left\lfloor (p-1)\mu_j\right\rfloor\Mm_j$, $t_p:=\frac{1}{p}\left\lfloor (p-1)t\right\rfloor$, and
$$\Bb_p:=(X',\Ff',B'_p,\Mm(p))(t_p).$$
By our construction of $p$, we have $t_p=0$ if and only if $t=0$. Since $K_{\Bb}$ is big, $K_{\Bb'}\geq h^*K_{\Bb}$ is big. We have that
$$\Bb'\geq\Bb_p\geq\Bb'^{1-\tau}$$
by our construction of $p$. By Lemma~\ref{lem: lc preserved under inequality} when $t<1$ and our assumption when $t=1$, $\Bb_p$ is lc. By Theorem~\ref{thm: big gap general}, $K_{\Bb_p}$ is big. Since $t_p<1$, $\lfloor B_p'\rfloor=0$, and $\Bb_p$ is foliated log smooth, $\Bb_p$ is klt. We let $\Ii_1:=\frac{1}{p}\mathbb N$ and let $N$ be a positive integer depending only on $d$ and $\Ii$ such that $p\mid N$ and
$$\tau':=\frac{1}{N}\leq\frac{1}{2}\min\left\{\tau_{\ilct}(d,\Ii_1),\tau_{\gt}(d,\Ii_1),\frac{1}{p^2}\right\}.$$ 
where $\tau_{\ilct}(d,\Ii_1)$ is the positive real number as in Theorem~\ref{thm: acc lct general version} and $\tau_{\gt}(d,\Ii_1)$ is the positive real number as in Theorem~\ref{thm: big gap general}. 

Since $\Bb_p$ is klt, by \cite[Lemma~2.13]{CLW26}, the set
$$\mathcal{S}:=\{E\mid E\text{ is prime}/X', a(E,\Bb_p)\leq (-1+\tau')(t_p\epsilon_{\Ff}(E)+(1-t_p))\}$$
is finite. Since $t_p\leq\frac{p-1}{p}$, $B_p'\in [0,\frac{p-1}{p}]$, and $\tau'<\frac{1}{p^2}$, any prime divisor in $\mathcal{S}$ is exceptional$/X'$. By \cite[Theorem~2.2.3]{Cas+25a}, there exists a projective birational morphism $g: X''\rightarrow X'$ such that $X''$ is $\mathbb Q$-factorial and $\Exc(g)$ contains exactly all prime divisors that are contained in $\mathcal{S}$. 

Since $\tau'<\tau_{\ilct}(d,\Ii_1)$ and $g^*\Bb_p$ is klt, by Theorem~\ref{thm: acc lct general version} and the definition of $\tau_{\ilct}(d,\Ii_1)$, we have that
$$\Bb_0'':=(g^{-1}_*\Bb_p,\Exc(g))$$
is lc. Moreover, $K_{\Bb_0''}\geq g^*K_{\Bb_p}$, so $K_{\Bb_0''}$ is big. We let
$$\Bb'':=\left(g^{-1}_*\Bb_p,(1-\tau')\Exc(g)\right).$$
Since $\tau'<\tau_{\gt}(d,\Ii_1)$ and $K_{\Bb''_0}$ is big, by Theorem~\ref{thm: big gap general} and the definition of $\tau_{\gt}(d,\Ii_1)$, $K_{\Bb''}$ is big. Moreover, we have $$K_{\Bb''}\leq g^*K_{\Bb_p}$$
by the construction of $\mathcal{S}$.

For any prime divisor $D$ over $X''$, if $D\not\in\mathcal{S}$, then
$$a(D,\Bb'')\geq a(D,\Bb_p)>(-1+\tau')(t_p\epsilon_{\Ff}(E)+(1-t_p))$$
hence
$$a(D,\Bb'')+\left(t_p\epsilon_{\Ff}(E)+(1-t_p)\right)>\tau'(t_p\epsilon_{\Ff}(E)+(1-t_p))\geq\frac{\tau'}{p}$$
as $t_p\leq\frac{p-1}{p}$. If $D\in\mathcal{S}$ and $D$ is not $\Ff$-invariant, then
$$a(D,\Bb'')+\left(t_p\epsilon_{\Ff}(E)+(1-t_p)\right)=\tau'>\frac{\tau'}{p}.$$
If $D\in\mathcal{S}$ and $D$ is $\Ff$-invariant, then
$$a(D,\Bb'')+\left(t_p\epsilon_{\Ff}(E)+(1-t_p)\right)=\tau'\left(t_p\epsilon_{\Ff}(E)+(1-t_p)\right)\geq\frac{\tau'}{p}$$
as $t_p\leq\frac{p-1}{p}$. Therefore, $\Bb''$ is $\mathbb Q$-factorial $\frac{\tau'}{p}$-lc and $K_{\Bb''}$ is big. We write
$$\Bb''=(X'',\Ff'',B'',\Mm(p))(t_p)$$
and let 
$$\Bb''(s)=(X'',\Ff'',B'',\Mm(p))(s)$$
for any $s\in [0,1]$. By \cite[Theorem~2.1.1]{Cas+25a}, we may run a $K_{\Bb''}$-MMP with scaling of an ample divisor which terminates with a good minimal model $\Bb_Y$ of $\Bb''$ with induced birational map $\phi\colon X''\dashrightarrow Y$. Let $\Bb_Y(s):=\phi_*\Bb''(s)$ for any $s\in [0,1]$. Then $\Bb_Y$ is $\mathbb Q$-factorial $\frac{\tau'}{p}$-lc. By \cite[Proposition~9.1]{Cas+25a}, $\Bb_Y(0)$ is $\frac{\tau'}{p}$-lc. In particular, $Y$ is $\frac{\tau'}{p}$-lc. Let $B_Y:=\phi_*B''$. We have the following.
\begin{enumerate}
    \item  $Y$ is $\frac{\tau'}{p}$-lc.
    \item $K_{\Bb_Y(t_p)}=K_{\Bb_Y}$ is big and nef.
    \item Since $\tau'=\frac{1}{N}$ and $p\mid N$, $NB''$ is integral. Thus $NK_{\Bb'}$ is integral, hence $$NpK_{\Bb_Y}=Np\phi_*K_{\Bb''}$$ is integral.
    \item By \cite[Theorem~5.1]{Cas+25a}, $K_{\Bb_Y\left(\frac{Npt_p}{Np-1}\right)}$ is big. 
Thus
    $$NpK_{\Bb_Y}-K_Y=B_Y+\Mm(p)_{Y}+NpK_{\Bb_Y(t_p)}-K_{\Bb_Y(0)}=B_Y+\Mm(p)_{Y}+(Np-1)K_{\Bb_Y\left(\frac{Npt_p}{Np-1}\right)}$$
is big.
\end{enumerate}
By \cite[Theorem~4.2]{Bir23a}, there exists a positive integer $m_1$ such that 
$$|mNpK_{\Bb_Y}|$$ defines a birational map for any integer $m\geq m_1$. Since $\phi$ is a sequence of steps of a $K_{\Bb''}$-MMP and $\phi_*K_{\Bb''}=K_{\Bb_Y}$,  $|mNpK_{\Bb''}|$ defines a birational map for any integer $m\geq m_1$. By \cite[Lemma~2.3.4]{HMX13}, $(2d+1)m_1NpK_{\Bb''}$ is potentially birational. Let
$$m_2:=(2d+1)m_1Np+Np,$$
then 
$$m_2>(2d+1)m_1Np+p\geq (2d+1)m_1Np+\frac{1}{1-t_p},$$
$m_2K_{\Bb_0''},m_2K_{\Bb''}$ are integral, and $m_2K_{\Bb_0''}-m_2K_{\Bb''}\geq 0$. By Lemma~\ref{lem: birational to +pe birational for afs},
$$|m_2K_{\Bb_0''}|$$
defines a birational map and $m_2K_{\Bb_0''}$ is integral. Since $K_{\Bb_0''}-g^*K_{\Bb_p}\geq 0$ and is exceptional$/X'$,
$$|m_2K_{\Bb_p}|$$
defines a birational map and $m_2K_{\Bb_p}$ is integral. Since $pt_p\in [0,p-1]\cap\mathbb N$, by Lemma~\ref{lem: birational to +pe birational for afs}, there exists a positive integer $m_3$ depending only on $d$ and $\Ii$ such that $\frac{m_3}{1-t_p}K_{\Bb_p}$ is integral, $p^2\mid\frac{m_3}{1-t_p}$, and for any integral pseudo-effective divisor $L'$ on $X'$, $$\left|\frac{m_3}{1-t_p}K_{\Bb_p}+L'\right|=\left|m_3\left(K_{X'}+B_p'+\Mm(p)_{X'}\right)+m_3\frac{t_p}{1-t_p}\left(K_{\Ff'}+B_p'^{\ninv}+\Mm(p)_{X'}\right)+L'\right|$$ defines a birational map. 

Let $m_0:=m_0(p,m_3,\Ii)$ be the positive integer as in Lemma~\ref{lem: elementary perturbation eb} and let $n_0:=p$. Now for any positive integers 
$$m\geq m_0 \quad \text{ and }\quad n\geq\min\left\{\left\lfloor\frac{mt}{1-t}\right\rfloor,n_0m=pm\right\}$$ 
and pseudo-effective integer divisor $L'$ on $X'$, we have
\begin{align*}
&D(m,n,L')
\coloneqq 
mK_{X'}+\left\lfloor mB'\right\rfloor+\sum\left\lfloor m\mu_j\right\rfloor\Mm_{j,X'}+nK_{\Ff'}+\left\lfloor nB'^{\ninv}\right\rfloor\\
+&\sum\left\lfloor n\mu_j\right\rfloor\Mm_{j,X'}+L'
=\frac{m_3}{1-t_p}K_{\Bb_p}+\left(\left\lfloor mB'\right\rfloor-mB_p'\right)\\
+&\left(\left\lfloor nB'^{\ninv}\right\rfloor-nB_p'^{\ninv}\right)
+\sum\left(\left\lfloor m\mu_j\right\rfloor-m\left\lfloor\frac{(p-1)\mu_j}{p}\right\rfloor\right)\Mm_{j,X'}\\
+&\sum\left(\left\lfloor n\mu_j\right\rfloor-n\left\lfloor\frac{(p-1)\mu_j}{p}\right\rfloor\right)\Mm_{j,X'}+L'\\
+&(m-m_3)\left(K_{X'}+B_p'+\Mm(p)_{X'}+\frac{n-\frac{m_3t_p}{1-t_p}}{m-m_3}\left(K_{\Ff'}+B_p'^{\ninv}+\Mm(p)_{X'}\right)\right).
\end{align*}
For each $j$, $\Mm_j$ is $\bb$-nef, so $\Mm_{j,X'}$ is pseudo-effective. By Lemma~\ref{lem: elementary perturbation eb}~(1), \begin{align*}
        &\left(\left\lfloor mB'\right\rfloor-mB'_p\right)+\left(\left\lfloor nB'^{\ninv}\right\rfloor-nB_p'^{\ninv}\right)+\sum\left(\lfloor m\mu_j\rfloor-m\left\lfloor\frac{(p-1)\mu_j}{p}\right\rfloor\right)\Mm_{j,X'}\\
&+\sum\left(\lfloor n\mu_j\rfloor-n\left\lfloor\frac{(p-1)\mu_j}{p}\right\rfloor\right)\Mm_{j,X'}+L'
    \end{align*}
    is pseudo-effective. By \cite[Theorem~5.1]{Cas+25a} and since $K_{\Bb_p}$ is big, we have that 
$$K_{X'}+B_p'+\Mm(p)_{X'}+\lambda(K_{\Ff'}+B_p'^{\ninv}+\Mm(p)_{X'})$$
    is big for any $\lambda\geq\frac{t_p}{1-t_p}$. Since $t_p<t$, $\left\lfloor\frac{mt_p}{1-t_p}\right\rfloor\leq\left\lfloor\frac{mt}{1-t}\right\rfloor$. Since $t_p\leq\frac{p-1}{p}$, $\left\lfloor\frac{mt_p}{1-t_p}\right\rfloor\leq m(p-1)$. Thus $n\geq\left\lfloor\frac{mt_p}{1-t_p}\right\rfloor$. By Lemma~\ref{lem: elementary perturbation eb}~(2), 
$$K_{X'}+B_p'+\Mm(p)_{X'}+\frac{n-\frac{m_3t_p}{1-t_p}}{m-m_3}\left(K_{\Ff'}+B_p'^{\ninv}+\Mm(p)_{X'}\right)$$
is big. Therefore, by our construction of $m_3$, $|D(m,n,L')|$ defines a birational map.

Now for any pseudo-effective divisor $L$ on $X$, $h_*D(m,n,0)+L$ is big and $h^{-1}_*L$ is pseudo-effective. By \cite[III. 5.10 Lemma]{Nak04}, there exists an $h$-exceptional divisor $E\geq 0$ such that
$$|h^{-1}_*(h_*D(m,n,0)+L)+F|=|h_*D(m,n,0)+L|$$
for any $h$-exceptional divisor $F\geq E$. We may find $F\geq E$ that is exceptional$/X$ and
$$h^{-1}_*(h_*D(m,n,0)+L)+F=D(m,n,h^{-1}_*L)+F'=D(m,n,h^{-1}_*L+F')$$
for some $F'\geq 0$. Since $|D(m,n,h^{-1}_*L+F')|$ defines a birational map, $|h_*D(m,n,0)+L|$ defines a birational map. The theorem follows.
\end{proof}

\begin{proof}[Proof of Theorem~\ref{thm: afs main}]
    It is the special case of Theorem~\ref{thm: afs eb} where $B=0$ and $\Mm=\bm{0}$.
\end{proof}

Note that we need to assume that $t<1$ in Theorem~\ref{thm: afs main}. When $t=1$, we have the following theorem which is also a special case of Theorem~\ref{thm: afs eb}.

\begin{thm}\label{thm: eb when t=1}
Let $d$ be a positive integer. Then there exist two positive integers $m_0,n_0$ depending only on $d$ satisfying the following.

Let $X$ be a projective lc variety and $\Ff$ an lc algebraically integrable foliation of general type on $X$. Then for any integers $m\geq m_0$ and $n\geq mn_0$, the linear system
$$|mK_X+nK_{\Ff}+L|$$
defines a birational map for any pseudo-effective integral divisor $L$.
\end{thm}
\begin{proof}
It is the special case of Theorem~\ref{thm: afs eb} for $\Ii=\{0,1\}$, $B=0$, $\Mm=\bm{0}$, and $t=1$.
\end{proof}

\section{Birational boundedness}\label{sec: birational boundedness}

The goal of this section is to prove the birational boundedness of algebraically integrable foliations of general type with bounded volumes (Theorem~\ref{thm: birationally bounded main}). We prove the following theorem, which is a more general version of Theorem~\ref{thm: birationally bounded main}.

\begin{thm}\label{thm: birational boundedness}
    Let $d$ be a positive integer, $\Ii\subset [0,+\infty)$ a DCC set, and $C$ a positive real number. Assume that
    \begin{enumerate}
        \item $\Bb=(X,\Ff,B,\Mm)(t)$ is a projective lc algebraically integrable normalized foliated structure,
        \item $\dim X=d$,
        \item $\Bb\in\Ii$ and $t\in (0,1)$, and
        \item $0<\vol(K_{\Bb})\leq (1-t)^dC$.
    \end{enumerate}
   Then $\Bb$ belongs to a log birationally bounded family.
\end{thm}

We first prove the following useful lemma.

\begin{lem}\label{lem: volume monotonicity}
    Let $\Bb\geq\Bb'$ be two $\mathbb Q$-factorial lc algebraically integrable normalized foliated structures of dimension $d$. Let $t$ and $t'$ be the parameters of $\Bb$ and $\Bb'$, respectively. Assume that $t<1$. Then 
    $$\left(\frac{1}{1-t}\right)^d\vol(K_{\Bb})\geq\left(\frac{1}{1-t'}\right)^d\vol(K_{\Bb'}).$$
\end{lem}
\begin{proof}
    Write 
    $$\Bb:=(X,\Ff,B,\Mm)(t),\quad \Bb':=(X,\Ff,B',\Mm')(t').$$
    Let $\Bb(s):=(X,\Ff,B,\Mm)(s)$ for any $s\in [0,1]$. Then
    $$\vol(K_{\Bb(t')})\geq\vol(K_{\Bb'}).$$
    We may assume that $K_{\Bb'}$ is big. Then $K_{\Bb(t')}$ is big. By \cite[Theorem~5.1]{Cas+25a}, $K_{\Bb(1)}$ is pseudo-effective. Since
    $$K_{\Bb(t)}=\frac{t-t'}{1-t'}K_{\Bb(1)}+\frac{1-t}{1-t'}K_{\Bb(t')},$$
    we have that
    $$\vol(K_{\Bb})=\vol(K_{\Bb(t)})\geq\vol\left(\frac{1-t}{1-t'}K_{\Bb(t')}\right)=\left(\frac{1-t}{1-t'}\right)^d\vol(K_{\Bb(t')})\geq\vol(K_{\Bb'}).$$
    The lemma follows.
\end{proof}

We need the following definition, which characterizes the boundedness of normalized foliated structures using numerical invariants.

\begin{defn}
Let $d,p$ be two positive integers, $r$ a positive real number, $\epsilon$ a non-negative real number, and $\Ii\subset\mathbb R$ a set. Let $\Bb=(X,\Ff,B,\Mm)(t)$ be a projective normalized foliated structure with ambient variety $X$. Let $\Bb(s):=(X,\Ff,B,\Mm)(s)$ for any $s\in [0,1]$. We say that $\Bb$ is a \emph{$(d,r,\Ii)$-normalized foliated structure} if
\begin{enumerate}
    \item $\dim X=d$, $\Bb\in\Ii$, and $t<1$, and
    \item there exists a very ample divisor $A$ on $X$ such that $$A^d\leq r \quad \text{and} \quad A-K_{\Bb(s)} \text{ is pseudo-effective for any }s\in [0,1].$$
\end{enumerate}
If $\Ii=\frac{1}{p}\mathbb N$, then we say that $\Bb$ is a \emph{$(d,r,p)$-normalized foliated structure}. 

We say that $\Bb$ is a \emph{$(d,r,\Ii,\epsilon)$} (resp. \emph{$(d,r,p,\epsilon)$})\emph{-normalized foliated structure} if $\Bb$ is a $(d,r,\Ii)$ (resp. $(d,r,p)$)-normalized foliated structure and $\Bb(0)$ is $\epsilon$-klt. We denote by 
$$\mathcal{N}_{d,r,\Ii}:=\{\Bb\mid \Bb\text{ is an algebraically integrable }(d,r,\Ii)\text{-normalized foliated structure}\},$$
$$\mathcal{N}_{d,r,\Ii,\epsilon}:=\{\Bb\mid \Bb\in\mathcal{N}_{d,r,\Ii},\Bb(0)\text{ is }\epsilon\text{-klt}\},\quad \mathcal{N}_{d,r,p}:=\mathcal{N}_{d,r,\frac{1}{p}\mathbb N},\quad \text{and}\quad \mathcal{N}_{d,r,p,\epsilon}:=\mathcal{N}_{d,r,\frac{1}{p}\mathbb N,\epsilon}.$$
\end{defn}

\begin{lem}\label{lem: drii0 is bounded}
Let $d$ be a positive integer, $r$ a positive real number, and $\Ii\subset [0,+\infty)$ a DCC set. Then $\mathcal{N}_{d,r,\Ii,0}$ is a bounded family.
\end{lem}
\begin{proof}
Pick $\Bb=(X,\Ff,B,\Mm)(t)\in\mathcal{N}_{d,r,\Ii}$ and let $\Bb(s):=(X,\Ff,B,\Mm)(s)$ for any $s\in [0,1]$. Let $A$ be a very ample divisor on $X$ such that $A^d\leq r$ and $A-K_{\Bb(s)}$ is pseudo-effective for any $s\in [0,1]$. Then $X$ is bounded. Possibly replacing $A$ by a bounded multiple, we may assume that $A+K_{\Bb(0)}$ and $A-(B+\Mm_X)$ are pseudo-effective. By \cite[Theorem~5.1]{Cas+25a}, $A+K_{\Bb(1)}$ is pseudo-effective, so $2A+K_{\Ff}$ is pseudo-effective. Since $A-K_{\Bb(1)}$ is pseudo-effective, $A-K_{\Ff}$ is pseudo-effective. By \cite[Proposition~3.36]{Cas+25a}, $(X,\Ff)$ belongs to a bounded family. Since $A-K_{\Bb(0)}$ is pseudo-effective and $\Bb\in\Ii$, the lemma follows from the definition (Definition~\ref{defn: bounded nfs}).
\end{proof}

We now prove Theorem~\ref{thm: birational boundedness}. We need the following theorem, which can be considered a more precise version of Theorem~\ref{thm: birational boundedness}. The key point of the following theorem is that, assuming the existence of a nice birational morphism $f\colon W\rightarrow X$ and a normalized foliated structure $\Bb_W$ on $W$, we can obtain a bounded birational model $\Bb'$ of $\Bb_W$ with ambient variety $X'$, such that the induced birational map $W\dashrightarrow X'$ does not extract any divisor.

\begin{thm}\label{thm: precise birational boundedness}
    Let $d$ and $p$ be two positive integers, $\epsilon$ a positive real number, and $C$ a positive real number. Then there exists a positive integer $r$ depending only on $d$, $p$, and $C$ satisfying the following. Assume that
    \begin{enumerate}
        \item $\Bb=(X,\Ff,B,\Mm)(t)$ is a projective klt algebraically integrable normalized foliated structure,
        \item $\dim X=d$, $\Bb\in\frac{1}{p}\mathbb N$, and $t\in (0,1)$.
        \item $pK_{\Bb}$ is integral and $|pK_{\Bb}|$ defines a birational map.
        \item $f: W\rightarrow X$ is a projective birational morphism such that
        $$f^*(pK_{\Bb})\sim A_W+R_W$$
        where $|A_W|$ is the movable part of $f^*|pK_{\Bb}|$, $A_W$ is base-point-free and defines a birational morphism $g: W\rightarrow Y$.
        \item $\Bb_W:=\left(f^{-1}_*\Bb,\left(1-\frac{1}{p}\right)\Exc(f)\right)$ is $\epsilon$-lc.
        \item $\vol(K_{\Bb})\leq C$.
    \end{enumerate}
Then there exists a birational map$/Y$ $\gamma: W\dashrightarrow X'$ which does not extract any divisor satisfying the following. Let $\Bb':=\gamma_*\Bb_W$, then $X'$ is $\mathbb Q$-factorial and
$$\Bb'\in \mathcal{N}_{d,r,p,\epsilon}.$$
In particular, $\Bb'$ belongs to a bounded family, and $\Bb$ and $\Bb_W$ belong to a log birationally bounded family.
\end{thm}
\begin{proof}
Possibly replacing $W$ by a small $\mathbb Q$-factorialization, we may assume that $W$ is $\mathbb Q$-factorial. Let $A_Y:=g_*A_W$. Then $A_Y$ is very ample and $A_Y^d\leq\vol(pK_{\Bb})\leq p^dC$. Therefore, $Y$ belongs to a bounded family. We may assume that $A_W\geq 0$.

Since $\Bb_W$ is $\mathbb Q$-factorial klt, by \cite[Theorem~2.1.1]{Cas+25a}, we may run a $(K_{\Bb_W}+3dA_W)$-MMP with scaling of an ample divisor. By the length of extremal rays, this is also a $K_{\Bb_W}$-MMP$/Y$ with scaling of an ample divisor, which terminates with a good minimal model $\Bb_Z/Y$ of $\Bb_W/Y$ associated with a birational map $\phi\colon W\dashrightarrow Z$, where $Z$ is the ambient variety of $\Bb_Z$. Let $g_Z: Z\rightarrow Y$ be the associated contraction and $A_Z:=\phi_*A_W$. Then $A_Z=g_Z^*A_Y$. 

Let $\Bb_W(s):=\left(f^{-1}_*\Bb(s),\left(1-\frac{1}{p}\right)\Exc(f)\right)$
and $\Bb_Z(s):=\phi_*\Bb_W(s)$ for any $s\in [0,1]$. Since $\Bb_W$ is $\epsilon$-lc, $\Bb_Z$ is $\epsilon$-lc. By \cite[Proposition~9.1]{Cas+25a}, $\Bb_Z(0)$ is $\epsilon$-lc. By the length of extremal rays, $K_{\Bb_Z(0)}+2dA_Z$ is pseudo-effective. Thus $K_{\Bb_Z(0)}+3dA_Z$ is big.

We run a $(K_{\Bb_{Z}(0)}+3dA_Z)$-MMP with scaling of an ample divisor. By the length of extremal rays, this MMP is a $K_{\Bb_Z(0)}$-MMP$/Y$ with scaling of an ample divisor and terminates with a good minimal model $\Bb'(0)/Y$ of $\Bb_Z(0)/Y$ associated with a birational map $\psi: Z\dashrightarrow X'$ and a contraction $g': X'\rightarrow Y$, where $\Bb':=\psi_*\Bb_Z$ and $\Bb'(s):=\psi_*\Bb_Z(s)$ for any $s\in [0,1]$. Let $A':=\psi_*A_Z$.

Since $\Bb_Z(0)$ is $\epsilon$-lc, $\Bb'(0)$ is $\epsilon$-lc. By Lemma~\ref{lem: volume monotonicity}, we have
\begin{align*}
    0&<\vol(K_{\Bb'(0)}+3dA')=\vol(K_{\Bb_Z(0)}+3dA_Z)\leq\left(\frac{1}{1-t}\right)^d\vol(K_{\Bb_Z}+3dA_Z)\\
    &=\left(\frac{1}{1-t}\right)^d\vol(K_{\Bb_W}+3dA_W)\leq l^d\vol(K_{\Bb_W}+3dA_W).
\end{align*}
Since $\Bb$ is klt and $pK_{\Bb}$ is integral, we have
$$\left\lfloor pK_{\Bb_W}\right\rfloor=pf^{-1}_*K_{\Bb}+p(\Exc(f)^{\ninv}+(1-t)\Exc(f)^{\inv})-\Exc(f)\geq \lfloor f^*(pK_{\Bb})\rfloor\sim A_W+\lfloor R_W\rfloor,$$
so 
$$\vol(K_{\Bb_W}+3dA_W)\leq (3dp+1)^d\vol(K_{\Bb_W})\leq (3dp+1)^d\vol(K_{\Bb})\leq (3dp+1)^dC.$$
Let $h\colon X'\rightarrow X''$ be the ample model of $K_{\Bb'(0)}+3dA'$. This is also the ample model$/Y$ of $K_{\Bb'(0)}$ by the length of extremal rays. Let $A'':=h_*A'$ and $\Bb''(s):=h_*\Bb'(s)$ for any $s\in [0,1]$. Then we have
$$\vol(K_{\Bb''(0)}+3dA'')=\vol(K_{\Bb'(0)}+3dA')\in\left(0,(3dp+1)^dC\right).$$
By \cite[Lemma~6.6]{BH22}, $(\Bb''(0),3d\overline{A_Y})$ is bounded. Thus there exists a positive integer $r_1$ depending only on $d$, $p$, and $\epsilon$, and a very ample divisor $H_1$ on $X''$, such that
$$H_1-2(K_{\Bb''(0)}+3dA'') \text{ is pseudo-effective}, \quad \text{and}\quad H_1^d\leq r_1.$$ We may assume that $H_1$ is general in $|H_1|$. Write
$$\Bb'(s):=(X',\Ff',B',\Mm)(s)$$
for any $s\in [0,1]$. By \cite[Theorem~2.3]{Bir24}, $(X',B'+h^*H_1)$ belongs to a bounded family. Thus there exists a positive integer $r_2$ and a very ample divisor $H_2$ on $X'$ such that $H_2^d\leq r_2$ and $H_2-h^*H_1$ is pseudo-effective. Therefore,
$$H_2-(K_{\Bb'(0)}+3dA')=H_2-h^*H_1+h^*(H_1-(K_{\Bb''(0)}+3dH_2))$$
is pseudo-effective. Therefore, $(\Bb''(0),3d\overline{A_Y})$ is bounded. Since $A_Y$ is very ample, there exists a positive real number $b$ depending only on $d$, $p$, and $\epsilon$, such that
$$bH_2-A'\quad \text{and}\quad bA'-H_2$$
are pseudo-effective. We let $\alpha: V\rightarrow Z$ and $\beta: V\rightarrow X'$ be a resolution of indeterminacies of $\psi$. Then we have
$$\beta_*\alpha^*(K_{\Bb_Z}+3dA_Z)=K_{\Bb'}+3dA'.$$
Since $K_{\Bb_Z}+3dA_Z$ and $\beta^*H_2$ are nef, we have
\begin{align*}
    (K_{\Bb'}+3dA')\cdot H_2^{d-1}=&\alpha^*(K_{\Bb_{Z}}+3dA_Z)\cdot(\beta^*H_2)^{d-1}\leq\vol(\alpha^*(K_{\Bb_{Z}}+3dA_Z)+\beta^*H_2)\\
    \leq&\vol(\alpha^*(K_{\Bb_{Z}}+3dA_Z)+b\beta^*A')=\vol(\alpha^*(K_{\Bb_{Z}}+(3d+b)A_Z))\\
    =&\vol(K_{\Bb_{Z}}+(3d+b)A_Z)=\vol(K_{\Bb_{W}}+(3d+b)A_W)\\
    \leq&(3dq+bp+1)^d\vol(K_{\Bb_W})\\
    \leq&(3dq+bp+1)^dC=:C_1.
\end{align*}
Since $p^2K_{\Bb_W}$ is an integral divisor and
$$\left\lfloor pK_{\Bb_W}\right\rfloor\geq \lfloor f^*(pK_{\Bb})\rfloor\sim A_W+\lfloor R_W\rfloor\geq 0,$$
we may write
$$pK_{\Bb_W}\sim L_W\geq 0$$
such that the coefficients of $L_W$ are $\geq\frac{1}{p}$. We let $L':=\psi_*\phi_*L_W$. Then $L'\sim pK_{\Bb'(t)}$ and $L'\cdot H_2^{d-1}\leq pC_1$. Therefore, $(X',L')$ belongs to a bounded family. Thus there exists a positive integer $c$ depending only on $d$, $p$, and $\epsilon$, such that
$$cH_2-(L'+3dA')$$
is pseudo-effective. Thus 
$$cpH_2-(K_{\Bb'(t)}+3dA')=cpH_2-t(K_{\Bb'(1)}+3dA')-(1-t)(K_{\Bb'(0)}+3dA')$$
is pseudo-effective. Since $K_{\Bb'(0)}+3dA'$ is nef and $t\geq\frac{1}{p}$,
$$cpH_2-K_{\Bb'(1)}$$
is pseudo-effective. Since $H_2-K_{\Bb'(0)}$ is pseudo-effective, 
$$cpH_2-K_{\Bb'(s)}$$
is pseudo-effective for any $s\in [0,1]$. 

Let $r:=(cp)^dr_2$ and $A:=cpH_2$. Then:
\begin{itemize}
    \item $\dim\Bb'=d,\Bb'\in\frac{1}{p}\mathbb N$, and $t<1$, and
    \item $A^d\leq r$ and $A-K_{\Bb'(s)}$ is pseudo-effective for any $s\in [0,1]$.
    \item $\Bb'(0)$ is $\epsilon$-lc.
\end{itemize}
Therefore, 
$$\Bb'\in\mathcal{N}_{d,r,p,\epsilon}.$$
By Lemma~\ref{lem: drii0 is bounded}, $\Bb'$ belongs to a bounded family.
Let $\gamma:=\psi\circ\phi$ be the induced birational map. Then $\gamma$ does not extract any divisor, hence $\Bb_W$ is log birationally bounded. By our definition of $\Bb_W$, we have that $\Supp B'$ contains all $f$-exceptional prime divisors. Thus $\Bb$ is log birationally bounded.
\end{proof}

The following lemma is complementary to Theorem~\ref{thm: precise birational boundedness}; it shows the existence of a nice birational model $W$ of $\Bb$.

\begin{lem}\label{lem: special model for eb}
Let $d$ and $p$ be positive integers. Then there exists a positive integer $q$ depending only on $d$ and $p$, and a projective birational morphism $f: W\rightarrow X$ satisfying the following. Assume that:
\begin{itemize}
    \item $\Bb=(X,\Ff,B,\Mm)(t)$ is a projective klt algebraically integrable normalized foliated structure.
    \item $\dim X=d$, $\Bb\in\frac{1}{p}\mathbb N$, and $t\in (0,1)$.
    \item $K_{\Bb}$ is big.
\end{itemize}
Then:
\begin{enumerate}
\item $p\mid q$, $qK_{\Bb}$ is integral, and $|qK_{\Bb}|$ defines a birational map.
\item We have
        $$f^*(qK_{\Bb})\sim A_W+R_W$$
    where $|A_W|$ is the movable part of $f^*|qK_{\Bb}|$, $A_W$ is base-point-free and defines a birational morphism $g: W\rightarrow Y$.
\item $\Bb_W:=\left(f^{-1}_*\Bb,\left(1-\frac{1}{q}\right)\Exc(f)\right)$ is $\frac{1}{pq}$-lc.
\end{enumerate}
\end{lem}
\begin{proof}
Possibly replacing $\Bb$ by a small $\mathbb Q$-factorialization, we may assume that $\Bb$ is $\mathbb Q$-factorial klt. By Theorem~\ref{thm: afs eb}, there exists a positive integer $q$ depending only on $d$ and $p$ such that $p\mid q$, $qK_{\Bb}$ is an integral divisor, and $|qK_{\Bb}|$ defines a birational map. We let $\phi_X: X\dashrightarrow \mathbb P^{n_X}$ be the birational map induced by $|qK_{\Bb}|$ and let $Y$ be the normalization of the closure of $\phi_X(X)$. Let $\phi\colon X\dashrightarrow Y$ be the induced birational map. 

We let $f': W'\rightarrow X$ and $g': W'\rightarrow Y$ be a resolution of indeterminacies of $\phi$ such that $f'$ is a log resolution of $\Bb$. We let
$$\Bb_{W'}:=\left(f'^{-1}_*\Bb,\left(1-\frac{1}{q}\right)\Exc(f')\right).$$
Since $\Bb$ is klt, $\Bb_{W'}$ is klt. We let 
$$\mathcal{S}:=\left\{E\middle| E\text{ is exceptional}/W', a(E,\Bb_{W'})\leq \left(-1+\frac{1}{q}\right)(t\epsilon_{\Ff}(E)+(1-t))\right\},$$
then $\mathcal{S}$ is a finite set by \cite[Lemma~2.13]{CLW26}. By \cite[Theorem~2.2.3]{Cas+25a}, there exists a projective birational morphism $\pi: W\rightarrow W'$ such that $\Exc(\pi)$ contains exactly all prime divisors that are contained in $\mathcal{S}$. Let
$$\Bb_W:=\left(\pi^{-1}_*\Bb_{W'},\left(1-\frac{1}{q}\right)\Exc(\pi)\right).$$
Then $\Bb_W$ is $\frac{1}{pq}$-lc. Let $f:=f'\circ\pi$ and $g:=g'\circ\pi$. By the construction of $\phi$, we have that
$$f^*(qK_{\Bb})\sim A_W+R_W$$
such that $A_W$ is base-point-free, $|A_W|$ is the movable part of $|f^*(qK_{\Bb})|=f^*|qK_{\Bb}|$, and defines $g: W\rightarrow Y$. The lemma follows.
\end{proof}

\begin{proof}[Proof of Theorem~\ref{thm: birational boundedness}] Possibly replacing $\Ii$ with $\Ii\cup\{1\}$ and replacing $\Bb$ with a $\mathbb Q$-factorial qdlt modification, we may assume that $\Bb$ is $\mathbb Q$-factorial qdlt. Write $\Mm=\sum\mu_j\Mm_j$ where each $\mu_j\in\Ii$ and each $\Mm_j$ is nef $\bb$-Cartier. By Theorem~\ref{thm: big gap general}, there exists a positive integer $l$ depending only on $d$ and $\Ii$ such that $(l-1)\gamma\geq 1$ for any $\gamma\in\Ii$ and
$$\Bb':=(X,\Ff,B(l),\Mm(l))(t(l))$$
is of general type, where 
$$B(l):=\frac{1}{l}\lfloor (l-1)B\rfloor,\Mm(l):=\frac{1}{l}\sum \lfloor (l-1)\mu_j\rfloor\Mm_j,\quad \text{and}\quad t(l):=\frac{1}{l}\lfloor (l-1)t\rfloor.$$
By Lemma~\ref{lem: volume monotonicity}, possibly replacing $\Ii$ with $\frac{1}{l}\mathbb N$ and replacing $\Bb$ with $\Bb'$, we may assume that
\begin{itemize}
    \item $\Bb$ is $\mathbb Q$-factorial klt, and
    \item $\Bb\in\frac{1}{l}\mathbb N$ and $t\in (0,1)$.
\end{itemize}
By Lemma~\ref{lem: special model for eb}, there exists a positive integer $p$ depending only on $d$ and $l$ hence depending only on $d$ and $\Ii$, and a projective birational morphism $f: W\rightarrow X$, such that
\begin{itemize}
    \item $l\mid p$, $pK_{\Bb}$ is integral, and $|pK_{\Bb}|$ defines a birational map. In particular, $\Bb\in\frac{1}{p}\mathbb N$.
    \item We have
        $$f^*(pK_{\Bb})\sim A_W+R_W$$
    where $|A_W|$ is the movable part of $f^*|pK_{\Bb}|$, $A_W$ is base-point-free and defines a birational morphism $g: W\rightarrow Y$.
\item $\Bb_W:=\left(f^{-1}_*\Bb,\left(1-\frac{1}{p}\right)\Exc(f)\right)$ is $\epsilon$-lc, where $\epsilon:=\frac{1}{lp}$.
\item $\vol(K_{\Bb})\leq (1-t)^dC<C$.
\end{itemize}
By Theorem~\ref{thm: precise birational boundedness}, $\Bb$ is log birationally bounded.
\end{proof}

\begin{proof}[Proof of Theorem~\ref{thm: birationally bounded main}] It is a special case of Theorem~\ref{thm: birational boundedness}.
\end{proof}

\begin{proof}[Proof of Theorem~\ref{thm: stable family birational boundedness}]
Let $\Ff$ be the foliation induced by $f$. 

We first show that $f$ is lc. We let $h_T: T'\rightarrow T$ be a resolution of $T$ and let $f': X'\rightarrow T'$ be the base change of $f$. Then $f': X'\rightarrow T'$ is BP stable (cf.~\cite[Definition~2.5]{ACSS21}). To see this, simply note that for any resolution $g_{T'}: T''\rightarrow T'$ and base change $f'': X''\rightarrow T''$ associated with $g: X''\rightarrow X'$ we have that $K_{X''/T''}=g^*K_{X'/T'}$. We can also simply take $T'$ to be a sufficiently high resolution and \cite[Theorem~1]{Sho23} applies.

Now let $\Ff'$ be the foliation induced by $f'$ and let $h: X'\rightarrow X$ be the associated birational morphism. Since $f$ is stable, 
$$h^*K_{X/T}=K_{X'/T'}=K_{\Ff'},$$
hence
$$K_{\Ff}=h_*K_{X'/T'}=K_{X/T}\quad \text{and}\quad h^*K_{\Ff}=K_{\Ff'}.$$
We have that $\Ff'$ satisfies Property $(*)$ and $f'$ is equidimensional, so by \cite[Proposition~3.7]{ACSS21}, $\Ff'$ is lc. By \cite[Proposition~2.15]{PX17}, $K_{\Ff'}$ is big and nef.

Since $K_{X'/T'}$ is $\mathbb Q$-Cartier and $f'^*K_{T'}$ is Cartier, it follows that $K_{X'}$ is $\mathbb Q$-Cartier. Since $\Ff'$ satisfies Property $(*)$, $X'$ is lc. By Theorem~\ref{thm: big gap general}, there exists a real number $t$ depending only on $d$ such that $K_{t}':=tK_{\Ff'}+(1-t)K_{X'}$ is big. Thus $K_t:=h_*K_t'$ is big. This implies (1). Since $h$ does not extract any divisor, under the conditions of (2),
$$0<\vol(K_t')\leq\vol(K_t)\leq C.$$
By Theorem~\ref{thm: birationally bounded main}, $(X',\Ff')$ is log birationally bounded. Thus $(X,\Ff)$ is log birationally bounded, which implies (2).
\end{proof}

\section{Boundedness of Fanos}\label{sec: application}

In this section, we prove Birkar's criterion for the boundedness of exceptional Fanos (Theorem~\ref{thm: exceptional afs main}) and Jiang's boundedness criterion (Theorem~\ref{thm: jiang afs main}) for algebraically integrable foliations. The latter gives boundedness for Fanos whose anti-canonical volume and Tian's $\alpha$-invariant are bounded away from zero.

\subsection{Boundedness of exceptional Fanos}

In this subsection, we prove Theorem~\ref{thm: exceptional afs main}.

\begin{defn}
Let $\Aa$ be a projective algebraically integrable adjoint foliated structure. We say that $\Aa$ is \emph{exceptional} if $|-K_{\Aa}|_{\mathbb R}\neq\emptyset$, and $(\Aa,D)$ is klt for any $0\leq D\sim_{\mathbb R}-K_{\Aa}$.

For any projective algebraically integrable normalized foliated structure $\Bb$, we say that $\Bb$ is \emph{exceptional} if $\log\Bb$ is exceptional.
\end{defn}

\begin{lem}\label{lem: gap exc threshold}
    Let $d$ be a positive integer and $\Ii\subset [0,+\infty)$ a DCC set of real numbers. Then $\tau=\tau_{\Rct}(d,\Ii\cup\{1\})$ as in Theorem~\ref{thm: rct general} satisfies the following. 
    
    Let $\Bb(t)=(X,\Ff,B,\Mm)(t)$ be a projective normalized foliated structure of dimension $d$ such that $X$ is of Fano type, $B\in\Ii,\Mm\in\bNef(X,\Ii)$, $t\in [1-\tau,1)$, and $\Bb(t)$ is $\mathbb R$-complementary but not exceptional. Then $\Bb(s):=(X,\Ff,B,\Mm)(s)$ is $\mathbb R$-complementary but not exceptional for any $s\in [0,1)$.
\end{lem}
\begin{proof}
We will show that we may take $\tau=\tau_{\Rct}(d,\Ii\cup\{1\})$ where $\tau_{\Rct}(d,\Ii\cup\{1\})$ is the number as in Theorem~\ref{thm: rct general}.

We may assume that $K_{\Bb(s)}$ is $\mathbb R$-Cartier for any $s\in [0,1]$; otherwise, $K_{\Bb(s)}$ is not $\mathbb R$-Cartier for any $s\neq t$ and we are done. Possibly replacing $X$ by a small $\mathbb Q$-factorialization, we may assume that $X$ is $\mathbb Q$-factorial klt. Since $\Bb(t)$ is $\mathbb R$-complementary but not exceptional, there exists an $\mathbb R$-complement $\Cc(t):=(\Bb(t),D)$ of $\Bb(t)$ that is lc but not klt. We let $E$ be an lc place of $\Cc(t)$. 

If $E$ is on $X$, then we let 
$$\Bb'(s):=(X,\Ff,B+(\mult_ED)E,\Mm)(s)$$ 
for any $s\in [0,1]$ and let $g: X\rightarrow X$ be the identity morphism. If $E$ is not on $X$, then since $t<1$, by \cite[Theorem~2.2.3]{Cas+25a}, there exists a projective birational morphism $g: Y\rightarrow X$ which extracts exactly $E$, and we let $$\Bb'(s):=(g^{-1}_*\Bb(s),E)$$
for any $s\in [0,1]$. Moreover, $Y$ is of Fano type.

In either case, $\Bb'(t)$ is $\mathbb R$-complementary. By Theorem~\ref{thm: rct general}, $\Bb'(s)$ is $\mathbb R$-complementary for any $s\in [0,1)$. Let $(\Bb'(s),D_s)$ be an $\mathbb R$-complement of $\Bb'(s)$, then $(g_*\Bb'(s),g_*D_s)$ is an $\mathbb R$-complement of $g_*\Bb'(s)$ that is not klt, and hence an $\mathbb R$-complement of $\Bb(s)$ that is not klt. Therefore, $\Bb(s)$ is $\mathbb R$-complementary but not exceptional for any $s\in [0,1)$.
\end{proof}

\begin{thm}\label{thm: exceptional bdd}
    Let $d$ be a positive integer and $\Ii\subset [0,+\infty)$ a DCC set of real numbers. Assume that $\Bb=(X,\Ff,B,\Mm)(t)$ is an exceptional algebraically integrable adjoint foliated structure of dimension $d$ such that $\Bb\in\Ii$, $t>0$, and $X$ is of Fano type. Then $\Bb$ belongs to a bounded family.
\end{thm}
\begin{proof}
Since $\Bb$ is exceptional, $\Bb$ is klt, so we may assume that $t<1$; otherwise, $\Ff=T_X$ and we are done by \cite[Theorem~1.11]{Bir19}.

Let $h\colon X'\rightarrow X$ be a small $\mathbb Q$-factorialization of $X$ and write 
$$\Bb':=h^*\Bb:=(X',\Ff',B',\Mm)(t).$$
Then $\Bb'$ is exceptional. Let $\Bb'(s):=(X',\Ff',B',\Mm)(s)$ for any $s\in [0,1]$. Let $\tau:=\tau_{\Rct}(d,\Ii\cup\{1\})$ as in Theorem~\ref{thm: rct general}. We let $t':=\min\{t,1-\tau\}$. By Lemma~\ref{lem: gap exc threshold}, $\Bb'(t')$ is exceptional.

Pick any $0\leq D'\sim_{\mathbb R}-K_{\Bb'(t')}$ and let $\Cc':=(\Bb'(t'),D')$. Then $\Cc'$ is klt and $K_{\Cc'}\sim_{\mathbb R}0$. Let $\Cc:=h_*\Cc'$, then $K_{\Cc}\sim_{\mathbb R}0$ and $\Cc'=h^*\Cc$. If $\Cc$ is $\tau$-lc, then by \cite[Theorem~2.4.1]{Cas+25a}, $(X,\Ff)$ belongs to a bounded family. In particular, $X$ belongs to a bounded family, so there exists a positive integer $r$ depending only on $d$ and $\Ii$, and a very ample divisor $H$ on $X$, such that $H^d\leq r$ and $-r\leq K_X\cdot H^{d-1}$. Let $\Cc'(s):=(\Bb'(s),D')$ for any $s\in [0,1]$. Since $K_{\Cc'}\sim_{\mathbb R}0$ and $\Cc'$ is klt, by \cite[Theorem~5.1]{Cas+25a}, $K_{\Cc'(1)}$ is pseudo-effective, hence $-K_{\Cc'(0)}$ is pseudo-effective. Thus
$$-K_{\Cc'(0)}\cdot(h^*H)^{d-1}\geq 0.$$
Therefore,
$$-r\leq K_X\cdot H^{d-1}\leq (K_X+B+\Mm_X)\cdot H^{d-1}\leq 0.$$
Since $\Bb\in\Ii$, possibly replacing $r$ by a multiple depending only on $\Ii$, we have that
$$-r\leq K_X\cdot H^{d-1}\leq (K_X+\Supp B+\Mm_X)\cdot H^{d-1}\leq 0.$$
Thus $\Bb$ belongs to a bounded family.

Therefore, we may assume that $\Cc$ is not $\tau$-lc, so $\Cc'$ is not $\tau$-lc. Let $E$ be a prime divisor over $X'$ such that $a(E,\Cc')=\tmld(\Cc')$. Since $1-t'\geq\tau$, $a(E,\Cc')<0$. If $E$ is on $X'$, then we let $g: X'\rightarrow X'$ be the identity morphism and let
$$\Bb'':=(X',\Ff',B'+(\mult_ED')E,\Mm)(t'),$$
and if $E$ is not on $X$, then since $t'<1$, by \cite[Theorem~2.2.3]{Cas+25a}, there exists a projective birational morphism $g: Y\rightarrow X'$ which extracts exactly $E$, and we let $$\Bb'':=(g^{-1}_*\Bb'(t'),E).$$
Moreover, $Y$ is of Fano type. 

In either case, by Theorem~\ref{thm: rct general}, $\Bb''$ is $\mathbb R$-complementary. Let $(\Bb'',L)$ be an $\mathbb R$-complement of $\Bb''$, then $(g_*\Bb'',g_*L)$ is an $\mathbb R$-complement of $g_*\Bb''$ that is not klt, and hence an $\mathbb R$-complement of $\Bb'(t')$ that is not klt. Therefore, $\Bb'(t')$ is not exceptional, a contradiction.
\end{proof}

\begin{proof}[Proof of Theorem~\ref{thm: exceptional afs main}]
 The condition $\alpha(\Aa)>1$ implies that $\Aa$ is exceptional. Thus Theorem~\ref{thm: exceptional afs main} is a special case of Theorem~\ref{thm: exceptional bdd}.
\end{proof}

\begin{rem}
For any exceptional Fano algebraically integrable adjoint foliated structure $\Aa$, we expect that $\alpha(\Aa)>1$ if and only if $\Aa$ is exceptional. This is not a trivial result, as it can be viewed as a foliated analogue of a question of Tian \cite[Question~1]{Tia90}, and even the proof of the classical case is non-trivial \cite[Theorem~1.7]{Bir21a}.
\end{rem}

\subsection{Jiang's criterion}

In this subsection, we prove Jiang's boundedness criterion for algebraically integrable foliations (Theorem~\ref{thm: jiang afs main}).

We state a more general version of the theorem, Theorem~\ref{thm: jiang principle general}. We remark that in the statement of Theorem~\ref{thm: jiang principle general}, there is no moduli part $\Mm$, although we expect the analogue with the moduli part $\Mm$ to also hold. This is because the references we need, cf.~\cite{LLX20,XZ21,XZ24}, only care about pairs. The corresponding results we cite are also expected to hold for generalized pairs.

\begin{lem}\label{lem: alpha bound}
  Let $(X,B)$ be a projective klt pair of dimension $d$, $D$ a semi-ample $\mathbb R$-Cartier $\mathbb R$-divisor on $X$, and $L$ a big $\mathbb R$-divisor on $X$. Then for any $\epsilon>0$, we may pick $0\leq D'\sim_{\mathbb R}D$ such that
  $$\alpha(X,B+D';L)\geq \alpha(X,B;L)-\epsilon.$$
\end{lem}
\begin{proof}
We may assume that $\epsilon<\alpha(X,B;L)$. Write $D=\sum_{i=1}^mr_iD_i$ where each $r_i\in (0,1)$ and each $D_i$ is base-point-free. Pick $l\gg 0$ and pick general $D_{i,j}\in |D_i|$ for each $1\leq j\leq l$ and take
$$D':=\sum_{i=1}^m\sum_{j=1}^l\frac{r_i}{l}D_{i,j}.$$
Then for any closed point $x\in X$, there exists a finite set $\Lambda=\Lambda(x)$ such that $\#\Lambda\leq d$ and $x\in\Supp D_{i,j}$ if and only if $(i,j)\in\Lambda$. Since $l\gg 0$ and $L$ is big, we have that
$$\epsilon L-\frac{d}{l}\sum_{i=1}^mD_i$$
is big, hence
$$\epsilon L-\sum_{(i,j)\in\Lambda(x)}\frac{r_i}{l}D_{i,j}$$
is big for any closed point $x\in X$. Thus we may pick
$$D_x:=\sum_{(i,j)\in\Lambda(x)}\frac{r_i}{l}D_{i,j}\leq H_x\sim_{\mathbb R}\epsilon L,$$
and we have
\begin{align*}
  \alpha(X\ni x,B+D';L)&=\alpha(X\ni x,B+D_x;L)\geq\alpha(X\ni x,B+H_x;L)\\
  &\geq \alpha(X,B+H_x;L)\geq \alpha(X,B;L)-\epsilon.
\end{align*}
Thus
$$\alpha(X,B+D';L)=\inf_{x\in X}\alpha(X\ni x,B+D';L)\geq\alpha(X,B;L)-\epsilon$$
and the lemma follows.
\end{proof}

\begin{thm}\label{thm: jiang principle general}
Let $d$ be a positive integer, $\Ii\subset [0,+\infty)$ a DCC set, and $\epsilon$ a positive real number. Then the set of projective varieties $X$ such that
\begin{enumerate}
    \item $\Bb:=(X,\Ff,B)(t)$ is a projective algebraically integrable normalized foliated structure,
    \item $\dim X=d$,
    \item $-K_{\Bb}$ is nef and $\alpha(\Bb)^d\cdot\vol(-K_{\Bb})\geq\epsilon$, and
    \item $t\leq 1-\epsilon$
\end{enumerate}
forms a bounded family. In addition, the set of $(X,\Ff)$ such that (1-4) holds and 
\begin{enumerate}
    \item[(5)] $t\geq\epsilon$
\end{enumerate}
forms a bounded family.
\end{thm}
\begin{proof}
We may assume that $\dim X\geq 2$, otherwise $X=\mathbb P^1$ and $\Ff=0$ or $T_X$ and there is nothing left to prove.

If $\Ff=T_X$ or $t=0$, the rational coefficient case was proven in \cite[Corollary~6.14]{LLX20}. The proof of the real coefficient case is similar: By \cite[Lemma~2.13]{XZ24} and \cite[Corollary~1.4]{XZ21}, there exists a positive real number $\epsilon_1$ such that $\nvol(X\ni x,B)\geq\epsilon_1$ for any closed point $x\in X$. Here $\nvol(X\ni x,B)$ is the local volume (normalized volume) of $(X,B)$ at $x$ (cf.~\cite{Li18,LLX20}). By \cite[Theorem~6.13]{LLX20}, $(X,B)$ is $\frac{\epsilon_1}{d^d}$-lc, and we are done by \cite[Theorem~1.1]{Bir21a}. In the following, we may assume that $\Ff\neq T_X$ and $t>0$. By (3), $-K_{\Bb}$ is big and nef. Since $\alpha(\Bb)>0$, $\Bb$ is klt. Thus $t<1$.

We let $\Bb(s):=(X,\Ff,B)(s)$ for any $s\in [0,1]$.

If $\Bb$ is canonical, then $\Bb$ is $\epsilon$-lc. By \cite[Theorem~9.2]{Cas+25a} and \cite[Corollary~1.4]{Bir21a}, $X$ belongs to a bounded family, and if $t\geq\epsilon$, then by \cite[Theorem~2.4.1]{Cas+25a}, $(X,\Ff)$ belongs to a bounded family, so the theorem follows. Therefore, we may assume that $\Bb$ is not canonical. We let $E$ be a prime divisor over $X$ such that
$$a(E,\Bb)+t\epsilon_{\Ff}(E)+(1-t)=\tmld(\Bb).$$
Then $a(E,\Bb)\leq 0$. By \cite[Theorem~1.10(1)]{Cas+24}, $X$ is potentially klt. By \cite[Theorem~2.2.3]{Cas+25a}, we may let $h\colon X'\rightarrow X$ be a small $\mathbb Q$-factorialization of $X$ if $E$ is on $X$, and a birational morphism such that $\Exc(h)=E$ and $X'$ is $\mathbb Q$-factorial if $E$ is not on $X$. We let
$$\Bb'(t):=h^*\Bb(t):=(X',\Ff',B')(t)$$
and let 
$$\Bb'(s):=(X',\Ff',B')(s)$$
for any $s\in [0,1]$. Since $\Bb$ is klt, $\Bb'(t)$ is klt. Since $t<1$, by Lemma~\ref{lem: lc preserved under inequality}, $\Bb'(s)$ is klt for any $s\in [0,t]$. Since $-K_{\Bb(t)}$ is big, $-K_{\Bb'(t)}$ is big, so $-K_{\Bb'(t-\delta)}$ is big for any $0\leq\delta\ll 1$. By \cite[Theorem~9.2]{Cas+25a}, $X$ and $X'$ are of Fano type. In particular, $-K_{\Bb'(t)}$ is semi-ample. 

By the length of extremal rays \cite[Theorem~1.3(2)]{Cas+24} and the contraction theorem \cite[Theorem~2.1.3]{Cas+25a}, we may pick an integer $n\gg 0$ such that $\Bb(t-1/n)$ is klt, $-K_{\Bb'(t-1/n)}$ is big, and we may run a $\left(-K_{\Bb'(t-1/n)}\right)$-MMP which terminates with a good minimal model and the MMP is $\left(-K_{\Bb'(t)}\right)$-trivial. We let $\phi\colon X'\dashrightarrow X''$ be this MMP and let $\Bb''(s):=\phi_*\Bb'(s):=(X'',\Ff'',B'')(s)$ for any $s\in [0,1]$. Note that $\phi$ is also a sequence of steps of a $(-K_{\Bb'(0)})$-MMP. Since $K_{\Bb'(0)}=K_{X'}+B'$, $\phi$ is also a sequence of steps of a $\left(-(K_{X'}+B')\right)$-MMP. By our construction, $-K_{\Bb'(t-1/n)}$ and $-K_{\Bb'(t)}$ are big and semi-ample, and
$$K_{X''}+B''+tn\left(-K_{\Bb''(t-1/n)}\right)=K_{\Bb''(0)}+tn\left(-K_{\Bb''(t-1/n)}\right)=(tn-1)(-K_{\Bb''(t)}).$$
Since $\phi$ and $h$ are $K_{\Bb'(t)}$-trivial, $\Bb''(t)$ and $\Bb(t)$ are crepant, $\alpha(\Bb''(t))=\alpha(\Bb)$, and $\vol(-K_{\Bb''(t)})=\vol(-K_{\Bb})$. In particular, $\Bb''(t)$ is klt, so by Lemma~\ref{lem: lc preserved under inequality}, $\Bb''(s)$ is klt for any $s\in [0,t]$. In particular, $\Bb''(0)$ is klt, so $(X'',B'')$ is klt. 

Pick any $s>\alpha(X'',B'';-K_{\Bb''(t)})$. Then we may pick $L_s\in |-K_{\Bb''(t)}|_{\mathbb R}$ such that $(X'',B''+sL_s)$ is not lc. In particular,
$$\left(X'',B''+s\left(L_s^{\ninv}+\frac{1}{1-t}L_s^{\inv}\right)\right)$$
is not lc. By Lemma~\ref{lem: lc preserved under inequality}, 
$$\left(X'',\Ff'',B''+s\left(L_s^{\ninv}+\frac{1}{1-t}L_s^{\inv}\right)\right)(t)$$
is not lc, hence $\alpha(X'',B'';-K_{\Bb''(t)})<s$. Therefore, 
$$\alpha(X'',B'';-K_{\Bb''(t)})\geq\alpha(\Bb''(t))=\alpha(\Bb).$$
By Lemma~\ref{lem: alpha bound} we may choose $0\leq D\sim_{\mathbb R}-tnK_{\Bb''(t-1/n)}$ such that 
$$\alpha(X'',B''+D;-K_{\Bb''(t)})\geq\frac{1}{2}\alpha(\Bb).$$ 
Since
$$K_{X''}+B''+D\sim_{\mathbb R}(tn-1)(-K_{\Bb''(t)}),$$
we have
$$\alpha(X'',B''+D;K_{X''}+B''+D)=\frac{1}{tn-1}\alpha(X'',B''+D;-K_{\Bb''(t)})>\frac{1}{2(tn-1)}\alpha(\Bb).$$
Moreover, we have
$$\vol(K_{X''}+B''+D)=(tn-1)^d\vol(-K_{\Bb''(t)})=(tn-1)^d\vol(-K_{\Bb}).$$
Therefore,
$$\alpha(X'',B''+D;K_{X''}+B''+D)^d\cdot\vol(K_{X''}+B''+D)>\frac{1}{2^d}\alpha(\Bb)^d\cdot\vol(-K_{\Bb})\geq\frac{\epsilon}{2^d}.$$
By \cite[Lemma~2.13]{XZ24}, for any closed point $x\in X''$, $\nvol(X''\ni x,B''+D)\geq\frac{\epsilon}{2^d}$. By \cite[Theorem~6.13]{LLX20}, $(X'',B''+D)$ is $\frac{\epsilon}{(2d)^d}$-lc. Thus $(X'',B'')$ is $\frac{\epsilon}{(2d)^d}$-lc. Since $\phi$ is a sequence of steps of a $-K_{\Bb'(t-1/n)}$-MMP and is $-K_{\Bb'(t)}$-trivial, $\phi$ is a sequence of steps of a $-K_{\Bb'(0)}$-MMP, hence $-(K_{X'}+B')$-negative. Therefore,
$$\frac{\epsilon}{(2d)^d}\leq \tmld(X'',B'')\leq\tmld(X',B').$$
Since $E$ is an irreducible component of $B'$, we have
$$1-\frac{\epsilon}{(2d)^d}\geq1-\tmld(X',B')\geq 1-a(E,X',B')=\mult_EB'=1-\frac{\tmld(\Bb)}{t\epsilon_{\Ff}(E)+(1-t)}$$
which implies that
$$\tmld(\Bb)\geq\frac{\epsilon}{(2d)^d}\cdot (1-t)\geq\frac{\epsilon^2}{(2d)^d}.$$
We are done by \cite[Theorem~2.4.1]{Cas+25a}.
\end{proof}

\begin{proof}[Proof of Theorem~\ref{thm: jiang afs main}]
 It is a special case of Theorem~\ref{thm: jiang principle general}.   
\end{proof}

\begin{ex}
    Consider $X=\mathbb P^2$ and $\Ff=0$ the trivial foliation. For any $t\in (0,1)$, we have that $\Bb(t):=(X,\Ff)(t)$ is klt, $\vol(-K_{\Bb(t)})=9(1-t)^2$, and $\alpha(\Bb(t))=\frac{1}{3}$. Then $\alpha(\Bb(t))^2\cdot\vol(-K_{\Bb(t)})=(1-t)^2$. When $t$ is very close to $1$, $\alpha(\Bb(t))^2\cdot\vol(-K_{\Bb(t)})$ does not have a lower bound. However, $(X,\Ff)$ is fixed hence belongs to a bounded family.
\end{ex}

\section{Lower bound on the volume}\label{sec: volume lower bound}

In this section, we prove Theorem~\ref{thm: volume lower bound}, Corollary~\ref{cor: low bound volume}, and Theorem~\ref{thm: bir auto group afs main}. 

\begin{thm}\label{thm: low bound volume}
  Let $d$ be a positive integer and let $\Ii\subset [0,+\infty)$ be a DCC set. Then there exists a positive real number $\epsilon$ depending only on $d$ and $\Ii$ satisfying the following. Assume that $\Bb$ is a projective lc algebraically integrable normalized foliated structure of general type of dimension $d$ such that $\Bb\in\Ii$. Then 
  $$\vol(K_{\Bb})\geq (1-t)^d\epsilon$$
  where $t$ is the parameter of $\Bb$.
\end{thm}
\begin{proof}
We may assume that $t<1$, otherwise the theorem is trivial. Possibly replacing $\Ii$ by $\Ii\cup\{1\}$, we may assume that $1\in\Ii$. Write 
$$\Bb=\left(X,\Ff,B,\Mm=\sum\mu_j\Mm_j\right)(t)$$
where $\mu_j\in\Ii$ and each $\Mm_j$ is nef $\bb$-Cartier. Let $h\colon X'\rightarrow X$ be a foliated log resolution of $\Bb$ and let $\Bb':=(h^{-1}_*\Bb,\Exc(h))$. Possibly replacing $\Bb$ with $\Bb'$, we may assume that $\Bb$ is foliated log smooth. By Lemma~\ref{lem: lc preserved under inequality} and Theorem~\ref{thm: big gap general}, there exists a positive integer $p$ depending only on $d$ and $\Ii$ such that 
$$\Bb_p:=(X,\Ff,B_p,\Mm(p))(t_p)$$
is lc and $K_{\Bb_p}$ is big, where
$$B_p:=\frac{1}{p}\left\lfloor pB\right\rfloor,\quad  \Mm(p):=\frac{1}{p}\sum\left\lfloor p\mu_j\right\rfloor\Mm_j,\quad \text{and}\quad t_p:=\frac{1}{p}\left\lfloor tp\right\rfloor.$$
By Lemma~\ref{lem: volume monotonicity}, possibly replacing $\Ii$ by $\frac{1}{p}\mathbb N$ and replacing $\Bb$ with $\Bb_p$, we may assume that $\Ii=\frac{1}{p}\mathbb N$. Let $m$ be a positive integer depending only on $d$ and $\Ii$ such that $p\mid m$ and $p\mid \frac{ms}{1-s}$ for any $s\in\Ii\cap [0,1)$. Let $\Bb(s):=(X,\Ff,B,\Mm)(s)$ for any $s\in [0,1]$. By Theorem~\ref{thm: afs eb}, there exists a positive integer $m_1$ depending only on $d$ and $\Ii$ such that $m\mid m_1$ and
$$\left|m_1K_{\Bb(0)}+\frac{m_1t}{1-t}K_{\Bb(1)}\right|$$
defines a birational map. In particular, 
$$\left(\frac{m_1}{1-t}\right)^d\vol(K_{\Bb})=\vol\left(m_1K_{\Bb(0)}+\frac{m_1t}{1-t}K_{\Bb(1)}\right)\geq 1.$$
Therefore, $\vol(K_{\Bb})\geq\frac{(1-t)^d}{m_1^d}$ and we may take $\epsilon:=\frac{1}{m_1^d}$.
\end{proof}

\begin{proof}[Proof of Theorem~\ref{thm: volume lower bound}]
 It is the special case of Theorem~\ref{thm: low bound volume} for $\Bb=(X,\Ff,0,\bm{0})(t)$.
\end{proof}

\begin{proof}[Proof of Corollary~\ref{cor: low bound volume}]
By our assumptions, $(X,\Ff)(t)$ is lc for any $t\in [0,1]$. By Theorem~\ref{thm: big gap general}, there exists a positive real number $t$ depending only on $d$ such that $\Bb:=(X,\Ff)(t)$ is lc and $K_{\Bb}$ is big. By Theorem~\ref{thm: volume lower bound}, there exists a positive real number $\epsilon_0$ depending only on $d$ such that
$$\vol(tK_{\Ff}+(1-t)K_X)\geq (1-t)^d\epsilon_0.$$
Since $-K_X$ is pseudo-effective,
$$t^d\vol(K_{\Ff})=\vol(tK_{\Ff})\geq\vol(tK_{\Ff}+(1-t)K_X)\geq(1-t)^d\epsilon_0.$$
Therefore, we may take $\epsilon:=\left(\frac{1-t}{t}\right)^d\epsilon_0$.
\end{proof}

\begin{proof}[Proof of Theorem~\ref{thm: bir auto group afs main}]
By Theorem~\ref{thm: big gap general} there exists a real number $t<1$ depending only on $d$ such that $(X,\Ff)(t)$ is klt and $tK_{\Ff}+(1-t)K_X$ is big. By \cite[Theorem~2.1.1]{Cas+25a}, $(X,\Ff)(t)$ has a good minimal model $(X',\Ff')(t)$. By \cite[Theorem~1]{PS02}, $\#\Bir(X,\Ff)<+\infty$. For any $g\in\Bir(X,\Ff)$, we have $g\in\Aut(X',\Ff')$. Let $(X'',\Ff''):=(X'/\Aut(X',\Ff'),\Ff'/\Aut(X',\Ff'))$ and let $q\colon X'\rightarrow X''$ be the quotient map. By \cite[Lemma~3.4]{Dru21},
$$tK_{\Ff'}+(1-t)K_{X'}=q^*(tK_{\Ff''}+(1-t)K_{X''}).$$
Let $\lambda:=\frac{t}{1-t}$. Since $(X'',\Ff'',t)$ is of general type, by Theorem~\ref{thm: low bound volume}, we have
$$\vol\left(K_{X''}+\lambda K_{\Ff''}\right)\geq\epsilon_0$$
for some positive real number $\epsilon_0$ depending only on $d$. Let $c:=\frac{1}{\epsilon_0}$, then
$$\#\Bir(X,\Ff)=\#\Aut(X',\Ff')=\deg q=\frac{\vol(tK_{\Ff'}+(1-t)K_{X'})}{\vol(tK_{\Ff''}+(1-t)K_{X''})}\leq c\vol\left(K_X+\lambda K_{\Ff}\right)$$
and the theorem follows.
\end{proof}

\section{Further discussions}\label{sec: discussion}

The results proved above provide a uniform package of ACC and boundedness statements for algebraically integrable adjoint foliated structures. 
In this final section, we briefly discuss possible extensions of the theory and formulate a collection of open problems motivated by the new invariants and thresholds introduced in the paper.

\subsection{Variation of our main theorems}\label{subsec: variation}

\begin{rem}[Non-algebraically integrable case]
As in \cite{Cas+25a}, we expect that our main theorems remain valid for foliations that are not necessarily algebraically integrable. 

There are, however, several obstacles to such an extension. The first major obstacle, and the best-known one, is the existence of foliated log resolutions, which is known in dimension $\leq 3$ \cite{Sei68,Can04,MP13} but remains widely open in dimension $\geq 4$, despite recent progress (cf.~\cite{AdSTW25}). For algebraically integrable foliations, the foliated log resolution we use is essentially the toroidal reduction \cite[Proposition~4.4]{AK00} arising as a step in the weak semi-stable reduction \cite[Theorem~2.1]{AK00}. Nevertheless, we expect that a suitable form of foliated log resolution should exist for arbitrary foliations.

A second obstacle, which appears to be more serious, is the failure of the adjunction formula. Although there is a satisfactory adjunction formula for foliations along non-invariant divisors \cite[Theorem~1.1]{CS25a}, such a formula fails in general for adjunction along invariant divisors (cf.~\cite[Example~3.20]{CS25a}). For foliated triples, this failure already causes substantial difficulties and requires delicate arguments to obtain a workable adjunction formula, even in dimension $3$ (cf.~\cite[Sections~3 and~4]{Spi20}, \cite[Section~3]{CS21}, \cite[Section~3]{SS22}, \cite[Section~4]{CS25b}). In fact, the failure of adjunction is the main obstacle in the proofs of the ACC and the global ACC for foliated triples in dimension $\leq 3$ (cf.~\cite{Che22,LLM23,LMX24a,LMX24b}). For adjoint foliated structures, the failure of adjunction is likely to create further difficulties, since the geometry of singularities is harder to control, especially when the parameter $t$ is very close to $0$.
\end{rem}

\begin{rem}[Foliated flags]
    In another talk of M\textsuperscript{c}Kernan \cite[50:55]{McK23}, the following question is raised: can one obtain an ACC statement for interpolated thresholds between two foliations, instead of between one foliation and $X$? More precisely, consider
    $$\lct(\Ff;\mathcal{G}):=\sup\{t\in [0,1]\mid (\mathcal{G},\Ff,t)\text{ is lc}\},$$
    where $(\mathcal{G},\Ff,t)$ is associated with the canonical class $tK_{\mathcal{G}}+(1-t)K_{\Ff}$, and ask whether $\lct(\Ff;\mathcal{G})$ belongs to an ACC set whenever $\mathcal{G}\subset\Ff$. Note that when $\Ff=T_X$ this is exactly Conjecture~\ref{conj: acc interpolated}. We expect this question to be approachable at least when $\Ff$ and $\mathcal{G}$ are algebraically integrable. To address it, however, one first needs an existence result for the minimal model program for $(\mathcal{G},\Ff,t)$, which is still open; see \cite[Conjecture~A.10]{Cas+24} for a precise formulation.

    We sketch a possible framework. Consider a structure
    $$\Aa:=(\Ff_1\subset\Ff_2\subset\cdots\subset\Ff_r,(a_1,\dots,a_r)),$$
    where each $\Ff_i$ is an (algebraically integrable) foliation, $a_i\in (0,1)$, and $\sum_{i=1}^r a_i=1$, and associate to it the canonical class
    $$K_{\Aa}:=\sum_{i=1}^r a_iK_{\Ff_i}.$$
    We call $\Aa$ a \emph{foliated flag} (or, by abuse of terminology, still an adjoint foliated structure). When the ambient variety $X$ is $\mathbb Q$-factorial klt and all $\Ff_i$ are algebraically integrable, the existence of a minimal model for such a structure is known in the following cases:
    \begin{itemize}
        \item $r=1$ and $\Ff_r=T_X$, by \cite{BCHM10}, since a foliated flag is then a klt pair.
        \item $r=1$ and $\Ff_r\neq T_X$, by \cite{ACSS21,CHLX23,CS25c,LMX25}, where a foliated flag is an algebraically integrable foliation.
        \item $r=2$ and $\Ff_r=T_X$, by \cite{Cas+25a}, where a foliated flag is an algebraically integrable adjoint foliated structure.
    \end{itemize}
    It is then natural to expect the following pattern for the existence of the minimal model program for algebraically integrable foliated flags:
    \begin{itemize}
        \item The case $r=r_0$ and $\Ff_r=T_X$ implies the case $r=r_0$ and $\Ff_r\neq T_X$.
        \item The case $r=r_0$ and $\Ff_r\neq T_X$ implies the case $r=r_0+1$ and $\Ff_r=T_X$.
    \end{itemize}
    Since we always have $r\leq \dim X+1$, this would yield the existence of the minimal model program for any such $\Aa$ by induction on $r$.
\end{rem}

\subsection{Open questions}

Many bounds on numerical invariants of algebraically integrable foliations have been obtained in this paper. It is then natural to ask explicit birational geometry questions about these invariants, i.e.\ to study their geography. We refer the reader to \cite{SS23,LWX25,Vas25,Xu25b} for related studies on surfaces. For this purpose, we adopt the following notation. For any positive integer $d$, let $\mathcal{P}_{1,d}$ be the set of foliated pairs $(X,\Ff)$ of dimension $d$ such that $X$ is lc. Set
$$\mathcal{P}_{2,d}:=\{(X,\Ff)\in \mathcal{P}_{1,d}\mid \Ff\text{ is algebraically integrable}\},$$
$$\mathcal{P}_{3,d}:=\{(X,\Ff)\in \mathcal{P}_{1,d}\mid X\text{ is klt}\},$$
and $\mathcal{P}_{4,d}:=\mathcal{P}_{2,d}\cap\mathcal{P}_{3,d}$.

\begin{ques}[Explicit geometry questions for adjoint foliated structures]
Fix $n\in\{1,2,3,4\}$ and $d\in\{2,3\}$. 
\begin{enumerate}
    \item (Gap of lc threshold) Find the smallest positive real number $\tau=\tau_{\ilct}(n,d)$ such that for any $(X,\Ff)\in\mathcal{P}_{n,d}$, if $\Ff$ is not lc, then $(X,\Ff,t)$ is not lc for any $t>1-\tau$. Furthermore:
    \begin{enumerate}
        \item Classify all (singularities of) $(X,\Ff)\in\mathcal{P}_{n,d}$ such that $(X,\Ff,1-\tau)$ is lc but $\Ff$ is not lc.
    \end{enumerate}

    \item (Gap of global lc threshold) Find the smallest positive real number $\tau=\tau_{\iglct}(n,d)$ such that for any projective $(X,\Ff)\in\mathcal{P}_{n,d}$, if $(X,\Ff,t)$ is lc, $t>1-\tau$, and $K_{(X,\Ff,t)}\equiv 0$, then $K_{\Ff}\equiv 0$. Furthermore:
    \begin{enumerate}
      \item Classify all projective $(X,\Ff)\in\mathcal{P}_{n,d}$ such that $(X,\Ff,1-\tau)$ is lc, $K_{(X,\Ff,1-\tau)}\equiv 0$, but $K_{\Ff}\not\equiv 0$.  
    \end{enumerate}

    \item (Gap of pseudo-effective threshold) Find the smallest positive real number $\tau=\tau_{\pet}(n,d)$ such that for any projective $(X,\Ff)\in\mathcal{P}_{n,d}$, if $K_{(X,\Ff,1-\tau)}$ is not pseudo-effective, then $K_{(X,\Ff,t)}$ is not pseudo-effective for any $t<1$. Furthermore:
    \begin{enumerate}
      \item Do we have $\tau_{\pet}(n,d)=\tau_{\iglct}(n,d)$?
      \item Classify all projective $(X,\Ff)\in\mathcal{P}_{n,d}$ such that $(X,\Ff,t)$ is lc, $K_{(X,\Ff,1-\tau)}$ is pseudo-effective, but $K_{(X,\Ff,t)}$ is not pseudo-effective for any $t<1-\tau$.
    \end{enumerate}

    \item (Gap of $\mathbb R$-complementary threshold) Find the smallest positive real number $\tau=\tau_{\Rct}(n,d)$ such that for any projective $(X,\Ff)\in\mathcal{P}_{n,d}$, if $(X,\Ff,t)$ is $\mathbb R$-complementary for some $t>1-\tau$, then $(X,\Ff,t)$ is $\mathbb R$-complementary for any $t\in (1-\tau,1)$. Furthermore:
    \begin{enumerate}
        \item Classify all projective $(X,\Ff)\in\mathcal{P}_{n,d}$ such that $(X,\Ff,1-\tau)$ is $\mathbb R$-complementary but $(X,\Ff,t)$ is not $\mathbb R$-complementary for any $t>1-\tau$. 
        \item Do we have $\tau_{\iglct}(n,d)=\tau_{\Rct}(n,d)$?
        \item Ask the same questions with the additional assumption that $X$ is of Fano type.
    \end{enumerate}

    \item (Explicit effective birationality) Find a set $\mathcal{S}_{\eb}\subset\mathbb N^+\times\mathbb N^+$ such that, for any projective $(X,\Ff)\in\mathcal{P}_{n,d}$, the linear system
    $$|mK_X+nK_{\Ff}|$$
    defines a birational map for every $(m,n)\in\mathcal{S}_{\eb}$.

    \item (Exceptional Fanos) Find the smallest positive real number $\tau=\tau_{\exc}(n,d)$ such that for any $(X,\Ff)\in\mathcal{P}_{n,d}$, if $X$ is of Fano type and $(X,\Ff,t)$ is exceptional for some $t\in (1-\tau,1)$, then $(X,\Ff,1-\tau)$ is exceptional. Furthermore:
    \begin{enumerate}
      \item Classify all exceptional $(X,\Ff)\in\mathcal{P}_{n,d}$ such that $X$ is of Fano type, $(X,\Ff,1-\tau)$ is exceptional, but $(X,\Ff,t)$ is not exceptional for any $t\in (1-\tau,1)$.
      \item Construct the moduli for $(X,\Ff)$ as in (a).
    \end{enumerate}
\end{enumerate}
One can also ask the same questions for structures with a polarization and with a boundary with standard coefficients. 
\end{ques}

We remark that \cite[Proposition~1.4, Example~2.16]{Xu25b} indicates that
$$\tau_{\ilct}(n,2)=\frac{1}{6},\quad \forall\ n\in\{1,2,3,4\},$$
while \cite[Theorem~5.11]{Vas25} shows that $\tau_{\pet}(1,d)\geq \tau_0\approx (10\uparrow\uparrow5)^{-1}$. 

We also remark that a more explicit study of the geometry of adjoint foliated structures on surfaces is expected to be closely related to the Poincar\'e problem for surface foliations, namely the problem of providing algebraicity criteria. L\"u and Tan \cite{LT24} recently made significant progress on the Poincar\'e problem by introducing new birational invariants of foliated surfaces, inspired by the formulas for modular invariants of families of algebraic curves (cf.~\cite{Tan96,Tan10}), and by establishing several algebraicity criteria for surface foliations. It is therefore natural to ask whether one can obtain a more explicit classification of surface foliations by combining the birational invariants arising from the Poincar\'e-problem viewpoint with those coming from the adjoint foliated surface framework.

Finally, we formulate the following questions that are related to complements. We expect these questions to have positive answers at least for algebraically integrable foliations, and we expect them to play an important role in the study of boundedness of complements for foliations. The first question concerns the existence of uniform rational polytopes (cf.~\cite{HLS24}), which allow one to perturb real coefficients to rational coefficients in a uniform way.

\begin{ques}[Uniform rational polytopes]
    Let $d$ be a positive integer, and let $m,n$ be two non-negative integers. Let
    $$\bm{v}_0:=(b_{1,0},\dots,b_{m,0},(1-t_0)\mu_{1,0},\dots,(1-t_0)\mu_{n,0},t_0)\in\mathbb R_{\geq 0}^{m+n}\times [0,1].$$
    Let $V$ be the rational envelope of $\bm{v}_0$ in $\mathbb R_{\geq 0}^{m+n+1}$. Does there exist an open subset $V_0\subset V$ containing $\bm{v}_0$ with the following properties? Assume that $$\Bb/U=\left(X,\Ff,B=\sum_{i=1}^mb_{i,0}B_i,\sum_{j=1}^n\mu_{j,0}\Mm_j\right)(t_0)/U$$
    is an (algebraically integrable) normalized foliated structure of dimension $d$, where $B_i\geq 0$ are Weil divisors and each $\Mm_j$ is nef$/U$ and $\bb$-Cartier. For any 
    $$\bm{v}=(b_1,\dots,b_m,(1-t)\mu_1,\dots,(1-t)\mu_n,t)\in\mathbb R^{m+n}\times [0,1],$$
    set
$$\Bb(\bm{v}):=\left(X,\Ff,B=\sum_{i=1}^mb_iB_i,\sum_{j=1}^n\mu_j\Mm_j\right)(t).$$
    Then:
    \begin{enumerate}
        \item[(1)] (Uniform lc rational polytope) If $\Bb$ is lc, then $\Bb(\bm{v})$ is lc for any $\bm{v}\in V_0$.
        \item[(2)] (Uniform $\mathbb R$-complementary rational polytope) If $X$ is of Fano type$/U$ and $\Bb/U$ is $\mathbb R$-complementary, then $\Bb(\bm{v})/U$ is $\mathbb R$-complementary for any $\bm{v}\in V_0$.
        \item[(3)] (Uniform pseudo-effective rational polytope) If $\Bb$ is lc and $K_{\Bb}$ is pseudo-effective$/U$, then $\Bb(\bm{v})$ is lc and $K_{\Bb(\bm{v})}$ is pseudo-effective$/U$ for any $\bm{v}\in V_0$.
    \end{enumerate}
\end{ques}

The second question concerns the behavior of $\mathbb R$-complements when $t$ approaches $1$. 

\begin{ques}\label{ques: rcomplement when t=1}
Assume that $\Bb/U:=(X,\Ff,B,\Mm)(t)/U$ is a normalized foliated structure and let $\Bb(s):=(X,\Ff,B,\Mm)(s)$ for any $s\in [0,1]$. Assume that $\Bb(s)/U$ is $\mathbb R$-complementary for any $0<s\ll 1$. Is $\Bb(1)/U$ also $\mathbb R$-complementary? If necessary, assume that $\Ff$ is algebraically integrable and/or that $X$ is of Fano type over $U$.
\end{ques}

Indeed, even the following special case of Question~\ref{ques: rcomplement when t=1} is interesting.

\begin{ques}\label{ques: rcomplement when t=1 fano}
    Let $\Aa/U=(X,\Ff,B,\Mm)/U$ be an algebraically integrable generalized foliated quadruple such that $X$ is of Fano type over $U$ and $-K_{\Aa}$ is nef$/U$. Does $\Aa/U$ have an $\mathbb R$-complement?
\end{ques}
When $B$ is a $\mathbb Q$-divisor, $\Mm=\bm{0}$, $U$ is a closed point, and $-K_{\Aa}$ is ample, Question~\ref{ques: rcomplement when t=1 fano} has a positive answer \cite[Theorem~1.1]{CJV24} even without assuming that $X$ is of Fano type. Note that in the statement of the ACC for $\mathbb R$-complementary thresholds (Theorem~\ref{thm: rct general}), we may assume that the parameter $t$ is $<1$, so this question does not arise there. However, the question seems crucial in the study of the boundedness of complements. 

Indeed, during the 2026 Workshop on Algebraic Geometry and Complex Geometry in Osaka, Chenyang Xu suggested to the second author that, in K-stability and K-moduli theory, the boundedness of complements plays an even more crucial role than BAB. It is reasonable to expect a similar phenomenon for the K-moduli theory of algebraically integrable foliations. We conclude the paper by proposing the following conjecture on (a version of) the boundedness of complements, which we hope to address in future work.

\begin{conj}[Boundedness of complements]\label{conj: boundedness of complements}
Let $d$ be a positive integer and $\Ii\subset [0,1]$ a DCC set of rational numbers such that $\overline{\Ii}\subset\mathbb Q$. Then there exists a positive integer $N$ depending only on $d$ and $\Ii$ satisfying the following. 

Assume that $\Bb:=(X,\Ff,B)(t)$ is a projective $\mathbb R$-complementary normalized foliated structure such that $X$ is of Fano type and $t\in\Ii$. Then there exists a monotonic $N$-complement $\Bb^+$ of $\Bb$ in the following sense.
\begin{enumerate}
    \item $\Bb^+=(X,\Ff,B^+)(t^+)\geq\Bb$, i.e.\ $B^+\geq B$ and $t^+\geq t$.
    \item $Nt^+\in\mathbb N$, $NB^+$ is an integral divisor, and $NK_{\Bb^+}\sim 0$.
\end{enumerate}
\end{conj}
Note that if $t$ belongs to a finite set, then we may further take $t^+=t$ after replacing $N$ by a bounded multiple. Indeed, $(X,\Ff,B)(0)$ is $\mathbb R$-complementary whenever $t<1$ (Proposition~\ref{prop: r complementary preserved under inequality}), and so a fixed interpolation of a bounded complement of $(X,B)$ (which exists by \cite[Theorem~1.7]{Bir19}) and $\Bb^+$ gives the desired bounded complement for $\Bb$ with $t=t^+$ fixed.

\begingroup
\sloppy
\emergencystretch=3em

\endgroup

\begin{thebibliography}{99}
\bibitem[AdSTW25]{AdSTW25} D. Abramovich, A. B. da Silva, M. Temkin, and J. W\l{}odarczyk, \textit{Principalization on logarithmically foliated orbifolds}, arXiv:2503.00926.

\bibitem[AK00]{AK00} D. Abramovich and K. Karu, \textit{Weak semistable reduction in characteristic $0$}, Invent. Math. \textbf{139} (2000), no. 2, 241--273.

\bibitem[ACSS21]{ACSS21} F. Ambro, P. Cascini, V. V. Shokurov, and C. Spicer, \textit{Positivity of the moduli part}, arXiv:2111.00423.

\bibitem[AB19]{AB19} K. Ascher and D. Bejleri, \textit{Moduli of fibered surface pairs from twisted stable maps}, Math. Ann. \textbf{374} (2019), no. 1--2, 1007--1032.

\bibitem[AB21]{AB21} K. Ascher and D. Bejleri, \textit{Moduli of weighted stable elliptic surfaces and invariance of log plurigenera}, Proc. Lond. Math. Soc. (3) \textbf{122} (2021), no. 5, 617--677. With an appendix by Giovanni Inchiostro.

\bibitem[AB22]{AB22} K. Ascher and D. Bejleri, \textit{Compact moduli of degree one del Pezzo surfaces}, Math. Ann. \textbf{384} (2022), no. 1--2, 881--911.

\bibitem[AB23]{AB23} K. Ascher and D. Bejleri, \textit{Compact moduli of elliptic K3 surfaces}, Geom. Topol. \textbf{27} (2023), no. 5, 1891--1946.

\bibitem[Asc$^+$23]{Asc+23} K. Ascher, D. Bejleri, H. Blum, K. DeVleming, G. Inchiostro, Y. Liu, and X. Wang, \textit{Moduli of boundary polarized Calabi-Yau pairs}, arXiv:2307.06522.

\bibitem[BFMT25]{BFMT25} B. Bakker, S. Filipazzi, M. Mauri, and J. Tsimerman, \textit{Baily--Borel compactifications of period images and the b-semiampleness conjecture}, arXiv:2508.19215.

\bibitem[Bej$^+$24]{Bej+24} D. Bejleri, J. Foster, A. F. Herrero, G. Inchiostro, S. Makarova, and J. Zhao, \textit{Moduli of elliptic surfaces of Kodaira dimension one fibered over rational curves}, arXiv:2407.05539.

\bibitem[BPRT22]{BPRT22} F. L. Bianco, J. V. Pereira, E. Rousseau, and F. Touzet, \textit{Rational endomorphisms of codimension one holomorphic foliations}, J. Reine Angew. Math. \textbf{789} (2022), 43--101.

\bibitem[Bir12]{Bir12} C. Birkar, \textit{Existence of log canonical flips and a special LMMP}, Publ. Math. IH\'ES \textbf{115} (2012), 325--368.

\bibitem[Bir19]{Bir19} C. Birkar, \textit{Anti-pluricanonical systems on Fano varieties}, Ann. of Math. (2) \textbf{190} (2019), no. 2, 345--463.

\bibitem[Bir21a]{Bir21a} C. Birkar, \textit{Singularities of linear systems and boundedness of Fano varieties}, Ann. of Math. (2) \textbf{193} (2021), no. 2, 347--405.

\bibitem[Bir21b]{Bir21b} C. Birkar, \textit{Boundedness and volume of generalised pairs}, arXiv:2103.14935.

\bibitem[Bir22]{Bir22} C. Birkar, \textit{Moduli of algebraic varieties}, arXiv:2211.11237.

\bibitem[Bir23a]{Bir23a} C. Birkar, \textit{Geometry of polarized varieties}, Publ. Math. IH\'ES \textbf{137} (2023), 47--105.

\bibitem[Bir23b]{Bir23b} C. Birkar, \textit{Singularities on Fano fibrations and beyond}, arXiv:2305.18770.

\bibitem[Bir24]{Bir24} C. Birkar, \textit{Boundedness of Fano type fibrations}, Ann. Sci. \'Ec. Norm. Sup\'er. (4) \textbf{57} (2024), no. 3, 787--840.

\bibitem[BCHM10]{BCHM10} C. Birkar, P. Cascini, C. D. Hacon, and J. M\textsuperscript{c}Kernan, \textit{Existence of minimal models for varieties of log general type}, J. Amer. Math. Soc. \textbf{23} (2010), no. 2, 405--468.

\bibitem[BDCS24]{BDCS24} C. Birkar, G. Di Cerbo, and R. Svaldi, \textit{Boundedness of elliptic Calabi-Yau varieties with a rational section}, J. Differential Geom. \textbf{128} (2024), no. 2, 463--519.

\bibitem[BH22]{BH22} C. Birkar and C. D. Hacon, \textit{Variations of generalised pairs}, arXiv:2204.10456.

\bibitem[BQ24]{BQ24} C. Birkar and S. Qu, \textit{Irrationality of degenerations of Fano varieties}, arXiv:2401.07233.

\bibitem[BZ16]{BZ16} C. Birkar and D.-Q. Zhang, \textit{Effectivity of Iitaka fibrations and pluricanonical systems of polarized pairs}, Publ. Math. IH\'ES \textbf{123} (2016), 283--331.

\bibitem[BJ20]{BJ20} H. Blum and M. Jonsson, \textit{Thresholds, valuations, and K-stability}, Adv. Math. \textbf{365} (2020), 107062.

\bibitem[BL24]{BL24} H. Blum and Y. Liu, \textit{Good moduli spaces for boundary polarized Calabi-Yau surface pairs}, arXiv:2407.00850.

\bibitem[Can04]{Can04} F. Cano, \textit{Reduction of the singularities of codimension one singular foliations in dimension three}, Ann. of Math. (2) \textbf{160} (2004), no. 3, 907--1011.

\bibitem[Cas21]{Cas21} P. Cascini, \textit{New directions in the Minimal Model Program}, Boll. Unione Mat. Ital. \textbf{14} (2021), no. 1, 179--190.

\bibitem[Cas$^+$24]{Cas+24} P. Cascini, J. Han, J. Liu, F. Meng, C. Spicer, R. Svaldi, and L. Xie, \textit{Minimal model program for algebraically integrable adjoint foliated structures}, arXiv:2408.14258.

\bibitem[Cas$^+$25a]{Cas+25a} P. Cascini, J. Han, J. Liu, F. Meng, C. Spicer, R. Svaldi, and L. Xie, \textit{On finite generation and boundedness of adjoint foliated structures}, arXiv:2504.10737.

\bibitem[Cas$^+$25b]{Cas+25b} P. Cascini, J. Liu, F. Meng, R. Svaldi, and L. Xie, \textit{Variation of algebraically integrable adjoint foliated structures}, arXiv:2510.02498.

\bibitem[CS21]{CS21} P. Cascini and C. Spicer, \textit{MMP for co-rank one foliations on threefolds}, Invent. Math. \textbf{225} (2021), no. 2, 603--690.

\bibitem[CS25a]{CS25a} P. Cascini and C. Spicer, \textit{Foliation adjunction}, Math. Ann. \textbf{391} (2025), no. 4, 5695--5727.

\bibitem[CS25b]{CS25b} P. Cascini and C. Spicer, \textit{On the MMP for rank one foliations on threefolds}, Forum Math. Pi \textbf{13} (2025), e20, 1--38.

\bibitem[CS25c]{CS25c} P. Cascini and C. Spicer, \textit{MMP for algebraically integrable foliations}, in: \textit{Higher Dimensional Algebraic Geometry: A Volume in Honor of V. V. Shokurov}, C. D. Hacon and C. Xu (eds.), London Math. Soc. Lecture Note Ser., Cambridge Univ. Press, Cambridge (2025), 69--84.

\bibitem[CHLX23]{CHLX23} G. Chen, J. Han, J. Liu, and L. Xie, \textit{Minimal model program for algebraically integrable foliations and generalized pairs}, arXiv:2309.15823.

\bibitem[Che20]{Che20} W. Chen, \textit{Boundedness of weak Fano pairs with alpha-invariants and volumes bounded below}, Publ. Res. Inst. Math. Sci. \textbf{56} (2020), no. 3, 539--559.

\bibitem[Che21]{Che21} Y.-A. Chen, \textit{Boundedness of minimal partial du Val resolutions of canonical surface foliations}, Math. Ann. \textbf{381} (2021), 557--573.

\bibitem[Che22]{Che22} Y.-A. Chen, \textit{ACC for foliated log canonical thresholds}, arXiv:2202.11346.

\bibitem[CJV24]{CJV24} Y.-A. Chen, D. Jiao, and P. Voegtli, \textit{Existence of complements for foliations}, arXiv:2408.11738.

\bibitem[CLW26]{CLW26} Y. Chen, J. Liu, and Y. Wang, \textit{Sarkisov program for algebraically integrable and threefold foliations}, Int. Math. Res. Not. IMRN \textbf{2026} (2026), no. 6, 1--19.

\bibitem[CPT25]{CPT25} G. Codogni, Z. Patakfalvi, and L. Tasin, \textit{Effective positivity of Hodge bundles and applications}, arXiv:2506.10515.

\bibitem[CTV23]{CTV23} G. Codogni, L. Tasin, and F. Viviani, \textit{Slope inequalities for KSB-stable and K-stable families}, Proc. Lond. Math. Soc. (3) \textbf{126} (2023), no. 4, 1394--1465.

\bibitem[CH88]{CH88} M. Cornalba and J. Harris, \textit{Divisor classes associated to families of stable varieties, with applications to the moduli space of curves}, Ann. Sci. \'Ec. Norm. Sup\'er. (4) \textbf{21} (1988), no. 3, 455--475.

\bibitem[DLM23]{DLM23} O. Das, J. Liu, and R. Mascharak, \textit{ACC for lc thresholds for algebraically integrable foliations}, arXiv:2307.07157.

\bibitem[DC16]{DC16} G. Di Cerbo, \textit{On Fujita's log spectrum conjecture}, Math. Ann. \textbf{366} (2016), 447--457.

\bibitem[DLI24]{DLI24} A. Di Lorenzo and G. Inchiostro, \textit{Stable maps to quotient stacks with a properly stable point}, arXiv:2411.16141.

\bibitem[Dru21]{Dru21} S. Druel, \textit{Codimension 1 foliations with numerically trivial canonical class on singular spaces}, Duke Math. J. \textbf{170} (2021), no. 1, 95--203.

\bibitem[Eng$^+$25]{Eng+25} P. Engel, S. Filipazzi, F. Greer, M. Mauri, and R. Svaldi, \textit{Boundedness of some fibered K-trivial varieties}, arXiv:2507.00973.

\bibitem[Fan25]{Fan25} Z. Fan, \textit{Volumes of foliations birationally bounded by algebraically integrable families}, arXiv:2512.21932.

\bibitem[Fil24]{Fil24} S. Filipazzi, \textit{On the boundedness of $n$-folds with $\kappa(X)=n-1$}, Algebr. Geom. \textbf{11} (2024), no. 3, 318--345.

\bibitem[FHS25]{FHS25} S. Filipazzi, C. D. Hacon, and R. Svaldi, \textit{Boundedness of elliptic Calabi--Yau threefolds}, J. Eur. Math. Soc. \textbf{27} (2025), no. 9, 3583--3650.

\bibitem[FS25]{FS25} S. Filipazzi and C. Spicer, \textit{On semi-ampleness of the moduli part}, Moduli \textbf{2} (2025), Paper No. e12, 1--27.

\bibitem[FO18]{FO18} K. Fujita and Y. Odaka, \textit{On the K-stability of Fano varieties and anticanonical divisors}, Tohoku Math. J. \textbf{70} (2018), 511--521.

\bibitem[Gon11]{Gon11} Y. Gongyo, \textit{On the minimal model theory for dlt pairs of numerical Kodaira dimension zero}, Math. Res. Lett. \textbf{18} (2011), no. 5, 991--1000.

\bibitem[HL21]{HL21} C. D. Hacon and A. Langer, \textit{On birational boundedness of foliated surfaces}, J. Reine Angew. Math. \textbf{770} (2021), 205--229.

\bibitem[HL23a]{HL23a} C. D. Hacon and J. Liu, \textit{Existence of flips for generalized lc pairs}, Camb. J. Math. \textbf{11} (2023), no. 4, 795--828.

\bibitem[HM06]{HM06} C. D. Hacon and J. M\textsuperscript{c}Kernan, \textit{Boundedness of pluricanonical maps of varieties of general type}, Invent. Math. \textbf{166} (2006), 1--25.

\bibitem[HMX13]{HMX13} C. D. Hacon, J. M\textsuperscript{c}Kernan, and C. Xu, \textit{Birational automorphisms of varieties of general type}, Ann. of Math. (2) \textbf{177} (2013), no. 3, 1077--1111.

\bibitem[HMX14]{HMX14} C. D. Hacon, J. M\textsuperscript{c}Kernan, and C. Xu, \textit{ACC for log canonical thresholds}, Ann. of Math. (2) \textbf{180} (2014), no. 2, 523--571.

\bibitem[HMX18]{HMX18} C. D. Hacon, J. M\textsuperscript{c}Kernan, and C. Xu, \textit{Boundedness of moduli of varieties of general type}, J. Eur. Math. Soc. \textbf{20} (2018), no. 4, 865--901.

\bibitem[HJLL24]{HJLL24} J. Han, J. Jiao, M. Li, and J. Liu, \textit{Volume of algebraically integrable foliations and locally stable families}, arXiv:2406.16604.

\bibitem[HL23b]{HL23b} J. Han and J. Liu, \textit{On effective birationality for sub-pairs}, Int. J. Math. \textbf{34} (2023), no. 6, 2350029.

\bibitem[HLS24]{HLS24} J. Han, J. Liu, and V. V. Shokurov, \textit{ACC for minimal log discrepancies of exceptional singularities}, Peking Math. J. (2024), 1--33.

\bibitem[Har80]{Har80} R. Hartshorne, \textit{Stable reflexive sheaves}, Math. Ann. \textbf{254} (1980), no. 2, 121--176.

\bibitem[HH25a]{HH25a} K. Hashizume and M. Hattori, \textit{On boundedness and moduli spaces of K-stable Calabi-Yau fibrations over curves}, Geom. Topol. \textbf{29} (2025), no. 3, 1619--1691.

\bibitem[HH25b]{HH25b} K. Hashizume and M. Hattori, \textit{K-moduli of quasimaps and on quasi-projectivity of moduli of K-stable Calabi-Yau fibrations over curves}, arXiv:2504.21519.

\bibitem[Hat26]{Hat26} M. Hattori, \textit{On positivity of CM line bundles on the moduli space of klt good minimal models with $\kappa=1$}, arXiv:2601.18208.

\bibitem[Hir26]{Hir26} Y. Hiroi, \textit{Quotients by $(p-1)/p$-klt foliations on surfaces}, arXiv:2602.20703.

\bibitem[Inc20]{Inc20} G. Inchiostro, \textit{Moduli of Weierstrass fibrations with marked section}, Adv. Math. \textbf{375} (2020), 107374.

\bibitem[ISZ25]{ISZ25} G. Inchiostro, R. Svaldi, and J. Zhao, \textit{Moduli of surfaces fibered in log Calabi-Yau pairs}, arXiv:2509.14145.

\bibitem[IZ25]{IZ25} G. Inchiostro and J. Zhao, \textit{Moduli of surfaces fibered in log Calabi-Yau pairs II: elliptic surfaces}, arXiv:2512.24017.

\bibitem[Jia20]{Jia20} C. Jiang, \textit{Boundedness of $\mathbb Q$-Fano varieties with degrees and alpha-invariants bounded from below}, Ann. Sci. \'Ec. Norm. Sup\'er. (4) \textbf{53} (2020), 1235--1248.

\bibitem[Jia23]{Jia23} X. Jiang, \textit{Boundedness of klt good minimal models}, arXiv:2312.03313.

\bibitem[JJZ25]{JJZ25} X. Jiang, J. Jiao, and M. Zhu, \textit{Boundedness of polarized log Calabi-Yau fibrations with bounded bases}, arXiv:2504.05243.

\bibitem[Jia25a]{Jia25a} J. Jiao, \textit{Boundedness of polarized Calabi-Yau fibrations}, J. Differential Geom. \textbf{130} (2025), no. 3, 635--675.

\bibitem[Jia25b]{Jia25b} J. Jiao, \textit{Boundedness of slc degenerations of polarized log Calabi-Yau pairs}, Forum Math. Sigma \textbf{13} (2025), e144, 1--14.

\bibitem[Kol23]{Kol23} J. Koll\'ar, \textit{Families of varieties of general type}, Cambridge Tracts in Math. \textbf{231}, Cambridge Univ. Press, Cambridge (2023). With the collaboration of Klaus Altmann and S\'andor Kov\'acs.

\bibitem[Kol$^+$92]{Kol+92} J. Koll\'{a}r et al., \textit{Flip and abundance for algebraic threefolds}, Ast\'{e}risque \textbf{211}, Soc. Math. France, Paris (1992).

\bibitem[KM98]{KM98} J. Koll\'{a}r and S. Mori, \textit{Birational geometry of algebraic varieties}, Cambridge Tracts in Math. \textbf{134}, Cambridge Univ. Press, Cambridge (1998).

\bibitem[KX20]{KX20} J. Koll\'ar and C. Xu, \textit{Moduli of polarized Calabi-Yau pairs}, Acta Math. Sin. (Engl. Ser.) \textbf{36} (2020), no. 6, 631--637.

\bibitem[KP17]{KP17} S. J. Kov\'acs and Z. Patakfalvi, \textit{Projectivity of the moduli space of stable log varieties and subadditivity of log-Kodaira dimension}, J. Amer. Math. Soc. \textbf{30} (2017), no. 4, 959--1021.

\bibitem[Li18]{Li18} C. Li, \textit{Minimizing normalized volumes of valuations}, Math. Z. \textbf{289} (2018), no. 1--2, 491--513.

\bibitem[LLX20]{LLX20} C. Li, Y. Liu, and C. Xu, \textit{A guided tour to normalized volume}, in: \textit{Geometric analysis}, Progr. Math. \textbf{333}, Birkh\"auser/Springer, Cham (2020), 167--219.

\bibitem[LLM23]{LLM23} J. Liu, Y. Luo, and F. Meng, \textit{On global ACC for foliated threefolds}, Trans. Amer. Math. Soc. \textbf{376} (2023), no. 12, 8939--8972.

\bibitem[LMX24a]{LMX24a} J. Liu, F. Meng, and L. Xie, \textit{Uniform rational polytope of foliated threefolds and the global ACC}, J. Lond. Math. Soc. (2) \textbf{109} (2024), no. 6, e12950.

\bibitem[LMX24b]{LMX24b} J. Liu, F. Meng, and L. Xie, \textit{Complements, index theorem, and minimal log discrepancies of foliated surface singularities}, Eur. J. Math. \textbf{10} (2024), no. 1, Paper No. 6, 29 pp.

\bibitem[LMX25]{LMX25} J. Liu, F. Meng, and L. Xie, \textit{Minimal model program for algebraically integrable foliations on klt varieties}, Compos. Math. \textbf{161} (2025), no. 12, 3213--3276.

\bibitem[LX25]{LX25} J. Liu and Z. Xu, \textit{Non-algebraicity of non-abundant foliations and abundance for adjoint foliated structures}, arXiv:2510.04419.

\bibitem[LW24]{LW24} J. Lu and X. Wang, \textit{On the 1-adjoint canonical divisor of a foliation}, Manuscripta Math. \textbf{175} (2024), 739--752.

\bibitem[LWX25]{LWX25} J. Lu, X. Wu, and S. Xu, \textit{Canonical models of adjoint foliated structures on surfaces}, arXiv:2501.00470.

\bibitem[L\"u25]{Lü25} X. L\"u, \textit{Unboundedness of foliated varieties}, Int. J. Math. \textbf{36} (2025), no. 6, 2550003.

\bibitem[LT22]{LT22} X. L\"u and S. Tan, \textit{On the sharp lower bounds of modular invariants and fractional Dehn twist coefficients}, J. Reine Angew. Math. \textbf{787} (2022), 163--195.

\bibitem[LT24]{LT24} X. L\"u and S. Tan, \textit{The Poincar\'e problem for a foliated surface}, arXiv:2404.16293.

\bibitem[MV26]{MV26} R. Mascharak and S. Vassiliadis, \textit{Birational rigidity and K-stability of Fano foliations}, personal communication, 2026.

\bibitem[M\textsuperscript{c}K22]{McK22} J. M\textsuperscript{c}Kernan, \textit{Log canonical thresholds for foliations}, \url{https://www.youtube.com/watch?v=ukF-oqJuypY}. JAMI 2022: Higher Dimensional Algebraic Geometry (2022 May 3--8), Johns Hopkins University, Baltimore -- An event in honor of Prof. Shokurov's 70th Birthday.

\bibitem[M\textsuperscript{c}K23]{McK23} J. M\textsuperscript{c}Kernan, \textit{The log canonical threshold revisited}, \url{https://www.youtube.com/watch?v=1oc2QYEJldc}. Conference on 100 Years of Noetherian Rings (2023 June 21st), Institute for Advanced Study, Princeton.

\bibitem[MP04]{MP04} J. M\textsuperscript{c}Kernan and Y. Prokhorov, \textit{Threefold thresholds}, Manuscripta Math. \textbf{114} (2004), 281--304.

\bibitem[McQ08]{McQ08} M. McQuillan, \textit{Canonical models of foliations}, Pure Appl. Math. Q. \textbf{4} (2008), no. 3, Special Issue: In honor of Fedor Bogomolov, Part 2, 877--1012.

\bibitem[MP13]{MP13} M. McQuillan and D. Panazzolo, \textit{Almost \'etale resolution of foliations}, J. Differential Geom. \textbf{95} (2013), no. 2, 279--319.

\bibitem[Nak04]{Nak04} N. Nakayama, \textit{Zariski-decomposition and abundance}, MSJ Memoirs \textbf{14}, Mathematical Society of Japan, Tokyo (2004).

\bibitem[OS12]{OS12} Y. Odaka and Y. Sano, \textit{Alpha invariants and K-stability of $\mathbb Q$-Fano varieties}, Adv. Math. \textbf{229} (2012), no. 5, 2818--2834.

\bibitem[Pap26]{Pap26} T. S. Papazachariou, \textit{K-stability of adjoint Fano foliated structures}, personal communication, 2026.

\bibitem[Pas24]{Pas24} A. Passantino, \textit{Numerical conditions for the boundedness of foliated surfaces}, arXiv:2412.05986.

\bibitem[PX17]{PX17} Z. Patakfalvi and C. Xu, \textit{Ampleness of the CM line bundle on the moduli space of canonically polarized varieties}, Algebr. Geom. \textbf{4} (2017), no. 1, 29--39.

\bibitem[PS02]{PS02} J. V. Pereira and P. F. S\'anchez, \textit{Transformation groups of holomorphic foliations}, Comm. Anal. Geom. \textbf{10} (2002), no. 5, 1115--1123.

\bibitem[PS19]{PS19} J. V. Pereira and R. Svaldi, \textit{Effective algebraic integration in bounded genus}, Algebr. Geom. \textbf{6} (2019), no. 4, 454--485.

\bibitem[Sei68]{Sei68} A. Seidenberg, \textit{Reduction of singularities of the differential equation $A\,dy = B\,dx$}, Amer. J. Math. \textbf{90} (1968), 248--269.

\bibitem[Sho92]{Sho92} V. V. Shokurov, \textit{Threefold log flips}, with an appendix in English by Y. Kawamata, Izv. Ross. Akad. Nauk Ser. Mat. \textbf{56} (1992), no. 1, 105--203.

\bibitem[Sho20]{Sho20} V. V. Shokurov, \textit{Existence and boundedness of $n$-complements}, arXiv:2012.06495.

\bibitem[Sho23]{Sho23} V. V. Shokurov, \textit{Log adjunction: moduli part}, Izv. Math. \textbf{87} (2023), no. 3, 616--640.

\bibitem[Spi20]{Spi20} C. Spicer, \textit{Higher dimensional foliated Mori theory}, Compos. Math. \textbf{156} (2020), no. 1, 1--38.

\bibitem[SS22]{SS22} C. Spicer and R. Svaldi, \textit{Local and global applications of the Minimal Model Program for co-rank 1 foliations on threefolds}, J. Eur. Math. Soc. \textbf{24} (2022), no. 11, 3969--4025.

\bibitem[SS23]{SS23} C. Spicer and R. Svaldi, \textit{Effective generation for foliated surfaces: Results and applications}, J. Reine Angew. Math. \textbf{795} (2023), 45--84.

\bibitem[SSV25]{SSV25} C. Spicer, R. Svaldi, and S. Velazquez, \textit{On moduli of foliated surfaces}, arXiv:2511.13491.

\bibitem[ST26]{ST26} C. Spicer and L. Tasin, \textit{Rank one foliations on toroidal varieties}, arXiv:2604.08100.

\bibitem[Tak06]{Tak06} S. Takayama, \textit{Pluricanonical systems of projective varieties of general type}, Invent. Math. \textbf{165} (2006), 551--587.

\bibitem[Tan96]{Tan96} S. Tan, \textit{On the invariants of base changes of pencils of curves. II}, Math. Z. \textbf{222} (1996), no. 4, 655--676.

\bibitem[Tan10]{Tan10} S. Tan, \textit{Chern numbers of a singular fiber, modular invariants and isotrivial families of curves}, Acta Math. Vietnam. \textbf{35} (2010), no. 1, 159--172.

\bibitem[Tia87]{Tia87} G. Tian, \textit{On K\"ahler-Einstein metrics on certain K\"ahler manifolds with $c_1(M)>0$}, Invent. Math. \textbf{89} (1987), no. 2, 225--246.

\bibitem[Tia90]{Tia90} G. Tian, \textit{On a set of polarized K\"ahler metrics on algebraic manifolds}, J. Differential Geom. \textbf{32} (1990), no. 1, 99--130.

\bibitem[Tsu07]{Tsu07} H. Tsuji, \textit{Pluricanonical systems of projective varieties of general type II}, Osaka J. Math. \textbf{44} (2007), no. 3, 723--764.

\bibitem[Vas25]{Vas25} S. Vassiliadis, \textit{Explicit bounds on foliated surfaces and the Poincar\'e problem}, arXiv:2511.08388.

\bibitem[Xia87]{Xia87} G. Xiao, \textit{Fibered algebraic surfaces with low slope}, Math. Ann. \textbf{276} (1987), 449--466.

\bibitem[Xu25a]{Xu25a} C. Xu, \textit{K-stability of Fano varieties}, New Mathematical Monographs \textbf{50}, Cambridge Univ. Press, Cambridge (2025).

\bibitem[XZ21]{XZ21} C. Xu and Z. Zhuang, \textit{Uniqueness of the minimizer of the normalized volume function}, Camb. J. Math. \textbf{9} (2021), no. 1, 149--176.

\bibitem[XZ24]{XZ24} C. Xu and Z. Zhuang, \textit{Boundedness of log Fano cone singularities and discreteness of local volumes}, arXiv:2404.17134.

\bibitem[Xu25b]{Xu25b} S. Xu, \textit{Adjoint log canonical foliated singularities on surfaces}, arXiv:2512.20744.

\bibitem[Zhu25]{Zhu25} M. Zhu, \textit{Boundedness of stable minimal models with klt singularities}, Int. Math. Res. Not. IMRN \textbf{2025} (2025), no. 2, rnae293.
\end{thebibliography}
\end{document}